\documentclass[10pt,reqno,hidelinks]{amsart}
\usepackage{amssymb}
\usepackage{amsmath, amssymb, verbatim, url, mathrsfs} 
\usepackage{stmaryrd}
\usepackage{tabularx}
\usepackage{mathtools}
\usepackage{verbatim}
\usepackage{amsthm}
\usepackage{framed}
\usepackage{cite}
\usepackage{wasysym}
\usepackage{upgreek}
\usepackage{color}
\usepackage[dvipsnames]{xcolor}
\usepackage{array}
\usepackage{tensor}
\usepackage{accents} 
\usepackage[T1]{fontenc}
\usepackage{dsfont}
\usepackage[colorlinks,linkcolor=blue,citecolor=blue]{hyperref}
\usepackage{enumerate}
\usepackage{enumitem}
\usepackage[normalem]{ulem}
\usepackage{longtable}
\usepackage{mathtools}

\numberwithin{equation}{section}

\theoremstyle{definition} 

\newtheorem{proposition}{Proposition}[section]
\newtheorem{lemma}[proposition]{Lemma}
\newtheorem{corollary}{Corollary}[section]
\newtheorem{theorem}{Theorem}[section]

\newtheorem{remark}{Remark}[section]

\newtheorem{definition}[proposition]{Definition}

\newtheorem*{theorem*}{Theorem}
\newtheorem*{mquestion*}{Main Question}
\newtheorem*{claim*}{Claim}
\newtheorem{claim}{Claim}[section]
\newtheorem*{intuition*}{Intuition}

\newcommand{\ca}{\mathsf{a}}
\newcommand{\ba}{\bar{\mathsf{a}}}
\newcommand{\cb}{\mathsf{b}}
\newcommand{\cg}{\mathsf{g}}

\newcommand{\cc}{\mathsf{c}}
\newcommand{\mft}{\mathfrak{t}}
\newcommand{\mfg}{\mathfrak{g}}
\newcommand{\cm}{\mathsf{m}}
\newcommand{\bc}{\bar{\mathsf{c}}}

\newcommand{\ufo}{\underline{f_0}}

\newcommand{\ck}{\mathsf{k}}
\newcommand{\mf}{\beta}

\newcommand{\vertiiii}[1]{{\left\vert\kern-0.25ex\left\vert\kern-0.25ex\left\vert\kern-0.25ex\left\vert #1 \right\vert\kern-0.25ex\right\vert\kern-0.25ex\right\vert\kern-0.25ex\right\vert}}
\newcommand{\vertiii}[1]{{\left\vert\kern-0.25ex\left\vert\kern-0.25ex\left\vert #1 \right\vert\kern-0.25ex\right\vert\kern-0.25ex\right\vert}}

\newcommand{\mfu}{\mathfrak{u}}

\newcommand{\mfv}{\mathfrak{v}}

\newcommand{\uf}{\underline{f}}

\newcommand{\htau}{\hat{\tau}}
\newcommand{\ttau}{\tilde{\tau}}

\newcommand{\txi}{\tilde{\zeta}}

\newcommand{\Rbb}{\mathbb{R}}
\newcommand{\Zbb}{\mathbb{Z}}
\newcommand{\Tbb}{\mathbb{T}}

\newcommand{\del}[1]{{\partial_{#1}}}

\newcommand{\AND}{{\quad\text{and}\quad}}

\newcommand{\Li}{L^\infty}

\newcommand{\supp}{\mathop{\mathrm{supp}}}

\newcommand{\p}[1]{
	\begin{pmatrix}
		#1
	\end{pmatrix}
}

\newcommand{\ta}{\mathtt{a}}

\newcommand{\tc}{\mathtt{c}}

\newcommand{\U}{\mathcal{U}}

\newcommand{\Pbp}{\mathbb{P}^{\perp}}

\newcommand{\be}{\begin{equation}}
	\newcommand{\ee}{\end{equation}}

\allowdisplaybreaks

\DeclareMathOperator{\diag}{diag}

\begin{document}

	\title[Emergence of nonlinear Jeans  instabilities]{The emergence of nonlinear Jeans-type instabilities for quasilinear wave equations}

	\author{Chao Liu}
	
	\address[Chao Liu]{Center for Mathematical Sciences and School of Mathematics and Statistics, Huazhong University of Science and Technology, Wuhan 430074, Hubei Province, China; Beijing International Center for	Mathematical Research (BICMR), Peking University, No.5 Yiheyuan Road Haidian District,
		Beijing, China.}
	\email{chao.liu.math@foxmail.com}

\begin{abstract} 
This article contributes a key ingredient to the longstanding open problem of understanding the fully nonlinear version of Jeans instability, as highlighted by A. Rendall \cite{Rendall2005}. %[Living Rev. Relativ. 8, 6 (2005)].  
We establish a family of  self-increasing  blowup solutions for the following class of quasilinear wave equations (a model of the Peebles' and Noh–Hwang's equations) that have not previously
been studied: 
\begin{equation*} 
	\partial^2_t  \varrho-  \biggl( \frac{ \mathsf{m}^2 (\partial_{t}\varrho  )^2}{(1+\varrho )^2}  + 4(\mathsf{k}-\mathsf{m}^2)(1+\varrho  )\biggr)  \Delta \varrho  =  F(t,\varrho,\partial_{\mu} \varrho)     
\end{equation*}   
where $F$ is given by
\begin{equation*}
	F(t,\varrho,\partial_{\mu} \varrho):=
	\underbrace{\frac{2}{3 }  \varrho  (1+  \varrho  ) }_{ \text{(i) self-increasing}} \underbrace{-\frac{1}{3} \partial_{t}\varrho  }_{ \text{(ii) damping}} 
	+ \underbrace{\frac{4}{3} \frac{(\partial_{t}\varrho )^2}{1+\varrho } }_{\text{(iii) Riccati}} +  \underbrace{ \biggl(\mathsf{m}^2 \frac{ (\partial_{t}\varrho )^2}{(1+\varrho )^2}  + 4(\mathsf{k}-\mathsf{m}^2) (1+\varrho  )  \biggr)  q^i \partial_{i}\varrho }_{\text{(iv) convection}} -  \mathtt{K}^{ij}   \partial_{i}\varrho\partial_{j}\varrho    .  
\end{equation*}
The result implies the solutions can attain arbitrarily large values  over time, leading to self-increasing singularities at some  future endpoints of null geodesics provided the inhomogeneous perturbations of data are sufficiently small. Moreover, the solution exhibits almost blowup behavior in the long-wavelength domain.   This phenomenon is referred to as the \textit{nonlinear Jeans-type instability} because this wave equation is closely related to the nonlinear version of the Jeans instability problem in the Euler--Poisson and Einstein--Euler systems, which characterizes the formation of nonlinear structures in the universe. The growth rate of $\varrho$ is significantly faster than that of the solutions to the classical linearized Jeans instability.

 \vspace{2mm}

{{\bf Keywords:} blowup, quasilinear wave equations, ODE blowup, Jeans instability, self-increase blowup}

\vspace{2mm}

{{\bf Mathematics Subject Classification:} Primary 35L05; Secondary 35A01, 35L02, 35L10 %, 83F05
}
\end{abstract}

	%\date{Version of \today}
	\maketitle
	%	\tableofcontents
	
	\setcounter{tocdepth}{2}
	
	\pagenumbering{roman} \pagenumbering{arabic}

\section{Introduction}
Let $\ck>0$, $\cm^2 \leq \ck$, $\beta \in (0,+\infty)$, $\beta_0 \in (0,+\infty)$,  $t_0 \in (0,+\infty)$ and $\mathbf{t}_0=\ln t_0$ be given constants, and let $q^i$ be a given vector field. Suppose $\psi \in C^1_0(\Rbb^n)$ and $\psi_0 \in C^1_0(\Rbb^n)$ are given positive-valued functions with %\footnote{This is a simplification, and it can be generalized to $\supp \psi \subset \subset B_1(0)$ and $\supp \psi_0 \subset \subset B_1(0)$. }
$\supp \psi = \supp \psi_0 =B_1(0)$  (where $B_1(0)$ denotes the unit open ball centered at the origin). Let $\mathtt{K}_{ij}$ be analytic functions in all their variables. Our goal is to find the self-increasing blowup  solutions to the following  non-covariant quasilinear wave equation with competing nonlinear terms\footnote{The damping is detrimental to the instability and does not play a dominant role in the analysis.}.
\begin{enumerate}[label={(Eq.\arabic*)}]
	\item\label{Eq1}  $
	\begin{gathered} 
		 \partial^2_t \varrho-  \cg \delta^{ij} \del{i}\del{j} \varrho  = 	\underbrace{\frac{2}{3 }  \varrho  (1+  \varrho  ) }_{ \text{(i) self-increas.}} \underbrace{-\frac{1}{3} \del{t}\varrho  }_{ \text{(ii) damping}} 
		+ \underbrace{\frac{4}{3} \frac{(\del{t}\varrho )^2}{1+\varrho } }_{\text{(iii) Riccati}} +  \underbrace{ \cg q^i \del{i}\varrho }_{\text{(iv) convec.}} -  \mathtt{K}^{ij}(t,\varrho,\del{\mu}\varrho)    \del{i}\varrho\del{j}\varrho     , \quad  \text{in}\; [\mathbf{t}_0,\mathbf{t}^\star) \times \Rbb^n,    \label{e:eq3}       \\
		  \varrho|_{t=\mathbf{t}_0}= \beta+ \psi(x^k)   \AND 	\del{t} \varrho|_{t=\mathbf{t}_0}=  \beta_0 +  \psi_0(x^k)   ,  \quad   \text{in}\; \{\mathbf{t}_0 \} \times \Rbb^n ,  %\label{e:eq3dt}
	\end{gathered}
	$\\
	where\footnote{In this article, we use the index convention given in \S\ref{s:AIN}, i.e., $\mu=0,\cdots,n$ and $i=1,\cdots,n$, $x_0=t$, and the Einstein summation convention. } 
	\begin{equation*}%\label{e:Fdef0}
	 \cg= \cg(   \varrho,\del{t}\varrho )  =  \cm^2 \frac{ (\del{t}\varrho )^2}{(1+\varrho )^2}  + 4(\ck-\cm^2) (1+\varrho  )    .
	\end{equation*}  \end{enumerate}

To study \ref{Eq1}, this article first applies a simple exponential-time coordinate transformation (where $e^t$ is treated as the new time) to convert it into \ref{Eq2}, and then analyzes \ref{Eq2}.
	\begin{enumerate}[label={(Eq.\arabic*)}] 
	\setcounter{enumi}{1}
	\item\label{Eq2}$ 
	\begin{gathered}
		\partial^2_t \varrho-  	\cg \delta^{ij} \del{i}\del{j} \varrho = 
		\frac{2}{3t^2}  \varrho  (1+  \varrho )-\frac{4}{3 t} \del{t}\varrho   + \frac{4}{3} \frac{(\del{t}\varrho )^2}{1+\varrho } + \cg q^i \del{i}\varrho   - \frac{1}{ t^2} \mathtt{K}^{ij}(t,\varrho,\del{\mu}\varrho ) \del{i}\varrho\del{j}\varrho      , \quad    \text{in}\; [t_0,t^\star) \times \Rbb^n,    %\label{e:maineq0}  
		 \\
	  \varrho|_{t=t_0}=   \beta + \psi(x^k)  \AND 	\del{t} \varrho|_{t=t_0}=     \beta_0 + \psi_0(x^k)   ,  \quad  \text{in}\; \{t_0 \} \times \Rbb^n ,  %\label{e:maineq1}
	\end{gathered}$ \\
	where  
	\begin{equation}\label{e:Fdef}
		\cg =\cg(t, \varrho,\del{t}\varrho )  :=     \cm^2\frac{ (\del{t}\varrho )^2}{(1+\varrho )^2}  + 4(\ck-\cm^2) \frac{1+\varrho  }{t^2} . 
	\end{equation} 
\end{enumerate}
For convenience, we denote $\mathtt{g}^{\alpha\beta} =-\delta^\alpha_0\delta^\beta_0+\cg\delta^{ij}\delta_i^\alpha\delta_j^\beta$ and view $\mathtt{g}_{\alpha\beta}$ as a Lorentzian metric.

\begin{claim}[see Appendix \ref{s:pfclm} for the proof]\label{t:clm2}
Equations \ref{Eq1} and \ref{Eq2} can be transformed into one another through an exponential-logarithmic time transformation.
\end{claim}  
\begin{remark}[Focusing on \ref{Eq2}]
	Since, according to Claim \ref{t:clm2},  \ref{Eq1} can be transformed into  \ref{Eq2}, the subsequent sections of this article focus on  \ref{Eq2}, with results that can be reformulated for  \ref{Eq1}. 
\end{remark}
\begin{remark}[Main goals] 
	The main goal of this article is to study the long-time behavior of solutions to equations \ref{Eq1} and \ref{Eq2}, which serve as toy models of the Euler--Poisson and Einstein--Euler systems after specific transformations and approximations (the Peebles and Noh–Hwang equations, see Remark \ref{t:origin}). We conclude that 
	\begin{enumerate}
		\item (Self-increasing blowup at some points) by making the initial inhomogeneities sufficiently small (i.e., with very long wavelengths), the solutions can attain arbitrarily large values, leading to self-increasing singularities at some  future endpoints of null geodesics; 
		\item (Almost blowups in the long-wavelength domain) the smaller the initial inhomogeneity is, the larger the long-wavelength domain can be ensured, thereby guaranteeing that the solution can grow arbitrarily large, provided the initial inhomogeneity is sufficiently small.  This is consistent with both linear physical analysis and observations. Outside the long-wavelength region, where the wavelengths become shorter, other types of singularities may also occur. 
		\item Additionally, we provide estimates of the growth rates.
	\end{enumerate}
	 For detailed results, see Theorems \ref{t:mainthm1} and \ref{t:mainthm2}.  
\end{remark}	 
\begin{remark}[Motivations]	 
	This work is motivated by the nonlinear Jeans-type instability problem for the Euler--Poisson system (and the Einstein--Euler system as well), a crucial aspect in astrophysics that characterizes the mass accretions of self-gravitating systems. Understanding the formation of stellar systems and nonlinear structures in the universe relies on the Jeans instability. However, except for our recent partial results \cite{Liu2022,Liu2022b,Liu2023a,Liu2023}, research on the Jeans instability has mainly focused on the linear regime since Jeans' work \cite{Jeans1902} in 1902 for Newtonian gravity, which was later extended to general relativity by Lifshitz \cite{Lifshitz1946} in 1946 and to the expanding background universe by Bonnor \cite{Bonnor1957} (see also \cite{Zeldovich1971,ViatcehslavMukhanov2013}). Despite advancements, the linearized Jeans instability is inadequate for dense regions, and its predicted growth rate (on the order of $\sim t^{\frac{2}{3}}$, see Peebles \cite{Peebles2020}, Weinberg \cite{Weinberg1976},  Mukhanov \cite{ViatcehslavMukhanov2013}, Zeldovich \cite{Zeldovich1971}, Bonnor \cite{Bonnor1957} and Liu \cite{Liu2022}) is too slow to explain observed large-scale inhomogeneities in the universe and galaxy formations. These inadequacies emphasize the need for a comprehensive study of the nonlinear analysis of Jeans instability. Despite Rendall's remark in \cite{Rendall2005} that there are currently no results on fully nonlinear Jeans instability, our series of works aim to address this \textit{long-standing open problem in astrophysics}. This article contributes to a key ingredient of this aim since the equations \ref{Eq1}--\ref{Eq2} play a crucial role (see the next Remark \ref{t:origin}) in understanding and modeling the nonlinear Jeans instability for the Euler–Poisson system (see, e.g.,  \cite{Peebles2020,ViatcehslavMukhanov2013,Zeldovich1971,Scherrer2001,Liu2022,Liu2023,Liu2023a,Noh2004,Hwang2013a,Hwang2005,Hwang2007,Hwang2006,Hwang1999,Noh2005,Fosalba1998} and the recent review \cite[\S $4.4$]{Fajman2024} for more details). Besides,  it has its own interest in studying the solutions of the quasilinear wave equation \ref{Eq1}.   
\end{remark}
\begin{remark}[The origin of \ref{Eq2}]	 \label{t:origin}	 
A large class of wave equations of the form of Equation \ref{Eq2} has appeared extensively in the astrophysical and cosmological literature related to the study of structure formation, nonlinear Jeans instability, and higher-order approximations of nonlinear Jeans instability. However, these works in the  cosmological literature either derive such equations under various specific conditions or solve them using numerical methods, various simplifications, approximations, or even linearizations. To date, there has been no substantial analytical study of this class of nonlinear wave equations.  For example, equations of the form of \ref{Eq2} have been derived in the renowned work \cite[\S II.9]{Peebles2020} by  Nobel laureate Peebles, where Peebles obtained a nonlinear equation for the density contrast $\delta$ (denoted as $\varrho$ in this article) from the cosmological Euler–Poisson system. Similar equations have also been derived and analyzed in our previous works \cite{Liu2022,Liu2022b,Liu2023a}. In addition, combining the covariant equations of Hawking and Ellis  \cite{Hawking1966,Ellis2009} with Bardeen's concept of gauge invariants \cite{Bardeen1980}, Noh and Hwang (see, e.g., \cite{Noh2004,Hwang2013a,Hwang2005,Hwang2007,Hwang2006,Hwang1999,Noh2005}) carried out a series of studies deriving second- and higher-order approximations of the Einstein–Euler system in various gauges. Under certain conditions and approximations, expanding the Peebles and Noh–Hwang equations yields a second-order form structurally similar to \ref{Eq2}. Thus, \textit{Equation \ref{Eq2} captures their essential terms} and serves as a \textit{simplified model and effective approximation} for this broad class of equations. Although the specific forms of the Peebles and Noh–Hwang equations vary with physical conditions, we aim to develop methods based on this model that are applicable to their general study.
  In the near future, we will return to the full Euler–Poisson equations. Beyond these, equations of this form can also be found in various other works, such as those by Fosalba, Gaztañaga, and Scherrer \cite{Fosalba1998,Scherrer2001}. In summary, such equations are widespread in the literature on cosmology and structure formation, highlighting the need for further study.  
\end{remark}
\begin{remark}[Instabilities] 
	If the initial data are zero, it is clear that \ref{Eq1} and \ref{Eq2} have the trivial solution $\varrho \equiv 0$. However, if the initial perturbations $\beta+\psi(x^i)$ and $\beta_0+\psi_0(x^i)$ are positive and the inhomogeneities $\psi(x^i)$ and $\psi_0(x^i)$ are small (long wavelength instabilities), we will conclude that the solutions can increase to arbitrarily large values, significantly deviating from the zero solution, which reflects the so-called ``instability''.
	In fact, this instability for homogeneous perturbations (i.e., $\psi=\psi_0=0$) has been previously shown in our previous work \cite{Liu2022b,Liu2023a} (see also Theorem \ref{t:mainthm0} for details) since in the homogeneous case, \ref{Eq2} simplifies to an ordinary differential equations (ODEs) as given by \eqref{e:feq0b}--\eqref{e:feq1b} below. 
\end{remark}
\begin{remark}[Classical Jeans instability for expanding universe] 
	Recall the classical Jeans instability for expanding universe: when considering very long wavelengths (much larger than the Jeans scales), the perturbations are nearly homogeneous. By linearizing the Euler–Poisson system and following standard procedures, such as those in \cite[\S$6.3$]{ViatcehslavMukhanov2013} with a Fourier transformation focusing on very low-frequency terms (frequencies near zero), we arrive at an approximated ODE:
\begin{equation}\label{e:LJeans}
	\partial^2_t f(t)+\frac{4}{3t} \del{t} f(t)-\frac{2}{3t^2} f(t)  =   0 . 
\end{equation}
	This is an Euler-type ODE with the general solution:
	$f(t)=c_1 t^{-1}+c_2 t^{\frac{2}{3}}$. 
	After linearizing \ref{Eq2} and focusing on the very low-frequency wave, \ref{Eq2} reduces to \eqref{e:LJeans} as well. In fact, the Euler–Poisson system can be simplified and modeled by \ref{Eq2} (see \cite{Liu2023} for similar ideas). Since the solution to \eqref{e:LJeans} includes a self-increasing term $t^{\frac{2}{3}}$, this behavior is termed the classical Jeans instability in astrophysics because the solution deviates from zero and grows as $t^{\frac{2}{3}}$ (see \cite{ViatcehslavMukhanov2013}). Correspondingly, we call similar nonlinear behaviors (if such behaviors can be found) in the original Euler–Poisson system nonlinear Jeans instability. As mentioned earlier, since \ref{Eq2} models the Euler–Poisson system, we refer to this self-increasing behavior as a nonlinear Jeans-type instability. We can also prove that the solution develops at least one self-increasing singularity.
\end{remark}
\begin{remark}[Without null conditions, etc.]
	It is evident that  \ref{Eq1} and \ref{Eq2} fail to satisfy Klainerman’s null condition \cite{Klainerman1984}, and John and Klainerman's conditions for  almost global existences \cite{John1984}. Although the global existence and lifespan of quasilinear wave equations have been extensively studied (see \cite{Li2017,Hoermander1997} and the references therein for a detailed review), \ref{Eq1} does not fit within any of the existing frameworks that have been explored.  
\end{remark} 
\begin{remark}[Competitive nonlinear terms]\label{t:nlterms} 
	On the right-hand side of \ref{Eq1}, as  labeled under the respective terms, there are four types of nonlinear terms that compete with each other. They are: 
	\begin{enumerate}[leftmargin=*,label={(\roman*)}]
		\item The self-increasing term (or Jeans-type term), which drives  the solution itself to grow similarly to the   classical Jeans instability \eqref{e:LJeans}; 
		\item The damping term, which suppresses the solution and counteracts the self-increasing term. However, the damping is not dominant, while the other nonlinear terms are  dominant and comparable, as discussed in Remark \ref{t:nlnrmech}; 
		\item  The Riccati term, which can cause blowups in the derivatives of $\varrho$ and lead to the formation of shock singularities  (see \cite{Speck2016b, Speck2016, Holzegel2016, Speck2016a, Miao2016, Miao2018} for details). If only term (iii) is present on the right-hand side of \ref{Eq1} and $\varrho$ is spatially homogeneous, then by multiplying the equation by $1/(1+\varrho)$, it reduces to
		$\partial_t \left( \partial_t \ln(1+  \varrho ) \right)   =    \frac{1}{3}  (\partial_t \ln(1+  \varrho ))^2$ 
		which is a Riccati equation. Consequently, we still refer to (iii) as the Riccati term, and this term results in finite-time blowups of both $\partial_t \ln(1+  \varrho )$ and $\varrho$  itself;  
		\item  The convection term describes the transport or advection of the wave due to the movement of the medium or field. It typically represents how the wave propagates through a medium with a velocity field, affecting both the speed and direction of wave propagation. The convection term introduces a directional bias in $\cg q^i$, causing the wave to propagate more strongly in a particular direction depending on the sign and magnitude of $\cg q^i$.
	\end{enumerate}
	If the right-hand side of \ref{Eq1} only includes the Riccati term and $\cm=0$, \ref{Eq1} reduces to an equation closely related to the shock formation problems studied in \cite{Speck2016b, Speck2016, Holzegel2016, Speck2016a}. However, due to the presence of the self-increasing, damping, and convection terms, this article concludes that the solution can reach arbitrarily large values and develop self-increasing singularities at certain future endpoints of null geodesics, provided the inhomogeneous part of the data is sufficiently small. Once the Riccati term eventually becomes dominant, $\varrho$ is no longer small. Even if it leads to a shock, it becomes a problem involving a large variable $\varrho$, as pointed in \cite[Remark 1.9]{Speck2016}.
\end{remark}
\begin{remark}[Coefficients] 
	In \ref{Eq1} and \ref{Eq2}, the coefficients $\frac{4}{3}$, $\frac{2}{3}$, and $\frac{4}{3}$ are used, though these values can be generalized in future studies. We specifically choose this set of coefficients because \ref{Eq2} with these values is closely related to the nonlinear version of the Jeans instability problem in the Euler–Poisson system (see \cite{ViatcehslavMukhanov2013, Zeldovich1971, Scherrer2001, Liu2023, Liu2023a} for details). As a result, \ref{Eq2} can be viewed as a simplified model of the nonlinear Jeans instability in expanding Newtonian universes.
\end{remark}

\subsection{Reference ODEs}\label{s:refsol}
To state the main results, we need the following reference ODEs studied in our previous article \cite{Liu2022b} for any given positive constants $\beta\in(0,+\infty)$ and $\beta_0\in(0,+\infty)$:
\begin{gather}
	\partial^2_\mft f(\mft)+\frac{4}{3 \mft} \del{\mft} f(\mft)-\frac{2}{3 \mft^2}  f(\mft)(1+  f(\mft))-\frac{4}{3}\frac{(\del{\mft} f(\mft))^2}{1+f(\mft)}=   0 , \label{e:feq0b}\\
	f(t_0)= \beta>0 \AND
	\del{t}f(t_0)=   \beta_0  >0.  \label{e:feq1b}
\end{gather} 
In this article, we refer to \eqref{e:feq0b}--\eqref{e:feq1b} as the \textit{reference ODE} for \ref{Eq1} and \ref{Eq2}. For general results on its solution, please refer to Appendix \ref{s:ODE} or \cite{Liu2022b}. We denote $t_m$ as the \textit{blowup time} of the solution $f(\mft)$, where $t_m>t_0>0$.

\subsection{Main Theorems}\label{s:mainthm}
Before presenting the main theorems, we define the following constants, which depend only on the initial data $t_0$, $\beta$, and $\beta_0$. These constants are used to characterize the self-increasing behavior of the solution. 
\begin{gather}
	\mathtt{A}:= \frac{3t_0 }{ 5 }\biggl(  \frac{t_0   \mf_0}{(1+\mf)^2} - \frac{2}{3}  \frac{\mf  } {1+\mf} \biggr), \quad 	\mathtt{B}:=   - \frac{3}{5 t_0^{\frac{2}{3} } } \Bigl(  \frac{\mf }{1+\mf}  + \frac{t_0 \mf_0}{(1+\mf)^2} \Bigr)<0,  \label{e:ttA}  \\ 
	\mathtt{C}:=   \frac{3} {5 t_0^{ \frac{2}{3}}} \Bigl( \ln( 1+\mf) +  \frac{t_0\mf_0}{1+\mf}\Bigr) >0 \AND 
	\mathtt{D}:=  
	\frac{2 t_0 }{5}  \Bigl(  \ln( 1+\mf)  - \frac{3} { 2} \frac{t_0\mf_0}{1+\mf}\Bigr).  \label{e:ttD}  
\end{gather}

To simplify the expressions in the statements and the calculations, we introduce \textbf{assumptions} on the initial data $\beta$ and $\beta_0$, as well as on the parameters $\ck$ and $q^i$ in \ref{Eq1} and \ref{Eq2}. It is important to note that these assumptions can be relaxed or even removed with additional effort and calculations using our method.
\begin{enumerate}[label={(A\arabic*)}]
	\item \label{Asp1} The initial data satisfies $\beta_0^2\geq 4\beta (1+\beta)^{2} t_0^{-2}$; 
	\item \label{Asp2}  The direction of convection is assumed to be constant and can be normalized $q^i=|q|\delta^i_1$ and $|q|\in(3,100)$; 
	\item \label{Asp3} $\ck=\frac{1}{4}$. 
\end{enumerate}

To present our main theorems, we introduce the following domains (illustrated in Fig. \ref{f:fig0}) to simplify the exposition. Specifically, these domains represent the domain of the \textit{inhomogeneous solution} $\mathcal{I}$ and the domain of the \textit{homogeneous solution} $\mathcal{H}$, denoted by
\begin{align}
	\mathcal{I}:=&\left\{(t,x)\in[t_0,t_m)\times \Rbb^n\;\bigg|\; |x|<1 +  \int_{t_0}^{t}   \sqrt{  \cg(y,f(y),f_0(y)) } dy \right\} ,  \label{e:cai}\\
	\mathcal{H}:=&\left\{(t,x)\in[t_0,t_m)\times \Rbb^n\;\bigg|\; |x| \geq 1 +  \int_{t_0}^{t}   \sqrt{  \cg(y,f(y),f_0(y)) } dy \right\} ,  \label{e:cah}
\end{align}
and 
the \textit{domain of influence} of the \textit{inhomogeneous data}, enclosed by a \textit{characteristic conoid surface} $\mathcal{C}$, 
\begin{equation}\label{e:char1}
	\mathcal{C}:=\left\{(t,x)\in[t_0,t_m)\times \Rbb^n\;\bigg|\; |x|=1 +  \int_{t_0}^{t}   \sqrt{  \cg(y,f(y),f_0(y)) } dy \right\}  . 
\end{equation}

The following theorems assert (see Fig. \ref{f:fig0b}) there is a hypersurface $\mathcal{T}$ (depending on the smallness of the inhomogeneous data) that separates the inhomogeneous domain $\mathcal{I}$ into two domains: an upper dark-shaded domain  and a lower light-shaded domain. The solution beneath the surface $\mathcal{T}$ (denoted as $\mathcal{K}$, see Theorem \ref{t:mainthm1} and \ref{t:mainthm2} for precise definitions) and within the homogeneous domain $\mathcal{H}$—that is, the solution in $\mathcal{K} \cup \mathcal{H}$ (comprising the light-shaded and white domains)—can be determined within this article. However, the solution in the upper dark-shaded domain, $\mathcal{I} \setminus \mathcal{K}$, remains unknown, which we refer to as the \textit{uncharted domain} $\mathcal{I} \setminus \mathcal{K}$. Moreover, we establish that the solution is inhomogeneous in the light-shaded domain $\mathcal{K} \cap \mathcal{I}$ but homogeneous in the white domain $\mathcal{H}$. Additionally, the known solution exhibits self-increasing behavior and \textit{reaches a self-increasing singularity at the point $p_m = (t_m, +\infty, 0, \dots, 0)$}. Although the solution itself grows within $\mathcal{K}\cup\mathcal{H}$, the spatial gradients of the solution remain small in this region. Consequently, we refer to $\mathcal{K}\cup\mathcal{H}$ as the \textit{long-wavelength domain}.

\medskip

\begin{tabularx}{0.95\textwidth} { 
		| >{\raggedright\arraybackslash}X 
		| >{\centering\arraybackslash}X 
		| >{\raggedleft\arraybackslash}X | }
	\hline		
	The white domain $\mathcal{H}$  & The homogeneous and self-increasing solution    \\
	\hline
	The light-shaded domain $\mathcal{K}\cap \mathcal{I}$ & The inhomogeneous and self-increasing solution    \\
	\hline
	The dark-shaded domain $\mathcal{I}\setminus \mathcal{K}$ &  The solution remains unknown  \\
	\hline
\end{tabularx}

\begin{figure}[htbp]
	\begin{minipage}[t]{0.5\linewidth}
		\centering
		\includegraphics[width=1\textwidth]{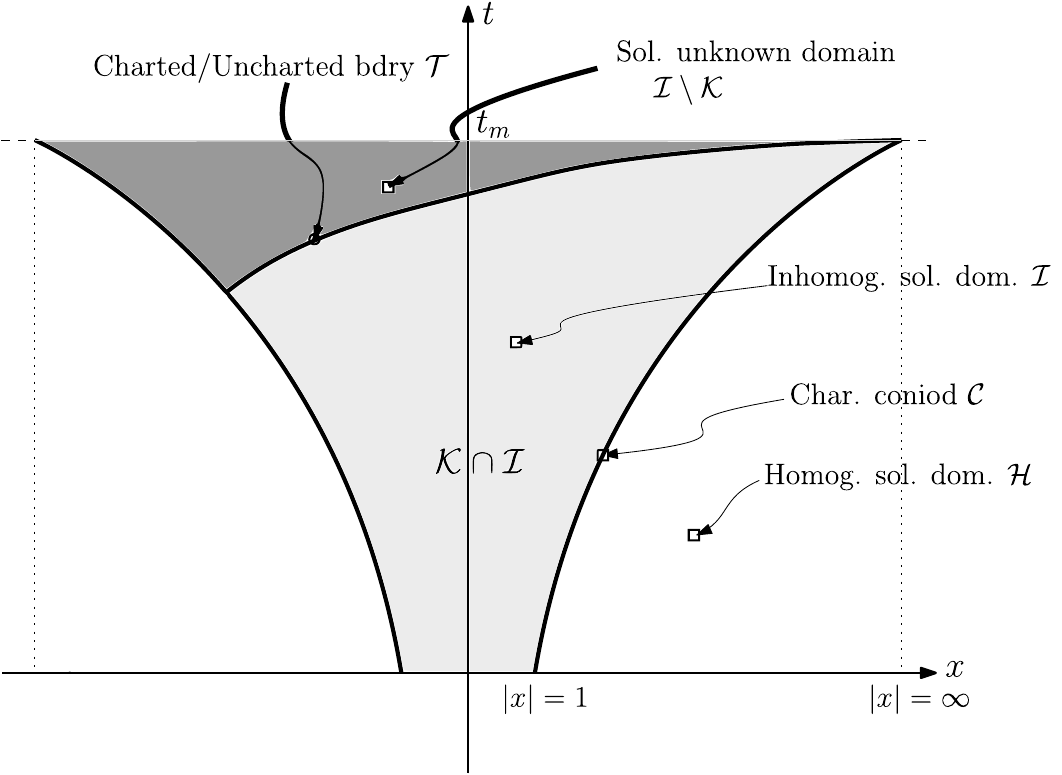}
		\caption{Domains in Main Theorems}
		\label{f:fig0}
	\end{minipage}%
	\begin{minipage}[t]{0.5\linewidth}
		\centering
		\includegraphics[width=1\textwidth]{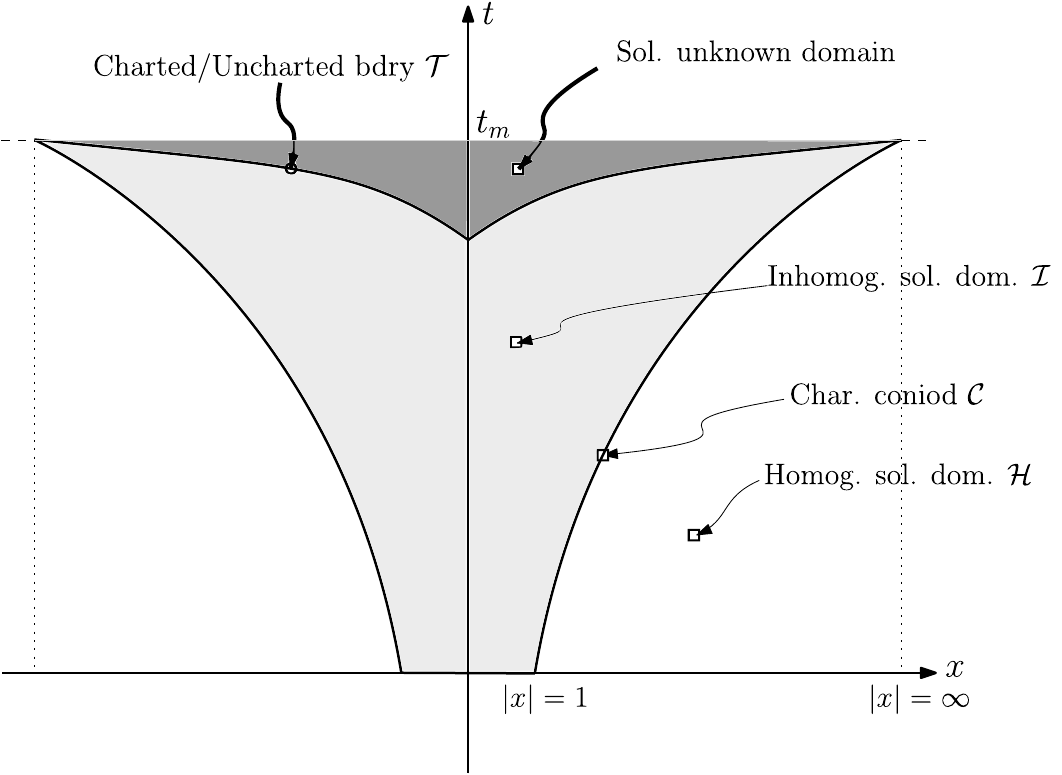}
		\caption{If $q^i\propto x^i/|x|$, domains of solutions}
		\label{f:fig0b}
	\end{minipage}
\end{figure}
 
 Now let us present  the main theorems and defer their proofs to  \S\ref{s:pfmthm}.  
  \begin{theorem}[The main theorem for \ref{Eq1}]\label{t:mainthm1}
 	Suppose $k\in\Zbb_{\frac{n}{2}+3}$, $\mathtt{A}, \;\mathtt{B},\;\mathtt{C},\; \mathtt{D}$ are constants depending on the initial data $\beta$ and $\beta_0$, as defined by \eqref{e:ttA}--\eqref{e:ttD}, and that Assumptions \ref{Asp1}--\ref{Asp3} hold. Let $(\psi,\psi_0)\in C^1_0(\Rbb^n) $  be given functions with $\supp (\psi,\psi_0)= B_1(0)$, $f(\mft)$  be the solution to  \eqref{e:feq0b}--\eqref{e:feq1b} as given by Theorem \ref{t:mainthm0}. 
 	%,  and $\mathfrak{g}$ be defined by \eqref{e:ctm2}.  
 	Then 
 	there exist sufficiently small constants $\sigma_0>0$ and  $\delta_0>0$, 	such that if  the initial data satisfy
 	\begin{equation*}
 		\|  \psi \|_{H^k(B_1(0))} +	\|\del{i} \psi \|_{H^{k}(B_1(0))} +  \|  \psi_0\|_{H^k(B_1(0))} \leq  e^{  - \frac{153}{\delta_0} }    \sigma_0^2   ,   
 	\end{equation*}
 	then  there exists a  hypersurface   $t=\mathcal{T}(x,\delta_0)$  satisfying
 	\begin{equation*}  
 		\Gamma_{\delta_0}:=\{(t,x)\in[\mathbf{t}_0,t_m)\times\Rbb^n\;|\;t=\mathcal{T}(x,\delta_0)\}  \subset\mathcal{I} ,  \quad \lim_{a\rightarrow +\infty} \mathcal{T}(a\delta^i_1,\delta_0)=t_m \AND \lim_{\delta_0\rightarrow 0+} \mathcal{T}(x,\delta_0) =t_m , 
 	\end{equation*}  such that  there is a solution $\varrho\in C^2(	\mathcal{K}\cup \mathcal{H})$ to \ref{Eq2}  where $ 	\mathcal{K}:=\left\{(t,x)\in[\mathbf{t}_0,t_m)\times \Rbb^n\; |\;t<\mathcal{T}(x,\delta_0)\right\} $ satisfying
 	\begin{enumerate}[leftmargin=*,label={(\arabic*)}] 	
 		\item\label{b1.1}  if we denote
 		\begin{equation*}
 			\mathbf{1}_{-}(x^1):=1-C \sigma_0^2 e^{  - \frac{50}{ \delta_0}} e^{ -\frac{x^1}{2} }   \searrow 1 \AND	\mathbf{1}_{+}(x^1):=1+C \sigma_0^2 e^{  - \frac{50}{ \delta_0}} e^{ -\frac{x^1}{2} }  \searrow 1 ,\quad \text{as }x^1\rightarrow +\infty , 
 		\end{equation*} 
 		then there are estimates 
 		for $(t,x)\in\mathcal{K}\cap\mathcal{I} $, 
 		\begin{align*}
 			\mathbf{1}_{-}(x^1) f_0\bigl(	e^{\mathbf{t}_0}+\mathbf{1}_{-}(x^1) (e^{t}-e^{\mathbf{t}_0})\bigr)  \leq 	\varrho_0(t,x) & \leq \mathbf{1}_{+}(x^1) f_0\bigl(	e^{\mathbf{t}_0}+\mathbf{1}_{+}(x^1) (e^{t}-e^{\mathbf{t}_0}) \bigr)  ,  \\
 			-C \sigma_0^2 e^{  - \frac{50}{ \delta_0}} e^{ -\frac{x^1}{2} }  (1+f \bigl(	e^{\mathbf{t}_0}+\mathbf{1}_{-}(x^1) (e^{t}-e^{\mathbf{t}_0})\bigr)  ) \leq 	\varrho_i (t,x)& \leq  C \sigma_0^2 e^{  - \frac{50}{ \delta_0}} e^{ -\frac{x^1}{2} }  \bigl(1+f \bigl(	e^{\mathbf{t}_0}+\mathbf{1}_{+}(x^1) (e^{t}-e^{\mathbf{t}_0}) \bigr)\bigr)  ,  \\
 			\mathbf{1}_{-}(x^1) f \bigl(	e^{\mathbf{t}_0}+\mathbf{1}_{-}(x^1) (e^{t}-e^{\mathbf{t}_0})\bigr)  \leq 	\varrho (t,x)& \leq \mathbf{1}_{+}(x^1) f \bigl(	e^{\mathbf{t}_0}+\mathbf{1}_{+}(x^1) (e^{t}-e^{\mathbf{t}_0}) \bigr)    . 
 		\end{align*}
 		Moreover, both $\varrho_0$ and $\varrho$ reach the self-increasing singularities at the point $p_m:=(t_m,+\infty,0,\cdots,0)$: 
 		\begin{gather*}
 			\lim_{\mathcal{K}\ni(t,x)\rightarrow p_m} \varrho=\lim_{\mathcal{K}\ni(t,x)\rightarrow p_m} f=+\infty  ,  \\
 			\lim_{\mathcal{K}\ni(t,x)\rightarrow p_m} \varrho_0=\lim_{\mathcal{K}\ni(t,x)\rightarrow p_m} f_0=+\infty 
 			\AND 
 			\lim_{\mathcal{K}\ni(t,x)\rightarrow p_m} \varrho_i=0. 
 		\end{gather*}    
 		\item \label{b5.1}  		if the parameter $\delta_0 \to 0^+$, which corresponds to the case where the initial inhomogeneities becomes sufficiently small. Then, along the boundary $\Gamma_{\delta_0}$ separating the charted and uncharted regions, the quantities $\varrho$ and $\varrho_0$ become arbitrarily large. More precisely, 
\begin{gather*}
	\lim_{\delta_0\rightarrow 0+} \varrho|_{\Gamma_{\delta_0}}=f(t_m)=+\infty  \AND 
	\lim_{\delta_0\rightarrow 0+} \varrho_0|_{\Gamma_{\delta_0}} =f_0(t_m)=+\infty  .
\end{gather*}   
 		This phenomenon is referred to as \emph{almost blowups}.
 		\item\label{b4.1}   $\varrho\equiv f$ for  $(t,x) \in \mathcal{H}$ where $\mathcal{H}$ is defined by \eqref{e:cah}. 
 		\item\label{b2.1}  the growth rate of $\varrho$ can be estimated using visualizable functions.   
 		\begin{align*}
 			\varrho(t,x)  \geq  & 
 			\mathbf{1}_{-}(x^1) f \bigl(	e^{\mathbf{t}_0}+\mathbf{1}_{-}(x^1) (e^{t}-e^{\mathbf{t}_0})\bigr) > 
 			\mathbf{1}_{-}(x^1) \bigl( e^ { \mathtt{C}  (	e^{\mathbf{t}_0}+\mathbf{1}_{-}(x^1) (e^{t}-e^{\mathbf{t}_0}))^{\frac{2}{3}}  }  -1\bigr)    
 			\intertext{and}
 			\varrho(t,x )   \leq &\mathbf{1}_{+}(x^1) f \bigl(	e^{\mathbf{t}_0}+\mathbf{1}_{+}(x^1) (e^{t}-e^{\mathbf{t}_0}) \bigr)   <\frac{3}{2} \left( \frac{1}{ 1+ \frac{\mathtt{A}}{	e^{\mathbf{t}_0}+\mathbf{1}_{+}(x^1) (e^{t}-e^{\mathbf{t}_0})}  + \mathtt{B}  (	e^{\mathbf{t}_0}+\mathbf{1}_{+}(x^1) (e^{t}-e^{\mathbf{t}_0}))^{\frac{2}{3}}   }   -1\right) 
 		\end{align*} 
 		for all $(t,x)\in \mathcal{K}\cap \mathcal{I}$. 
 		\item\label{b3.1}  if the initial data satisfy
 		$ \breve{\beta} :=  \frac{e^{\mathbf{t}_0} \mf_0 }{1+\mf}-1>0  $, then $\varrho$ has an improved lower bound, indicating finite-time blowups.
 		\begin{equation*}
 			\varrho(t,x)  \geq 
 			\mathbf{1}_{-}(x^1) f \bigl(	e^{\mathbf{t}_0}+\mathbf{1}_{-}(x^1) (e^{t}-e^{\mathbf{t}_0})\bigr)  > 
 			\mathbf{1}_{-}(x^1)  \left( \frac{1+\mf}{\Bigl(  \frac{\beta_0 e^{\frac{4}{3}\mathbf{t}_0}}{1+\beta}    \bigl(	e^{\mathbf{t}_0}+\mathbf{1}_{-}(x^1) (e^{t}-e^{\mathbf{t}_0})\bigr)^{-\frac{1}{3}} -\breve{\beta} \Bigr)^{3}}   -	1 \right)  
 		\end{equation*}  
 		for all  $(t,x^k)\in \mathcal{K}\cap \mathcal{I}$.     
 	\end{enumerate} 
 \end{theorem} 
 
 \begin{theorem}[The main theorem for \ref{Eq2}] \label{t:mainthm2}
	Suppose $k\in\Zbb_{\frac{n}{2}+3}$, $\mathtt{A}, \;\mathtt{B},\;\mathtt{C},\; \mathtt{D}$ are constants depending on the initial data $\beta$ and $\beta_0$, as defined by \eqref{e:ttA}--\eqref{e:ttD}, and that Assumptions \ref{Asp1}--\ref{Asp3} hold. Let $(\psi,\psi_0)\in C^1_0(\Rbb^n) $ be given functions with $\supp (\psi,\psi_0)= B_1(0)$, $f(\mft)$  be the solution to  \eqref{e:feq0b}--\eqref{e:feq1b} as given by Theorem \ref{t:mainthm0}.  
	Then 
	there exist sufficiently small constants $\sigma_0>0$ and  $\delta_0>0$, 	such that if  the initial data satisfy
	\begin{equation*}
		\|  \psi \|_{H^k(B_1(0))} +	\|\del{i} \psi \|_{H^{k}(B_1(0))} +  \|  \psi_0\|_{H^k(B_1(0))} \leq  e^{  - \frac{153}{\delta_0} }    \sigma_0^2   ,   
	\end{equation*} then  there exists a  hypersurface   $t=\mathcal{T}(x,\delta_0)$  satisfying
	\begin{equation*}   
		\Gamma_{\delta_0}:=\{(t,x)\in[t_0,t_m)\times\Rbb^n\;|\;t=\mathcal{T}(x,\delta_0)\}  \subset\mathcal{I} ,  \quad \lim_{a\rightarrow +\infty} \mathcal{T}(a\delta^i_1,\delta_0)=t_m \AND \lim_{\delta_0\rightarrow 0+} \mathcal{T}(x,\delta_0) =t_m,
	\end{equation*}
	such that  there is a solution $\varrho\in C^2(	\mathcal{K}\cup \mathcal{H})$ to \ref{Eq2}  where $ 	\mathcal{K}:=\left\{(t,x)\in[t_0,t_m)\times \Rbb^n\; |\;t<\mathcal{T}(x,\delta_0)\right\} $ satisfying
	\begin{enumerate}[leftmargin=*,label={(\arabic*)}] 	 	
		\item\label{b1}  if we denote
		\begin{equation*}
			\mathbf{1}_{-}(x^1):=1-C \sigma_0^2 e^{  - \frac{50}{ \delta_0}} e^{ -\frac{x^1}{2} }   \searrow 1 \AND	\mathbf{1}_{+}(x^1):=1+C \sigma_0^2 e^{  - \frac{50}{ \delta_0}} e^{ -\frac{x^1}{2} }  \searrow 1 ,\quad \text{as }x^1\rightarrow +\infty , 
		\end{equation*} 
		then there are estimates 
		for $(t,x)\in\mathcal{K}\cap\mathcal{I} $, 
		\begin{align*}
			\mathbf{1}_{-}(x^1) f_0\bigl(	t_0+\mathbf{1}_{-}(x^1) (t-t_0)\bigr)  \leq 	\varrho_0(t,x) & \leq \mathbf{1}_{+}(x^1) f_0\bigl(	t_0+\mathbf{1}_{+}(x^1) (t-t_0) \bigr) , \\
			-C \sigma_0^2 e^{  - \frac{50}{ \delta_0}} e^{ -\frac{x^1}{2} }  (1+f \bigl(	t_0+\mathbf{1}_{-}(x^1) (t-t_0)\bigr)  ) \leq 	\varrho_i (t,x)& \leq  C \sigma_0^2 e^{  - \frac{50}{ \delta_0}} e^{ -\frac{x^1}{2} }  \bigl(1+f \bigl(	t_0+\mathbf{1}_{+}(x^1) (t-t_0) \bigr)\bigr)  ,  \\
			\mathbf{1}_{-}(x^1) f \bigl(	t_0+\mathbf{1}_{-}(x^1) (t-t_0)\bigr)  \leq 	\varrho (t,x)& \leq \mathbf{1}_{+}(x^1) f \bigl(	t_0+\mathbf{1}_{+}(x^1) (t-t_0) \bigr)    . 
		\end{align*}
		Moreover, both $\varrho_0$ and $\varrho$ reach the self-increasing singularities at the point $p_m:=(t_m,+\infty,0,\cdots,0)$: 
		\begin{gather*}
			\lim_{\mathcal{K}\ni(t,x)\rightarrow p_m} \varrho=\lim_{\mathcal{K}\ni(t,x)\rightarrow p_m} f=+\infty  ,  \\
			\lim_{\mathcal{K}\ni(t,x)\rightarrow p_m} \varrho_0=\lim_{\mathcal{K}\ni(t,x)\rightarrow p_m} f_0=+\infty 
			\AND 
			\lim_{\mathcal{K}\ni(t,x)\rightarrow p_m} \varrho_i=0. 
		\end{gather*}    
		\item \label{b5}  
		if the parameter $\delta_0 \to 0^+$, which corresponds to the case where the initial inhomogeneities becomes sufficiently small. Then, along the boundary $\Gamma_{\delta_0}$ separating the charted and uncharted regions, the quantities $\varrho$ and $\varrho_0$ become arbitrarily large. More precisely, 
		\begin{gather*}
			\lim_{\delta_0\rightarrow 0+} \varrho|_{\Gamma_{\delta_0}}=f(t_m)=+\infty  \AND 
			\lim_{\delta_0\rightarrow 0+} \varrho_0|_{\Gamma_{\delta_0}} =f_0(t_m)=+\infty  .
		\end{gather*}   
		This phenomenon is referred to as \emph{almost blowups}.
		\item\label{b4}   $\varrho\equiv f$ for  $(t,x) \in \mathcal{H}$ where $\mathcal{H}$ is defined by \eqref{e:cah}. 
		\item\label{b2}  the growth rate of $\varrho$ can be estimated using visualizable functions
		\begin{align*}
			\varrho(t,x)  \geq  & 
			\mathbf{1}_{-}(x^1) f \bigl(	t_0+\mathbf{1}_{-}(x^1) (t-t_0)\bigr) > 
			\mathbf{1}_{-}(x^1) \bigl( e^ { \mathtt{C}  (	t_0+\mathbf{1}_{-}(x^1) (t-t_0))^{\frac{2}{3}}  }  -1\bigr)    
			\intertext{and}
			\varrho(t,x )   \leq &\mathbf{1}_{+}(x^1) f \bigl(	t_0+\mathbf{1}_{+}(x^1) (t-t_0) \bigr)   <\frac{3}{2} \left( \frac{1}{ 1+ \frac{\mathtt{A}}{	t_0+\mathbf{1}_{+}(x^1) (t-t_0)}  + \mathtt{B}  (	t_0+\mathbf{1}_{+}(x^1) (t-t_0))^{\frac{2}{3}}   }   -1\right) 
		\end{align*} 
		for all $(t,x)\in \mathcal{K}\cap \mathcal{I}$. 
		\item\label{b3}  if the initial data satisfy
		$ \breve{\beta} :=  \frac{t_0 \mf_0 }{1+\mf}-1>0  $, then $\varrho$ has an improved lower bound, indicating finite-time blowups.
		\begin{equation*}
			\varrho(t,x)  \geq 
			\mathbf{1}_{-}(x^1) f \bigl(	t_0+\mathbf{1}_{-}(x^1) (t-t_0)\bigr)  > 
			\mathbf{1}_{-}(x^1)  \left( \frac{1+\mf}{\Bigl(  \frac{\beta_0 t_0^{\frac{4}{3}}}{1+\beta}    \bigl(	t_0+\mathbf{1}_{-}(x^1) (t-t_0)\bigr)^{-\frac{1}{3}} -\breve{\beta} \Bigr)^{3}}   -	1 \right)  
		\end{equation*}  
		for all  $(t,x^k)\in \mathcal{K}\cap \mathcal{I}$.     
	\end{enumerate} 
\end{theorem} 
\begin{remark}[Almost blowups]
The conclusion \ref{b5.1} suggests an almost blowup in the long-wavelength region. This is consistent with both linear physical analysis and observations. Outside the long-wave region $\mathcal{K} \cup \mathcal{H}$, where the wavelengths become shorter, other types of singularities may also occur. However, the smaller the initial inhomogeneity is, the larger the long-wavelength domain can be ensured, thereby guaranteeing that the solution can grow arbitrarily large, provided the initial inhomogeneity is sufficiently small.
\end{remark}
 \begin{remark}[The direction bias and generalization]
	The directional bias (i.e., the singularity at $p_m$) in the result arises from the chosen vector field $q^i$ in the convection term (see Assumption \ref{Asp2}). If $q^i$ is proportional to $x^i/|x|$, the same method can be used to demonstrate that, by rotating the coordinate system, the solution experiences self-increasing blowups at all future endpoints of null geodesics originating from $B_1(0)$ (see Fig. \ref{f:fig0b}).    
\end{remark}
 \begin{remark}[Main nonlinear mechanisms]\label{t:nlnrmech}
	Although the right-hand side of \ref{Eq1} contains four types of nonlinear terms (recall remark \ref{t:nlterms}), our proofs will demonstrate that the dominant terms in the domains $\mathcal{K}\cup\mathcal{H}$ are (i) the self-increasing term, (iii) the Riccati term, and (iv) the convection term. The damping effects are overwhelmed by these dominant terms. This is evident from the reference ODEs \eqref{e:feq0b}--\eqref{e:feq1b} and Lemma \ref{t:iden1}.\eqref{e:keyid3} (i.e. $f_0\sim f^{\frac{1}{2}}(1+f)/\mft $,  as will be proved later)
	\begin{equation*}
		\frac{4}{3}\frac{(\del{\mft} f )^2}{1+f }  \sim  	\frac{2}{3 \mft^2}  f (1+  f ) \quad \text{but}\quad -\frac{4}{3 \mft} \del{\mft} f \sim \frac{1}{\mft^2} f^{\frac{1}{2}}(1+f) . 
	\end{equation*}  
\end{remark}  
\begin{remark}[Uncharted domain $\mathcal{I} \setminus \mathcal{K}$ and future directions]\label{t:Ptdevp}
	As mentioned in \S\ref{s:mainthm}, the solution in the upper dark-shaded domain, $\mathcal{I} \setminus \mathcal{K}$, remains unknown. However, this region can be made arbitrarily small by choosing $\delta_0$ sufficiently small, since $\lim_{\delta_0 \rightarrow 0^+} \mathcal{T}(x,\delta_0) = t_m$ as shown in Theorems \ref{t:mainthm1} and \ref{t:mainthm2}. 
	Understanding the solution in this  domain $\mathcal{I} \setminus \mathcal{K}$ is both interesting and necessary. Additional conditions might be required to fully characterize this domain, given the interplay, competitions and dominance of the nonlinear terms. Intuitively and based on numerical tests, gradients in $\mathcal{I} \setminus \mathcal{K}$  may increase, breaking the long-wavelength profiles and potentially leading to shock formation or Rendall's instability (see \cite{Oliynyk2024}). 
	As inhomogeneities grow, the spacetime gradients of $\varrho$  could become significantly larger than $\varrho$  itself, which might desynchronize the blowup times. This could cause the Riccati term to dominate the nonlinear effects in \ref{Eq1}, possibly resulting in shock formation. The shock formation has been extensively studied in the context of quasilinear systems \cite{Speck2016b, Speck2016, Holzegel2016, Speck2016a, Miao2016, Miao2018} and (relativistic) Euler equations \cite{Christodoulou2007, Christodoulou2014, Speck2017a, Luk2021, Luk2018, Abbrescia2022, Abbrescia2023, Ginsberg2024, Shkoller2024, Neal2023, Neal2023a, Buckmaster2021, Buckmaster2022, Buckmaster2022a, Buckmaster2020}. 
	Given the fundamentally different mechanisms and required methods, the short-wavelength domain will be treated separately. This work is devoted to the long-wavelength regime and the emergence of a self-increasing singularity, associated with the nonlinear Jeans instability in astrophysics. 
\end{remark}

\subsection{Related works}\label{s:rewks} 
We have already mentioned many related works in the previous remarks. Therefore, this section focuses on additional aspects that extend beyond those works.

In our previous article \cite{Liu2022b}, we investigated a wave equation similar to \ref{Eq2}, where the specific functions $\cg $ and $-K^{ij}\del{i}\varrho\del{j}\varrho$ were replaced by 
\begin{equation*}%\label{e:Fdef2}
	\cg (t) := \frac{\cm^2 (\del{t} f(t))^2}{(1+f(t))^2} \quad \text{and} \quad F(t) := \ck \frac{(\del{t} f(t))^2}{1+f(t)} . 
\end{equation*}
This replacement  neutralizes some effects of the Riccati term and thus synchronized the blowup times and further simplified the proofs. In our subsequent article \cite{Liu2023}, with the help of our other papers \cite{Liu2022,Liu2023a}, we utilized this idea to demonstrate the Newtonian nonlinear Jeans instability for the Euler--Poisson system with specific sources that provided synchronizing effects.
In the current article, however, $\cg $ is given by \eqref{e:Fdef}, making equation \ref{Eq2} quasilinear, which complicates the problem. Additionally, by setting $F=0$, the synchronizing effect is absent. This fundamentally alters the core problem addressed in \cite{Liu2022b} and introduces significant difficulties. Due to the \textit{absence of a synchronizing effect}, certain nonlinear terms may dominate and potentially alter the behavior in the uncharted domain. This is one reason why our analysis in this article does not extend to that domain.

On the other hand, the blowup problem has been extensively studied for various hyperbolic equations, particularly those of the wave type. For an introduction to blowups in various hyperbolic systems, we refer readers to \cite{Alinhac1995,Kichenassamy2021}. For a comprehensive review of various blowup results and their fundamental proof techniques, see Speck \cite{Speck2020} and the references therein. In addition to the previously mentioned shock problems, there are related works on self-increasing solutions for Euler equations. Notable examples include studies on spherically symmetric Euler flows with physical vacuum boundaries \cite{Guo2022,Guo2018} and investigations into the implosion of a compressible fluid \cite{Merle2022,Merle2022a,Buckmaster2022b}. Another closely related work is \cite{Oliynyk2024}, which, under certain assumptions, concludes the blow-up of the fractional density gradient for relativistic fluids in exponentially expanding FLRW spacetimes, as conjectured by Rendall.

\subsection{Overviews, ideas and outlines}\label{s:overview}
We focus on \ref{Eq2}, and \ref{Eq1} can subsequently be concluded using an exponential-logarithmic time transform (see Appendix \ref{s:pfclm}).

\subsubsection{Reviews on fundamental tools in our previous article \cite{Liu2022b}}\label{s:review}
Before proceeding, to better understand the method used in this article, it is helpful to first revisit some key features from \cite{Liu2022b} (see \S\ref{s:rewks}). In \cite{Liu2022b}, as mentioned earlier,\textit{ the presence of the source term synchronizes the blowup time}. To begin, we introduce a \textit{reference ODE} \eqref{e:feq0}--\eqref{e:feq1}, which is detailed in Appendix \ref{s:ODE0} (and also in \cite[\S$2$]{Liu2022b}). We anticipate that the solution to this wave equation will asymptotically approach the reference solution of the reference ODE \eqref{e:feq0}--\eqref{e:feq1}. This approach allows the problem to be transformed into a global existence and stability problem by analyzing the equation for the difference between $\varrho$ and the reference solution $f$. It is important to note that the solution $f$ to the reference ODE \eqref{e:feq0}--\eqref{e:feq1} is a self-increasing solution that blows up at time $t_m$, which can be either finite or infinite depending on the initial data (see Appendix \ref{t:refsol} for details and estimates on the solution $f$). Therefore, once it is established that $\varrho$ is close to $f$ on $(t,x)\in[t_0,t_m)\times \Rbb^n$, we obtain self-increasing estimates for $\varrho$ that are controlled by the reference solution $f$. To proceed, we must address the following questions.

\underline{Question $1$:   How can we prove that $\varrho$ is close to $f$ on $(t,x)\in[t_0,t_m)\times \Rbb^n$?}   
Given that this is a global existence and stability problem, various tools could be employed. Our approach utilizes the \textit{Fuchsian global initial value problem (Fuchsian GIVP)}, originally introduced by Oliynyk \cite[Appendix B]{Oliynyk2016a}, which is equivalent to energy methods. We provide a precise statement of the Fuchsian GIVP in Appendix \ref{s:fuc}, a specialized version of the theorem established in \cite{Beyer2020}. This method is particularly effective, as evidenced by its broader generalizations and applications in works such as \cite{Liu2018,Liu2028c, Liu2018a, Beyer2020b, Fajman2021, LeFloch2021, Ames2022, Ames2022a, Marshall2023, Liu2022, Liu2022a, Liu2023, Beyer2023a, Beyer2023b, Fajman2023, Oliynyk2024}. A Fuchsian system is a singular symmetric hyperbolic system of the form: 
\begin{equation}\label{e:Fucmodel} 
	B^{\mu}(t,x,u)\partial_{\mu}u = \frac{1}{\tau}\textbf{B}(t,x,u)\textbf{P}u+H(t,x,u) . 
\end{equation} 
There are structural assumptions on the coefficients $B^\mu$, $\mathbf{B}$ and $H$ (detailed in Conditions \ref{c:2}--\ref{c:7} in Appendix \ref{s:fuc}), and these assumptions ensure that the asymptotics of this PDE are determined by a simple ODE (see \cite[\S$3.5$]{Beyer2020} for further details). We emphasize the following key point: 
\begin{enumerate}[leftmargin=*,label={(\arabic*)}]
	\item $B^i$ and $H$ can also include singular terms involving $1/\tau$, but in a specific manner (see Conditions \ref{c:3} and \ref{c:4}). 
	\item one of the most important assumption \ref{c:5} is 
	\begin{equation}\label{e:Bineq1}
		\frac{1}{\gamma_{1}}\mathds{1}\leq B^{0}\leq \frac{1}{\kappa} \textbf{B} \leq\gamma_{2}\mathds{1} 
	\end{equation}
	where $\kappa,\,\gamma_{1},\,\gamma_{2}$ are positive constants. We emphasize this assumption because it is one of the most difficult to satisfy when transforming a hyperbolic equation into a Fuchsian formulation. This condition involves the decay rates of the solution, making it particularly difficult to construct. In the current work, we will demonstrate the multiple transformations necessary to achieve this assumption. 	 
\end{enumerate}

The conclusion of the Fuchsian formulation can be roughly expressed (see Theorem \ref{t:fuc} for the precise statement) as follows: the solution to the Fuchsian IVP can be extended to $\tau \in [-1,0)$, meaning the solution exists globally up to the singular time  $\tau=0$. In most applications, a time transformation is required to convert $t=+\infty$ to $\tau=0$. Therefore, this result implies the existence of a global solution.

\underline{Question $2$:   How to rewrite the wave equation as a Fuchsian system?} To convert the wave equation into a Fuchsian system, we first need to compactify the time coordinate $t\in[t_0,t_m)$ to $\tau\in[-1,0)$, mapping $t_m$ to $0$. Next, we select variables that are weighted by appropriate decaying rates to serve as the Fuchsian variables. The ``correct'' choice of Fuchsian variables, meaning those with the appropriate decay rates, facilitates obtaining the Fuchsian formulation. Additionally, in \cite{Liu2022b}, we employed certain hidden relations to distinguish between the singular term $\frac{1}{\tau}\mathbf{B} \mathbf{P} u$ and the regular term $H$ in the Fuchsian system \eqref{e:Fucmodel}. Below, we provide a detailed explanation. 
\begin{enumerate}[leftmargin=*,label={(\arabic*)}]
	\item  In our article \cite{Liu2022b}, we perform time compactification using the transformation (since,  by Lemma \ref{t:f0fg}, $f$ is strictly increasing)  
	\begin{equation}\label{e:ctm2.a}
		\tau=\mathfrak{g}(\mft) =-\exp\Bigl(-A\int^{\mft}_{t_0} \frac{f(s)(f(s)+1)}{s^2 f_0(s)} ds \Bigr)=- \Bigl(1+ \frac{2}{3} B  \int^{\mft}_{t_0} s^{-\frac{2}{3}} f(s)(1+f(s))^{-\frac{1}{3}}  ds \Bigr)^{-\frac{3A}{2}}  \in[-1,0)  . 
	\end{equation}  
	This transformation is derived by assuming a transformation $\mathfrak{g}(\mft)$ to be determined and ensuring it matches the required form for achieving the Fuchsian formulation. 
	\item The Fuchsian variables are the differences between $\varrho$ and $f$, scaled by $f$, along with their derivatives (see \cite{Liu2022b} for details); 
	\item Key quantities such as $\chi_\uparrow$, $1/(f\mfg)$, and $1/(\mfg f^{\frac{1}{2}})$,  along with fundamental identities involving $\chi$, $\tau$, $f$ and $f_0$ from our previous article \cite{Liu2022b}, are now listed in Appendix \ref{t:refsol}. For example,  
	\begin{equation*} 
		\chi_\uparrow(\mft):=\frac{\mft^{\frac{2}{3}} f_0(\mft)}{(1+f(\mft))^{\frac{2}{3}} f(\mft) (-\mfg(\mft))^{\frac{2}{3A}}} \overset{\eqref{e:f0aa}}{=} \frac{  (-\mfg(\mft) )^{-\frac{4}{3A}}\mft^{-\frac{2}{3}}}{B f(\mft) (1+f(\mft))^{-\frac{2}{3}}}  ,
	\end{equation*}
	and $\chi_\uparrow $ satisfies (see Proposition \ref{t:limG})
	\begin{equation*} 
		\chi_\uparrow(\mft)=\frac{2\cb B}{3-2\cc}+\mathfrak{G}(\mft)
	\end{equation*}
	where $\lim_{\mft\rightarrow t_m}\mathfrak{G}(\mft)=0$. The key point is that $\chi_\uparrow(t)$, $1/(f\mfg)$ and $1/(\mfg f^{\frac{1}{2}})$ help  distinguish between terms that belong to the singular term $\frac{1}{\tau}\mathbf{B} \mathbf{P} u$ and those that belong to the regular term $H$ in the Fuchsian system \eqref{e:Fucmodel}. 
\end{enumerate}
Using the above facts, one can transform the wave equation into a Fuchsian system and further conclude the existence of a global solution for $\tau\in[-1,0)$, which implies $\varrho$ is close to $f$ for $t\in[t_0,t_m)$. This completes the proof.

Summarizing the methods introduced above in \cite{Liu2022b}, we identify four fundamental tools that also prove to be crucial in the current article.
\begin{enumerate}[label={(S\arabic*)}]
	\item\label{S1} Fuchsian GIVP; 
	\item\label{S2} The fundamental properties of the reference ODE \eqref{e:feq0b}--\eqref{e:feq1b}. 
	\item\label{S3} The time compactification. 
	\item\label{S4} Key quantities $\chi_\uparrow$, $1/(f\mfg)$, $1/(\mfg f^{\frac{1}{2}})$ and basic identities. 
\end{enumerate}

\subsubsection{Overviews of the methods}\label{s:oview} 
Let us return to equation \ref{Eq2} in this section. 
First, a simple intuitive observation is that if the initial inhomogeneities $\|(\psi,\psi_0)\|_{\mathtt{X}}$ tend to vanish (where $\|\cdot\|_{\mathtt{X}}$ denotes a suitable norm), then equation \ref{Eq2} reduces to \eqref{e:feq0b}--\eqref{e:feq1b}. In other words, the independent variable $\mft$ of $f$ aligns with the time coordinate $t$, and the solution $\varrho(t,x^k)$ to \ref{Eq2} becomes the solution $f(\mft)=f(t)$ to the reference ODEs \eqref{e:feq0b}--\eqref{e:feq1b}. The term ``almost vanishing initial inhomogeneities'' refers to very long wavelengths or very low frequencies, which align with the classical Jeans instability requiring the \textit{Jeans scale}. 
Based on this observation, we have the following \textbf{intuitions}: 
\begin{intuition*}\label{t:int1}
	If the initial inhomogeneities $\|(\psi,\psi_0)\|_{\mathtt{X}} $ are sufficiently small, then $\mft$, which can be regarded as $\mft = \mft(t, x^k)$ with $\mft(t_0, x^k) = t_0$, remains close to $t$, and the solution $\varrho(t, x^k)$ to \ref{Eq2} is sufficiently close to the solution $f(\mft)$ of the reference ODEs \eqref{e:feq0b}--\eqref{e:feq1b}, at least within a sufficiently large domain $\mathcal{D}$ close to $ [t_0, t_m)\times \Rbb^n$ (recalling that $f$ exists for $t\in[t_0,t_m)$).  
\end{intuition*}
Our main Theorem \ref{t:mainthm2} not only confirms this intuition but also demonstrates that $\mathcal{D}=\mathcal{K}\cup \mathcal{H}$, the long-wavelength domain, reaches a self-increasing singularity at  $p_m:=(t_m,+\infty,0,\cdots,0)$. In fact, despite considerable effort, Theorem \ref{t:mainthm2}  demonstrates that \textit{only one self-increasing singularity, $p_m$, can be reached} within the long-wavelength domain. As noted in Remark \ref{t:Ptdevp}, in the uncharted domain $\mathcal{I} \setminus \mathcal{K}$, different behaviors and mechanisms may \textit{prevent} the formation of self-increasing singularities,  requiring further investigation using different techniques.

As mentioned in  \S\ref{s:review}, Fuchsian GIVP is a useful tool for  addressing the global existence problem. However, it can \textit{not} be applied \textit{directly} to our current goal. To \textit{employ the Fuchsian-based method} and utilize the tools prepared in \S\ref{s:review} for proving Theorem \ref{t:mainthm2}, we must overcome the following main difficulties (only the primary challenges that impact the core proofs are listed here; other difficulties will be addressed in appropriate sections of the article).

\underline{Difficulty $1$:  The current problem is a partially global existence problem.} Unlike the scenarios discussed in  \S\ref{s:review}, where $\varrho$ remains close to $f$ over the entire $[t_0,t_m)\times \Rbb^n$,  in this case,  $\varrho$ remains close to $f$ only within  the long-wavelength domain  $\mathcal{K}\cup\mathcal{H}$. 
This domain, $\mathcal{K}\cup\mathcal{H}$, extends to the infinity point  $p_m$, where a self-increasing singularity occurs.  
To apply the Fuchsian method, we need to compactify the blowup time $t=t_m$ to $\tau=-1$ (see Question $1$ in \S\ref{s:review}). However, in this situation, we only reach one blowup point,  $p_m$; the other points at $t=t_m$ are hidden within the uncharted domain  $\mathcal{I}\setminus \mathcal{K}$, which we cannot access at this stage. This phenomenon is largely due to the absence of the synchronizing effect, as discussed in \S\ref{s:review}. Consequently, we can not expect  $\varrho$ to remain close to $f$ throughout the entire domain  $[t_0,t_m)\times \Rbb^n$. This limitation prevents us from extending the problem to a fully global existence problem, as in \S\ref{s:review}, and confines us to a global existence only at one self-increasing singularity. 

An intuitive idea to overcoming this difficulty is to \textit{solve a global existence problem for a revised hyperbolic system} and ensure that this revised system is consistent with \ref{Eq2} (or its variant) within a sufficiently large lens-shaped domain  (see Theorem \cite[Theorem $4.5$]{Lax2006}). Outside the lens-shaped domain, we have considerable \textit{flexibility to modify} the hyperbolic equation to align with the Fuchsian formulations, allowing us to use the Fuchsian GIVP for the revised system to achieve a global solution. 

However, rewriting \ref{Eq2} as a singular symmetric hyperbolic  system and modifying the system in the domain outside the lens-shaped domain bring new difficulties.  
As outlined in \S\ref{s:review}, obtaining a Fuchsian system first requires a suitable time compactification, which we will address in Difficulty $2$. Additionally, modifying the domain outside the lens-shaped region to achieve a Fuchsian system presents further difficulties, which we will explore in Difficulty $3$.

\underline{Difficulty $2$: How can time be compactified due to the absence of the synchronizing effect?}  
In the previous \S\ref{s:review}, the presence of the synchronizing source ensures that the blowup times of $\varrho$ and $f$ coincide (i.e., they both blow up at $t=t_m$  and are defined on $[t_0,t_m)\times \Rbb^n$). As a result, the time transformation \eqref{e:ctm2.a} compactifies the blowup time for both $\varrho$ and $f$, mapping $t=t_m$  to $\tau=0$. However, the absence of the synchronizing effect in the current settings disrupts this structure. Consequently, we require two separate time transformations to compactify the time coordinates $t$ and $\mft$ for  $\varrho(t,x)$ and $f(\mft)$, respectively.

Since our focus is on the long-wavelength domain $\mathcal{K}\cup\mathcal{H}$, where we expect $\varrho$ and its derivatives to closely approximate  $f$, we anticipate only a slight modification to the structure described in  \S\ref{s:review}.   Based on this observation, we can still define $\tau=\mfg(\mft)$ using \eqref{e:ctm2.a} and define $\tau=g(t,x)$ by imitating \eqref{e:ctm2.a},
\begin{equation}\label{e:ctm2.b}
	\tau=	g(t,x)= - \Bigl(1+ \frac{2}{3} B  \int^t_{t_0} s^{-\frac{2}{3}} \varrho(s,x)(1+\varrho(s,x))^{-\frac{1}{3}}  ds \Bigr)^{-\frac{3A}{2}} \in[-1,0)   .  
\end{equation}
This transformation is appropriate because: $(1)$ when $\varrho$ is homogeneous  and equal to $f$, \eqref{e:ctm2.b} reduces to \eqref{e:ctm2.a};  and $(2)$ if $\varrho$ blows up at some point, the blowup time in terms of $\tau$ becomes   $\tau=0$. Thus, in the $\tau$ coordinate, we can synchronize the blowup time to  $\tau=0$.

However, unlike in \eqref{e:ctm2.a}, the definition in \eqref{e:ctm2.b} involves the unknown variable $\varrho$ within the transformation itself. This means that the time transformation is defined through an evolution equation. In other words, we must \textit{solve \ref{Eq2} concurrently with the coordinate equation \eqref{e:ctm2.b}}. To facilitate this, it is preferable to convert the integral equation \eqref{e:ctm2.b} into a partial differential equation (PDE) so that it can be solved alongside the wave equation. By differentiating \eqref{e:ctm2.b}, we obtain the following (see \S\ref{s:2} for more details and the ODE expression for $\mfg(\mft)$, which is useful in calculations):
\begin{equation}\label{e:dtgtrs}
	\del{t}g(t,x^i)= 	 	\frac{A B   \varrho(t,x^i) \left(-g(t,x^i)\right){}^{\frac{2}{3 A}+1}}{t^{\frac{2}{3}}   (  \varrho(t,x^i)+1)^{\frac{1}{3}}}  \quad \text{with} \quad 
	g(t_0,x^i) =   -1 .   
\end{equation}

In fact,  \eqref{e:dtgtrs} provides crucial information about the Jacobian of the transformation $\tau=g(t,x)$. However, when converting the wave equation into a singular hyperbolic system, it becomes more practical to use the inverse of the Jacobian. Therefore, we introduce the inverse of $g$, denoted as $\mathsf{b}$, and use $\mathsf{b}$ and its derivative $\mathsf{b}_\zeta:=\del{\zeta}\mathsf{b}$ to characterize the coordinate transformation from the original coordinates $(t,x)$ to the time-compactified coordinates $(\tau,\zeta)$ (see \S \ref{s:2} and Lemma \ref{t:gb1} for details). Moreover, by \eqref{e:dtgtrs}, we derive the evolution equation for $\mathsf{b}_\zeta$, which describes the evolution of the Jacobian. We solve this equation in conjunction with the singular hyperbolic system derived from the wave equation in the $(\tau, \zeta)$ coordinates.
The same discussions applies to the coordinate transformation $\tau=\mfg(\mft)$ as well.

Furthermore, it is preferable to compare $\varrho$ and $f$ in terms of the $\tau$ coordinate rather than $t$, as they share a common blowup time at $\tau = 0$. Thus, when transforming \ref{Eq2} into a singular hyperbolic system, we define the variables for the system based on the differences, such as $u(\tau,\zeta):=(\underline{\varrho}(\tau,\zeta) - \uf(\tau))/\uf(\tau)$ (notations defined in \S\ref{iandc}). 
Further details on the variables can be found in \S\ref{s:sngexp}.

Following a similar spirit to \cite{Liu2022b}, as discussed in \S\ref{s:review}, we are able to derive a singular symmetric hyperbolic system in the $(\tau, \zeta)$ coordinates from the wave equation \ref{Eq2} and the developing equations for the coordinate transform and its Jacobian. It is important to note that these derivations are neither trivial nor straightforward, as they rely on hidden identities established in \S\ref{s:2} and Appendix \ref{t:refsol}. These identities arise from the analysis of $f$, $f_0$, and key quantities such as $\chi_\uparrow$, $1/(f\mfg)$, and $1/(\mfg f^{\frac{1}{2}})$ (as discussed in \S\ref{s:review}.\ref{S4}), as well as from the evolution of the Jacobian.

\S\ref{s:3} provides the derivations, leading to the singular symmetric hyperbolic system given by (simplified form, see Lemma \ref{t:mainsys1} for details)
\begin{equation}\label{e:mainsys1.a}
	\mathbf{A}^0\del{\tau}U+\frac{1}{A\tau} \mathbf{A}^i \del{\zeta^i} U =\frac{1}{A\tau} \mathbf{A}   U +\mathbf{F}(U) . 
\end{equation}
Compared to  \cite{Liu2022b}, the presence of the source term here implies that, at this stage, the system is a Fuchsian system. Therefore, using results for the Fuchsian GIVP, we can conclude that $\varrho$ is close to $f$ on $[t_0,t_m)\times \Rbb^n$. However, in this circumstance, \eqref{e:mainsys1.a} fails to be a Fuchsian system because $\mathbf{A}$ in the singular term $\frac{1}{A\tau} \mathbf{A} U$ is \textit{not} a \textit{positive definite} matrix. 
Recalling the key condition for a Fuchsian system, as discussed in  \S\ref{s:review},   $\mathbf{A}$ must be positive definite (see \eqref{e:Bineq1}, and this condition is closely related to the decay rates). This issue introduces a third difficulty in our analysis. 
Based on the ideas discussed in Difficulty 1, there is currently no obvious way to modify the system outside the lens-shaped domain while still ensuring that this domain includes $p_m$. 

\underline{Difficulty $3$:  Positive definiteness of $\mathbf{B}$ in Fuchsian system \eqref{e:Fucmodel}.}  
The \textit{central idea} for overcoming this difficulty is to introduce a \textit{zoom-in (or blowup) coordinate transformation} and \textit{corresponding zoom-in variables} that effectively ``zooms in'' on the region where $\varrho$ exhibits blow-up behavior.

To introduce the specific forms of these transformations more naturally and clearly, let us first state an idea from the \textit{algebraic perspective} to overcome this difficulty. To convey the basic ideas and provide an intuitive hint, we simplify \eqref{e:mainsys1.a} by setting $n=1$ and adjusting the coefficient matrices\footnote{See Lemma  \ref{t:mainsys1} for the precise coefficient matrices. } (noting that $\ck=1/4$) to: 
\begin{equation*}\label{e:coef1}
	\mathbf{A}^0=\p{1&U\\U&\frac{1}{4}}, \quad \mathbf{A}^1=\p{0& 1\\1 & 0}\AND \mathbf{A}=\p{-\frac{14}{3} & - q  \\
		0 & 1}  . 
\end{equation*}
The same ideas and spirits can be applied to the original system \eqref{e:mainsys1} in Lemma  \ref{t:mainsys1} and after applying the following procedures \ref{F1} and \ref{F2}, the exact form derived from \eqref{e:mainsys1.a} is provided in Lemma \ref{t:sigsys}. 

Firstly, note two useful facts. 
\begin{enumerate}[leftmargin=*,label={(F\arabic*)}]
	\item\label{F1} By introducing a \textit{zoom-in (or blowup) coordinate transformation}\footnote{Intuitively, as shown in  Fig.  \ref{f:fig3a} and Fig. \ref{f:fig3b}, this coordinate transformation helps  to amplify and expand the corner toward   $p_m$ in $\mathcal{K}\cup\mathcal{H}$, that is, it effectively ``zooms in'' on the region where $\varrho$ is blowing up.} defined by $\txi^1=\frac{\alpha}{A} \ln(-\tau) +\zeta^1$ (where $\alpha>0$) and $\ttau=\tau$, we obtain the following transformations:  $\del{\tau}=\del{\ttau}+\frac{\alpha}{A\ttau}\del{\txi^1}$ and $\del{\zeta^1}=\del{\txi^1}$ (see \S\ref{s:riv1} for details). 
	\item\label{F2} By introducing a \textit{variable transformation}  $\mathfrak{U}=e^{\theta \txi^1} U$ ( where $\theta>0$), the derivative transforms as follows:  $\frac{1}{A\ttau}   \del{\txi^1}U=e^{-\theta \txi^1} \frac{1}{A\ttau}   \del{\txi^1}\mathfrak{U}-\theta e^{-\theta \txi^1} \frac{1}{A\ttau}   \mathfrak{U}$ (see \S\ref{s:riv2} for details). 
\end{enumerate}

Combining these two facts, \ref{F1} contributes an additional term $\frac{\alpha}{A\ttau}\mathbf{A}^0 \del{\txi^i}U$ from $\mathbf{A}^0\del{\tau}U$. This term effectively changes the zero diagonal elements in $\mathbf{A}^1$ into nonzero values. Consequently,  \eqref{e:mainsys1.a} becomes 
\begin{align*}
	\mathbf{A}^0\del{\ttau}U+\frac{1}{A\ttau} \underbrace{\bigl(\alpha \mathbf{A}^0 +\mathbf{A}^1\bigr)}_{=\p{\alpha & 1+\alpha U \\1+\alpha U &\frac{1}{4}\alpha}} \del{\txi^1} U =\frac{1}{A\ttau} \mathbf{A}   U +\mathbf{F}(U) . 
\end{align*}
Then \ref{F2} introduces an additional term,  $\frac{1}{A\ttau} \mathfrak{U}$, which moves to the RHS to complement   $\frac{1}{A\ttau} \mathbf{A}   U$.  This adjustment can lead to that the Fuchsian $\mathbf{B}$ is positive definite.  That is, using \ref{F2} in the above equation, we obtain
\begin{align*}
	&\underbrace{\p{1&  e^{-\theta \txi^1}   \mathfrak{U}\\e^{-\theta \txi^1}   \mathfrak{U}&\frac{1}{4}} }_{=:\mathbf{B}^0} \del{\ttau} \mathfrak{U} +  \frac{1}{A\ttau}  \underbrace{\p{\alpha & 1 +\alpha e^{-\theta \txi^1}   \mathfrak{U} \\ 1+\alpha e^{-\theta \txi^1}   \mathfrak{U} &\frac{1}{4}\alpha} }_{=:\mathbf{B}^1} \del{\txi^1}\mathfrak{U}  \notag  \\
	=&\frac{1}{A\ttau} \underbrace{\p{\alpha \theta-\frac{14}{3} & \theta  +\alpha \theta e^{-\theta \txi^1}   \mathfrak{U}- q  \\
		\theta  +\alpha \theta e^{-\theta \txi^1}   \mathfrak{U}& 1+\frac{1}{4}\alpha \theta }   }_{=:\mathbf{B}} \mathfrak{U}  + 	e^{\theta \txi^1} \mathbf{F}(e^{-\theta \txi^1} \mathfrak{U}) . 
\end{align*}
Since $\mathbf{F}(0)=0$, the Taylor expansion of $e^{\theta \txi^1} \mathbf{F}(e^{-\theta \txi^1} \mathfrak{U})$ is given by $a_1 \mathfrak{U}+a_2 e^{-\theta \txi^1}  \mathfrak{U}^2+a_3 e^{-2\theta \txi^1}  \mathfrak{U}^3+\cdots$.
We verify the positivity of  
\begin{equation*}
	\eta^T \mathbf{B}\eta :=\eta^T\p{\alpha \theta-\frac{14}{3} & \theta  +\alpha \theta e^{-\theta \txi^1}   \mathfrak{U}- q  \\
		\theta  +\alpha \theta e^{-\theta \txi^1}   \mathfrak{U}& 1+\frac{1}{4}\alpha \theta }    \eta =\eta^T\p{\alpha \theta-\frac{14}{3} & \theta  +\alpha \theta e^{-\theta \txi^1}   \mathfrak{U}- \frac{q }{2} \\
		\theta  +\alpha \theta e^{-\theta \txi^1}   \mathfrak{U} - \frac{q }{2}  & 1+\frac{1}{4}\alpha \theta }    \eta =: \eta^T \mathbb{B}   \eta . 
\end{equation*}

Neglecting $\alpha \theta e^{-\theta \txi^1}   \mathfrak{U}$ for the moment, for $q\in\Rbb$, $\theta>0$, choosing $\alpha>(\sqrt{36 \theta ^2+9 q^2-36 \theta  q+169}+1)/3 \theta $ ensures that both eigenvalues of  $\mathbb{B}$ are positive, and hence	$\eta^T \mathbf{B}\eta>0$, achieving our goal.  However, the presence of $e^{-\theta \txi^1}$ not only disrupts the positive definiteness of $\mathbf{B}$ but also causes all coefficients and remainder terms to blow up as $\txi^1\rightarrow-\infty$ (which is not allowed for the Fuchsian GIVP), leading to the next difficulty. 

\begin{remark}
	If \ref{F1} is not used, meaning  $\txi^1=\zeta^1$,  the transformation \ref{F2} introduces an additional term $\frac{1}{A\ttau} \mathfrak{U}$, which moves to the RHS to complement $\frac{1}{A\tau} \mathbf{A}   U$. However, the diagonal elements in $\mathbf{A}^1$ are all zero, which can not compensate for the  $-\frac{14}{3}$ term in $\mathbf{A}$. Therefore,  \ref{F1} is necessary to ensure that the diagonal elements are nonzero. After this adjustment, \ref{F2} can be applied as described above. 
\end{remark} 
\begin{remark}
	 There is an additional constraint on the constant $\alpha$ in  \ref{F1}: the new $\ttau$ axis, defined by $\txi^1 = \frac{\alpha}{A} \ln(-\tau) + \zeta^1 = 0$, must be timelike  with respect to the metric $\mathtt{g}_{\alpha\beta}$ (otherwise, Fig. \ref{f:fig3b} fails). 
	 According to the theory of characteristic curves and surfaces (see, for example, \cite[\S$3.4$ and \S$4.2$]{Lax2006}), the normalized matrix $\mathbf{B}^1$—obtained by normalizing $\mathbf{B}^0$ to $\mathds{1}$ and calculating the corresponding $\mathbf{B}^1$, while neglecting the term $\alpha e^{-\theta \txi^1} \mathfrak{U}$ for the moment—must have both positive and negative eigenvalues\footnote{In brief, since the $\ttau$ axis is timelike and enclosed by the characteristic conoid, the normal vectors $(1,\pm y)$ (with $y\geq0$) of the characteristic conoid satisfy that $\mathds{1}\pm y \mathbf{B}^1$ is not invertible (see  \cite[\S$3.4$ and \S$4.2$]{Lax2006}). }. Specifically, $
	\mathbf{B}^1= \frac{1}{A\ttau} \p{\alpha & 2 \\ 2 & \alpha}$
has eigenvalues  $\frac{1}{A\ttau}(\alpha - 2 ) >0$ and $\frac{1}{A\ttau}(\alpha + 2 ) < 0$.  This leads to the condition $0 < \alpha < 2$, which imposes further requirements on $q$ and $\theta$ (we omit the details).   
\end{remark}

\underline{Difficulty $4$:  Singularity of $e^{-\theta \txi^1}$ as $\txi^1 \rightarrow -\infty$.} 
To address this issue, we note that the singularity of $e^{-\theta \txi^1}$   occurs as $\txi^1 = -\infty$ , which is outside the long wavelength domain of interest. Therefore, we can apply the approach from Difficulty $1$ by introducing a cutoff function $\phi(\txi^1)$ to exclude and revise the system near $\txi^1 = -\infty$ (see  \eqref{e:phi1} and \S\ref{s:revFuc} for details). That is, we  revise the system for $\txi^1\leq -\frac{1}{\delta_0}$ (near $\txi^1=-\infty$), where $\delta_0>0$ is sufficiently small. This adjustment transforms the singular symmetric hyperbolic system \eqref{e:mainsys1.a}   into \eqref{e:mainsys5} , which closely resembles the Fuchsian system, except that the domain is not a closed manifold. This introduces a fifth difficulty. Additionally, the revised system does not exactly match the original wave equation \ref{Eq2} due to these modifications.

\underline{Difficulty $5$:  Compactifying the Domain.} 
To address this difficulty, we introduce new coordinates:  $
\hat{\tau}=\ttau\in[-1,0)  $ and $\hat{\zeta}^i= \arctan (\gamma\txi^i)\in\left(-\frac{\pi}{2} ,\frac{\pi}{2} \right)$.  This transformation compactifies the domain (further obtain a torus $\Tbb^n$) and leads to a revised and extended Fuchsian system, as detailed in  \eqref{e:mainsys6} (see \S\ref{s:revFuc1}). 

Once we have the revised and extended Fuchsian system, we can apply the global existence theorem for Fuchsian GIVP (see Theorem \ref{t:fuc} for details). It allows us to conclude that, after transforming back to the $(t,x)$ coordinates, the solution to the revised system $\varrho$, along with their derivatives, remain close to $f$ within $[t_0,t_m)\times \Rbb^n$. However, the revised system is not always consistent with the original wave equation \ref{Eq2}, but they do agree within a lens-shaped domain  (see Fig. \ref{f:fig1}), which corresponds to the long-wavelength domain described in Theorem \ref{t:mainthm2}. This provides a rough outline of the ideas behind the proof of Theorem  \ref{t:mainthm2}.

\subsection{Summaries and outlines}
In this article, we employ four coordinates and three coordinate transformations to convert the wave equation into a Fuchsian system. To clarify, let's summarize the role of each coordinate transformation.
\begin{enumerate}[leftmargin=*,label={(\arabic*)}]
	\item The transformation $(t,x) \rightarrow (\tau,\zeta)$ compactifies the time coordinate, aligning the blow-up time with $\tau=0$ and making the Fuchsian system feasible. 
	\item  The transformations $(\tau,\zeta) \rightarrow (\ttau,\txi)$  and $\mathfrak{U}=e^{\theta \txi^1} U$ ensure that the Fuchsian $\mathbf{B}$ in \eqref{e:Fucmodel} becomes positive definite. 
	\item  The transformation  $(\ttau,\txi) \rightarrow (\htau,\hat{\zeta})$ compactifies the spatial coordinates, obtaining a torus, which is necessary because the Fuchsian GIVP requires the domain to be a closed manifold. 
\end{enumerate}

Let us now \textit{outline the structure} of this article. In \S\ref{s:2}, we introduce a compactified time coordinate transformation $(t,x) \rightarrow (\tau,\zeta)$ and $\mft\rightarrow \tau$  in \S\ref{s:2.1} and \S\ref{s:2.2}, respectively. Additionally, we present some hidden but useful quantities and identities in \S\ref{s:iden1}, derived through these coordinate transformations. These identities play a crucial role in deriving the Fuchsian formulations of \ref{Eq2}, as highlighted in Difficulty $2$. 

In \S\ref{s:3}, beginning with the wave equation \ref{Eq2} and aiming to derive a Fuchsian system, we present the nontrivial derivations that lead to a singular symmetric hyperbolic system \eqref{e:mainsys1.a} (see Difficulty $2$). However, this system falls short of being a true Fuchsian system. We also give the the domain of influence of the inhomogeneous data at \S\ref{s:DoI}. 

In \S\ref{s:4}, we provide a detailed exposition and derivation of the revised and extended Fuchsian system. \S\ref{s:riv1} and \S\ref{s:riv2} focus on the derivations related to Difficulty $3$. \S\ref{s:revFuc}  addresses Difficulty $4$ by detailing the revised Fuchsian system, while \S\ref{s:revFuc1} covers the extended Fuchsian system in relation to Difficulty $5$. \S\ref{s:reorg} reconstructs the solution to the original system by constructing the lens-shaped domain, addressing ideas from Difficulty $1$. Finally, \S\ref{s:verfuc} provides a comprehensive proof for verification of Fuchsian system.

In \S\ref{s:pfmthm}, we focus on the proofs of the main Theorems \ref{t:mainthm1} and \ref{t:mainthm2}.

Appendix \ref{s:pfclm}  contains proofs of propositions from this article, ensuring the continuity of the statements in the main text. Appendix \ref{s:ODE0} compiles results from our previous works \cite{Liu2022b, Liu2023}, which serve as tools for this article. Appendix \ref{s:fuc} provides the foundational statements of the Fuchsian system, a tailored version of the theorem established in \cite{Beyer2020} and originally established by \cite{Oliynyk2016a}.

\subsection{Notations}\label{s:AIN}
Unless stated otherwise, the following notational conventions will be used throughout this article without further clarification in subsequent sections.

\subsubsection{Indices and coordinates}\label{iandc}
Unless stated otherwise, our indexing convention will be as follows: lower case Latin letters (e.g. $i, j,k$) denote spatial indices running from $1$ to $n$,  while lower case Greek letters (e.g. $\alpha, \beta, \gamma$) represent spacetime indices running from $0$ to $n$.   We will adhere to the \textit{Einstein summation convention}, meaning that repeated lower and upper indices are implicitly summed over. The standard coordinates on $\Rbb^n$ or $\mathbb{T}^n$ are denoted by $x^{i}$ (or $\zeta^i$, etc.) for $i=1, \cdots, n$,  and the time coordinate on the interval $[t_0,\infty)$ is denoted by
$t = x^{0}$ (or $\tau=\zeta^0$, etc.). 
We use $\zeta^i$ with a \textit{running index} $i$  (or simply $\zeta$ ) to represent the coordinates of a point $\zeta=(\zeta^1,\cdots,\zeta^n)$, and $|\zeta|$ to denote its length, defined by  $	|\zeta|^2:=\delta_{jk}\zeta^j \zeta^k  =\zeta_k\zeta^k$ where $\zeta_i:=\zeta^j\delta_{ij}$.

This article introduces four coordinate systems: $(t,x)$, $(\tau,\zeta)$, $(\ttau,\txi)$ and $(\hat{\tau},\hat{\zeta})$. Various variables, objects, and maps are defined within each of these coordinate systems, with specific notations used accordingly. Unless otherwise specified, we denote an object or map $Q$ in the $(\tau,\zeta)$ coordinates as  $\underline{Q}$,  in the $(\ttau,\tilde{\zeta})$ coordinates as $\widetilde{Q}$,  and in the $(\htau,\hat{\zeta})$ coordinates as $\widehat{Q}$. For example, we use $\varrho$ in $(\tau,\zeta)$, $\widetilde{\mathbf{A}^0}$ in $(\ttau,\txi)$ and $\widehat{\mathfrak{U}}$ in $(\htau,\hat{\zeta})$. 
Additionally, for instance, the relationships between these variables  are expressed as follows: 
\begin{equation*} %\label{e:udl}
	\underline{\varrho}(\tau,\zeta):=  \varrho\bigl(t(\tau,\zeta), x(\tau,\zeta)\bigr), \quad 
	\widetilde{\mathbf{A}^0}\bigl(\ttau,\txi\bigr)=\mathbf{A}^0\bigl(\tau(\ttau,\txi),\zeta(\ttau,\txi)\bigr)   \AND  \widehat{\mathfrak{U}}\bigl(\htau,\hat{\zeta}\bigr)=\mathfrak{U}\bigl(\ttau(\htau,\hat{\zeta}),\txi(\htau,\hat{\zeta})\bigr)  . 
\end{equation*}

\subsubsection{Remainder terms}\label{s:rmdrs}
Firstly, we say that a function $f(x,y)$ \textit{vanishes to the $n^{\text{th}}$ order in $y$} if it satisfies $f(x,y)\sim \mathrm{O}(y^n)$ as $y\rightarrow 0$. In other words, there exists a positive constant $C$ such that $|f(x,y)|\leq C|y|^n$ as $y \rightarrow 0$.

To simplify the treatment of remainder terms where the exact form is not important, we will, unless otherwise stated, use uppercase script letters, such as  $\mathscr{S}(\tau, U;V)$, $\mathscr{Z}(\tau, U;V)$, $\mathscr{H}(\tau, U;V)$,  to denote
analytic maps of the variable
$(\tau, U, V)$. These maps vanish to the  $1^{\text{st}}$ order in $V$;. The domain of analyticity for these maps will be clear from the context.

\subsubsection{Derivatives}\label{s:der}
Partial derivatives with respect to the  coordinates $(x^\mu)=(t,x^i)$ are denoted by   $\partial_\mu = \partial/\partial x^\mu$.  We use 
$Du=(\partial_j u)$  to represent the spatial gradient and $\partial u = (\partial_\mu u)$ for the spacetime gradient. 
Greek letters are also used to denote multi-indices, e.g., 
$\alpha = (\alpha_1,\alpha_2,\ldots,\alpha_n)\in \mathbb{Z}_{\geq 0}^n$, with the standard notation  $D^\alpha = \partial_{1}^{\alpha_1} \partial_{2}^{\alpha_2}\cdots
\partial_{n}^{\alpha_n}$ for spatial partial derivatives. The context will clarify whether a Greek letter represents a spacetime coordinate index or a multi-index.

\subsubsection{Function spaces, inner-products and matrix inequalities}\label{s:funsp}
Given a finite dimensional vector space $V$, we denote by
$H^s(\mathbb{R}^n,V)$
the space of maps from $\mathbb{R}^n$ to $V$ with $s$ derivatives belongs to $L^2(\Rbb^n)$, where $s\in \mathbb{Z}_{\geq 0}$. When the
vector space $V$  is understood from the context--such as when  $V=\Rbb^N$--we simplify the notation to  $H^s(\mathbb{R}^n)$ instead of $H^s(\mathbb{R},V)$.
We define the standard $L^2$ inner product by
\begin{equation*}
	\langle{u,v\rangle} = \int_{\mathbb{R}^n} (u(x),v(x))\, d^n x,
\end{equation*}
where $(\cdot,\cdot)$
denotes the Euclidean inner product on $\Rbb^N$ (i.e., $(\xi,\zeta)=\xi^T\zeta$ for any $\xi, \zeta\in \Rbb^N$).  The $H^s$ norm is then defined by
\begin{equation*}
	\|u\|_{H^s}^2 = \sum_{0\leq |\alpha|\leq s} \langle D^\alpha u, D^\alpha u \rangle.
\end{equation*}

For matrices $A,B\in \mathbb{M}_{N\times N}$, we define
\begin{equation*}
	A\leq B \quad \Leftrightarrow \quad  (\zeta,A\zeta)\leq (\zeta,B\zeta),  \quad \forall \zeta\in \Rbb^N.
\end{equation*}

%%%--------------NEW SEC--------------

\section{Compactified time transformations and analysis of the reference equation}
\label{s:2} 
\subsection{Compactified time transformations}\label{s:2.1}
As discussed in \S\ref{s:oview} (particularly,  Difficulty $2$), our goal is to derive a Fuchsian system. To achieve this, we first introduce a compactified time coordinate system $(\tau, \zeta^i)$ defined by: 
\begin{equation}\label{e:coord2}
\tau =g(t,x^i) \AND 	\zeta^i=x^i
\end{equation}
where the function $g(t, x^i)$ is determined by the solution of the following Cauchy problem: 
\begin{align}
	\del{t}g(t,x^i)= 	&	\frac{A B   \varrho(t,x^i) \left(-g(t,x^i)\right){}^{\frac{2}{3 A}+1}}{t^{\frac{2}{3}}   (  \varrho(t,x^i)+1)^{\frac{1}{3}}} ,  \label{e:tmeq1}   \\
	g(t_0,x^i) = & -1 ,  \label{e:tmeq2} 
\end{align}
and $0<A<2$ and $B:= (1+\mf)^{\frac{4}{3}}/( t_0^{\frac{4}{3}} \mf_0)>0 $ are constants. 
Note the compactified time $\tau=g(t,x)$ is defined in conjunction with solving for   $\varrho$. Therefore, the coordinates $(\tau,\zeta)$ and $\varrho$ are defined simultaneously by solving the coupled  system \ref{Eq2} and \eqref{e:tmeq1}--\eqref{e:tmeq2}.

\begin{lemma}\label{t:gb1}
	If the coordinate  $(\tau,\zeta^i)$ is given by \eqref{e:coord2}, $g$ is defined by \eqref{e:tmeq1}--\eqref{e:tmeq2} and $0<\varrho \in C^2([t_0,t_\star]\times \Rbb)$ solves the main equation \ref{Eq2}, then 
	\begin{enumerate}[leftmargin=*,label={(\arabic*)}]
		\item\label{l:1.1} $\tau=g(t,x^i)$ can be expressed as
		\begin{equation*}\label{e:tmdef}
			\tau=	g(t,x^i)= - \Bigl(1+ \frac{2}{3} B  \int^t_{t_0} s^{-\frac{2}{3}} \varrho(s,x^i)(1+\varrho(s,x^i))^{-\frac{1}{3}}  ds \Bigr)^{-\frac{3A}{2}} \in[-1,0)   .  
		\end{equation*}
	\item\label{l:1.2} the inverse transformation of \eqref{e:coord2} is
	\begin{equation}\label{e:coordi2}
		  t=\mathsf{b}(\tau,\zeta^i) \AND x^i=\zeta^i
	\end{equation}
	where $\mathsf{b}(\tau,\zeta^i)$ satisfies the following ODE:
	\begin{align}\label{e:tmeqi1}
		\del{\tau} \mathsf{b} (\tau, \zeta^i) = & \frac{\mathsf{b}^{\frac{2}{3}} (\tau, \zeta^i)  ( 1 + \underline{\varrho}(\tau,\zeta^i))^{\frac{1}{3}}}{A B  \underline{\varrho}(\tau,\zeta^i) \left(-\tau\right){}^{\frac{2}{3 A}+1}}   ,   \\
		\mathsf{b}(-1,\zeta^i)= & t_0 . 
	\end{align} 
\item\label{l:1.3} denoting $\mathsf{b}_i:=\del{\zeta^i}\mathsf{b}$, $\mathsf{b}_i$ satisfies
	\begin{equation}\label{e:bzeq1}
	\del{\tau} \mathsf{b}_{i}	
	=   \frac{2 (-\tau )^{-\frac{2}{3 A}-1} \mathsf{b}_i \left(\underline{\varrho } +1\right)^{\frac{1}{3}}}{3 A B  \mathsf{b}^{\frac{1}{3}} \underline{\varrho } }       -\frac{(-\tau )^{-\frac{2}{3 A}-1} \mathsf{b}^{\frac{2}{3}} \del{\zeta^i} \underline{\varrho }   \left(\underline{\varrho } +1\right)^{\frac{1}{3}}}{A B  \underline{\varrho }^2}   +\frac{(-\tau )^{-\frac{2}{3 A}-1} \mathsf{b}^{\frac{2}{3}} \del{\zeta^i} \underline{\varrho }  }{3 A B  \underline{\varrho }  \left(\underline{\varrho } +1\right)^{\frac{2}{3}}}   . 
\end{equation}  
\item\label{l:1.4} for any function $F(t,x^i)$, there are relations
\begin{equation*}
	\underline{\del{t}F}=    \frac{A B  \underline{\varrho}  \left(-\tau\right){}^{\frac{2}{3 A}+1}}{\mathsf{b}^{\frac{2}{3}}   ( \underline{\varrho} +1)^{\frac{1}{3}}}\del{\tau} \underline{F}   \AND \underline{\del{x^i}F} =  -  \frac{A B  \underline{\varrho}  \left(-\tau\right){}^{\frac{2}{3 A}+1} \mathsf{b}_i  } {\mathsf{b}^{\frac{2}{3}}   ( \underline{\varrho} +1)^{\frac{1}{3}}} \del{\tau} \underline{F}     + \del{\zeta^i} \underline{F}  .
\end{equation*}
	\end{enumerate}	 
\end{lemma}
\begin{proof}
	Firstly, solving the ODE system \eqref{e:tmeq1}--\eqref{e:tmeq2} directly yields the result in \ref{l:1.1}. 
	
	To prove \ref{l:1.2}, we apply the inverse mapping theorem. We begin by calculating the Jacobian of the transformation  \eqref{e:coord2} (slightly abusing  notation),
	\begin{equation*}
		\p{  \partial_t \tau   &  \partial_{x^i} \tau  \\
	\partial_t \zeta^j  &  \partial_{x^i} \zeta^j   }= \p{ \partial_t g   &  \partial_{x^i} g    \\ 0 & \delta^j_i}   . 
	\end{equation*}
Then the determinant of this Jacobian matrix is $\del{t}g  $. By using \eqref{e:tmeq1} and \ref{l:1.1}, we obtain $\del{t}g>0$ for any $(t,r)\in [t_0,t_\star]\times \Rbb$ and $g\in C^1$. By the inverse mapping theorem, there is an inverse transformation given by \eqref{e:coordi2}, where $\mathsf{b}\in C^1$, and the Jacobian of the inverse transformation \eqref{e:coordi2} is 
\begin{equation}\label{e:Joc}
	\p{\del{\tau} t  &   \del{\zeta^i} t \\ \del{\tau} x^j &   \del{\zeta^i} x^j}=\p{\del{\tau} \mathsf{b}  &   \del{\zeta^i} \mathsf{b}  \\ 0 &   \delta^j_i}=\underline{ \p{ \partial_t g   &  \partial_{x^j} g    \\ 0 & \delta^l_j}^{-1}} =\left(
	\begin{array}{cc}
		\frac{1}{\underline{\partial_t g }} & -\frac{\underline{\partial_{x^i}  g} }{\underline{\partial_t g} } \\
		0 & \delta^j_i \\
	\end{array}
	\right)
\end{equation}
and 
\begin{align*}
	\p{ \underline{\partial_t g}   &  \underline{\partial_{x^i} g}    \\ 0 & \delta^j_i}= \p{\del{\tau} \mathsf{b}  &   \del{\zeta^j} \mathsf{b}  \\ 0 &   \delta^l_j }^{-1}  = \left(
	\begin{array}{cc}
		\frac{1}{\del{\tau} \mathsf{b}  } & -\frac{\del{\zeta^i} \mathsf{b} }{\del{\tau} \mathsf{b}  } \\
		0 & \delta^j_i \\
	\end{array}
	\right)  .  
\end{align*}
Then
\begin{equation}\label{e:dtb}
\del{\tau} \mathsf{b}=\frac{1}{\underline{\del{t}g}}, \quad \del{\zeta^i} \mathsf{b}=-\frac{\underline{\partial_{x^i} g} }{\underline{\partial_t g }}, \quad  \underline{\partial_t g }=	\frac{1}{\del{\tau} \mathsf{b}  }  \AND   \underline{\partial_{x^i} g  }= -\frac{\del{\zeta^i} \mathsf{b} }{\del{\tau} \mathsf{b}  } . 
\end{equation}
By \eqref{e:tmeq1}, we arrive at \eqref{e:tmeqi1}.  Then by using \eqref{e:tmeq2},   $t_0=\mathsf{b}(g(t_0,x^j),\zeta^i)=\mathsf{b}(-1,\zeta^i)$.

	\ref{l:1.3} is obtained by directly differentiating \eqref{e:tmeqi1} with respect to $\zeta^i$. 
	
	Let us prove \ref{l:1.4} at the end. By the chain rule, along with \eqref{e:tmeqi1} and \eqref{e:dtb}, we arrive at
	\begin{align*}
		\underline{\del{t}F} = &    \frac{1}{\del{\tau}\mathsf{b} } \del{\tau} \underline{F}   =      \frac{A B  \underline{\varrho}  \left(-\tau\right){}^{\frac{2}{3 A}+1}}{\mathsf{b}^{\frac{2}{3}}   ( \underline{\varrho} +1)^{\frac{1}{3}}}\del{\tau} \underline{F}  ,   \\
	\underline{	\del{x^i}F } = &    - \frac{\del{\zeta^i}\mathsf{b}}{\del{\tau}\mathsf{b} } \del{\tau} \underline{F}     + \del{\zeta^i} \underline{F}  =-  \frac{A B  \underline{\varrho}  \left(-\tau\right){}^{\frac{2}{3 A}+1} \mathsf{b}_i  } {\mathsf{b}^{\frac{2}{3}}   ( \underline{\varrho} +1)^{\frac{1}{3}}} \del{\tau} \underline{F}     + \del{\zeta^i} \underline{F}   .
	\end{align*}
	This completes the proof.   
\end{proof}

\subsection{Time transformations for  reference solutions}\label{s:2.2}
For the reference equation \eqref{e:feq0b}--\eqref{e:feq1b}, we introduce a compactified time transformation which was presented in our previous work \cite[eq. $(1.13)$]{Liu2022b}. 
\begin{equation}\label{e:ctm2}
	\tau=\mathfrak{g}(\mft)=- \Bigl(1+ \frac{2}{3} B  \int^{\mft}_{t_0} s^{-\frac{2}{3}} f(s)(1+f(s))^{-\frac{1}{3}}  ds \Bigr)^{-\frac{3A}{2}}  \overset{\text{Lemma \ref{t:gmap}}}{\in}[-1,0)  . 
\end{equation}

\begin{lemma}\label{t:gb2}
	Suppose $f\in C^2([t_0,t_1))$ ($t_1>t_0$) solves the reference ODE \eqref{e:feq0b}--\eqref{e:feq1b}, $\mfg(t)$ is defined by \eqref{e:ctm2}, and  denote $f_0(\mft):=\del{\mft} f(\mft)$, then 
	\begin{enumerate}[leftmargin=*,label={(\arabic*)}]
		\item\label{l:2.1} $	\tau=\mfg(\mft)$ can be expressed by 
		\begin{equation*}
			\tau=\mfg(\mft)=-\exp\Bigl(-A\int^{\mft}_{t_0} \frac{f(s)(f(s)+1)}{s^2 f_0(s)} ds \Bigr)<0. 
		\end{equation*}
	\item\label{l:2.2} $\mfg(\mft)$ solves 
		\begin{align*}%\label{e:dtg1}
		\del{\mft}\mfg(\mft)= &-A \mfg(\mft)   \frac{f(\mft)(f(\mft)+1)}{\mft^2f_0(\mft)}= \frac{A B f(\mft) (-  \mfg(\mft) )^{1+\frac{2}{3A}}  }{  \mft^{\frac{2}{3}}  (1+f(\mft))^{\frac{1}{3}}} , \\
		\mfg(t_0) = &   -1 .
	\end{align*}
\item\label{l:2.3} the inverse transformation of \eqref{e:ctm2} is
\begin{equation*}%\label{e:ctmi2}
 \mft=\mathsf{b}_\uparrow(\tau) ,  
\end{equation*}
where $\mathsf{b}_\uparrow(\tau) $ satisfies an ODE
\begin{align}\label{e:dtbup1} 
		\del{\tau}\mathsf{b}_\uparrow (\tau) = &  \frac{\mathsf{b}_\uparrow^{\frac{2}{3}}(\tau) (1+\underline{f}(\tau) )^{\frac{1}{3}}}{AB  \underline{f}(\tau)  (-\tau)^{\frac{2}{3A}+1}} , \\
		\mathsf{b}_\uparrow(-1)= & t_0  . 
\end{align}  
\item\label{l:2.4} in terms of the compactified time $\tau$,   the reference ODE \eqref{e:feq0b}--\eqref{e:feq1b} becomes  
\begin{equation*}\label{e:dtf0eq1} 	\del{\tau}\underline{f_0} + \frac{4}{3}  \frac{\mathsf{b}_\uparrow^{-\frac{1}{3}} (1+ \uf)^{\frac{1}{3}}}{A B  \uf (-\tau)^{\frac{2}{3A}+1}} \underline{f_0}  -\frac{2}{3  }  \frac{\mathsf{b}_\uparrow^{-\frac{4}{3}} (1+ \uf )^{\frac{4}{3}} }{A B  (-\tau)^{\frac{2}{3A}+1}} -\frac{4}{3 } \frac{\mathsf{b}_\uparrow^{\frac{2}{3}}  \underline{f_0}^2 (1+ \uf )^{-\frac{2}{3}}}{A B  \uf (-\tau)^{\frac{2}{3A}+1}}= 0   
\end{equation*} 
where we denote $\uf(\tau):=f\circ\mathsf{b}_\uparrow(\tau)$ and $\ufo(\tau):=f_0\circ\mathsf{b}_\uparrow(\tau)$. 
\item\label{l:2.5} $\del{\tau}\uf $ and $\underline{f_0}$ have the relation
\begin{equation*}%\label{e:dtfeq1}
	\del{\tau} \uf  = \frac{\mathsf{b}_\uparrow^{\frac{2}{3}} (1+ \uf )^{\frac{1}{3}}}{A B  \uf   (-\tau)^{\frac{2}{3A}+1}}	\underline{f_0}  . 
\end{equation*}
\item\label{l:2.6} in terms of the compactified time $\tau$, 
\begin{equation*}%\label{e:f0frl}
	\underline{f_0}(\tau) =B^{-1} (\mathsf{b}_{\uparrow}(\tau))^{-\frac{4}{3}} (-\tau)^{-\frac{2}{3A}} (1+ \uf(\tau))^{\frac{4}{3} } >0  . 
\end{equation*} 
	\end{enumerate}
\end{lemma}
\begin{proof}
	Firstly, \ref{l:2.1}, \ref{l:2.2} are directly from Lemma \ref{t:f0fg} and Remark \ref{t:dtg1}. 
	
	To prove \ref{l:2.3},  we  use the \textit{inverse function theorem}.  Since $\mathfrak{g}\in C^1$ and $\del{\mft}\mfg>0$ by \eqref{e:ctm2} and \ref{l:2.2}.  The inverse function theorem implies there is a inverse function $\mft=\mathsf{b}_\uparrow(\tau)$ such that $\tau= \mfg\circ\mathsf{b}_\uparrow(\tau)$, $\mathsf{b}_\uparrow \in C^1$ and $\del{\tau} \mathsf{b}_\uparrow =(\underline{\del{\mft}\mfg })^{-1}$. Then with the help of \ref{l:2.2}, we conclude \eqref{e:dtbup1}.  Then by \eqref{e:ctm2} (i.e., $\tau=\mfg(t_0)=-1$), we obtain $t_0=\mathsf{b}_\uparrow(\mfg(t_0))=\mathsf{b}_\uparrow(-1)$.

\ref{l:2.4} can be derived by the chain rule $\del{\tau} \underline{f_0} = \del{\tau} \mathsf{b}_\uparrow \underline{\partial_{\mft}^2f}$ and \eqref{e:feq0b}.

\ref{l:2.5} can be reached by the chain rule $\del{\tau} \uf = \del{\tau} \mathsf{b}_\uparrow \underline{\del{\mft}f}$ and \eqref{e:dtbup1}.

\ref{l:2.6} can be obtained by rewriting Lemma \ref{t:f0fg}.$(1)$ in terms of $\tau$.  It completes the proof. 
\end{proof}

\subsection{Useful identities}\label{s:iden1}
Let us first define an important quantity $\chi(\mft)$ which was  introduced in our previous work \cite{Liu2022b} (see \eqref{e:Gdef0} in Appendix \ref{t:refsol} for its general form), 
\begin{equation*}\label{e:Gdef1}
	\chi_\uparrow(\mft):=\frac{\mft^{\frac{2}{3}} f_0(\mft)}{(1+f(\mft))^{\frac{2}{3}} f(\mft) (-\mfg(\mft))^{\frac{2}{3A}}} \overset{\eqref{e:f0aa}}{=} \frac{  (-\mfg(\mft) )^{-\frac{4}{3A}}\mft^{-\frac{2}{3}}}{B f(\mft) (1+f(\mft))^{-\frac{2}{3}}}  .
\end{equation*}
In terms of the compactified time $\tau$, it becomes
\begin{equation}\label{e:Gdef2}
	\underline{\chi_\uparrow}(\tau):=\frac{\mathsf{b}_\uparrow^{\frac{2}{3}} \underline{f_0}(\tau)}{(1+\uf(\tau))^{\frac{2}{3}} \uf(\tau) (-\tau)^{\frac{2}{3A}}} = \frac{  (-\tau )^{-\frac{4}{3A}}\mathsf{b}_\uparrow^{-\frac{2}{3}}}{B \uf(\tau) (1+\uf(\tau))^{-\frac{2}{3}}}  .
\end{equation}
We highlight an important and frequently used property of $\chi$ is given by Proposition \ref{t:limG}.

\begin{lemma}\label{t:iden1}
		Suppose $f\in C^2([t_0,t_1))$ (with $t_1>t_0$) solves the reference ODE \eqref{e:feq0b}--\eqref{e:feq1b}, let $\mfg(t)$ be defined by \eqref{e:ctm2}, and denote $f_0(\mft):=\del{\mft} f(\mft)$, then  the following identities hold: 
	\begin{gather}
		\frac{ \mathsf{b}_\uparrow^{-\frac{1}{3}} (1+\uf)^{\frac{1}{3}}   }{A B \uf  (-\tau)^{\frac{2}{3A}+1}  }    =  \frac{ \underline{\chi_\uparrow}^{\frac{1}{2}} }{A (-\tau) B^{\frac{1}{2}} \uf^{\frac{1}{2}}    }  ,     \label{e:iden3} \\
		\frac{1}{\underline{f_0}}   \frac{ \mathsf{b}_\uparrow^{-\frac{4}{3}} (1+\uf)^{\frac{4}{3}}   }{A B  (-\tau)^{\frac{2}{3A}+1}  }=       \frac{    1   }{A   (-\tau)   } , \label{e:iden1} \\
		\frac{ \mathsf{b}_\uparrow^{\frac{2}{3}} (1+\uf)^{-\frac{2}{3}}   }{A B \uf  (-\tau)^{\frac{2}{3A}+1}  }\underline{f_0} =  
		\frac{  \underline{\chi_\uparrow}(\tau)     }{A B    (-\tau)  }  ,  \label{e:iden2}  \\
		\frac{ \mathsf{b}_\uparrow  \underline{f_0} }{(1+\uf) }=\frac{\underline{\chi_\uparrow}^{\frac{1}{2}}}{B^{\frac{1}{2}} } \uf^{\frac{1}{2}}  . \label{e:keyid3}
	\end{gather}
\end{lemma}
\begin{proof}
	These identities can be obtained by direct calculations with the help of \eqref{e:Gdef2} and Lemma \ref{t:gb2}\ref{l:2.6}. We omit the details. 
\end{proof}

%%%--------------NEW SEC-----------------

\section{Singular symmetric hyperbolic formulations}\label{s:3}
In this section, we begin by rewriting the main equation \ref{Eq2} in terms of the compactified coordinates $(\tau,\zeta)$ to obtain a first-order \textit{singular symmetric hyperbolic system} (singular at $\tau=0$). This transformation is a crucial step toward deriving the Fuchsian formulations of the system. Although we omit details, this section still involves complex calculations due to the nontrivial transformations and the use of hidden relations.

\subsection{Derivation of the singular expression for  \ref{Eq2}}\label{s:sngexp}
We start with introducing the following variables:
\begin{align}
	u(\tau,\zeta^k) = 	\frac{\underline{\varrho}(\tau,\zeta^k)- \uf(\tau) }{\uf(\tau )}  \quad \Leftrightarrow & \quad   	\underline{\varrho}(\tau,\zeta^k)= 
	\uf(\tau) + \uf(\tau )u(\tau,\zeta^k)    ,\label{e:v1}  
	\\
	u_0(\tau,\zeta^k)   =  	\frac{ \underline{\varrho_0}(\tau,\zeta^k)  - \underline{f_0}(\tau) }{ \underline{f_0}(\tau) }   \quad \Leftrightarrow & \quad  	\underline{\varrho_0}(\tau,\zeta^k)  =    
	\underline{f_0}(\tau)  +\underline{f_0}(\tau)   u_0(\tau,\zeta^k)     
	, \label{e:v2}  \\
	u_i(\tau,\zeta^k)  =  \frac{ \underline{\varrho_i} (\tau,\zeta^k) }{1+\underline{f}(\tau)}   \quad \Leftrightarrow & \quad  
	\underline{\varrho_i} (\tau,\zeta^k) =     (1+\underline{f}(\tau)) u_i(\tau,\zeta^k)   
	, \label{e:v3}   \\
	z(\tau,\zeta^k) =      \biggl(\frac{\mathsf{b}(\tau,\zeta^k)}{\mathsf{b}_\uparrow(\tau)}\biggr)^{\frac{1}{3}}-1 \quad \Leftrightarrow & \quad 	\mathsf{b}(\tau,\zeta^k)  =   \bigl(1+ z (\tau,\zeta^k) \bigr)^3 \mathsf{b}_\uparrow(\tau) 
	, \label{e:v4}   \\
	\mathcal{B}_{j} (\tau,\zeta^k)   =   	\frac{(\mathsf{b}_{\uparrow}(\tau))^{-\frac{2}{3}}(1+\uf(\tau))^{\frac{2}{3}}}{  B^{-1}  (-\tau)^{-\frac{2}{3A}}  }     \mathsf{b}_j(\tau,\zeta^k)  \quad \Leftrightarrow &  \quad 	\mathsf{b}_j(\tau,\zeta^k)= \frac{  B^{-1}  (-\tau)^{-\frac{2}{3A}}  }  {(\mathsf{b}_{\uparrow}(\tau))^{-\frac{2}{3}}(1+\uf(\tau))^{\frac{2}{3}}}	\mathcal{B}_j (\tau,\zeta^k)  . \label{e:v5}  
\end{align}
where $\varrho_\mu:=\del{x^\mu}\varrho$.  
The following lemma provides a relation that will be frequently used later.

\begin{lemma}\label{t:id0}
	Let $\ell_0,\ell_1,\ell_2,\ell_3$ be constants, then
\begin{equation}\label{e:id1}
	\frac{    \mathsf{b}^{\frac{\ell_1}{3}} (1+\uf+\uf u)^{\ell_2}}{A B  (\uf+\uf u)^{\ell_0} (-\tau)^{\frac{\ell_3}{3A}+1}}
	=      \frac{   \mathsf{b}_\uparrow^{\frac{\ell_1}{3}}   (1+\uf)^{\ell_2}   }{A B \uf^{\ell_0} (-\tau)^{\frac{\ell_3}{3A}+1}}  \frac{ (   1+ z  )^{\ell_1} (1+\frac{\uf}{1+\uf} u)^{\ell_2}  }{ (1+ u)^{\ell_0}  }   . 
\end{equation}	
Moreover, there is an identify
\begin{align}\label{e:id3}
		\frac{  B^{-1}  (-\tau)^{-\frac{2}{3A}}  }  {(\mathsf{b}_{\uparrow}(\tau))^{-\frac{2}{3}}(1+\uf(\tau))^{\frac{2}{3}}}=\frac{\mathsf{b}_\uparrow^2(\tau) \underline{f_0}(\tau)}{(1+\uf(\tau))^2}	 = \frac{\underline{\chi_\uparrow}(\tau) }{B  } \frac{ \uf (\tau) }{\ufo (\tau )} . 
\end{align}
Furthermore, \eqref{e:id3} implies
\begin{equation}\label{e:id2}
	\mathsf{b}_j(\tau,\zeta^k)= 	\frac{\mathsf{b}_\uparrow^2(\tau) \underline{f_0}(\tau)}{(1+\uf(\tau))^2}	 \mathcal{B}_j (\tau,\zeta^k) = \frac{\underline{\chi_\uparrow}(\tau) }{B  } \frac{ \uf (\tau) }{\ufo (\tau )}\mathcal{B}_j (\tau,\zeta^k). 
\end{equation}
\end{lemma} 
\begin{proof}
	 Identity \eqref{e:id1} can be verified by direct calculations. Identities \eqref{e:id3} and \eqref{e:id2} can be verified by Lemma \ref{t:gb2}.\ref{l:2.6} and \eqref{e:keyid3} from Lemma \ref{t:iden1}. 
\end{proof}

To present the derivations of the singular system concisely, we first collect several important quantities that will arise in subsequent calculations and provide their identities in Lemma \ref{t:coef1}. We define the following frequently used quantities:
\begin{gather}
	\mathscr{R}^j :=	 \underline{	\cg} \delta^{ij} \mathsf{b}_{i} \frac{(1+\underline{f} ) }{\underline{f_0} } \overset{\eqref{e:id2}}{=} \underline{	\cg} \delta^{ij}  \frac{\uf(1+\underline{f} ) }{\underline{f_0}^2 } \frac{\underline{\chi_\uparrow}  }{B  } \mathcal{B}_i , \quad 	H^{ij} :=  \underline{\cg} \delta^{ij}    \frac{\mathsf{b}_\uparrow^{2} }{\uf}       \frac{ \bigl( 1 +     z \bigr)^{2} \Bigl(1+\frac{\uf}{1+\uf} u \Bigr)^{\frac{1}{3}}  }{ (1+ u) } ,   \label{e:Rdef}  
	\\
	\mathscr{S}=\mathscr{S}(\tau):=   \frac{ (\ck-\cm^2) }{\tau}  \frac{B}{\underline{\chi_\uparrow}} \biggl(  \frac{4}{ \uf}     -  \frac{\underline{\mathfrak{G}}}{ B} \biggr)  , \quad  S=S(\tau):=  \ck  +  \tau \mathscr{S}  , \label{e:S1} \\  
	\mathscr{L}=	\mathscr{L}(\tau;u_0,u,z)  :=   \cm^2   \left(  \frac{  (1 +   u_0  )^2}{ \bigl(1+ \frac{\uf}{1+\uf } u  \bigr)^2} -1\right)  + 4(\ck-\cm^2)  \frac{B}{\underline{\chi_\uparrow}} \biggl(1+\frac{1}{\uf}\biggr) \left(\frac{\bigl(1+\frac{\uf}{1+\uf}  u \bigr)  }{(1+ z  )^6}-1\right)  .   \label{e:L1}   
\end{gather}

\begin{lemma}\label{t:coef1} 
	Under the definitions \eqref{e:Rdef}--\eqref{e:L1}, there are expansion in terms of the variables \eqref{e:v1}--\eqref{e:v5}, 
	\begin{align*}
		\mathscr{R}^j=&\mathscr{R}^j(\tau,u_0,u,z;\mathcal{B}_k ):= R \mathcal{B}_i \delta^{ij} = (S+\mathscr{L}) \frac{\uf}{1+\uf} \frac{\underline{\chi_\uparrow}}{B} \mathcal{B}_i \delta^{ij}  ,  \\
		H^{ij}=&H^{ij}(\tau,u_0,u,z) :=   S \delta^{ij} \frac{\underline{\chi_\uparrow}}{B}  + \mathscr{H} \delta^{ij} ,  
	\end{align*}
	where
	\begin{align}		
		R :=& R(\tau,u_0,u,z)= (S+\mathscr{L} ) \frac{\uf}{1+\uf} \frac{\underline{\chi_\uparrow}}{B},  \label{e:R1}\\
		\mathscr{H} :=	&\mathscr{H}(\tau; u_0,u, z)= \mathscr{L}     \frac{\underline{\chi_\uparrow}}{B} + (S+\mathscr{L}  )  \frac{\underline{\chi_\uparrow}}{B} \left( \frac{ \bigl( 1 +     z \bigr)^{2} (1+\frac{\uf}{1+\uf} u)^{\frac{1}{3}}  }{ (1+ u) }-1\right), 
	\end{align}
	and $\mathscr{S}(\tau)$ and $S(\tau)$ satisfies $\mathscr{S}(\tau)<-\frac{(\ck-\cm^2)}{\tau}$ and $\cm^2<S(\tau)\leq \ck\bigl(1+ \frac{1}{\beta}\bigr)$. 
\end{lemma}
\begin{proof}
	Using \eqref{e:Fdef}, \eqref{e:v1}--\eqref{e:v5} and \eqref{e:keyid3} from Lemma \ref{t:iden1}, we calculate $\underline{\cg}  $, 
	\begin{align}\label{e:grrexp}
		\underline{	\cg} 
		= &  \frac{\cm^2 \underline{\varrho_0}^2}{(1+\underline{\varrho} )^2}  + 4(\ck-\cm^2) \frac{1+\underline{\varrho}  }{\mathsf{b}^2}  =   \frac{  \cm^2\underline{f_0}^2 (1 +   u_0  )^2}{(1+\uf )^2 \bigl(1+ \frac{\uf}{1+\uf } u  \bigr)^2}  + 4(\ck-\cm^2) \frac{1+\uf+\uf u  }{\mathsf{b}^2}   \notag  \\
		= & \frac{ \cm^2 \underline{\chi_\uparrow} \uf }{B \mathsf{b}_\uparrow^2}  \frac{  (1 +   u_0  )^2}{ \bigl(1+ \frac{\uf}{1+\uf } u  \bigr)^2}  + 4(\ck-\cm^2) \frac{(1+\uf)}{ \mathsf{b}_\uparrow^2}  \frac{\bigl(1+\frac{\uf}{1+\uf}  u \bigr)  }{(1+ z  )^6}    . 
	\end{align}
	For later use, we note an expansion based on Proposition \ref{t:limG}, 
	\begin{align}\label{e:chig}
		\frac{\underline{\chi_\uparrow}}{B}  =4+\frac{\underline{\mathfrak{G}}}{B}  . 
	\end{align}	
Before proceeding,  note that by substituting the expression \eqref{e:grrexp} into the following terms, and using \eqref{e:keyid3} and \eqref{e:chig}, we obtain
\begin{align}\label{e:tm1}
	\delta_{jk}\underline{\cg} \delta^{ij}\Bigl(\frac{1+\uf}{\ufo}\Bigr)^2  
	= & \delta_{jk}   \biggl(  \cm^2    \frac{  (1 +   u_0  )^2}{ \bigl(1+ \frac{\uf}{1+\uf } u  \bigr)^2}  + 4(\ck-\cm^2) \frac{B}{\underline{\chi_\uparrow }} \frac{(1+\uf)}{\uf} \frac{\bigl(1+\frac{\uf}{1+\uf}  u \bigr)  }{(1+ z  )^6} \biggr) \delta^{ij}  \notag  \\
	= &	 \delta^{i}_k \biggl(\ck+ \cm^2   \biggl(  \frac{  (1 +   u_0  )^2}{ \bigl(1+ \frac{\uf}{1+\uf } u  \bigr)^2} -1\biggr)  + 4(\ck-\cm^2)  \frac{B}{\underline{\chi_\uparrow}} \biggl(1+\frac{1}{\uf}\biggr) \biggl(\frac{\bigl(1+\frac{\uf}{1+\uf}  u \bigr)  }{(1+ z  )^6}-1\biggr)  \notag  \\
	&   +  (\ck-\cm^2)    \biggl(\Bigl(1+\frac{\underline{\mathfrak{G}}}{4B}\Bigr)^{-1} -1 \biggr)+ 4(\ck-\cm^2)  \frac{B}{\underline{\chi_\uparrow}}   \frac{1}{\uf}  \biggr)  \notag  \\ 
	= &	  (\ck  +\mathscr{L}+\tau\mathscr{S})   \delta^{i}_k   . 
\end{align}	
Directly using this \eqref{e:tm1} yields
\begin{equation*} 
	\mathscr{R}^j=\underline{\cg} \delta^{ij}\Bigl(\frac{1+\uf}{\ufo}\Bigr)^2  \frac{\uf}{1+\uf} \frac{\underline{\chi_\uparrow}}{B} \mathcal{B}_j=(\ck+\mathscr{L}+\tau\mathscr{S}) \delta^{ij} \frac{\uf}{1+\uf} \frac{\underline{\chi_\uparrow}}{B} \mathcal{B}_j. 
\end{equation*}

Next, we calculate $H^{ij}$. 
Using \eqref{e:keyid3}, \eqref{e:Rdef} and \eqref{e:tm1}, we obtain
\begin{align*}
	H^{ij} = &   \underline{\cg} \delta^{ij} \frac{\mathsf{b}_\uparrow^{2} }{\uf}     \frac{ \bigl( 1 +     z \bigr)^{2} (1+\frac{\uf}{1+\uf} u)^{\frac{1}{3}}  }{ (1+ u) } = \underline{\cg} \delta^{ij}  \biggl(\frac{1+\uf}{\ufo}\biggr)^2 \frac{\mathsf{b}_\uparrow^{2} }{\uf}  \biggl(\frac{\ufo}{1+\uf}\biggr)^2     \frac{ \bigl( 1 +     z \bigr)^{2} (1+\frac{\uf}{1+\uf} u)^{\frac{1}{3}}  }{ (1+ u) } \notag  \\ 
	= & (\ck+\mathscr{L}    + \tau\mathscr{S}  )\delta^{ij} \frac{\underline{\chi_\uparrow}}{B}  \frac{ \bigl( 1 +     z \bigr)^{2} (1+\frac{\uf}{1+\uf} u)^{\frac{1}{3}}  }{ (1+ u) } 
	=   (\ck  + \tau\mathscr{S}  )\delta^{ij} \frac{\underline{\chi_\uparrow}}{B}  + \mathscr{H} \delta^{ij}  . 
\end{align*}
 In the end, noting $1/\uf>0>-1$, we verify the inequality 
\begin{equation*}
	\tau	\mathscr{S}=   (\ck-\cm^2)  \frac{B}{\underline{\chi_\uparrow}} \biggl(  \frac{4}{ \uf}     -  \frac{\underline{\mathfrak{G}}}{ B} \biggr)  
	=  (\ck-\cm^2)   \biggl(  \frac{1}{ \uf}     -  \frac{\underline{\mathfrak{G}}}{ 4 B} \biggr) /\biggl( 1+ \frac{\underline{\mathfrak{G}}}{4B}\biggr)   
	>  -(k-\cm^2)    . 
\end{equation*}
In addition, by lemma \ref{t:Thpst}.\eqref{r:1}, we estimate $
	S\leq \ck+(\ck-\cm^2)\frac{1}{f}\leq \ck\bigl(1+ \frac{1}{\beta}\bigr)$. 
It completes the proof. 
\end{proof}

 We now begin deriving the main singular symmetric hyperbolic system. Lemmas \ref{t:Seq1} through \ref{e:Seq6} provide each first-order equation in the complete singular system. Since the derivations and calculations are not always straightforward, we outline the key steps for each lemma below.

\begin{lemma}\label{t:Seq1}
	In terms of the variables \eqref{e:v1}--\eqref{e:v5}, the main equation \ref{Eq2} becomes 
	\begin{align*}%\label{e:u0eq0}
		&      \del{\tau}  u_0   +  \mathscr{R}^j \del{\tau} u_j  +		\frac{1}{A\tau}  H^{ij}   \del{\zeta^i}  u_j  \notag  \\  
		= &  \biggl( -\frac{14  }{3    }+\mathscr{Z}_{11}(\tau; u_0,u,z)  \biggr)  \frac{ 1 }{A \tau }    u_0   -\frac{ 4\ck q^j    }{A  \tau}   u_j      +\frac{1}{A\tau} \mathscr{Z}^j_{12}(\tau; u_0,u,z; u_i, \mathcal{B}_i) u_j +  \frac{1}{A\tau} \biggl(8+\mathscr{Z}_{13}(\tau;  u_0,u,z) \biggr)  u       \notag  \\
		& +   \frac{     1     }{A    \tau  }  (-8 +\mathscr{Z}_{15}(\tau;  u_0,u,z) ) z     +\mathfrak{F}_{u_0}    . 
	\end{align*}   
	where  
	\begin{equation*}
		\mathfrak{F}_{u_0}= 	\mathfrak{F}_{u_0}(\tau, u_0,u,  z):=\frac{\underline{\mathfrak{G}}}{\tau} \mathscr{S}_{11}( \tau;  u_0,u,  z )  +\frac{1}{\tau \uf^{\frac{1}{2}}} \mathscr{S}_{12}(\tau,u_0; u, z )   
	\end{equation*}
and  $\mathscr{Z}^j_{12}(\tau; u_0,u,z; 0,0) =0$, $\mathscr{Z}_{1\ell}(\tau;  0,0,0)=0$ ($\ell=1,3,5$) conform to the conventions outlined in  \S\ref{s:rmdrs}.    
\end{lemma}
\begin{proof}
	The proof involves  lengthy and tedious  calculations utilizing the variables \eqref{e:v1}--\eqref{e:v5},  \eqref{e:grrexp},  and results from Lemmas  \ref{t:gb1}.\ref{l:1.4},  \ref{t:gb2}.\ref{l:2.4}, \ref{l:2.5}, \ref{l:2.6} and \ref{t:iden1}. Below, we outline only a few key steps. By substituting  \eqref{e:v1}--\eqref{e:v5} into \ref{Eq2} and applying Lemma \ref{t:gb1}.\ref{l:1.4} to switch the coordinates to   $(\tau,\zeta)$, we obtain
	\begin{align*}
		&  \frac{A B  \uf (1+u) \left(-\tau\right){}^{\frac{2}{3 A}+1}}{\mathsf{b}^{\frac{2}{3}}   ( \uf+\uf u +1)^{\frac{1}{3}}}\del{\tau} (\underline{f_0}  +\underline{f_0}  u_0 )  + \underline{\cg} \delta^{ij}\frac{A B  \uf (1+  u)  \left(-\tau\right){}^{\frac{2}{3 A}+1} \mathsf{b}_i } {\mathsf{b}^{\frac{2}{3}}   ( \uf (1+  u)  +1)^{\frac{1}{3}}} \del{\tau} \bigl( (1+\underline{f} )  u_j\bigr)  -\underline{\cg} \delta^{ij}\del{\zeta^i}   \bigl(  (1+\underline{f} ) u_j \bigr)  \notag  \\
		&  \hspace{0.5cm} +\frac{4}{3 \mathsf{b}} \underline{f_0}  (1+ u_0 ) -
		\frac{2}{3\mathsf{b}^2} \uf ( 1 +  u) (1+  \uf + \uf u  )   -\frac{4}{3} \frac{(\underline{f_0}  +\underline{f_0}  u_0 )^2}{1+\uf + \uf u  }    +\frac{(1+\uf)^2}{\mathsf{b}^2} K^{ij}  u_i u_j\notag  \\
		&\hspace{0.5cm}  -\biggl( \frac{ \cm^2 \underline{\chi_\uparrow} \uf }{B \mathsf{b}_\uparrow^2}  \frac{  (1 +   u_0  )^2}{ \bigl(1+ \frac{\uf}{1+\uf } u  \bigr)^2}  + 4(\ck-\cm^2) \frac{(1+\uf)}{ \mathsf{b}_\uparrow^2}  \frac{\bigl(1+\frac{\uf}{1+\uf}  u \bigr)  }{(1+ z  )^6} \biggr)  (1+\uf) q^i u_i   =  0   , 
	\end{align*} 
	where
	\begin{equation*}
		K^{ij}=K^{ij}(\tau,u, u_\mu,z):=\mathtt{K}^{ij}\bigl( (1+ z   )^3 \mathsf{b}_\uparrow,  \uf + \uf u,  \underline{f_0}   +\underline{f_0}  u_0, (1+\underline{f} ) u_k\bigr) .
	\end{equation*}

Next, expanding the above equation and applying Lemmas \ref{t:gb2}.\ref{l:2.4} and \ref{l:2.5} to replace $\del{\tau}\underline{f_0}$ and $\del{\tau}\uf$ yields,  
\begin{align*} 
	&     \del{\tau}  u_0   + \underline{\cg} \delta^{ij} \mathsf{b}_i   \frac{(1+\underline{f} ) }{\underline{f_0} }  \del{\tau} u_j  -  \underline{\cg} \delta^{ij} \frac{\mathsf{b}^{\frac{2}{3}}   ( \uf + \uf u  +1)^{\frac{1}{3}} (1+\underline{f} )  }{A B (\uf + \uf u ) \left(-\tau\right){}^{\frac{2}{3 A}+1}  \underline{f_0}  } \del{\zeta^i}  u_j   +  \underline{\cg} \delta^{ij} \mathsf{b}_i  \frac{\mathsf{b}_\uparrow^{\frac{2}{3}} (1+ \uf )^{\frac{1}{3}}}{A B  \uf   (-\tau)^{\frac{2}{3A}+1}}	   u_j  \notag  \\
	& \hspace{1cm} +\biggl(  -\frac{4}{3}  \frac{\mathsf{b}_\uparrow^{-\frac{1}{3}} (1+ \uf)^{\frac{1}{3}}}{A B  \uf (-\tau)^{\frac{2}{3A}+1}}  + \frac{2}{3  }  \frac{\mathsf{b}_\uparrow^{-\frac{4}{3}} (1+ \uf )^{\frac{4}{3}} }{A B  (-\tau)^{\frac{2}{3A}+1} \underline{f_0} } + \frac{4}{3 } \frac{\mathsf{b}_\uparrow^{\frac{2}{3}}   \underline{f_0}     (1+ \uf )^{-\frac{2}{3}}}{A B  \uf (-\tau)^{\frac{2}{3A}+1}}\biggr)  u_0  \notag  \\
	&  \hspace{1cm} +\frac{4}{3  }\frac{\mathsf{b}^{-\frac{1}{3}}   ( \uf + \uf u  +1)^{\frac{1}{3}}}{A B  (\uf + \uf u ) \left(-\tau\right){}^{\frac{2}{3 A}+1}}    u_0       +\biggl(\frac{4}{3  }\frac{\mathsf{b}^{-\frac{1}{3}}   ( \uf+\uf u +1)^{\frac{1}{3}}}{A B  \underline{\varrho}  \left(-\tau\right){}^{\frac{2}{3 A}+1}}  -\frac{4}{3}  \frac{\mathsf{b}_\uparrow^{-\frac{1}{3}} (1+ \uf)^{\frac{1}{3}}}{A B  \uf (-\tau)^{\frac{2}{3A}+1}}  \biggr)  \notag  \\ 
	& \hspace{1cm} + \frac{2}{3  }  \frac{\mathsf{b}_\uparrow^{-\frac{4}{3}} (1+ \uf )^{\frac{4}{3}} }{A B  (-\tau)^{\frac{2}{3A}+1} \underline{f_0} }-
	\frac{2}{3}  \frac{\mathsf{b}^{-\frac{4}{3}}   ( \uf + \uf u  +1)^{\frac{1}{3}}}{A B  (\uf + \uf u )  \left(-\tau\right){}^{\frac{2}{3 A}+1}} \frac{ \uf  ( 1+  \uf )}{\underline{f_0} }     \notag  \\
	& \hspace{1cm} -
	\frac{2}{3}  \frac{\mathsf{b}^{-\frac{4}{3}}   ( \uf + \uf u  +1)^{\frac{1}{3}}}{A B  (\uf + \uf u ) \left(-\tau\right){}^{\frac{2}{3 A}+1}} \frac{ \uf  ( 1+  \uf )}{\underline{f_0} } \biggl[ (1+    u) \biggl(1 + \frac{\uf}{1+  \uf} u  \biggr) -1\biggr] \notag\\
	&\hspace{1cm} -\frac{4}{3}\frac{\mathsf{b}^{\frac{2}{3}}   ( \uf + \uf u  +1)^{\frac{1}{3}}}{A B  (\uf + \uf u ) \left(-\tau\right){}^{\frac{2}{3 A}+1}} \frac{\underline{f_0} }{(1+\uf ) } \biggl( \frac{(1 +  u_0 )^2}{(1+ \frac{\uf}{1+\uf} u )} -1\biggr)  + \frac{4}{3 } \frac{\mathsf{b}_\uparrow^{\frac{2}{3}}   \underline{f_0}   (1+ \uf )^{-\frac{2}{3}}}{A B  \uf (-\tau)^{\frac{2}{3A}+1}} \notag  \\
	&\hspace{1cm}  -\frac{4}{3}\frac{\mathsf{b}^{\frac{2}{3}}   ( \uf + \uf u  +1)^{\frac{1}{3}}}{A B  (\uf + \uf u ) \left(-\tau\right){}^{\frac{2}{3 A}+1}} \frac{\underline{f_0} }{(1+\uf ) }    +\frac{\mathsf{b}^{-\frac{4}{3}}   ( \uf+\uf u +1)^{\frac{1}{3}}}{A B  \uf (1+u) \left(-\tau\right){}^{\frac{2}{3 A}+1}}  \frac{(1+\uf)^2}{\underline{f_0}} K^{ij}  u_i u_j \notag  \\
	& - \frac{\mathsf{b}^{\frac{2}{3}}(\uf+\uf u+1)^{\frac{1}{3}}}{AB\uf(1+u)(-\tau)^{\frac{2}{3A}+1}} \frac{1}{\ufo} \biggl( \frac{ \cm^2 \underline{\chi_\uparrow} \uf (1+\uf)}{B \mathsf{b}_\uparrow^2}  \frac{  (1 +   u_0  )^2}{ \bigl(1+ \frac{\uf}{1+\uf } u  \bigr)^2}  + 4(\ck-\cm^2) \frac{(1+\uf)^2}{ \mathsf{b}_\uparrow^2}  \frac{\bigl(1+\frac{\uf}{1+\uf}  u \bigr)  }{(1+ z  )^6} \biggr)   q^i u_i   
	=    0 . 
\end{align*}

Then by applying Lemmas \ref{t:id0} and \ref{t:iden1}, we obtain  
\begin{align}\label{e:u0eq0b}
	&     \del{\tau}  u_0   +  \underline{\cg} \delta^{ij}\mathsf{b}_i    \frac{(1+\underline{f} ) }{\underline{f_0} }  \del{\tau} u_j  -   \underline{\cg} \delta^{ij}\frac{\mathsf{b}_\uparrow^{2} }{\uf}    \frac{    1   }{A   (-\tau)   }    \frac{ \bigl( 1 +     z \bigr)^{2} (1+\frac{\uf}{1+\uf} u)^{\frac{1}{3}}  }{ (1+ u) }   \del{\zeta^i}  u_j   \notag  \\   
	= & -\biggl(  -\frac{4}{3}  \frac{ \underline{\chi_\uparrow}^{\frac{1}{2}} }{A (-\tau) B^{\frac{1}{2}} \uf^{\frac{1}{2}}    }  + \frac{2}{3  }  \frac{    1   }{A   (-\tau)   }  + \frac{4}{3 } \frac{  \underline{\chi_\uparrow}   }{A B    (-\tau)  }  \biggr)  u_0   -  \underline{\cg} \delta^{ij} \mathsf{b}_i	\frac{  \underline{\chi_\uparrow}    }{A B    (-\tau)  }   \frac{(1+ \uf ) }{\underline{f_0}}	   u_j  \notag  \\
	& - \frac{4}{3  }  \frac{ \underline{\chi_\uparrow}^{\frac{1}{2}} }{A (-\tau) B^{\frac{1}{2}} \uf^{\frac{1}{2}}    }   \frac{  \bigl( 1+  z  \bigr)^{-1} (1+\frac{\uf}{1+\uf} u)^{\frac{1}{3}}  }{ (1+ u)  }      u_0  -\frac{4}{3  }  \frac{ \underline{\chi_\uparrow}^{\frac{1}{2}} }{A (-\tau) B^{\frac{1}{2}} \uf^{\frac{1}{2}}    }   \biggl[ \frac{  \bigl( 1+  z  \bigr)^{-1} (1+\frac{\uf}{1+\uf} u)^{\frac{1}{3}}  }{ (1+ u)  }  - 1 \biggr]  \notag  \\  
	&+\frac{2}{3}      \frac{    1   }{A   (-\tau)   }    \Biggl( \frac{(1+\frac{\uf}{1+\uf} u)^{\frac{4}{3}}  }{  ( 1+   z  )^{4}  }    -1  \Biggr)   +\frac{4}{3}  	\frac{  \underline{\chi_\uparrow}  }{A B    (-\tau)  }    \biggl(  \frac{  \bigl( 1 +  z  \bigr)^{2}  (1 +  u_0 )^2 }{ (1+ u) (1+ \frac{\uf}{1+\uf} u )^{\frac{2}{3}}}     - 1   \biggr)     \notag  \\
	&     -\frac{ 1 }{A  \tau}   q^i u_i   \biggl(  \cm^2\frac{\underline{\mathfrak{G}}}{B}  +4(\ck-\cm^2)    \frac{1}{ \uf }      \biggr)      -\frac{ 4\ck q^i    }{A  \tau}   u_i       +\frac{1}{A\tau} \frac{1+\uf}{\uf} \frac{ \bigl(1+\frac{\uf}{1+\uf} u\bigr)^{\frac{1}{3}}}{(1+z)^{4}(1+u) } K^{ij}  u_i u_j \notag  \\
	&      -\frac{ 1 }{A  \tau}   q^i u_i     \frac{ \cm^2 \underline{\chi_\uparrow} }{B  }  \biggl(\frac{  (1 +   u_0  )^2}{ \bigl(1+ \frac{\uf}{1+\uf } u  \bigr)^{\frac{5}{3}}} \frac{ (   1+ z  )^{2}    }{ (1+ u)   } -1\biggr) -\frac{ 4(\ck-\cm^2) }{A  \tau}        \frac{(1+\uf) }{ \uf }   \biggl(\frac{   (1+\frac{\uf}{1+\uf} u)^{\frac{4}{3}}  }{ (1+ z  )^4(1+ u)   } -1\biggr) q^i u_i        .  
\end{align}

Then, by applying Lemma \ref{t:coef1} and substituting \eqref{e:id2} (for replacing $\mathsf{b}_i$), along with \eqref{e:Rdef},   \eqref{e:grrexp} and \eqref{e:chig} into \eqref{e:u0eq0b}, and expanding \eqref{e:u0eq0b},
we arrive at 
\begin{align*}%\label{e:u0eq0}
	&      \del{\tau}  u_0   +  \mathscr{R}^j \del{\tau} u_j  +		\frac{1}{A\tau}  H^{ij}   \del{\zeta^i}  u_j  \notag  \\   
	= &  -  \frac{14}{3}  	\frac{   1     }{A     \tau  }  u_0       -\frac{ 4\ck q^i    }{A  \tau}   u_i        	+\frac{1}{A\tau} \frac{1+\uf}{\uf} \frac{ \bigl(1+\frac{\uf}{1+\uf} u\bigr)^{\frac{1}{3}}}{(1+z)^{4}(1+u) } K^{ij}  u_i u_j + 	\frac{  8 }{A     \tau  }  u   -     \frac{   8      }{A    \tau  }    z       +  	\frac{  \underline{\chi_\uparrow}    }{A B    \tau   }    \mathscr{R}^i    u_j  \notag  \\
	&  -  \frac{8}{3}  	\frac{  1 }{A   \tau  }  \frac{\underline{\mathfrak{G}}}{B}   z  -  \frac{4}{3}  	\frac{   1     }{A   \tau  } \frac{\underline{\mathfrak{G}}}{B}  u_0 + \frac{20}{9}  	\frac{  1 }{A \tau  }\frac{\underline{\mathfrak{G}}}{B}u      -\frac{4}{3  }  \frac{ \underline{\chi_\uparrow}^{\frac{1}{2}} }{A (-\tau)  \uf^{\frac{1}{2}}   B^{\frac{1}{2}}   }   (1 + u_0) \Biggl[\frac{    \bigl(1+\frac{\uf}{1+\uf} u\bigr)^{\frac{1}{3}}   }{ (1+ u) \bigl( 1+ z  \bigr) }       -  1 \Biggr]      \notag  \\ 
	&    - \frac{2}{3}      \frac{    1   }{A  \tau   }  \Biggl[         \frac{ \bigl(1 + \frac{\uf}{1+  \uf} u  \bigr)^{\frac{4}{3}}}{\bigl( 1+  z \bigr)^{4} } -1-\frac{4}{3}u +4z\Biggr]       -   \frac{4}{3}  	\frac{  \underline{\chi_\uparrow}  }{A B   \tau  }  \Biggl[ \frac{  \bigl( 1 + z  \bigr)^{2} (1 +  u_0 )^2 }{ (1+ u) \bigl(1+ \frac{\uf}{1+\uf} u \bigr)^{\frac{2}{3} } }         -1-2u_0+\frac{5}{3}u  -2z \Biggr]  \notag  \\
	&      -\frac{ 1 }{A  \tau}   q^i u_i     \frac{ \cm^2 \underline{\chi_\uparrow} }{B  }  \biggl(\frac{  (1 +   u_0  )^2}{ \bigl(1+ \frac{\uf}{1+\uf } u  \bigr)^{\frac{5}{3}}} \frac{ (   1+ z  )^{2}    }{ (1+ u)   } -1\biggr) -\frac{ 4(\ck-\cm^2) }{A  \tau}        \frac{(1+\uf) }{ \uf }   \biggl(\frac{   (1+\frac{\uf}{1+\uf} u)^{\frac{4}{3}}  }{ (1+ z  )^4(1+ u)   } -1\biggr) q^i u_i    \notag  \\
	&       -\frac{ 1 }{A  \tau}   q^i u_i   \biggl(  \cm^2\frac{\underline{\mathfrak{G}}}{B}  +4(\ck-\cm^2)    \frac{1}{ \uf }      \biggr)      . 
\end{align*}   
This equation completes the proof of the lemma.  
\end{proof}

To transform the second-order equation \ref{Eq2} into a first-order hyperbolic system, additional equations for the variables $(u_i,u,z,\mathcal{B}_k)$ are required. We will derive these equations in the following lemmas.

\begin{lemma}\label{t:Seq2}
		In terms of variables \eqref{e:v1}--\eqref{e:v5}, there is an equation,  
		\begin{align*}%\label{e:dtur6} 
			& ( S+ \mathscr{L}  )  \delta^{i}_k \del{\tau}u_i 	 +  \mathscr{R}_k \del{\tau}u_0   +  \frac{ 1}{A   \tau    } H^i_k \del{\zeta^i}   u_0    \notag  \\
			= &  \frac{1}{A\tau }  \bigl(4 \ck +\mathscr{Z}_{22}(\tau; u_0,u,z)\bigr)  \delta^{i}_k  u_i          +        \frac{  1    }{A   \tau   }    \bigl(24 \ck   + \mathscr{Z}_{24}(\tau;u_0,u,z)\bigr)   \delta^{i}_k \mathcal{B}_i   +	\mathfrak{F}_{u_k}   , 
		\end{align*}  
	where $\mathscr{R}_k:=\delta_{jk}\mathscr{R}^j $,  $H^i_k:=\delta_{jk} H^{ij} $, 
	and
	\begin{align*}
		\mathfrak{F}_{u_k}=\frac{\underline{\mathfrak{G}}}{\tau} \mathscr{S}_{k;21}(\tau; u_k,\mathcal{B}_k)+\frac{1}{\tau \uf^{\frac{1}{2}}} \mathscr{S}_{k;22}(\tau,u_0,u,z; u_k,\mathcal{B}_k) . 
	\end{align*}
\end{lemma}
\begin{proof}
	Differentiating \eqref{e:v3} with respect to  $\tau$, and utilizing  \eqref{e:v3} along with  Lemma \ref{t:gb2},\ref{l:2.5},  we obtain
\begin{equation}\label{e:dtur}
	\del{\tau}u_i=   \del{\tau}\biggl(\frac{1}{1+\underline{f}} \underline{\varrho_i} \biggr)   
	=   -  \frac{\mathsf{b}_\uparrow^{\frac{2}{3}} (1+ \uf )^{-\frac{2}{3}}}{A B  \uf   (-\tau)^{\frac{2}{3A}+1}}	\underline{f_0}   u_i +\frac{1  }{\uf+1}\del{\tau}\underline{\varrho _i}  . 
\end{equation}

Next, Next, we  establish the relationship between $\del{\tau} \underline{\varrho_i}$ and $\del{\zeta^i} \underline{\varrho_0}= \underline{f_0} \del{\zeta^i}   u_0  $ (by \eqref{e:v2}).  First, observe that by applying Lemma \ref{t:gb1}.\ref{l:1.4},  
\begin{equation}\label{e:dzrho0}
		\underline{\varrho_i}  =      \del{\zeta^i} \underline{\varrho}   -     \mathsf{b}_i \underline{\varrho_0} \quad \Leftrightarrow  \quad        \del{\zeta^i} \underline{\varrho} =	\underline{\varrho_i} +\mathsf{b}_{i} \underline{\varrho_0}    .
\end{equation}
Further differentiating it with respect to $\tau$ implies
	\begin{equation}\label{e:dtrr}
		\del{\tau}\underline{\varrho_i}  =    \del{\zeta^i} \del{\tau}  \underline{\varrho}   -     (  \del{\tau}\mathsf{b}_i ) \underline{\varrho_0}  -     \mathsf{b}_i \del{\tau} \underline{\varrho_0} . 
	\end{equation}
Substituting  \eqref{e:dtrr} into \eqref{e:dtur} yields
\begin{equation}\label{e:dtur2}
	\del{\tau}u_i 
	=   -  \frac{\mathsf{b}_\uparrow^{\frac{2}{3}} (1+ \uf )^{-\frac{2}{3}}}{A B  \uf   (-\tau)^{\frac{2}{3A}+1}}	\underline{f_0}   u_i +\frac{1  }{\uf+1} (  \del{\zeta^i} \del{\tau}  \underline{\varrho}   -     (  \del{\tau}\mathsf{b}_i ) \underline{\varrho_0}  -     \mathsf{b}_i \del{\tau} \underline{\varrho_0} ) . 
\end{equation}

To calculate  \eqref{e:dtur2}, we require expressions for  $\del{\tau} \underline{\varrho_0}$, and $ 	\del{\zeta^i}\del{\tau} \underline{\varrho} - (	\del{\tau} \mathsf{b}_i ) \underline{\varrho_0}$. First, note the relationship between  $\del{\tau}\underline{\varrho}$ and $\underline{\varrho_0}$ as provided by Lemma  \ref{t:gb1}.\ref{l:1.4} and its proof: 
\begin{equation}\label{e:dtrho}
\del{\tau} \underline{\varrho}  =  (\del{\tau} \mathsf{b} ) \underline{\varrho_0}   =  \frac{\mathsf{b}^{\frac{2}{3}}   ( \underline{\varrho} +1)^{\frac{1}{3}}}{A B  \underline{\varrho}  \left(-\tau\right){}^{\frac{2}{3 A}+1}} \underline{\varrho_0}   	 .      
\end{equation}
Then applying Lemma \ref{t:gb1}.\ref{l:1.2} to replace $\del{\tau}\mathsf{b}$ and using the variables \eqref{e:v1}--\eqref{e:v5}, differentiating \eqref{e:dtrho} with respect to $\zeta^i$ results in
\begin{align}\label{e:dzdtr1}
		\del{\zeta^i}\del{\tau} \underline{\varrho} - (	\del{\tau} \mathsf{b}_i ) \underline{\varrho_0} = (	\del{\tau} \mathsf{b} )\del{\zeta^i} \underline{\varrho_0}   
		 =\frac{ (1+ z  )^2 \mathsf{b}_\uparrow^{\frac{2}{3}}   ( 1 + \uf+\uf u )^{\frac{1}{3}}}{A B  \uf (1+  u)  \left(-\tau\right){}^{\frac{2}{3 A}+1}}  \underline{f_0}  \del{\zeta^i}   u_0 . 
\end{align}

Now let us calculate $\del{\tau} \underline{\varrho_0}$ in \eqref{e:dtrr}. By using  \eqref{e:v2} and applying Lemma \ref{t:gb2}.\ref{l:2.4} to  substitute $\del{\tau} \underline{f_0}$, we arrive at
\begin{align}\label{e:dtrho0}
	\del{\tau}  \underline{\varrho_0} 
	=  & -\frac{4}{3}  \frac{\mathsf{b}_\uparrow^{-\frac{1}{3}} (1+ \uf)^{\frac{1}{3}}}{A B  \uf (-\tau)^{\frac{2}{3A}+1}} \underline{f_0}  +\frac{2}{3  }  \frac{\mathsf{b}_\uparrow^{-\frac{4}{3}} (1+ \uf )^{\frac{4}{3}} }{A B  (-\tau)^{\frac{2}{3A}+1}}+\frac{4}{3 } \frac{\mathsf{b}_\uparrow^{\frac{2}{3}} (  \underline{f_0}   )^2 (1+ \uf )^{-\frac{2}{3}}}{A B  \uf (-\tau)^{\frac{2}{3A}+1}}  \notag \\
	& + \biggl(-\frac{4}{3}  \frac{\mathsf{b}_\uparrow^{-\frac{1}{3}} (1+ \uf)^{\frac{1}{3}}}{A B  \uf (-\tau)^{\frac{2}{3A}+1}} \underline{f_0}  +\frac{2}{3  }  \frac{\mathsf{b}_\uparrow^{-\frac{4}{3}} (1+ \uf )^{\frac{4}{3}} }{A B  (-\tau)^{\frac{2}{3A}+1}}+\frac{4}{3 } \frac{\mathsf{b}_\uparrow^{\frac{2}{3}} (  \underline{f_0}   )^2 (1+ \uf )^{-\frac{2}{3}}}{A B  \uf (-\tau)^{\frac{2}{3A}+1}}  \biggr)  u_0  +\underline{f_0}   \del{\tau}u_0    .
\end{align}

Substituting $	\del{\zeta}\del{\tau} \underline{\varrho} - (	\del{\tau} \mathsf{b}_\zeta ) \underline{\varrho_0}$ from \eqref{e:dzdtr1}  and $\del{\tau}  \underline{\varrho_0}$ from \eqref{e:dtrho0} into \eqref{e:dtur2},  and using the variables \eqref{e:v1}--\eqref{e:v5},  direct calculations imply that
\begin{align}\label{e:dtur3} 
	& \del{\tau}u_i 	 + \frac{   \underline{f_0}   }{\uf+1}  \mathsf{b}_i   \del{\tau}u_0   - \frac{   (1+ z )^2 \mathsf{b}_\uparrow^{\frac{2}{3}}   ( 1 + \uf+\uf u )^{\frac{1}{3}}}{A B  \uf (\uf+1) (1+  u)  \left(-\tau\right){}^{\frac{2}{3 A}+1}}  \underline{f_0}  \del{\zeta^i}   u_0    \notag  \\
	= &  -  \frac{\mathsf{b}_\uparrow^{\frac{2}{3}} (1+ \uf )^{-\frac{2}{3}}}{A B  \uf   (-\tau)^{\frac{2}{3A}+1}}	\underline{f_0}   u_i         +    \frac{4  }{3}  \frac{\mathsf{b}_\uparrow^{-\frac{1}{3}} (1+ \uf)^{-\frac{2}{3}}}{A B  \uf (-\tau)^{\frac{2}{3A}+1}} \underline{f_0}  \mathsf{b}_i   -      \frac{2    }{3  }  \frac{\mathsf{b}_\uparrow^{-\frac{4}{3}} (1+ \uf )^{\frac{1}{3}} }{A B  (-\tau)^{\frac{2}{3A}+1}}  \mathsf{b}_i  -     \frac{4  }{3 } \frac{\mathsf{b}_\uparrow^{\frac{2}{3}}  \underline{f_0}^2 (1+ \uf )^{-\frac{5}{3}}}{A B  \uf (-\tau)^{\frac{2}{3A}+1}} \mathsf{b}_i \notag \\
	&  + \biggl(   \frac{4  }{3}  \frac{\mathsf{b}_\uparrow^{-\frac{1}{3}} (1+ \uf)^{-\frac{2}{3}}}{A B  \uf (-\tau)^{\frac{2}{3A}+1}} \underline{f_0}  \mathsf{b}_i -       \frac{2  }{3  }  \frac{\mathsf{b}_\uparrow^{-\frac{4}{3}} (1+ \uf )^{\frac{1}{3}} }{A B  (-\tau)^{\frac{2}{3A}+1}} \mathsf{b}_i  -     \frac{4  }{3 } \frac{\mathsf{b}_\uparrow^{\frac{2}{3}}   \underline{f_0}^2 (1+ \uf )^{-\frac{5}{3}}}{A B  \uf (-\tau)^{\frac{2}{3A}+1}} \mathsf{b}_i \biggr)  u_0    . 
\end{align}
Utilizing the definition of $	\underline{\chi_\uparrow}$ from \eqref{e:Gdef2} and the expression for $\mathcal{B}_i$ from \eqref{e:id2} in Lemma \ref{t:id0}, \eqref{e:dtur3} becomes   
\begin{align}\label{e:dtur5} 
	& \del{\tau}u_i	 + \frac{   \underline{f}   }{\uf+1}  \frac{\underline{\chi_\uparrow} }{B  } \mathcal{B}_i   \del{\tau}u_0   -  \frac{    \underline{\chi_\uparrow}}{A B (-\tau)   } \frac{   (1+ z  )^2   \bigl( 1  + \frac{\uf}{1+\uf} u \bigr)^{\frac{1}{3}}}{ (1+  u)  }   \del{\zeta^i}   u_0    \notag  \\
	= &  -  \frac{\mathsf{b}_\uparrow^{\frac{2}{3}} (1+ \uf )^{-\frac{2}{3}}  \underline{f_0}   }{A B  \uf   (-\tau)^{\frac{2}{3A}+1}}	u_i        +   \biggl(  \frac{4  }{3}  \frac{\mathsf{b}_\uparrow^{-\frac{1}{3}} (1+ \uf)^{-\frac{2}{3}}}{A B    (-\tau)^{\frac{2}{3A}+1}}      -      \frac{2    }{3  }  \frac{\mathsf{b}_\uparrow^{-\frac{4}{3}} (1+ \uf )^{\frac{1}{3}} }{A B  (-\tau)^{\frac{2}{3A}+1}}   \frac{ \uf   }{\ufo }    -     \frac{4  }{3 } \frac{\mathsf{b}_\uparrow^{\frac{2}{3}}  \underline{f_0}  (1+ \uf )^{-\frac{5}{3}}}{A B    (-\tau)^{\frac{2}{3A}+1}} \biggr)  \frac{\underline{\chi_\uparrow} }{B  }  \mathcal{B}_i  \notag  \\
	&     + \biggl(   \frac{4  }{3}  \frac{\mathsf{b}_\uparrow^{-\frac{1}{3}} (1+ \uf)^{-\frac{2}{3}}}{A B    (-\tau)^{\frac{2}{3A}+1}}   \frac{\underline{\chi_\uparrow} }{B  }  \mathcal{B}_i        -       \frac{2  }{3  }  \frac{\mathsf{b}_\uparrow^{-\frac{4}{3}} (1+ \uf )^{\frac{1}{3}} }{A B  (-\tau)^{\frac{2}{3A}+1}} \frac{ \uf  }{\ufo  } \frac{\underline{\chi_\uparrow} }{B  } \mathcal{B}_i   -     \frac{4  }{3 } \frac{\mathsf{b}_\uparrow^{\frac{2}{3}}   \underline{f_0}  (1+ \uf )^{-\frac{5}{3}}}{A B   (-\tau)^{\frac{2}{3A}+1}} \frac{\underline{\chi_\uparrow} }{B  }  \mathcal{B}_i  \biggr)  u_0   . 
\end{align} 
Next, applying Lemma \ref{t:iden1},  \eqref{e:dtur5} further becomes
\begin{align}\label{e:dtur6} 
	& \del{\tau}u_i 	 + \frac{   \underline{f}   }{\uf+1}  \frac{\underline{\chi_\uparrow} }{B  } \mathcal{B}_i   \del{\tau}u_0   + \frac{1}{A  \tau  }   \frac{    \underline{\chi_\uparrow}}{ B  } \frac{   (1+ z )^2   \bigl( 1  + \frac{\uf}{1+\uf} u \bigr)^{\frac{1}{3}}}{ (1+  u)  }   \del{\zeta^i}   u_0    \notag  \\
	= &  \frac{1}{A\tau } \frac{  \underline{\chi_\uparrow}     }{ B   }     u_i          +      \frac{2    }{3  }    \frac{    1   }{A   \tau   }   \frac{\underline{\chi_\uparrow} }{B  }  \mathcal{B}_i +     \frac{4  }{3 } 
	\frac{  1  }{A \tau  }  \biggl(  \frac{\underline{\chi_\uparrow} }{B  }  \biggr)^2 \mathcal{B}_i   +    \frac{2    }{3  }     \frac{    1   }{A    \tau  }     \frac{\underline{\chi_\uparrow} }{B  } \biggl(1+2\frac{\underline{\chi_\uparrow} }{B  }\biggr) u_0  \mathcal{B}_i      \notag  \\
	&  -  \frac{2    }{3  }    \frac{    1   }{A   \tau  (1+\uf)  }    \frac{\underline{\chi_\uparrow} }{B  }  \mathcal{B}_i   -   \frac{4  }{3}  \frac{1}{A  \tau    \uf^{\frac{1}{2}}    }   \frac{\uf}{1+\uf}    \biggl( \frac{\underline{\chi_\uparrow} }{B  } \biggr)^{\frac{3}{2}} \mathcal{B}_i     -    \frac{4  }{3 } 
	\frac{  1 }{A  \tau (1+\uf) }    \biggl( \frac{\underline{\chi_\uparrow} }{B  } \biggr)^2 \mathcal{B}_i   \notag  \\
	&  -    \frac{4  }{3} \frac{ 1 }{A  \tau    \uf^{\frac{1}{2}}    }   \frac{\uf}{1+\uf}    \biggl( \frac{\underline{\chi_\uparrow} }{B  } \biggr)^{\frac{3}{2}} \mathcal{B}_i    u_0     -   \frac{2    }{3  }     \frac{    1   }{A   \tau  (1+\uf)  }     \frac{\underline{\chi_\uparrow} }{B  }  \mathcal{B}_i   u_0   -  \frac{4  }{3 }
	\frac{  1 }{A   \tau (1+\uf) }   \biggl( \frac{\underline{\chi_\uparrow} }{B  } \biggr)^2 \mathcal{B}_i     u_0    . 
\end{align} 

Note that, with the help of \eqref{e:keyid3} (from Lemma \ref{t:iden1}) and \eqref{e:Rdef},  the following term can be simplified, 
\begin{equation}\label{e:tm2}
	 \delta_{jk}\underline{\cg} \delta^{ij}\Bigl(\frac{1+\uf}{\ufo}\Bigr)^2 \frac{    \underline{\chi_\uparrow}}{ B     } \frac{   (1+ z )^2   \bigl( 1  + \frac{\uf}{1+\uf} u \bigr)^{\frac{1}{3}}}{ (1+  u)  }  =  \delta_{jk}\underline{\cg} \delta^{ij}\frac{\mathsf{b}_\uparrow^2}{\uf}  \frac{   (1+ z )^2   \bigl( 1  + \frac{\uf}{1+\uf} u \bigr)^{\frac{1}{3}}}{ (1+  u)  }  =\delta_{jk} H^{ij} . 
\end{equation}	

Next, by multiplying  $ \delta_{jk}\underline{\cg} \delta^{ij} \Bigl(\frac{1+\uf}{\ufo}\Bigr)^2$ on both sides of  \eqref{e:dtur6},  applying  Lemma \ref{t:coef1}, \eqref{e:Rdef} and \eqref{e:chig}, and substituting  \eqref{e:tm1} and \eqref{e:tm2}, we expand \eqref{e:dtur6}as follows,  
   \begin{align}\label{e:dtur6} 
 	&( \ck + \mathscr{L}+\tau\mathscr{S}   )  \delta^{i}_k \del{\tau}u_i 	 +   \delta_{jk} \mathscr{R}^j  \del{\tau}u_0   + \frac{1}{A\tau}  \delta_{jk} H^{ij}    \del{\zeta^i}   u_0    \notag  \\
 	= &  \frac{4 \ck }{A\tau }   \delta^{i}_k      u_i          +      \frac{     24   \ck    }{A   \tau   }   \delta^{i}_k    \mathcal{B}_i  +\frac{1}{A\tau }   \ck   \frac{\underline{\mathfrak{G}}}{B}     u_k  +      \frac{2    }{3  }    \frac{    1   }{A   \tau   }   \ck   \frac{\underline{\mathfrak{G}}}{B}   \mathcal{B}_k    +     \frac{ 64  }{3 } 
 	\frac{  1  }{A \tau  }  \frac{\underline{\mathfrak{G}}}{4B} \biggl(  2+\frac{\underline{\mathfrak{G}}}{4B} \biggr)   \ck \delta^{i}_k  \mathcal{B}_i  +\frac{1}{A\tau }   ( \mathscr{L}   + \tau\mathscr{S} )   \delta^{i}_k    \frac{  \underline{\chi_\uparrow}     }{ B   }     u_i      \notag  \\
 	& +      \frac{2    }{3  }    \frac{    1   }{A   \tau   }    ( \mathscr{L}   + \tau\mathscr{S} )   \delta^{i}_k    \frac{\underline{\chi_\uparrow} }{B  }  \mathcal{B}_i  +     \frac{4  }{3 } 
 	\frac{  1  }{A \tau  }  \biggl(  \frac{\underline{\chi_\uparrow} }{B  }  \biggr)^2  ( \mathscr{L}   + \tau\mathscr{S} )     \delta^{i}_k    \mathcal{B}_i    +    \frac{2    }{3  }     \frac{    1   }{A    \tau  }     \frac{\underline{\chi_\uparrow} }{B  } \biggl(1+2\frac{\underline{\chi_\uparrow} }{B  }\biggr) ( \ck  +  \mathscr{L}   + \tau\mathscr{S}      ) u_0   \mathcal{B}_k      \notag  \\
 	&  -  \frac{2    }{3  }    \frac{    ( \ck  +  \mathscr{L}   + \tau\mathscr{S}      )   }{A   \tau  (1+\uf)  }    \frac{\underline{\chi_\uparrow} }{B  }  \mathcal{B}_k   -   \frac{4  }{3}  \frac{( \ck  + \mathscr{L}   + \tau\mathscr{S}      ) }{A  \tau    \uf^{\frac{1}{2}}    }   \frac{\uf}{1+\uf}    \biggl( \frac{\underline{\chi_\uparrow} }{B  } \biggr)^{\frac{3}{2}} \mathcal{B}_k     -    \frac{4  }{3 } 
 	\frac{ ( \ck +  \mathscr{L}   + \tau\mathscr{S}    ) }{A  \tau (1+\uf) }    \biggl( \frac{\underline{\chi_\uparrow} }{B  } \biggr)^2  \mathcal{B}_k   \notag  \\
 	&  -    \frac{4  }{3} \frac{  ( \ck  +  \mathscr{L}   + \tau\mathscr{S}       ) }{A  \tau    \uf^{\frac{1}{2}}    }   \frac{\uf}{1+\uf}    \biggl( \frac{\underline{\chi_\uparrow} }{B  } \biggr)^{\frac{3}{2}}\mathcal{B}_k    u_0     -   \frac{2    }{3  }     \frac{      ( \ck  +  \mathscr{L}   + \tau\mathscr{S}       )   }{A   \tau  (1+\uf)  }     \frac{\underline{\chi_\uparrow} }{B  }\mathcal{B}_k   u_0   -  \frac{4  }{3 }
 	\frac{ ( \ck  +  \mathscr{L}   + \tau\mathscr{S}     ) }{A   \tau (1+\uf) }   \biggl( \frac{\underline{\chi_\uparrow} }{B  } \biggr)^2 \mathcal{B}_k    u_0    . 
 \end{align} 
This equation verifies the lemma, thereby completing the proof.
\end{proof}

\begin{lemma}\label{t:Seq3}
	In terms of variables \eqref{e:v1}--\eqref{e:v5}, there is an equation, 
	\begin{align*}
			2 \del{\tau}  u
		=    &       \frac{1}{A\tau}   (	- 8   +\mathscr{Z}_{31}(\tau; u) ) u_0        +   \frac{1}{  A \tau } \biggl( \frac{ 40 }{3 } +\mathscr{Z}_{33}(\tau;u) \biggr) u   + \frac{ 1 }{A    \tau    } (-16+\mathscr{Z}_{35}(\tau; u_0,u,z))     z    +\mathfrak{F}_{u} , 
	\end{align*} 
	where
\begin{align*}
	\mathfrak{F}_{u}=\frac{\underline{\mathfrak{G}}}{\tau} \mathscr{S}_{31}(\tau;u_0,u,z)  +\frac{1}{\tau \uf^{\frac{1}{2}}} \mathscr{S}_{32}(\tau; u_0, u, z)   . 
\end{align*}
\end{lemma}
\begin{proof}
	By substituting  \eqref{e:v1}--\eqref{e:v2} into \eqref{e:dtrho}, we obtain
	\begin{equation*}
		\del{\tau} (\uf + \uf u )  =\del{\tau} \uf +u  \del{\tau} \uf   +  \uf \del{\tau} u    
		=     
		\frac {\mathsf{b}^{\frac{2}{3}}   ( \uf + \uf u +1)^{\frac{1}{3}} \underline{f_0} }{A B  \uf (1 +   u ) \left(-\tau\right){}^{\frac{2}{3 A}+1}}  + 	\frac {\mathsf{b}^{\frac{2}{3}}   ( \uf + \uf u +1)^{\frac{1}{3}}  \underline{f_0}  u_0   }{A B  \uf (1 +   u ) \left(-\tau\right){}^{\frac{2}{3 A}+1}}    .
	\end{equation*}
According to Lemma \ref{t:gb2}.\ref{l:2.5}, substituting $\del{\tau} \uf$ yields
\begin{align}\label{e:dtu.a}
	\del{\tau} u    
	=    & 
	\frac {\mathsf{b}^{\frac{2}{3}}   ( \uf + \uf u +1)^{\frac{1}{3}} \underline{f_0} }{A B  \uf (1 +   u ) \left(-\tau\right){}^{\frac{2}{3 A}+1}}   \frac{1}{\uf }+ 	\frac {\mathsf{b}^{\frac{2}{3}}   ( \uf + \uf u +1)^{\frac{1}{3}}  \underline{f_0}  }{A B  \uf (1 +   u ) \left(-\tau\right){}^{\frac{2}{3 A}+1}}  \frac{1}{\uf }  u_0    -	\frac{\mathsf{b}_\uparrow^{\frac{2}{3}} (1+ \uf )^{\frac{1}{3}}\underline{f_0} }{A B  \uf   (-\tau)^{\frac{2}{3A}+1}} \frac{1}{\uf } -\frac{\mathsf{b}_\uparrow^{\frac{2}{3}} (1+ \uf )^{\frac{1}{3}}\underline{f_0} }{A B  \uf   (-\tau)^{\frac{2}{3A}+1}} \frac{1}{\uf }	u    . 
\end{align}
By Lemma \ref{t:id0},  \eqref{e:dtu.a} becomes
\begin{align}\label{e:dtu.b}
	\del{\tau} u     
	=    & 
   \frac{   \mathsf{b}_\uparrow^{\frac{2}{3}}   (1+\uf)^{-\frac{2}{3}}   }{A B \uf  (-\tau)^{\frac{2}{3A}+1}}  \frac{ \bigl( 1 +   z   \bigr)^{2} (1+\frac{\uf}{1+\uf} u)^{\frac{1}{3}}  }{ (1+ u) }   \underline{f_0} 	\frac{1+\uf}{\uf} -\frac{\mathsf{b}_\uparrow^{\frac{2}{3}} (1+ \uf )^{-\frac{2}{3}}}{A B  \uf   (-\tau)^{\frac{2}{3A}+1}}	\underline{f_0} 	\frac{1+\uf}{\uf}   \notag  \\
	& + 	    \frac{   \mathsf{b}_\uparrow^{\frac{2}{3}}   (1+\uf)^{-\frac{2}{3}}   }{A B \uf  (-\tau)^{\frac{2}{3A}+1}}  \frac{ \bigl( 1 +    z  \bigr)^{2} (1+\frac{\uf}{1+\uf} u)^{\frac{1}{3}}  }{ (1+ u) }  \underline{f_0}  u_0  	\frac{1+\uf}{\uf} -\frac{\mathsf{b}_\uparrow^{\frac{2}{3}} (1+ \uf )^{-\frac{2}{3}}}{A B  \uf   (-\tau)^{\frac{2}{3A}+1}}	\underline{f_0}  u  	\frac{1+\uf}{\uf} .
\end{align}

Using Lemma \ref{t:iden1} and  \eqref{e:chig} to expand \eqref{e:dtu.b}, and after lengthy but direct calculations, we arrive at
\begin{align*}
	\del{\tau} u     
	=    & 
	- \frac{  4   }{A    \tau  }    u_0  		+\frac{20}{3} \frac{ 1   }{A    \tau  }  u  	- \frac{ 8 }{A    \tau    }      z    	  	- \frac{  \underline{\chi_\uparrow}   }{A B  \tau   }   \biggl[ \frac{    (1+\frac{\uf}{1+\uf} u)^{\frac{1}{3}}  }{ (1+ u) } -1+\frac{2}{3}u \Biggr]       - 	\frac{  \underline{\chi_\uparrow}     }{A B  \tau }   \biggl[\frac{   (1+\frac{\uf}{1+\uf} u)^{\frac{1}{3}}  }{ (1+ u) } -1\biggr]   u_0  
	\notag \\	
	& 	- \frac{  \underline{\chi_\uparrow}   }{A B    \tau    }    2z   \Biggl( \frac{   \bigl(    1+  \frac{1}{2} z  \bigr)  (1+\frac{\uf}{1+\uf} u)^{\frac{1}{3}} (1+  u_0  ) }{ (1+ u) }  -1  \Biggr) - \frac{  1  }{A    \tau  }   \frac{\underline{\mathfrak{G}}}{B}  u_0   +\frac{5}{3} \frac{ 1   }{A    \tau  }   \frac{\underline{\mathfrak{G}}}{B}    u   	 	- \frac{  1  }{A    \tau    }    \frac{\underline{\mathfrak{G}}}{B}   2 z    \notag  \\
	&    -
	\frac{  \underline{\chi_\uparrow}    }{A B   \tau \uf }   \biggl(  \frac{ \bigl( 1 +   z   \bigr)^{2} (1+\frac{\uf}{1+\uf} u)^{\frac{1}{3}}  (1+ u_0 ) }{ (1+ u) }  -1 - u \biggr)  . 
\end{align*} 
This equation can be reformulated using this lemma, thereby completing the proof.
\end{proof}

\begin{lemma}\label{e:Seq5}
In terms of variables \eqref{e:v1}--\eqref{e:v5}, there is an equation, 
	\begin{equation*}
	\del{\tau}	\mathcal{B}_i 
=     \frac{1}{A\tau} \biggl( \frac{2  }{3       }  +\mathscr{Z}_{42}(\tau;  u, z)\biggr) u_i  +  \frac{1}{A\tau}   \biggl(  \frac{2}{3   }  +\mathscr{Z}_{44}(\tau; u_0,u,z)  \biggr) \mathcal{B}_i          +\mathfrak{F}_{\mathcal{B}_i} , 
	\end{equation*} 
	where
	\begin{equation*}
		\mathfrak{F}_{\mathcal{B}_i}= \frac{1}{\tau \uf^{\frac{1}{2}}} \mathscr{S}_{i;42}(\tau, u_0,u, z ; u_i,\mathcal{B}_k) . 
	\end{equation*}
\end{lemma}
\begin{proof}
Starting with the equation \eqref{e:bzeq1} from Lemma \ref{t:gb1}.\ref{l:1.3}, and utilizing \eqref{e:dzrho0} (i.e., $\del{\zeta^i} \underline{\varrho} =	\underline{\varrho_i} +\mathsf{b}_i \underline{\varrho_0} $), we obtain 
		\begin{equation}\label{e:dtbz1}
		\del{\tau} \mathsf{b}_{i}	
		=   \frac{2  \mathsf{b}_i \left(\underline{\varrho } +1\right)^{\frac{1}{3}}}{3 A B (-\tau )^{\frac{2}{3 A}+1} \mathsf{b}^{\frac{1}{3}} \underline{\varrho } }       -\frac{  \mathsf{b}^{\frac{2}{3}} (\underline{\varrho_i} +\mathsf{b}_i \underline{\varrho_0} ) \left(\underline{\varrho } +1\right)^{\frac{1}{3}}}{A B  (-\tau )^{\frac{2}{3 A}+1} \underline{\varrho }^2}   +\frac{(-\tau )^{-\frac{2}{3 A}-1} \mathsf{b}^{\frac{2}{3}} (\underline{\varrho_i} +\mathsf{b}_i \underline{\varrho_0} ) }{3 A B  \underline{\varrho }  \left(\underline{\varrho } +1\right)^{\frac{2}{3}}}   . 
	\end{equation} 
First, we calculate the relationship between  $\del{\tau}\mathcal{B}_i$ and $	\del{\tau} \mathsf{b}_{i}	$.  By differentiating \eqref{e:v5} with respect to $\tau$ and applying \eqref{e:id3}, we derive the first equality in the following calculations
\begin{align}\label{e:dtbz2}
	&	\del{\tau}\mathcal{B}_i=  \del{\tau }\biggl(\frac{\mathsf{b}_{\uparrow}^{-\frac{2}{3}}(1+\uf)^{\frac{2}{3}}}{  B^{-1}  (-\tau)^{-\frac{2}{3A}}  }  \biggr) \mathsf{b}_i +\frac{(1+\uf )^2}{\mathsf{b}_\uparrow^2  \underline{f_0} } \del{\tau } \mathsf{b}_i  \notag  \\
		= & -\frac{2 B (-\tau )^{\frac{2}{3 A}} \left(\underline{f}+1\right)^{2/3} \del{\tau} \mathsf{b}_\uparrow }{3 \mathsf{b}_\uparrow^{5/3}}   \mathsf{b}_i +\frac{2 B (-\tau )^{\frac{2}{3 A}} \del{\tau} \underline{f} }{3 \mathsf{b}_\uparrow^{2/3} (\underline{f} +1)^{\frac{1}{3}}}   \mathsf{b}_i  -\frac{2 B (-\tau )^{\frac{2}{3 A}-1} \left(\underline{f} +1\right)^{2/3}}{3 A \mathsf{b}_\uparrow^{2/3}}   \mathsf{b}_i  +\frac{(1+\uf )^2}{\mathsf{b}_\uparrow^2  \underline{f_0} } \del{\tau } \mathsf{b}_i  .   
\end{align} 
In terms of  \eqref{e:v1}--\eqref{e:v4},  and  replacing $\del{\tau}\mathsf{b}_\uparrow$ and $\del{\tau}\uf$ with Lemmas \ref{t:gb2}.\ref{l:2.3} and \ref{t:gb2}.\ref{l:2.5}, respectively, and substituting $\del{\tau}\mathsf{b}_i$ with \eqref{e:dtbz1}, we
apply \eqref{e:id1} from  Lemma \ref{t:id0}. After  lengthy but direct calculations,  \eqref{e:dtbz2} becomes   
\begin{align*}%\label{e:dtbz3}
	\del{\tau}\mathcal{B}_i
	= & -\frac{2     \left(\underline{f}+1\right)  }{3 A  \underline{f}  (-\tau)  \mathsf{b}_\uparrow }      \mathsf{b}_i +\frac{2  \underline{f_0}    }{3   A   \uf   (-\tau) }    	  \mathsf{b}_i -\frac{2 B  \left(\underline{f} +1\right)^{\frac{2}{3}  }}{3 A \mathsf{b}_\uparrow^{\frac{2}{3} } (-\tau )^{-\frac{2}{3 A}+1} }   \mathsf{b}_i  + \frac{2}{3}  \frac{(1+\uf )^2}{\mathsf{b}_\uparrow^2  \underline{f_0} }   
	\frac{   \mathsf{b}_\uparrow^{-\frac{1}{3}}   (1+\uf)^{\frac{1}{3}}   }{A B \uf  (-\tau)^{\frac{2}{3A}+1}}  \frac{  (1+\frac{\uf}{1+\uf} u)^{\frac{1}{3}}  }{ (1+z) (1+ u)   }  
	\mathsf{b}_i  \notag  \\
	& \hspace{-1cm} - \frac{(1+\uf )^3}{\mathsf{b}_\uparrow^2  \underline{f_0} }
	\frac{   \mathsf{b}_\uparrow^{\frac{2}{3}}   (1+\uf)^{\frac{1}{3}}   }{A B \uf^{2} (-\tau)^{\frac{2}{3A}+1}}  \frac{\bigl(   1+ z \bigr)^{2} (1+\frac{\uf}{1+\uf} u)^{\frac{1}{3} }  }{ (1+ u)^{2}  }    u_i      - \frac{(1+\uf )^2}{\mathsf{b}_\uparrow^2    } 
	\frac{   \mathsf{b}_\uparrow^{\frac{2}{3}}   (1+\uf)^{\frac{1}{3}}   }{A B \uf^{2} (-\tau)^{\frac{2}{3A}+1}}  \frac{\bigl(   1+  z \bigr)^{2} (1+\frac{\uf}{1+\uf} u)^{\frac{1}{3} } 	  (1 +  u_0 )  }{ (1+ u)^{2}  }  
\mathsf{b}_i  \notag   \\
	&  \hspace{-1cm} +  \frac{1}{3} \frac{(1+\uf )^3 }{\mathsf{b}_\uparrow^2  \underline{f_0} }  
	\frac{   \mathsf{b}_\uparrow^{\frac{2}{3}}   (1+\uf)^{-\frac{2}{3}}   }{A B \uf  (-\tau)^{\frac{2}{3A}+1}}  \frac{\bigl(   1+ z \bigr)^{2}   }{ (1+ u)   \bigl(1+\frac{\uf}{1+\uf} u\bigr)^{\frac{2}{3}}}  
	  u_i       +  \frac{1}{3} \frac{(1+\uf )^2}{\mathsf{b}_\uparrow^2   }    \frac{   \mathsf{b}_\uparrow^{\frac{2}{3}}   (1+\uf)^{-\frac{2}{3} }   }{A B \uf  (-\tau)^{\frac{2}{3A}+1}}  \frac{\bigl(   1+  z \bigr)^{2} (1 +  u_0 ) }{ (1+\frac{\uf}{1+\uf} u)^{ \frac{2}{3}  }  (1+ u)   }      \mathsf{b}_i   .
\end{align*} 
By appropriately substituting  $\mathsf{b}_i$ using \eqref{e:v5} and \eqref{e:id2} as follows, we obtain
\begin{align*}%\label{e:dtbz3}
	\del{\tau}\mathcal{B}_i
	= & -\frac{2     }{3 A    (-\tau)   \uf^{\frac{1}{2}}}      \biggl(\frac{\underline{\chi_\uparrow} }{B  }\biggr)^{\frac{1}{2}}    \mathcal{B}_i +\frac{2  \underline{f_0}    }{3   A   \uf   (-\tau) }    	  \frac{\underline{\chi_\uparrow} }{B  } \frac{ \uf   }{\ufo  }\mathcal{B}_i -\frac{2 B \left(\underline{f} +1\right)^{\frac{2}{3}  }}{3 A \mathsf{b}_\uparrow^{\frac{2}{3} }(-\tau )^{-\frac{2}{3 A}+1}  }   \frac{  B^{-1}  (-\tau)^{-\frac{2}{3A}}  }  {(\mathsf{b}_{\uparrow} )^{-\frac{2}{3}}(1+\uf )^{\frac{2}{3}}}	\mathcal{B}_i \notag  \\
	& + \frac{2}{3}  \frac{(1+\uf )^2}{\mathsf{b}_\uparrow^2  \underline{f_0} }   
	\frac{   \mathsf{b}_\uparrow^{-\frac{1}{3}}   (1+\uf)^{\frac{1}{3}}   }{A B \uf  (-\tau)^{\frac{2}{3A}+1}}  \frac{\bigl(   1+  z \bigr)^{-1} (1+\frac{\uf}{1+\uf} u)^{\frac{1}{3}}  }{ (1+ u)   }  
		\frac{\mathsf{b}_\uparrow^2  \underline{f_0} }{(1+\uf )^2}	 \mathcal{B}_i  \notag  \\
	&  - \frac{(1+\uf )^3}{\mathsf{b}_\uparrow^2  \underline{f_0} }
	\frac{   \mathsf{b}_\uparrow^{\frac{2}{3}}   (1+\uf)^{\frac{1}{3}}   }{A B \uf^{2} (-\tau)^{\frac{2}{3A}+1}}  \frac{\bigl(   1+ z \bigr)^{2} (1+\frac{\uf}{1+\uf} u)^{\frac{1}{3} }  }{ (1+ u)^{2}  }    u_i   \notag   \\
	&  - \frac{(1+\uf )^2}{\mathsf{b}_\uparrow^2    } 
	\frac{   \mathsf{b}_\uparrow^{\frac{2}{3}}   (1+\uf)^{\frac{1}{3}}   }{A B \uf^{2} (-\tau)^{\frac{2}{3A}+1}}  \frac{\bigl(   1+  z \bigr)^{2} (1+\frac{\uf}{1+\uf} u)^{\frac{1}{3} } 	(1 +  u_0 )  }{ (1+ u)^{2}  }  
	\frac{\mathsf{b}_\uparrow^2 \underline{f_0} }{(1+\uf )^2}	 \mathcal{B}_i  \notag   \\
	&   +  \frac{1}{3} \frac{(1+\uf )^3 }{\mathsf{b}_\uparrow^2  \underline{f_0} }  
	\frac{   \mathsf{b}_\uparrow^{\frac{2}{3}}   (1+\uf)^{-\frac{2}{3}}   }{A B \uf  (-\tau)^{\frac{2}{3A}+1}}  \frac{\bigl(   1+ z \bigr)^{2} (1+\frac{\uf}{1+\uf} u)^{-\frac{2}{3}}  }{ (1+ u)   }  
	u_i    \notag   \\
	&   +  \frac{1}{3} \frac{(1+\uf )^2}{\mathsf{b}_\uparrow^2   }    \frac{   \mathsf{b}_\uparrow^{\frac{2}{3}}   (1+\uf)^{-\frac{2}{3} }   }{A B \uf  (-\tau)^{\frac{2}{3A}+1}}  \frac{\bigl(   1+  z \bigr)^{2} (1 +  u_0 ) }{ (1+\frac{\uf}{1+\uf} u)^{ \frac{2}{3}  }  (1+ u)   }      	\frac{\mathsf{b}_\uparrow^2 \underline{f_0} }{(1+\uf )^2}	 \mathcal{B}_i    .
\end{align*}
Then applying Lemma \ref{t:iden1} and utilizing \eqref{e:chig}, we perform lengthy but direct calculations and expansions to arrive at 
\begin{align*}%\label{e:dtbz3}
	\del{\tau}\mathcal{B}_i
	= &  \frac{2}{3}  \frac{ 1 }{A  \tau   }    u_i       + \frac{2}{3} \frac{1 }{ A  \tau   }    	\mathcal{B}_i  +\frac{2     }{3 A     \tau   \uf^{\frac{1}{2}}}      \biggl(\frac{\underline{\chi_\uparrow} }{B  }\biggr)^{\frac{1}{2}}   \Biggl(1- \frac{  \bigl(1+\frac{\uf}{1+\uf} u\bigr)^{\frac{1}{3}}  }{ ( 1+  z) (1+ u)   }  
	\Biggr) \mathcal{B}_i    +    \frac{  \underline{\chi_\uparrow}    }{A B   \tau \uf }   
	\mathcal{B}_i  \notag  \\
	&  + \biggl(   \frac{5}{3}   +\frac{1}{\uf}\biggr)  \frac{ 1 }{A  \tau   \uf }    u_i    + \biggl(1+\frac{1}{\uf}\biggr)^2  \frac{ 1 }{A  \tau   }  \biggl( \frac{\bigl(   1+ z \bigr)^{2} \bigl(1+\frac{\uf}{1+\uf} u\bigr)^{\frac{1}{3} }  }{ (1+ u)^{2}  }   -1  \biggr) u_i   \notag   \\
	&  + \biggl(1+ \frac{1 }{\uf}\biggr) \frac{  \underline{\chi_\uparrow}    }{A B   \tau  }  \Biggl( \frac{ (   1+  z  )^{2} \bigl(1+\frac{\uf}{1+\uf} u \bigr)^{\frac{1}{3} } 	(1 +  u_0 )   }{ (1+ u)^{2}  }  -1 \Biggr)
	\mathcal{B}_i  \notag   \\
	&   -  \frac{1}{3} \biggl(1+ \frac{1 }{\uf}\biggr)  \frac{  1  }{A \tau  }  \biggl( \frac{ (   1+ z  )^{2}   }{ (1+ u) \bigl(1+\frac{\uf}{1+\uf} u\bigr)^{\frac{2}{3}}   }  -1 \biggr)
	u_i        -  \frac{1}{3}    	  \frac{  \underline{\chi_\uparrow}   }{A B \tau   }  \biggl( \frac{\bigl(   1+  z \bigr)^{2} (1 +  u_0 ) }{ (1+\frac{\uf}{1+\uf} u)^{ \frac{2}{3}  }  (1+ u)   }  -1 \biggr)    \mathcal{B}_i  . 
\end{align*}
This equation verifies this lemma and we thereby finish the proof.   
\end{proof}

\begin{lemma}\label{e:Seq6}
		In terms of variables \eqref{e:v1}--\eqref{e:v5}, there is an equation,  
	\begin{equation*}
		\del{\tau} z+	\frac{\mathtt{d}^i}{A\tau} \del{\zeta^i}  z=   \mathfrak{F}_z  , 
	\end{equation*}
where $\mathtt{d}^i$ are any given constants and
\begin{align*}
	\mathfrak{F}_z= \frac{1}{\tau \uf^{\frac{1}{2}}} \mathscr{S}_{52}(\tau; u,  \mathcal{B}_i, z) :=  - \frac{1}{3} 	\frac{1 }{A \tau    \uf^{\frac{1}{2}}    }    \biggl( \frac{\underline{\chi_\uparrow} }{B}\biggr)^{\frac{1}{2}}   \Biggl(    \frac{  (1+\frac{\uf}{1+\uf} u)^{\frac{1}{3}}  }{ (1+ u)   }    -  1 - z -  \frac{  \uf }{(1+\uf) } \frac{ \mathtt{d}^i     	\mathcal{B}_i }{(1+z)^2} \Biggr)  . 
\end{align*}    
\end{lemma}
\begin{proof}
On one hand,  using \eqref{e:v4} 
and differentiating $\bigl( \mathsf{b} /\mathsf{b}_\uparrow  \bigr)^{\frac{1}{3}}$ with respect to $\tau$ yield,  
\begin{equation}\label{e:dtbb1}
	\del{\tau }	\biggl(\frac{\mathsf{b} }{\mathsf{b}_\uparrow }\biggr)^{\frac{1}{3}}  =   \del{\tau } z  .   
\end{equation}
On the other hand, note
\begin{equation}\label{e:dtbb2}
	\del{\tau }	\biggl(\frac{\mathsf{b} }{\mathsf{b}_\uparrow }\biggr)^{\frac{1}{3}}   
	= \frac{1}{3} \biggl(\frac{\mathsf{b} }{\mathsf{b}_\uparrow }\biggr)^{-\frac{2}{3}}  		\biggl(\frac{\del{\tau }\mathsf{b} }{\mathsf{b}_\uparrow }- \frac{\mathsf{b}}{\mathsf{b}_\uparrow^2}\del{\tau} \mathsf{b}_\uparrow\biggr)  . 
\end{equation} 
By utilizing  \eqref{e:dtbup1} from Lemma \ref{t:gb2}.\ref{l:2.3} and \eqref{e:tmeqi1} from Lemma \ref{t:gb1}.\ref{l:1.2}, and applying  \eqref{e:v4} once more,  \eqref{e:dtbb2} becomes 
	\begin{equation}\label{e:dtbb3}
		\del{\tau }	\biggl(\frac{\mathsf{b} }{\mathsf{b}_\uparrow }\biggr)^{\frac{1}{3}} 
 	=   \frac{1}{3} \biggl(\frac{\mathsf{b} }{\mathsf{b}_\uparrow }\biggr)^{\frac{1}{3}}  	  \biggl(  \frac{\mathsf{b}^{-\frac{1}{3}}   ( 1 + \underline{\varrho} )^{\frac{1}{3}}}{A B  \underline{\varrho}  \left(-\tau\right){}^{\frac{2}{3 A}+1}}   - \frac{\mathsf{b}_\uparrow^{-\frac{1}{3}}  (1+\underline{f}  )^{\frac{1}{3}}}{AB  \underline{f}  (-\tau)^{\frac{2}{3A}+1}}\biggr)     
		=    \frac{1}{3} \bigl(1+ z \bigr)	  \biggl(  \frac{\mathsf{b}^{-\frac{1}{3}}   ( 1 + \underline{\varrho} )^{\frac{1}{3}}}{A B  \underline{\varrho}  \left(-\tau\right){}^{\frac{2}{3 A}+1}}   - \frac{\mathsf{b}_\uparrow^{-\frac{1}{3}}  (1+\underline{f}  )^{\frac{1}{3}}}{AB  \underline{f}  (-\tau)^{\frac{2}{3A}+1}}\biggr)   . 
	\end{equation}
Inserting \eqref{e:v1} into \eqref{e:dtbb3} and applying \eqref{e:id1} from Lemma \ref{t:id0} imply
\begin{align}\label{e:dtbb5}
		\del{\tau }	\biggl(\frac{\mathsf{b} }{\mathsf{b}_\uparrow }\biggr)^{\frac{1}{3}}   
	= &   \frac{1}{3} \bigl(1+z \bigr)	 \frac{   \mathsf{b}_\uparrow^{-\frac{1}{3}}   (1+\uf)^{\frac{1}{3}}   }{A B \uf  (-\tau)^{\frac{2}{3A}+1}}   \Biggl(    \frac{\bigl(   1+ z \bigr)^{-1} (1+\frac{\uf}{1+\uf} u)^{\frac{1}{3}}  }{ (1+ u)   }    - 1 \Biggr)  .
\end{align}
Combining \eqref{e:dtbb1} and \eqref{e:dtbb5} together, with the help of Lemma \ref{t:iden1}, we arrive at
\begin{equation}\label{e:dtz1}
	\del{\tau} z =     - \frac{1}{3} 	\frac{1 }{A \tau    \uf^{\frac{1}{2}}    }    \biggl( \frac{\underline{\chi_\uparrow} }{B}\biggr)^{\frac{1}{2}}   \Biggl(    \frac{  (1+\frac{\uf}{1+\uf} u)^{\frac{1}{3}}  }{ (1+ u)   }    -  1 - z  \Biggr) . 
\end{equation}

Next let us find an expression for $\del{\zeta^i} z$. 
First by \eqref{e:iden3}, note
\begin{equation}\label{e:fb1}
 \mathsf{b}_{\uparrow}^{-\frac{1}{3}}  = \frac{\underline{\chi_\uparrow}^{\frac{1}{2}}}{B^{\frac{1}{2}}	\uf^{\frac{1}{2}}}  \frac{B\uf(-\tau)^{\frac{2}{3A}}}{(1+\uf)^{\frac{1}{3}}}   . 
\end{equation}
By differentiating \eqref{e:v4} with respect to $\zeta^i$ and substituting $\mathsf{b}$ and  $\mathsf{b}_i$ using \eqref{e:v4} and \eqref{e:v5}, along with \eqref{e:fb1}, we obtain
\begin{equation}\label{e:dzz1}
	\del{\zeta^i}  z=    \frac{1}{3}  \mathsf{b}_\uparrow^{-\frac{1}{3}}  \mathsf{b}^{-\frac{2}{3}}    \mathsf{b}_i = \frac{1}{3}  \mathsf{b}_\uparrow^{-\frac{1}{3}} (1+z)^{-2} \mathsf{b}_\uparrow^{-\frac{2}{3}}    \frac{  B^{-1}  (-\tau)^{-\frac{2}{3A}}  }  {\mathsf{b}_{\uparrow}^{-\frac{2}{3}}(1+\uf )^{\frac{2}{3}}}	\mathcal{B}_i  
	=   \frac{1}{3}  \biggl(\frac{\underline{\chi_\uparrow}}{B}\biggr)^{\frac{1}{2}} \frac{ 1}{  \uf^{\frac{1}{2}}}  \frac{ \uf }{(1+\uf) }      	\frac{\mathcal{B}_i }{  (1+  z  )^{2}  }   . 
\end{equation}
Next, introduce an arbitrary vector $q^i$ to be determined.  With this,  \eqref{e:dzz1} becomes
\begin{equation}\label{e:dzz2}
	\frac{\mathtt{d}^i}{A\tau} \del{\zeta^i}  z	=  	\frac{1}{A\tau  \uf^{\frac{1}{2}}} \frac{1}{3}  \biggl(\frac{\underline{\chi_\uparrow}}{B}\biggr)^{\frac{1}{2}}    \frac{ \uf }{(1+\uf) } \frac{ \mathtt{d}^i  	\mathcal{B}_i }{(1+z)^2}   . 
\end{equation}

Combine \eqref{e:dtz1} and \eqref{e:dzz2}, we obtain 
\begin{equation*}
	\del{\tau} z+	\frac{\mathtt{d}^i}{A\tau} \del{\zeta^i}  z=     - \frac{1}{3} 	\frac{1 }{A \tau    \uf^{\frac{1}{2}}    }    \biggl( \frac{\underline{\chi_\uparrow} }{B}\biggr)^{\frac{1}{2}}   \Biggl(    \frac{  (1+\frac{\uf}{1+\uf} u)^{\frac{1}{3}}  }{ (1+ u)   }    -  1 - z -  \frac{  \uf }{(1+\uf) } \frac{ \mathtt{d}^i     	\mathcal{B}_i }{(1+z)^2} \Biggr)  . 
\end{equation*}
This is consistent with the lemma and completes the proof. 
\end{proof}

Gathering Lemma \ref{t:Seq1}--\ref{e:Seq6} together, we arrive at the following \textit{singular} system. 
\begin{lemma}\label{t:mainsys1}
	The main equation  \ref{Eq2} can be reformulated as a first order singular symmetric hyperbolic system as follows, 
	\begin{align}\label{e:mainsys1}
		\mathbf{A}^0\del{\tau}U+\frac{1}{A\tau} \mathbf{A}^i \del{\zeta^i} U =\frac{1}{A\tau} \mathbf{A}   U +\mathbf{F}  , 
	\end{align}
where $U:=\p{u_0,u_j,u,\mathcal{B}_l,z}^T$,  $\mathbf{F}=\p{\mathfrak{F}_{u_0}, \mathfrak{F}_{u_j},  \mathfrak{F}_{u}, \mathfrak{F}_{\mathcal{B}_i}, \mathfrak{F}_{z}  }^T$, 
\begin{align*}
	\mathbf{A}^0=\p{1 & \mathscr{R}^j  & 0 & 0 & 0 \\
	\mathscr{R}_k      &  (S+\mathscr{L}) \delta^j_k  & 0 & 0 & 0 \\
  0&0&2&0&0\\
0&0&0&\delta^l_s&0\\
0&0&0&0&1 } , \qquad
	\mathbf{A}^i=\p{0 & H^{ij} & 0 & 0 & 0 \\
	H^i_k   & 0 & 0 & 0 & 0\\
0 & 0 &0&0& 0 \\
0&0&0&0&0\\
0&0&0&0&\mathtt{d}^i }  , 
\end{align*} 
\begin{align*}
	\mathbf{A}=\p{-\frac{14  }{3    }+\mathscr{Z}_{11}&   -  4\ck q^j             + \mathscr{Z}^j_{12}  &8+\mathscr{Z}_{13}  & 0 & -8+\mathscr{Z}_{15}   \\
	0 &(4\ck+\mathscr{Z}_{22})\delta^j_k &0&(24   \ck + \mathscr{Z}_{24}) \delta^l_k & 0 \\
	- 8   +\mathscr{Z}_{31} & 0 & \frac{ 40 }{3 } +\mathscr{Z}_{33} & 0 & -16+\mathscr{Z}_{35} \\0&(\frac{2  }{3       }  +\mathscr{Z}_{42})\delta^j_{s} &0& (\frac{2}{3   }  +\mathscr{Z}_{44} ) \delta_{s}^l  & 0 \\ 0 & 0 &0 &0 &0}     .  
\end{align*}   
\end{lemma} 
\begin{remark}
	We note that the system \eqref{e:mainsys1}, as presented in Lemma \ref{t:mainsys1}, is not a Fuchsian system as defined in Appendix \ref{s:fuc} because it does not satisfy Condition \ref{c:5} from Appendix \ref{s:fuc}. Consequently, we cannot apply Theorem \ref{t:fuc} to this system to derive global results. To utilize Theorem \ref{t:fuc}, we need to address two modifications of the system, which will be discussed in the next section.
\end{remark}

\subsection{Domain of influence of the inhomogeneous data}\label{s:DoI}
We will use the domain of influence of the inhomogeneous data, also referred to as the \textit{inhomogeneous domain} in this article, in various coordinate systems. To facilitate this, we first calculate it in the  $(\tau,\zeta)$ coordinate system, which can be  transformed to subsequent coordinate systems such as $(\ttau,\txi)$ and $(\htau,\hat{\zeta})$. For simplicity, we set  $\ck=\frac{1}{4}$ and $\mathtt{d}^i=0$ based on our requirements. 

\begin{lemma}\label{t:char1}
	Suppose there is a classical solution $U$ to the system \eqref{e:mainsys1}  within some spacetime domain, if the initial data satisfies $\supp U(\zeta)|_{\tau=-1}=B_1(0)$, then the domain of influence of the initial support $B_1(0)$ is enclosed within or on the characteristic conoid $\underline{\mathcal{C}} $ which is given by a hypersurface
	\begin{equation*}
			|\zeta|= 1 +\frac{1}{A} \int_{-1}^{\tau} \frac{1}{-s} \sqrt{S} \frac{\underline{\chi_\uparrow}}{B}ds=1 -\frac{2}{A} \ln(-\tau)+\frac{1}{A} \int_{-1}^{\tau} \frac{\Xi(s)}{-s} ds ,
	\end{equation*}
	where 	  
	\begin{equation*}
		\Xi 
		:=   2\sqrt{\biggl(1+\frac{\underline{\mathfrak{G}}}{4B}\biggr)\biggl(1+\cm^2\frac{\underline{\mathfrak{G}}}{B} +\frac{1-4\cm^2}{\uf} \biggr)}-2  \geq   0   \AND \lim_{\tau\rightarrow 0}\Xi(\tau)=0 . 
	\end{equation*} 
	and the integral $\frac{1}{A} \int_{-1}^{\tau} \frac{\Xi(s)}{-s} ds$ is nonnegative and bounded for $\tau\in[-1,0)$. We denote the maximum value of this integral as $\theta:=\max_{\tau\in[-1,0]} \frac{1}{A} \int_{-1}^{\tau} \frac{\Xi(s)}{-s} ds$. 
\end{lemma}
\begin{proof}
	On the boundary of the domain of influence, $U=0$ because $U$ is the classical solution and the boundary is where the perturbation $U$ ends.  Consequently,  $\mathbf{A}^0=\mathbf{A}^0|_{U=0}$ and $\mathbf{A}^i=\mathbf{A}^i|_{U=0}$ can be simplified along this boundary. By applying standard normalization techniques, we can normalize $\mathbf{A}^0$  to the identity matrix $\mathds{1}$ (see, e.g., \cite[Ch. VI. \S$3.8$]{Courant2008}). According to \cite[\S $4.2$--$4.3$ and Theorem $4.6$]{Lax2006}, this enables us to construct the boundary of the domain of influence for the inhomogeneous data. 
	\begin{equation}\label{e:chareq1}
		\frac{d\zeta^i}{d\tau}=-\frac{1}{A\tau} \sqrt{S} \frac{\chi_\uparrow}{B} \frac{\zeta^i}{|\zeta|}=-\frac{1}{A\tau} \bigl(2+\Xi(\tau)\bigr) \frac{\zeta^i}{|\zeta|} \quad \overset{\text{multiplying } \zeta_i/|\zeta|}{\Rightarrow} \quad \frac{d|\zeta|}{d\tau}=-\frac{1}{A\tau} \sqrt{S} \frac{\chi_\uparrow}{B}  =-\frac{1}{A\tau} \bigl(2+\Xi(\tau)\bigr) . 
	\end{equation}

Integrating equation  \eqref{e:chareq1} concludes the expression for the hypersurface $\underline{\mathcal{C}}$. 
Lemma \ref{t:Thpst} ($\underline{\mathfrak{G}}>0$) implies $\frac{1}{A} \int_{-1}^{\tau} \frac{\Xi(s)}{-s} ds\geq0$. We next point out $\frac{1}{A} \int_{-1}^{\tau} \frac{\Xi(s)}{-s} ds$ is bounded. Specifically, for  $\tau\in[-1,0]$, we have $
	0\leq	\Xi(\tau) \leq C_1 \underline{\mathfrak{G}}+C_2 \uf^{-1}$. 
Applying  Lemma \ref{t:Gest2} and Proposition \ref{s:gf1/2}, we obtain
		\begin{equation*}
	0\leq 	\frac{1}{A} \int_{-1}^{\tau} \frac{\Xi(s)}{-s} ds \leq 	\frac{C_1}{A} \int_{-1}^{\tau} \frac{ \underline{\mathfrak{G}}}{-s} ds+	\frac{C_2 }{A} \int_{-1}^{\tau} \frac{1}{(-s) \uf(s)} ds\leq 	\frac{C_1}{A} \int_{-1}^{\tau} \frac{1}{(-s)^{\frac{1}{2}}} ds+	\frac{C_2 }{A} \int_{-1}^{\tau} (-s)   ds 
		\leq C .
	\end{equation*}
This completes the proof. 
\end{proof}

\begin{corollary}
In terms of the coordinate $(t,x^i)$, the hypersurface $\underline{\mathcal{C}}$ becomes $\mathcal{C}$ defined by \eqref{e:char1}.  
\end{corollary}
\begin{proof}
Applying the coordinate transformation given in   \eqref{e:coord2},  along with the results from \eqref{e:keyid3},  Lemma \ref{t:Thpst} and Lemma \ref{t:gb2}.\ref{l:2.2},  we derive
\begin{align*}
	|x|= 	|\zeta| =&1 +\frac{1}{A} \int_{-1}^{\tau} \frac{1}{-s} \sqrt{S} \frac{\underline{\chi_\uparrow}}{B}ds   
	\overset{\text{Lem.}\ref{t:Thpst}}{=}  1 +\frac{1}{A} \int_{-1}^{\tau} \frac{2}{-s}  \sqrt{\frac{\underline{\chi_\uparrow}}{B}\left(\frac{\cm^2}{4}\frac{\underline{\chi_\uparrow}}{B}+\left(\frac{1}{4}-\cm^2\right)\frac{1+\uf}{\uf}\right)} ds\notag  \\ 
\overset{\eqref{e:keyid3}}{=}& 1 +\frac{1}{A} \int_{t_0}^{t} \frac{1}{-\mfg(y)}  \sqrt{\frac{y^2f_0^2(y)}{(1+f(y))^2f(y)}\left( \cm^2 \frac{y^2f_0^2(y)}{(1+f(y))^2f(y)}+\left(1-4\cm^2\right)\frac{1+f(y)}{f(y)}\right)} d\mfg(y)\notag  \\ 
\overset{\text{Lem.} \ref{t:gb2}.\eqref{l:2.2}}{=} & 1 +  \int_{t_0}^{t}   \sqrt{\frac{y^2f_0^2(y)}{(1+f(y))^2f(y)}\left( \cm^2 \frac{y^2f_0^2(y)}{(1+f(y))^2f(y)}+\left(1-4\cm^2\right)\frac{1+f(y)}{f(y)}\right)}  \frac{f(y)(f(y)+1)}{y^2f_0(y)}dy \notag  \\ 
	=& 1 +  \int_{t_0}^{t}   \sqrt{   \frac{\cm^2 f_0^2(y)}{(1+f(y))^2 } +\left(1-4\cm^2\right)\frac{1+f(y)}{y^2 }   } dy 	 
	=  1 +  \int_{t_0}^{t}   \sqrt{  \cg(y,f(y),f_0(y)) } dy ,  
\end{align*}
where we recall the definition  \eqref{e:Fdef},  $\cg(y,\varrho,\varrho_0)=\cm^2\frac{ (\varrho_0 )^2}{(1+\varrho )^2}  + 4(\ck-\cm^2) \frac{1+\varrho  }{y^2} $. This completes the proof.
\end{proof}

\begin{remark}
The hypersurface $\mathcal{C}$ can also be directly derived from the wave operator $\partial^2_t  -  	\cg \delta^{ij}  \del{i}\del{j} $ where $\cg $ is defined by \eqref{e:Fdef}. This is because the optical function for $|x_0|\geq 1$ is given by $-|x|+|x_0|+\int^t_{t_0}\sqrt{g}dy$. Moreover, $x^1=x^1_0+\int^t_{t_0}\sqrt{g}dy$ are null geodesics. 
\end{remark}

\begin{corollary}\label{t:homdom}
	Suppose $\mathcal{H}$ is defined by \eqref{e:cah}, then $\varrho=f$ on $\mathcal{H} $. 
\end{corollary}
\begin{proof}
	This can be proven by directly using the symmetric hyperbolic formulation \eqref{e:mainsys1} from Lemma \ref{t:mainsys1} and  \cite[Theorem $4.6$]{Lax2006}. 
	 We omit the details. 
\end{proof}

%%%---------------NEW SEC----------------- 

\section{Fuchsian Formulations for  revised and extended systems}\label{s:4}
As remarked,  \eqref{e:mainsys1} is not a Fuchsian system, To transform it into a Fuchsian system, as described in Appendix \ref{s:fuc}, we will apply  Theorem \ref{t:fuc} to achieve global results for the Fuchsian fields. This approach will also provide estimates for the behavior of the variables $\varrho$ and its derivatives. In this section, let us revise \eqref{e:mainsys1} in two modifications to derive a Fuchsian system.  \S\ref{s:riv1} and \S\ref{s:riv2} address the modifications related to Difficulty $3$ outlined in \S\ref{s:oview}.

\subsection{Revision $1$: The zoom-in (or blowup) coordinate system}\label{s:riv1} 
Recalling \ref{F1} in \S\ref{s:oview}, the \textit{first revision} to achieve the Fuchsian system involves introducing a \textit{zoom-in (or blowup) coordinate system} $(\ttau,\txi)$ which is defined by 
\begin{equation}\label{e:coord5}
	\ttau=\ttau(\tau,\zeta )=\tau  
\AND
	\txi^i=\txi^i(\tau,\zeta  )=     \frac{\ta \tc^i }{A} \ln(- \tau)+  \zeta^i  , 
\end{equation}
where $\tc^i$ and $\ta$ are constants to be determined, and $\txi^i\in(-\infty,\infty)$. 
Its inverse transformation is
\begin{equation}\label{e:coordi5}
	\tau=	\tau(\ttau,\txi )=\ttau \AND \zeta^i= \zeta^i(\ttau,\txi )  =  \txi^i -    \frac{\ta \tc^i }{A}\ln(-\ttau)  . 
\end{equation}
The Jacobian of this transformation is 
\begin{equation}\label{e:Jb2}
	\p{\widetilde{\frac{\partial \ttau}{\partial \tau}} &  \widetilde{\frac{\partial \ttau}{\partial \zeta^j} }  \\  \widetilde{ \frac{\partial \txi^i}{\partial \tau} } & \widetilde{\frac{\partial \txi^i}{\partial  \zeta^j} }  }= \p{1 & 0 \\
		\frac{\ta \tc^i  }{ A \ttau }  &  \delta^i_j  }  \AND \det	\p{\widetilde{\frac{\partial \ttau}{\partial \tau}} &  \widetilde{\frac{\partial \ttau}{\partial \zeta^j} }  \\  \widetilde{ \frac{\partial \txi^i}{\partial \tau} } & \widetilde{\frac{\partial \txi^i}{\partial  \zeta^j} }  }= 1 > 0. 
\end{equation}  
For a function $z(\tau,\zeta )$,  let $\tilde{z}(\ttau,\txi )$ be defined as $\tilde{z}(\ttau,\txi ):=z\bigl(\tau(\ttau,\txi ),\zeta(\ttau,\txi )\bigr)$. Using the Jacobian matrix given by equation \eqref{e:Jb2}, we obtain  
\begin{align}
	\widetilde{\del{\tau}z}= &   \widetilde{\frac{\partial \ttau}{\partial \tau}} \del{\ttau} \tilde{z}+  \widetilde{ \frac{\partial \txi^i}{\partial \tau} }  \del{\txi^i} \tilde{z} =  \del{\ttau} \tilde{z}   + \frac{\ta \tc^i  }{A \ttau }     \del{\txi^i }  \tilde{z}  , \label{e:transf1} \\
	\widetilde{\del{\zeta^i}z}= &   \widetilde{\frac{\partial \ttau}{\partial \zeta^i} } \del{\ttau} \tilde{z}+  \widetilde{\frac{\partial \txi^j}{\partial  \zeta^i} }  \del{\txi^j} \tilde{z} =      \del{\txi^i}  \tilde{z} . \label{e:transf2}
\end{align}

\begin{figure}[htbp]
    \begin{minipage}[t]{0.5\linewidth}
	\centering
	\includegraphics[width=1\textwidth]{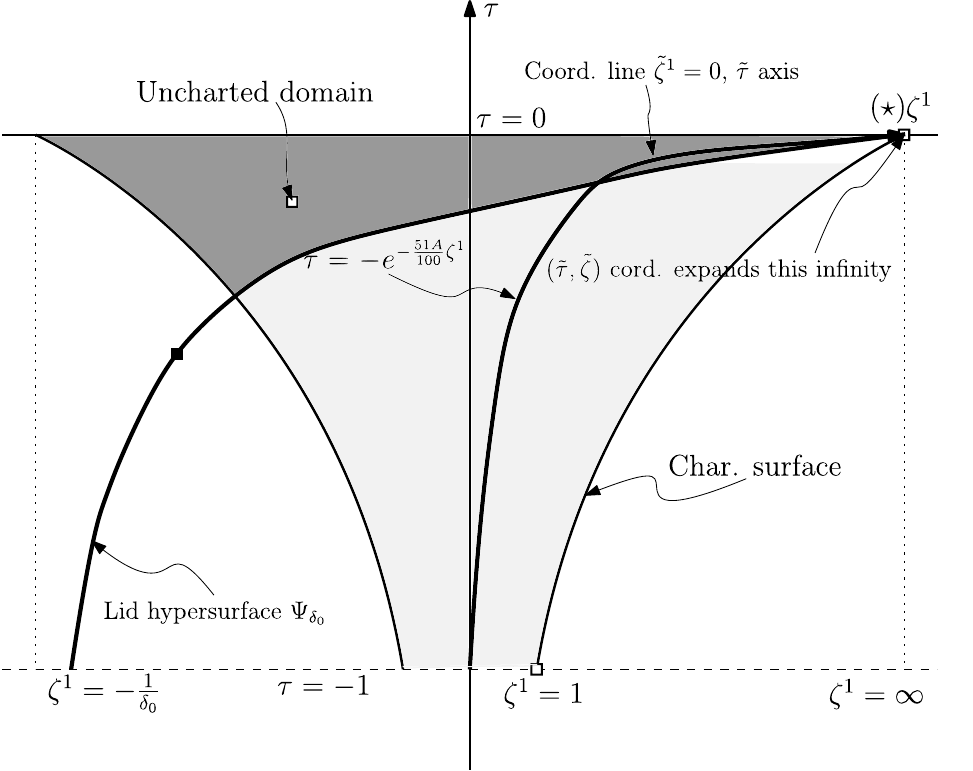}
	\caption{The $t$-compactified  coord.  $(\tau,\zeta)$}
	\label{f:fig3a}
\end{minipage}%
\begin{minipage}[t]{0.5\linewidth}
	\centering
	\includegraphics[width=1\textwidth]{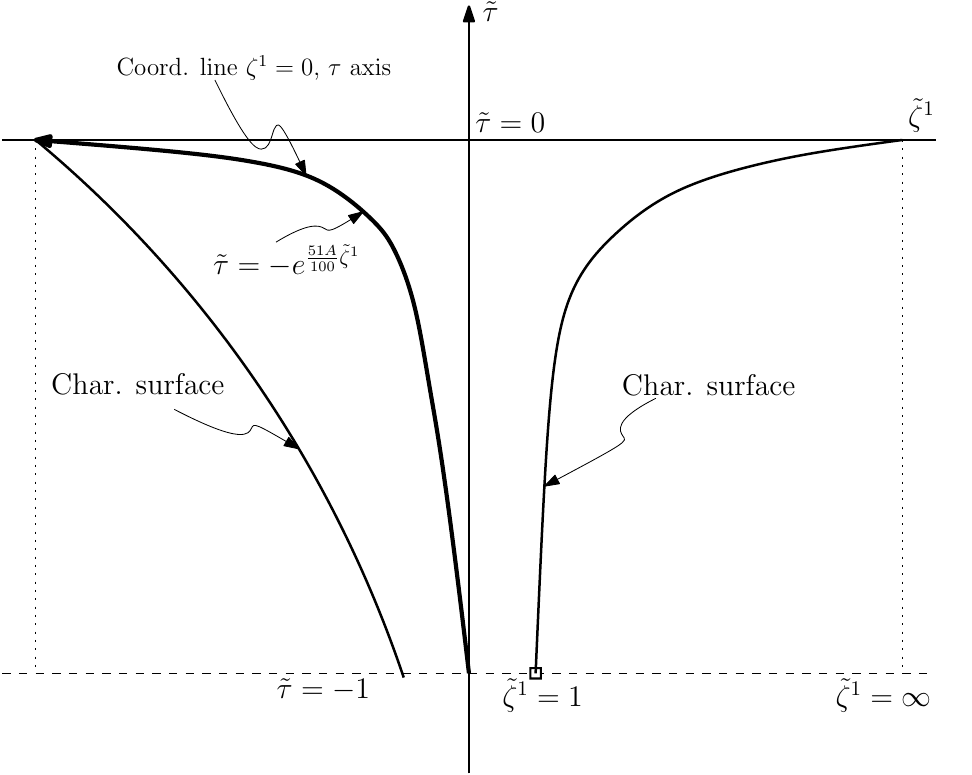}
	\caption{The zoom-in coord. $(\ttau,\txi)$}
	\label{f:fig3b}
\end{minipage}
\end{figure}

\begin{lemma}\label{t:mainsys2}
		In terms of the zoom-in (or blowup) coordinate \eqref{e:coord5}, the singular system \eqref{e:mainsys1} becomes
		\begin{equation}\label{e:mainsys2a}
	\widetilde{\mathbf{A}^0} \del{\ttau} \tilde{U}+\frac{1}{A \ttau } 	\tilde{\mathbf{A}}^{i} \del{\txi^i}  \tilde{U} =\frac{1}{A\ttau}   \widetilde{\mathbf{A}}   \tilde{U} +\tilde{\mathbf{F}} , 
		\end{equation}
	where 
	\begin{align*}
		\tilde{\mathbf{A}}^{i}:=&  
  \ta \tc^i \widetilde{\mathbf{A}^0}   +   \widetilde{\mathbf{A}^i}    
		=      \p{\ta \tc^i  &\ta \tc^i \tilde{\mathscr{R}}^j +  \tilde{H} ^{ij}   & 0 & 0 & 0 \\
\ta \tc^i \tilde{\mathscr{R}}_k   +  \tilde{H}^i_k   & \ta \tc^i (\ck +\mathscr{L}+\tau\mathscr{S}) \delta^j_k   & 0 & 0 & 0 \\
			0&0& 2 \ta \tc^i  &0&0\\
			0&0&0&\ta \tc^i  \delta^l_s &0\\
			0&0&0&0&  \ta \tc^i + \mathtt{d}^i   }  . 
	\end{align*}
\end{lemma}
\begin{proof}
	This lemma can be proven directly using the system \eqref{e:mainsys1} and the transformations \eqref{e:transf1}--\eqref{e:transf2}. 
\end{proof}

\subsection{Revision $2$: The zoom-in   variables}\label{s:riv2}
As discussed in \ref{F2} in \S\ref{s:oview}, the \textit{second revision} involves introducing a \textit{revised set of variables} based on those defined in  \eqref{e:v1}--\eqref{e:v5}. 
\begin{align}
\mfu_0(\ttau,\txi ) = 	\frac{1}{\mu(\txi )} \widetilde{u_0}(\ttau,\txi )   
\quad \Leftrightarrow&\quad \widetilde{u_0}(\ttau,\txi )    =   	\mu(\txi) \mfu_0(\ttau,\txi ) 	, \label{e:fv1}  \\
\mfu_i(\ttau,\txi )   = 	\frac{1}{\mu(\txi )} \widetilde{u_i}(\ttau,\txi )       
\quad \Leftrightarrow&\quad  \widetilde{u_i}(\ttau,\txi )    = 	  	\mu(\txi )	\mfu_i(\ttau,\txi )  , \label{e:fv2}   \\
	\mfu(\ttau,\txi )  =  \tilde{u}(\ttau,\txi )    \quad \Leftrightarrow&\quad    
	\tilde{u}(\ttau,\txi ) = 	 	 	\mfu(\ttau,\txi ) ,\label{e:fv3}  \\	\mfv(\ttau,\txi )  = 	  \frac{1}{\mu(\txi )} \tilde{u}(\ttau,\txi )    \quad \Leftrightarrow&\quad   		\tilde{u}(\ttau,\txi ) = 	 	 \mu (\txi ) 	\mfv(\ttau,\txi )   ,\label{e:fv7}  \\
	\mathfrak{B}_j(\ttau,\txi^I)  =   \frac{1}{\mu(\txi )}	\tilde{\mathcal{B}}_j (\ttau,\txi ) \quad \Leftrightarrow&\quad  	\tilde{\mathcal{B}}_j (\ttau,\txi )  =   	\mu(\txi ) \mathfrak{B}_j(\ttau,\txi )   ,    \label{e:fv4} 
	\\ 
\mathfrak{z}(\ttau,\txi )  =    \frac{1}{\eta(\txi )} \tilde{z}(\ttau,\txi )    \quad \Leftrightarrow&\quad  \tilde{z}(\ttau,\txi)   =    \eta(\txi )  \mathfrak{z}(\ttau,\txi )  ,  \label{e:fv6}    
\end{align} 
where  
\begin{equation}\label{e:eta1}
 	\mu(\txi ):=  \hat{\mu} e^{\omega \kappa_i  \txi^i } ,  \quad 	\eta(\txi ):=  \hat{ \eta } e^{\theta \kappa_i \txi^i}   \quad \Rightarrow \quad \frac{\del{\txi^i}\mu}{\mu} = \omega \kappa_i <0 ,   \quad \frac{\del{\txi^i}\eta}{\eta} = \theta \kappa_i   .   
\end{equation}  
and $\hat{\mu}$,  $ \hat{ \eta } $, $\omega>0$ and $\kappa_i<0$ are constants to be determined, satisfying
\begin{equation}\label{e:ck}
	\tc^i \kappa_i =1 .
\end{equation}

\begin{lemma}\label{t:sigsys}
In terms of the revised variables defined in  \eqref{e:fv1}--\eqref{e:fv6}, the singular system \eqref{e:mainsys2a} becomes
		\begin{align}\label{e:mainsys3}
		\mathfrak{A}^0\del{\ttau}\mathfrak{U}+\frac{1}{A\ttau} \mathfrak{A}^{i} \del{\txi^i} \mathfrak{U} =\frac{1}{A\ttau} \mathfrak{A}  \mathfrak{U} +\mathfrak{F} , 
	\end{align} 
	where   $\mathfrak{U}:=\p{\mfu_0,\mfu_j,\mfu,\mathfrak{B}_l, \mathfrak{z},\mfv}^T$,  
\begin{align*}
	\mathfrak{A}^0:=	& \p{1 & \mu \tilde{R}  \mathfrak{B}_i \delta^{ij}  & 0 & 0 & 0 & 0 \\
	\mu \tilde{R} \mathfrak{B}_k     &  (\tilde{S} +\widetilde{\mathscr{L}}  ) \delta^j_k  & 0 & 0 & 0 & 0\\
		0&0&2&0&0&0\\
		0&0&0&\delta^l_s&0 & 0 \\
		0&0&0&0&1 & 0\\
		0&0&0&0&0&2}  ,  
\end{align*}  
\begin{align*}
	\mathfrak{A}^{i}:=  \p{\ta \tc^i  & 	   \tilde{S}  \frac{\widetilde{\underline{\chi_\uparrow}}}{B}  \delta^{ij}  + 	\widetilde{\mathscr{Z}}^{ij}   & 0 & 0 & 0 &0\\
	(  \tilde{S} \frac{\widetilde{\underline{\chi_\uparrow}}}{B}  \delta^{ij}  + 	\widetilde{\mathscr{Z}}^{ij}  ) \delta_{jk}   & \ta \tc^i (\tilde{S} +\widetilde{\mathscr{L}} ) \delta^j_k   & 0 & 0 & 0 &0\\
		0&0& 2 \ta \tc^i  &0&0&0\\
		0&0&0&\ta \tc^i  \delta^l_s &0&0\\
		0&0&0&0&  \ta \tc^i + \mathtt{d}^i  & 0 \\0&0&0&0&0& 2 \ta\tc^i }  , 
\end{align*}  
\begin{align*}
	&\mathfrak{A}:=  \notag  \\
&     	  \p{-   \omega  \ta   -\frac{14  }{3    }+	\widetilde{\mathscr{Z}_{11}}&   -4\ck q^j -4     \omega \ck \kappa_i \delta^{ij}     +	\widetilde{\mathscr{Z}^j_{-} }&0  & 0 & \frac{\eta}{\mu}   (-8+	\widetilde{\mathscr{Z}_{15}} ) & 8+	\widetilde{\mathscr{Z}_{13}}  \\
	-  4 	\omega  \ck \kappa_i \delta^{i}_k    +\widetilde{\mathscr{Z}_k^+}   & ( 4 \ck-  \ck \omega  \ta    +\widetilde{\mathscr{Z}}_\star)\delta^j_k  &0&(24  \ck  + 	\widetilde{\mathscr{Z}_{24}} ) \delta^l_k & 0 & 0\\
	\mu(- 8   +	\widetilde{\mathscr{Z}_{31}} ) & 0 & \frac{ 40 }{3 } +	\widetilde{\mathscr{Z}_{33}} & 0 & \eta( -16+	\widetilde{\mathscr{Z}_{35} })  & 0  \\0&(\frac{2  }{3       }  +	\widetilde{\mathscr{Z}_{42}} )\delta^j_{s} &0& (\frac{2}{3   } -   \omega    \ta +	\widetilde{\mathscr{Z}_{44} }      ) \delta^l_s  & 0 & 0 \\ 0 & 0 &0 &0 &    - \theta (\ta   + \kappa_i \mathtt{d}^i)     &0  \\
	- 8   +	\widetilde{\mathscr{Z}_{31} } & 0 & 0 & 0 & 	\frac{\eta}{\mu}  (-16+	\widetilde{\mathscr{Z}_{35}} ) & \frac{ 40 }{3 } -  2 \omega   \ta   +	\widetilde{\mathscr{Z}_{33} }  }       ,
\end{align*}
\begin{equation*}
	\mathfrak{F}:=  \p{ - \frac{ \omega }{A}      \widetilde{\mathscr{S}}^j\mfu_j+	\mu^{-1}	\widetilde{\mathfrak{F}_{u_0}}\\ 	  -    \frac{\omega  \ta }{A }   \widetilde{\mathscr{S}}   \mfu_k - \frac{  \omega  }{A}     \widetilde{\mathscr{S}}^j\delta_{jk} \mfu_0 +	\mu^{-1} \widetilde{\mathfrak{F}_{u_k}} \\  		\widetilde{\mathfrak{F}_{u}}\\ 		\mu^{-1}\widetilde{\mathfrak{F}_{\mathcal{B}_i}} \\ 		\eta^{-1} \widetilde{\mathfrak{F}_{z} } \\
		\mu^{-1}	\widetilde{\mathfrak{F}_u}} , 
\end{equation*}
\begin{equation}
	\widetilde{\mathscr{Z}}^{ij}:=	\widetilde{\mathscr{Z}}^{ij}(\ttau; \mu \mfu_0, \mu \mfv, \eta \mathfrak{z},\mu\mathfrak{B}_i)  =  \frac{\mu \ta \tc^i (\tilde{S}+\widetilde{\mathscr{L}}) \tilde{\uf}}{1+\tilde{\uf}} \frac{\widetilde{\underline{\chi_\uparrow}}}{B}   \mathfrak{B}_k  \delta^{kj} + \widetilde{\mathscr{H}} (\ttau; \mu \mfu_0,\mu \mfv, \eta \mathfrak{z}) \delta^{ij} ,  \label{e:ZIj}  
\end{equation}
\begin{equation}\label{e:Zj}
	\widetilde{\mathscr{Z}}^j:= \kappa_i  	\widetilde{\mathscr{Z}}^{ij}  \AND 
	\widetilde{\mathscr{S}}^j:= 	\ck \kappa_i \delta^{ij}  \frac{\tilde{\underline{\mathfrak{G}}}}{B} + \kappa_i\delta^{ij} \frac{\tilde{\underline{\chi}}_\uparrow}{B} \ttau \widetilde{\mathscr{S}}   , 
\end{equation} 
\begin{equation}\label{e:Z+-}
	\widetilde{\mathscr{Z}_k^+}=	-  \omega    	\widetilde{\mathscr{Z}}^j\delta_{jk} , \quad  \widetilde{\mathscr{Z}^j_{-} } =\widetilde{\mathscr{Z}^j_{12} }  -    \omega 	  \widetilde{\mathscr{Z}}^j \AND \widetilde{\mathscr{Z}}_\star:= -   \omega  \ta       \widetilde{\mathscr{L}}     +	\widetilde{\mathscr{Z}_{22}} , 
\end{equation}
and 
\begin{align*}
	\mu^{-1}	\widetilde{\mathfrak{F}_{u_0}}= & \frac{\tilde{\underline{\mathfrak{G}}}}{\ttau} \widehat{\mathscr{S}_{11}}(\ttau,\mu,\eta,\eta \mu^{-1} ; \mfu_0, \mfv,  \mathfrak{z})  +\frac{1}{\ttau \tilde{\uf}^{\frac{1}{2}}} \widehat{\mathscr{S}_{12}}(\ttau,\mu,\eta,\eta \mu^{-1} ; \mfu_0, \mfv,  \mathfrak{z})  ;\\
	\mu^{-1}	\widetilde{\mathfrak{F}_{u_k}}= & \frac{\tilde{\underline{\mathfrak{G}}}}{\ttau} \widehat{\mathscr{S}_{k;21}}(\ttau,\mu,\eta; \mfu_i,   \mathfrak{B}_k)  +\frac{1}{\ttau \tilde{\uf}^{\frac{1}{2}}} \widehat{\mathscr{S}_{k;22}}(\ttau,\mu,\eta ,\mu \mfu_0,\mfu,\eta \mathfrak{z}; \mfu_i,   \mathfrak{B}_k)  ;\\
	\mu^{-1}\widetilde{\mathfrak{F}_{\mathcal{B}_i}} = & \frac{1}{\ttau \tilde{\uf}^{\frac{1}{2}}} \widehat{\mathscr{S}_{l;42}}(\ttau,\mu,\eta ,\mu \mfu_0,\mfu,\eta \mathfrak{z}; \mfu_i,   \mathfrak{B}_k);\\
	\eta^{-1} \widetilde{\mathfrak{F}_{z} }= &  \frac{1}{\ttau \tilde{\uf}^{\frac{1}{2}}} \widehat{\mathscr{S}_{52}}(\ttau,\mu\eta^{-1}; \mfv,  \mathfrak{B}_i, \mathfrak{z}) ;  \\
		\mu^{-1}	\widetilde{\mathfrak{F}_u} 
	=& \frac{\underline{\widetilde{\mathfrak{G}}}}{\ttau} \widehat{\mathscr{S}_{31}}(\ttau,\mu,\eta,\eta \mu^{-1} ; \mfu_0, \mfv,  \mathfrak{z})  +\frac{1}{\ttau \tilde{\uf}^{\frac{1}{2}}} \widehat{\mathscr{S}_{32}} (\ttau,\mu,\eta,\eta \mu^{-1} ; \mfu_0, \mfv,  \mathfrak{z}) .  
\end{align*} 
\end{lemma}
\begin{proof} 
	By isolating the third line of \eqref{e:mainsys2a} and then reintegrating it into the system, we arrive at
	\begin{equation} \label{e:mrxeq1}
		\p{\widetilde{\mathbf{A}^0} & \\&2}\del{\ttau} \p{\tilde{U}\\\tilde{u}}+\frac{1}{A \ttau } 	\p{\tilde{\mathbf{A}}^{i} & \\& 2 \ta\tc^i   } \del{\txi^i}  \p{\tilde{U}\\\tilde{u}} =\frac{1}{A\ttau}   \widetilde{\mathbf{A}}^\star  \p{\tilde{U}\\\tilde{u}} +\p{\tilde{\mathbf{F}}\\\widetilde{\mathfrak{F}_u}}  , 
	\end{equation}
	where 
	\begin{align*}%\label{e:Astr}
		\widetilde{\mathbf{A}}^\star  := \p{-\frac{14  }{3    }+	\widetilde{\mathscr{Z}_{11}}&  -4\ck q^j+	\widetilde{\mathscr{Z}^j_{12} } &0  & 0 & -8+	\widetilde{\mathscr{Z}_{15}}  & 8+	\widetilde{\mathscr{Z}_{13}}  \\
			0 &(4\ck+	\widetilde{\mathscr{Z}_{22}} )\delta^j_k &0&(24 \ck   + 	\widetilde{\mathscr{Z}_{24}} ) \delta^l_k & 0 & 0\\
			- 8   +	\widetilde{\mathscr{Z}_{31}}  & 0 & \frac{ 40 }{3 } +	\widetilde{\mathscr{Z}_{33}} & 0 & -16+	\widetilde{\mathscr{Z}_{35} }  & 0  \\0&(\frac{2  }{3       }  +	\widetilde{\mathscr{Z}_{42}} )\delta^j_{i} &0& (\frac{2}{3   }  +	\widetilde{\mathscr{Z}_{44} }  ) \delta_{i}^l  & 0 & 0 \\ 0 & 0 &0 &0 &0  &0  \\
			- 8   +	\widetilde{\mathscr{Z}_{31} } & 0 & 0 & 0 & -16+	\widetilde{\mathscr{Z}_{35}} & \frac{ 40 }{3 } +	\widetilde{\mathscr{Z}_{33} }}  . 
	\end{align*}

Let  
\begin{align*}
	E:= & \diag\left\{\frac{1}{\mu^2(\txi^i)},\frac{1}{\mu^2(\txi^i)},1,\frac{1}{\mu^2(\txi^i)},\frac{1}{\eta^2(\txi^i)},\frac{1}{\mu^2(\txi^i)}\right\}  \AND 
	Q:=   \diag\left\{\mu(\txi^i) ,\mu(\txi^i) ,1,\mu(\txi^i) ,\eta(\txi^i),\mu(\txi^i)  \right\}   . 
\end{align*}

Using $EQ$ acting on \eqref{e:mrxeq1} and substituting $\mathfrak{U}:=Q^{-1} \p{\tilde{U},\tilde{u}}^T$ into \eqref{e:mrxeq1}  yields  
		\begin{equation} \label{e:mrxeq1b}
		 	EQ	\p{\widetilde{\mathbf{A}^0} & \\&2} Q \del{\ttau} \mathfrak{U}+\frac{1}{A \ttau } EQ	\p{\tilde{\mathbf{A}}^{i} & \\& 2 \ta\tc^i  } Q  \del{\txi^i}  \mathfrak{U} =  \frac{1}{A\ttau}  EQ \widetilde{\mathbf{A}}^\star Q   \mathfrak{U} -\frac{1}{A \ttau } EQ	\p{\tilde{\mathbf{A}}^{i} & \\& 2 \ta\tc^i  }   \del{\txi^i} Q  \mathfrak{U}  +EQ \p{\tilde{\mathbf{F}}\\\widetilde{\mathfrak{F}_u}}   .
		\end{equation}  
Note that, by using Lemmas \ref{t:coef1} and \ref{t:Seq1}, we have
\begin{equation}\label{e:aRH1}
	\ta \tc^i \tilde{\mathscr{R}}^j +  \tilde{H} ^{ij}	  
	  =  	(\ck  + \tau\mathscr{S}  )\delta^{ij} \frac{\underline{\chi_\uparrow}}{B}   +\widetilde{\mathscr{Z}}^{ij}   ,
\end{equation} 
where $\widetilde{\mathscr{Z}}^{ij}$ is given by \eqref{e:ZIj}.  
Furthermore, by using Lemma \ref{t:coef1}, \eqref{e:aRH1} and \eqref{e:ck} (i.e.,  $\kappa_i \tc^i =1$), we obtain
\begin{equation}\label{e:aRH2}
  \ta  \tilde{\mathscr{R}}^j + \kappa_i \tilde{H} ^{ij} =  \kappa_i (	\ta \tc^i \tilde{\mathscr{R}}^j +  \tilde{H} ^{ij}	)
	=    4 \ck  \kappa_i \delta^{ij}   +    	\ttau \widetilde{\mathscr{S}}^j  +	\widetilde{\mathscr{Z}}^j  , 
\end{equation} 
where $	\widetilde{\mathscr{Z}}^j$ and $	\widetilde{\mathscr{S}}^j$ are given by \eqref{e:Zj}.

After a lengthy but straightforward calculations, and with the help of \eqref{e:ck} (i.e.,  $\kappa_i\tc^i=1$) and  \eqref{e:aRH1}--\eqref{e:aRH2}, by expanding and rearranging the terms in  \eqref{e:mrxeq1b}, we can conclude this lemma. Additionally, since terms like  $\mu^{-1} \widetilde{\mathscr{S}_{32}} (\ttau;\mu\mfu_0,\mu\mfv,\eta \mathfrak{z}) $ vanish to the $1^\text{st}$ order in $(\mu\mfu_0,\mu\mfv,\eta \mathfrak{z})$, we can express
\begin{align*}
	\mu^{-1}	\widetilde{\mathfrak{F}_u}= &	\mu^{-1}	\biggl(\frac{\underline{\widetilde{\mathfrak{G}}}}{\ttau} \widetilde{\mathscr{S}_{31}}(\ttau;\mu\mfu_0,\mu\mfv,\eta \mathfrak{z})  +\frac{1}{\ttau \tilde{\uf}^{\frac{1}{2}}} \widetilde{\mathscr{S}_{32}} (\ttau;\mu\mfu_0,\mu\mfv,\eta \mathfrak{z})  \biggr) \notag  \\
	=& \frac{\underline{\widetilde{\mathfrak{G}}}}{\ttau} \widehat{\mathscr{S}_{31}}(\ttau,\mu,\eta,\eta \mu^{-1} ; \mfu_0, \mfv,  \mathfrak{z})  +\frac{1}{\ttau \tilde{\uf}^{\frac{1}{2}}} \widehat{\mathscr{S}_{32}} (\ttau,\mu,\eta,\eta \mu^{-1} ; \mfu_0, \mfv,  \mathfrak{z}) . 
\end{align*}
Similarly, we obtain $\mu^{-1}	\widetilde{\mathfrak{F}_{u_0}},\mu^{-1}	\widetilde{\mathfrak{F}_{u_k}},\mu^{-1}	\widetilde{\mathfrak{F}_z},\mu^{-1}	\widetilde{\mathfrak{F}_u}$. 
This completes the proof. 
\end{proof}

\subsection{The revised singular system}\label{s:revFuc}
As highlighted in Difficulty $4$ in \S\ref{s:oview},  the issue arises with the singularity of  $\mu$ and $\eta$ as $\txi^1 \rightarrow -\infty$ in the system \eqref{e:mainsys3}. 
To address this, we revise the singular system \eqref{e:mainsys3} by excluding the neighborhood where $\txi^1 \rightarrow -\infty$, using a cutoff function. This modification yields a revised system that can  lead to a Fuchsian system.

We first select the parameters to be determined as below. The basic idea for choosing these parameters is discussed in  Difficulty $3$ of \S\ref{s:oview}.  
\begin{equation}\label{e:para} 
\ck=\frac{1}{4}, \quad 	\kappa_i= - 51  \delta^1_i ,  \quad	\tc^i= -\frac{1}{51}  \delta_1^i  ,   \quad	\omega=\theta=1>0 ,\quad \ta=  - 100 , \quad  \mathtt{d}^i = 0  . 
\end{equation} 
In addition, we set $\hat{\mu}=\sigma_0 e^{- \frac{153}{\delta_0} } $ and $\hat{\eta}=\sigma_0^2 e^{-  \frac{153}{\delta_0} }  $ in \eqref{e:eta1} for  small constants $\sigma_0, \delta_0 \in(0,1)$. That is,
\begin{equation}\label{e:eta2}
	\mu(\txi^1):=    \sigma_0 e^{- \frac{153}{\delta_0} }  e^{-51 \txi^1 }   ,  \quad 	\eta(\txi^1):=    \sigma_0^2 e^{-  \frac{153}{\delta_0} }    e^{-51 \txi^1 }     \quad \Rightarrow \quad \frac{\del{\txi^i}\mu}{\mu} =-51\delta^1_i<0 ,   \quad \frac{\del{\txi^i}\eta}{\eta} =-51\delta^1_i < 0  .   
\end{equation}  
It is clear that for $\txi^1>-\frac{3}{\delta_0}$, both $	\mu(\txi^1)$ and $\eta(\txi^1)$ remain small. 
\begin{equation}\label{e:mueta2}
	\mu(\txi^1) \leq \sigma_0 \AND \eta(\txi^1)   \leq \sigma_0^2 . 
\end{equation}
Additionally, for later reference, we calculate the ratio $\eta/\mu$ as follows, 
\begin{equation*}%\label{e:eonm}
	\frac{\eta}{\mu} = \frac{   \sigma_0^2 e^{-  \frac{153}{\delta_0}  }    e^{-51 \txi^1 }  }{   \sigma_0 e^{- \frac{153}{ \delta_0} }  e^{- 51 \txi^1 }}=  \sigma_0  .  
\end{equation*}

Let $\phi(\txi^1)$ be a \textit{smooth cut-off function} satisfying 
\begin{equation}\label{e:phi1}
    \phi\in C^\infty\bigl(\Rbb;[0,1]\bigr),\quad \phi |_{[-\delta_0^{-1} , +\infty)}=1  \AND 
    \supp \phi\subset 	[-2\delta_0^{-1} , +\infty)    \subset \Rbb . 
\end{equation}
Then in this section, we consider the following revised system \eqref{e:mainsys5} given by
	\begin{equation}\label{e:mainsys5}
		 	\mathfrak{A}^0_\phi\del{\ttau}  \mathfrak{U} +\frac{1}{A\ttau}    \mathfrak{A}^{i}_\phi \del{\txi^i}  \mathfrak{U}  =\frac{1}{A\ttau}  	 \mathfrak{A}_\phi   \mathfrak{U}  + 	 \mathfrak{F}_\phi  , 
	\end{equation} 
	where  
		\begin{align}\label{e:A0} 	\mathfrak{A}^0_\phi:=	& \p{1 & \phi\mu \tilde{R}  \mathfrak{B}_i \delta^{ij}  & 0 & 0 & 0 & 0 \\
		\phi	\mu \tilde{R} \mathfrak{B}_k     &  ( \tilde{S}+ \widetilde{\mathscr{L} }_\phi ) \delta^j_k  & 0 & 0 & 0 & 0\\
			0&0&2&0&0&0\\
			0&0&0&\delta^l_s&0 & 0 \\
			0&0&0&0&1 & 0\\
			0&0&0&0&0&2} ,  
	\end{align} 
	\begin{align}
		\mathfrak{A}^{i}_\phi:=  \p{ \frac{100}{51} \delta^i_1 & 	  \tilde{S} \frac{\widetilde{\underline{\chi_\uparrow}}}{B}  \delta^{ij}  + 	\widetilde{\mathscr{Z}}_\phi^{ij}   & 0 & 0 & 0 &0\\
		(  \tilde{S} \frac{\widetilde{\underline{\chi_\uparrow}}}{B}  \delta^{ij}  +  	\widetilde{\mathscr{Z}_\phi^{ij}}  ) \delta_{jk}   & \frac{100}{51} \delta^i_1 (S+ \widetilde{\mathscr{L} }_\phi   ) \delta^j_k   & 0 & 0 & 0 &0\\
		0&0& \frac{200}{51} \delta^i_1  &0&0&0\\
			0&0&0&\frac{100}{51} \delta^i_1  \delta^l_s &0&0\\
			0&0&0&0&  \frac{100}{51} \delta^i_1  & 0 \\0&0&0&0&0& \frac{200}{51} \delta^i_1  }  , 
	\end{align}  
	and  
	\begin{align}\label{e:Aph}
		 \mathfrak{A}_\phi:=   
		&   	  \p{  \frac{286}{3}+ 	\widetilde{\mathscr{Z}_{11,\phi}}&   - q^j+  51   \delta^{1j}  +     	\widetilde{\mathscr{Z}^j_{-,\phi} }&0  & 0 &  \frac{\eta}{\mu}   (-8+ 	\widetilde{\mathscr{Z}_{15,\phi}} ) & 8+ 	\widetilde{\mathscr{Z}_{13,\phi}}  \\
		51  \delta^{1}_k   + \widetilde{\mathscr{Z}_{k,\phi}^+}   & ( 26 + \widetilde{\mathscr{Z}}_{\star,\phi})\delta^j_k  &0&(6    +  	\widetilde{\mathscr{Z}_{24,\phi}} ) \delta^l_k & 0 & 0\\
		\phi\mu(- 8   + 	\widetilde{\mathscr{Z}_{31,\phi}} ) & 0 & \frac{ 40 }{3 } + 	\widetilde{\mathscr{Z}_{33,\phi}} & 0 & \phi\eta( -16+ 	\widetilde{\mathscr{Z}_{35, \phi} })  & 0  \\0&(\frac{2  }{3       }  + 	\widetilde{\mathscr{Z}_{42, \phi}} )\delta^j_{s} &0& (\frac{302}{3}+ 	\widetilde{\mathscr{Z}_{44,\phi} }      ) \delta^l_s  & 0 & 0 \\ 0 & 0 &0 &0 & 100   &0  \\
		- 8   + 	\widetilde{\mathscr{Z}_{31,\phi} } & 0 & 0 & 0 & 	 \frac{\eta}{\mu}  (-16+ 	\widetilde{\mathscr{Z}_{35,\phi}} ) & \frac{640}{3} + 	\widetilde{\mathscr{Z}_{33,\phi} }  }       ,
	\end{align}
	\begin{equation}
		\mathfrak{F}_\phi:= \p{ - \frac{1}{A}      \widetilde{\mathscr{S}}^j\mfu_j+	\mu^{-1}	\widetilde{\mathfrak{F}_{u_0,\phi }}\\ 	  \frac{100}{A } \widetilde{\mathscr{S}} (\ttau)   \mfu_k - \frac{1}{A}    \widetilde{\mathscr{S}}^j\delta_{jk} \mfu_0 +	\mu^{-1} \widetilde{\mathfrak{F}_{u_k,\phi }} \\  		\widetilde{\mathfrak{F}_{u,\phi }}\\ 		\mu^{-1}\widetilde{\mathfrak{F}_{\mathcal{B}_i,\phi }} \\ 		\eta^{-1} \widetilde{\mathfrak{F}_{z,\phi } } \\
			\mu^{-1}	\widetilde{\mathfrak{F}_{u,\phi }}} , 
	\end{equation} 
	and the notation $\widetilde{\mathfrak{F}_{u,\phi }}$ denotes $\widetilde{\mathfrak{F}_{u}}$ with  $\mu$, $\eta$ and $\eta \mu^{-1}$ replaced by $\phi\mu$, $\phi\eta$ and $\phi\eta \mu^{-1}$, respectively.  Similarly,  $\widetilde{\mathfrak{F}_{\mathcal{B}_i,\phi }}$, $\widetilde{\mathfrak{F}_{z,\phi } }$, $\widetilde{\mathfrak{F}_{u_k,\phi }}$,  $\widetilde{\mathfrak{F}_{u_0,\phi }}$, $\widetilde{\mathscr{Z}_\phi}$ and  $\widetilde{\mathscr{L} }_\phi$, etc. are obtained by making analogous substitutions.

\subsection{The revised and extended Fuchsian system on a torus $\Tbb^n$}\label{s:revFuc1}
As discussed in Difficulty $5$ of \S\ref{s:oview}, the singular system \ref{e:mainsys5} closely resembles the Fuchsian system, except that the domain is not a closed manifold.  We have to compactify the domain and further obtain a torus. This involves introducing a new coordinate system, defined as follows\footnote{Note that the notation  $\hat{\zeta}^i=\arctan \txi^i$ should be understood as $\hat{\zeta}^i=h^i( \txi^k)$, where $h^i=\arctan$ for each $i$. }
\begin{equation}\label{e:coord6b}
	\hat{\tau}=\ttau\in[-1,0)  \AND \hat{\zeta}^i= \arctan (\gamma\txi^i)\in\left(-\frac{\pi}{2} ,\frac{\pi}{2} \right)
\end{equation}
for a constant $\gamma>0$ to be determined (see \eqref{e:bt1} in \S\ref{s:verfuc} for  the selection of $\gamma$), 
and its inverse is
\begin{equation}\label{e:coordi6}
	\ttau=\hat{\tau} \AND \txi^i=\frac{1}{\gamma} \tan  \hat{\zeta}^i  . 
\end{equation}
The Jacobian of this transformation is
\begin{equation*}%\label{e:Jb3}
	\p{\widehat{\frac{\partial \hat{\tau}}{\partial \tilde{\tau}}} &  \widehat{\frac{\partial \hat{\tau}}{\partial \tilde{\zeta}^j} }  \\  \widehat{ \frac{\partial \hat{\zeta}^i}{\partial \tilde{\tau}} } & \widehat{\frac{\partial \hat{\zeta}^i}{\partial  \tilde{\zeta}^j} }  }= \p{1 & 0 \\
	0  & \gamma \cos^2( \hat{\zeta}^i  ) \delta^i_j  }  \AND \det\p{\widehat{\frac{\partial \hat{\tau}}{\partial \tilde{\tau}}} &  \widehat{\frac{\partial \hat{\tau}}{\partial \tilde{\zeta}^j} }  \\  \widehat{ \frac{\partial \hat{\zeta}^i}{\partial \tilde{\tau}} } & \widehat{\frac{\partial \hat{\zeta}^i}{\partial  \tilde{\zeta}^j} }  }=\gamma^n \prod_{i=1}^n \cos^{2}  \hat{\zeta}^i  > 0. 
\end{equation*}  
For a function $\tilde{z}(\ttau,\txi )$, if denoting  $\hat{z}(\hat{\tau},\hat{\zeta} ):=\tilde{z}\bigl(\ttau(\hat{\tau},\hat{\zeta} ),\txi(\hat{\tau},\hat{\zeta} )\bigr)$, then using the above Jocobian matrix described above, we obtain
\begin{equation}
	\widehat{\del{\ttau}\tilde{z}}=   \widehat{\frac{\partial \hat{\tau}}{\partial \ttau}} \del{\hat{\tau}} \hat{z}+  \widehat{ \frac{\partial \hat{\zeta}^i}{\partial \ttau} }  \del{\hat{\zeta}^i} \hat{z} =  \del{\hat{\tau}} \hat{z}    \AND 
	\widehat{\del{\txi^i}\tilde{z}}=   \widehat{\frac{\partial \hat{\tau}}{\partial \txi^i} } \del{\hat{\tau}} \hat{z}+  \widehat{\frac{\partial \hat{\zeta}^j}{\partial  \txi^i} }  \del{\hat{\zeta}^j} \hat{z} = \gamma \cos^2 (\hat{\zeta}^i) \del{\hat{\zeta}^i}  \hat{z} . \label{e:transf3}
\end{equation}

\begin{remark}
	This coordinate system  $(\hat{\tau},\hat{\zeta})$ compacifies the spatial coordinate. After this compactification, we are able to identify $\hat{\zeta}^i=-\frac{\pi}{2}$ and $\frac{\pi}{2}$ to obtain a torus $\Tbb^n_{[-\frac{\pi}{2},\frac{\pi}{2}]}$. The revised singular system \eqref{e:mainsys5} can then be reformulated on this torus, as $\hat{\phi}\hat{\mu}\rightarrow 0$ and $\hat{\phi}\hat{\eta}\rightarrow 0$ when $\hat{\zeta}^1\rightarrow\pm \frac{\pi}{2}$.  Consequently, we obtain a Fuchsian system (see the proof in \S\ref{s:verfuc}). 
\end{remark}

The revised and extended system is 
\begin{equation}\label{e:mainsys6}
	\widehat{\mathfrak{A}}^0_\phi\del{\hat{\tau}}  \widehat{\mathfrak{U}} +\frac{1}{A\hat{\tau}}    \gamma\cos^2 \hat{\zeta}^i  \widehat{\mathfrak{A}}^{i}_\phi   \del{\hat{\zeta}^i}  \widehat{\mathfrak{U} } =\frac{1}{A\hat{\tau} }  	 \widehat{\mathfrak{A}}_\phi  \widehat{ \mathfrak{U}}  + 	 \widehat{\mathfrak{F}}_\phi , 
\end{equation}  
where
\begin{align}\label{e:A0b} 	\widehat{\mathfrak{A}}^0_\phi:=	& \p{1 & \hat{\phi} \hat{\mu} \hat{R}  \hat{\mathfrak{B}}_i \delta^{ij}  & 0 & 0 & 0 & 0 \\
		\hat{\phi}	\hat{\mu} \hat{R} \hat{\mathfrak{B}}_k     &  (\hat{S}+  \widehat{\mathscr{L}_\phi}  ) \delta^j_k  & 0 & 0 & 0 & 0\\
		0&0&2&0&0&0\\
		0&0&0&\delta^l_s&0 & 0 \\
		0&0&0&0&1 & 0\\
		0&0&0&0&0&2} ,  
\end{align} 
\begin{align}
	\widehat{\mathfrak{A}}^{i}_\phi:=  \p{ \frac{100}{51} \delta^i_1 & 	   \hat{S} \frac{\widehat{\underline{\chi_\uparrow}}}{B}  \delta^{ij}  +  	\widehat{\mathscr{Z}_\phi}^{ij}   & 0 & 0 & 0 &0\\
		(  \hat{S} \frac{\widehat{\underline{\chi_\uparrow}}}{B}  \delta^{ij}  +  	\widehat{\mathscr{Z}_\phi}^{ij}  ) \delta_{jk}   & \frac{100}{51} \delta^i_1 (\hat{S}+  \widehat{\mathscr{L}_\phi }   ) \delta^j_k   & 0 & 0 & 0 &0\\
		0&0& \frac{200}{51} \delta^i_1  &0&0&0\\
		0&0&0&\frac{100}{51} \delta^i_1  \delta^l_s &0&0\\
		0&0&0&0&  \frac{100}{51} \delta^i_1  & 0 \\0&0&0&0&0& \frac{200}{51} \delta^i_1  }  , 
\end{align}  
and  
\begin{align}\label{e:Aphb}
	\widehat{\mathfrak{A}}_\phi:=   
	&   	  \p{  \frac{286}{3}+ 	\widehat{\mathscr{Z}_{11,\phi}}&     -q^j+51   \delta^{1j}  +    	\widehat{\mathscr{Z}^j_{-,\phi} }&0  & 0 &   \frac{\hat{\eta}}{\hat{\mu}}   (-8+ 	\widehat{\mathscr{Z}_{15,\phi}} ) & 8+ 	\widehat{\mathscr{Z}_{13,\phi}}  \\
	51  \delta^{1}_k   +  \widehat{\mathscr{Z}_{k,\phi}^+}   & ( 26 +  \widehat{\mathscr{Z}}_{\star,\phi})\delta^j_k  &0&(6    +  	\widehat{\mathscr{Z}_{24,\phi}} ) \delta^l_k & 0 & 0\\
	\hat{\phi} \hat{\mu} (- 8   + 	\widehat{\mathscr{Z}_{31,\phi}} ) & 0 & \frac{ 40 }{3 } + 	\widehat{\mathscr{Z}_{33,\phi}} & 0 & \hat{\phi} \hat{\eta} ( -16+ 	\widehat{\mathscr{Z}_{35,\phi} })  & 0  \\0&(\frac{2  }{3       }  + 	\widehat{\mathscr{Z}_{42,\phi}} )\delta^j_{s} &0& (\frac{302}{3}+ 	\widehat{\mathscr{Z}_{44,\phi} }      ) \delta^l_s  & 0 & 0 \\ 0 & 0 &0 &0 & 100   &0  \\
	- 8   + 	\widehat{\mathscr{Z}_{31,\phi} } & 0 & 0 & 0 & 	  \frac{\hat{\eta}}{\hat{\mu}}  (-16+ 	\widehat{\mathscr{Z}_{35,\phi}} ) & \frac{640}{3} + 	\widehat{\mathscr{Z}_{33,\phi} }  }       ,
\end{align}
\begin{equation}
	\widehat{\mathfrak{F}}_\phi:=  \p{ - \frac{1}{A}      \widehat{\mathscr{S}}^j\hat{\mfu}_j+	\hat{\mu}^{-1}	\widehat{\mathfrak{F}_{u_0,\phi}}\\ 	  \frac{100}{A } \widehat{\mathscr{S}}    \hat{\mfu}_k - \frac{1}{A}    \widehat{\mathscr{S}}^j\delta_{jk} \hat{\mfu}_0 +	\hat{\mu}^{-1} \widehat{\mathfrak{F}_{u_k,\phi}} \\  		\widehat{\mathfrak{F}_{u,\phi}}\\ 		\hat{\mu}^{-1}\widehat{\mathfrak{F}_{\mathcal{B}_i,\phi}} \\ 		\hat{\eta}^{-1} \widehat{\mathfrak{F}_{z,\phi} } \\
	\hat{\mu}^{-1}	\widehat{\mathfrak{F}_{u,\phi}}} , 
\end{equation}

Correspondingly, for later reference, the extended system of the original one \eqref{e:mainsys3} (without using the cutoff function $\phi$) is 
\begin{equation}\label{e:mainsys8}
	\widehat{\mathfrak{A}}^0 \del{\hat{\tau}}  \widehat{\mathfrak{U}} +\frac{1}{A\hat{\tau}}    \gamma\cos^2 \hat{\zeta}^i  \widehat{\mathfrak{A}}^{i}   \del{\hat{\zeta}^i}  \widehat{\mathfrak{U} } =\frac{1}{A\hat{\tau} }  	 \widehat{\mathfrak{A}}   \widehat{ \mathfrak{U}}  + 	 \widehat{\mathfrak{F}} . 
\end{equation}  

\begin{remark}
	Note $\cos^2(\hat{\zeta}^i)$,  $\hat{\phi}\hat{\mu}$ and $\hat{\phi}\hat{\eta}$, and their derivatives of any order approach zero as $\hat{\zeta}^1$ tends to $\pi /2$ and $-\pi /2$.  Therefore, this system is well-defined on the extended spacetime domain $[-1,0)\times \Tbb^n_{[-\frac{\pi}{2} ,\frac{\pi}{2} ]}$ where   $\Tbb^n_{[-\frac{\pi}{2} ,\frac{\pi}{2} ]}$ denotes the $n$-dimensional torus obtained from identifying the end points of the every interval  $
	[-\frac{\pi}{2} ,\frac{\pi}{2} ]\subset \Rbb$. 
\end{remark}

After constructing the singular system \eqref{e:mainsys6}, we claim the following proposition, with the proof deferred to \S\ref{s:verfuc}.  
\begin{proposition}\label{t:verfuc}
	Suppose $k\in\Zbb_{> \frac{n}{2}+3}$, $\widehat{\mathfrak{U}}|_{\hat{\tau}=-1}=\widehat{\mathfrak{U}}_0$,  $\widehat{\mathfrak{U}}_{0}\in H^{k}(\Tbb^n_{[-\frac{\pi}{2} ,\frac{\pi}{2} ]})$, then  
	\begin{enumerate}[leftmargin=*]
		\item the system \eqref{e:mainsys6} is a Fuchsian system given by Appendix \ref{s:fuc}. 
		\item  there exist constants $\sigma_1, \sigma>0$ satisfying $\sigma<\sigma_1$, such that if
		\begin{equation*}
			\|\widehat{\mathfrak{U}}_0\|_{H^k} \leq \sigma,
		\end{equation*}
		then there exists a unique solution
		\begin{equation}\label{e:solreg}
			\widehat{\mathfrak{U}} \in C^0([-1,0),H^k(\Tbb^n_{[-\frac{\pi}{2},\frac{\pi}{2}]}) ) \cap C^1([-1,0),H^{k-1}(\Tbb^n_{[-\frac{\pi}{2},\frac{\pi}{2}]}) ) \cap \Li ([-1,0),H^k(\Tbb^n_{[-\frac{\pi}{2},\frac{\pi}{2}]}))
		\end{equation}
		of the initial value problem \eqref{e:mainsys6} and $\widehat{\mathfrak{U}}|_{\ttau=-1}=\widehat{\mathfrak{U}}_0$. Moreover, for $-1\leq \hat{\tau}<0$, the solution $\widehat{\mathfrak{U}}$ satisfies the energy estimate
		\begin{equation*} 
			\|\widehat{\mathfrak{U}}(\hat{\tau})\|_{H^k(\Tbb^n_{[-\frac{\pi}{2},\frac{\pi}{2}]})}^2  - \int^{\hat{\tau}}_{-1}\frac{1}{s}\| \widehat{\mathfrak{U}}(s)\|^2_{H^k(\Tbb^n_{[-\frac{\pi}{2},\frac{\pi}{2}]})}ds \leq C(\sigma_1,\sigma_1^{-1}) \|\widehat{\mathfrak{U}}_0\|^2_{H^k(\Tbb^n_{[-\frac{\pi}{2},\frac{\pi}{2}]})} .
		\end{equation*}
	\end{enumerate}
\end{proposition}

\subsection{Constructing solution to the original system}\label{s:reorg}
This section focuses on reconstructing the solution to the original system by utilizing a lens-shaped domain, addressing the difficulty outlined in Difficulty  $1$ of \S\ref{s:oview}. 
We have previously constructed a solution to the revised and extended equation  \eqref{e:mainsys6} on $[-1,0)\times \Tbb^n_{[-\frac{\pi}{2},\frac{\pi}{2}]}$. To recover the solution to the original system, our idea involves defining a lens-shaped domain within $[-1,0)\times \Tbb^n_{[-\frac{\pi}{2},\frac{\pi}{2}]}$.  In this domain, we aim for the revised system to align with the original equation, ensuring that 
this domain intersects with $\mathcal{I}$, including the point $p_m$ (recall $\mathcal{I}$ and $p_m$ in Theorem \ref{t:mainthm2}). If successful, this implies that within this lens-shaped domain, the solution to the revised system will match the solution to the original system due to \cite[Theorem $4.5$]{Lax2006}. 
 
Next, we will construct that, within a specific domain, the solution to the extended system \eqref{e:mainsys6} coincides with the solution to \eqref{e:mainsys8}. We define a lens-shaped domain (see Fig. \ref{f:fig1}, the shaded region), denoted by $\mathit{\Lambda}_{\delta_0}$. This lens-shaped domain is enclosed by weakly spacelike hypersurfaces, with its top boundary intersecting the characteristic boundary $\mathcal{C}$ (defined by \eqref{e:char1}) at a future endpoint $p_m$ of a null geodesic (see Fig. \ref{f:fig1}). 
We begin by denoting a hypersurface $\widehat{\Gamma}_{\delta_0}$ in $(\htau,\hat{\zeta})$, which serves as the boundary between charted and uncharted domains for the solutions to the original equation \ref{Eq2}. For any given small constant $\delta_0>0$, we define 
\begin{equation}\label{e:surf0}
\widehat{\Gamma}_{\delta_0}:=\left\{(\htau,\hat{\zeta})\in[-1,0 )\times\Tbb^n_{[-\frac{\pi}{2}, \frac{\pi}{2}]} \;\bigg|\; \widehat{\Psi}(\hat{\tau},\hat{\zeta}):= \hat{\tau}-\widehat{\mathfrak{T}}_{\delta_0}( \hat{\zeta}) = 0 \right\}  , 
\end{equation}
where  
\begin{align*}
	\widehat{\mathfrak{T}}_{\delta_0}(\hat{\zeta} ):=
		\begin{cases}
		 -	\exp\left(- \cc \left( \frac{51A}{2}-  \frac{51A}{4\delta_0 \gamma^{-1} \tan  \hat{\zeta}^1+8 }\right) \left(   \frac{1}{\gamma} \tan  \hat{\zeta}^1 + \frac{1}{\delta_0} \right)\right)   ,  \quad &-\arctan\bigl(\frac{\gamma}{\delta_0}\bigr)  \leq    \hat{\zeta}^1 \leq \arctan\bigl( \frac{2-51 \Xi_0 }{102  \Xi_0} \frac{\gamma}{\delta _0} \bigr)  \\
	 -\exp\left(- \left( \frac{51A}{2}-\frac{51A}{ 4\delta_0  \gamma^{-1} \tan  \hat{\zeta}^1  + 8}  \right) \left(   \frac{1}{\gamma} \tan  \hat{\zeta}^1 + \frac{1}{2\delta_0} \right)\right)  ,  \quad  &     \hat{\zeta}^1 >\arctan\bigl( \frac{2-51 \Xi_0 }{102  \Xi_0} \frac{\gamma}{\delta _0} \bigr) 
	\end{cases} , 
\end{align*}
for a constant $\cc \in (0, \bigl(1+\frac{51}{2} \Xi_0\bigr)^{-1})$ and $\Xi_0:=\Xi(t_0)=\sup_{t\in[t_0,t_m)} \Xi(t)$ (Lemma \ref{t:Thpst} ensures this).  We point out $\frac{2-51 \Xi_0 }{102  \Xi_0 } >-\frac{1}{2  }$.   
Then we define the lens-shaped domain $ \widehat{\mathit{\Lambda}}_{\delta_0}$, 
\begin{equation*}%\label{e:Lbddef}
 \widehat{\mathit{\Lambda}}_{\delta_0} :=\left\{(\htau,\hat{\zeta})\in[-1,0 )\times\Tbb^n_{[-\frac{\pi}{2}, \frac{\pi}{2}]}\;|\; \hat{\tau}\leq  \widehat{\mathfrak{T}}_{\delta_0}( \hat{\zeta})   \right\} , 
\end{equation*}
and the boundary of $ \widehat{\mathit{\Lambda}}_{\delta_0}$ is
\begin{equation*}%\label{e:ttsurf1}
	\partial  \widehat{\mathit{\Lambda}}_{\delta_0} =\widehat{\Sigma}_0 \cup \widehat{\Gamma}_{\delta_0}\cup\widehat{\Sigma}_1 \left(\cup_{k=2}^n \widehat{\Sigma}_{\pm k}\right) ,
\end{equation*}
where
\begin{equation*}%\label{e:ttsurf2}
	\widehat{\Sigma}_{k}:=  \left\{(\hat{\tau},\hat\zeta^i) \;\Big|\; \hat{\zeta}^k=\frac{\pi}{2}\right\}\cap \widehat{\mathit{\Lambda}}_{\delta_0} , \quad \widehat{\Sigma}_{-k}:=  \left\{(\hat{\tau},\hat\zeta^i) \;\Big|\; \hat{\zeta}^k=-\frac{\pi}{2}\right\}\cap \widehat{\mathit{\Lambda}}_{\delta_0} \AND \widehat{\Sigma}_0:=\{\hat{\tau}=-1\} \cap \widehat{\mathit{\Lambda}}_{\delta_0}  . 
\end{equation*}

\begin{remark} 
	Given the significance of the hypersurface $\widehat{\Gamma}_{\delta_0}$, it is essential to reexpress it in alternative coordinate systems—namely, $(\ttau,\txi)$, $(\tau,\zeta)$, and $(t,x)$—for ease of reference. These reexpressions are provided in Lemmas \ref{t:Tsurf2}--\ref{t:Tsurf5} in Appendix \ref{s:ctbdry}. Additionally, an important result that $p_m \in \overline{\mathit{\Lambda}_{\delta_0} \cap \mathcal{I}}$ is established in Lemma \ref{t:inI} in the same appendix.
\end{remark}

\begin{figure}[h]
	\centering
	\includegraphics[width=10cm]{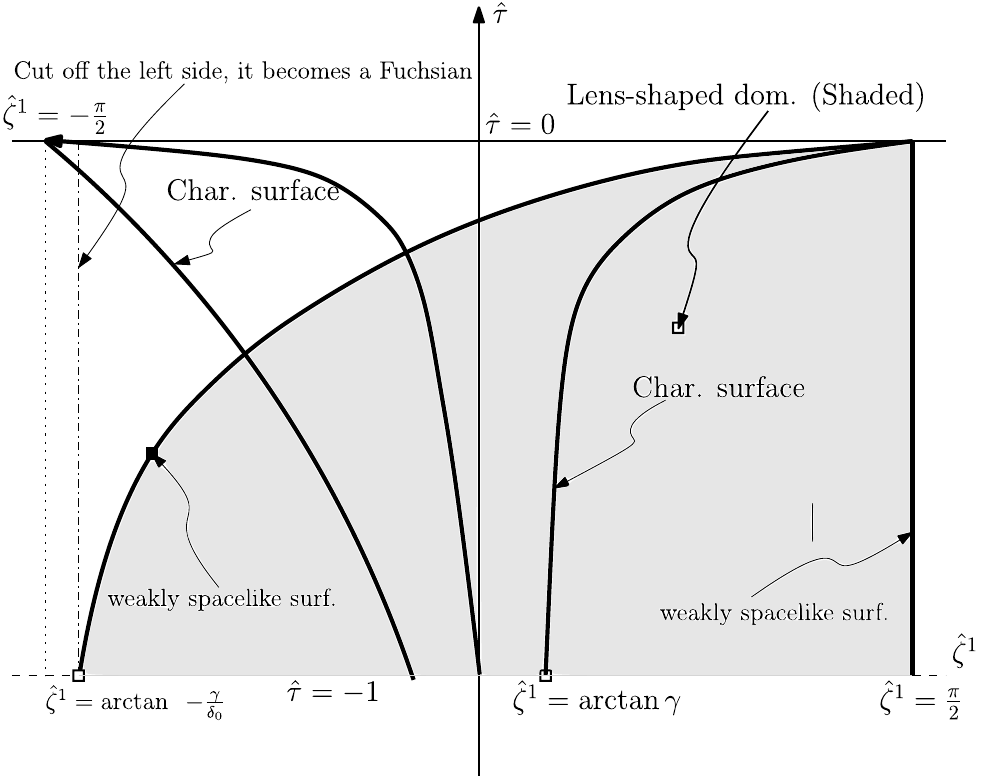}
	\caption{Domains $ \widehat{\mathit{\Lambda}_{\delta_0}}$ and $\widehat{\mathcal{C}}$}
	\label{f:fig1}
\end{figure}

\begin{lemma}\label{t:extori2}
	If $\widehat{\mathfrak{U}}(\htau,\hat{\zeta})$ is a solution to the revised and extended system \eqref{e:mainsys6} on the spacetime $[0,1)\times \Tbb^n_{[-\frac{\pi}{2},\frac{\pi}{2}]}$, then its restriction to the lens-shaped domain $\widehat{\mathit{\Lambda}}_{\delta_0}$, denoted as $\widehat{\mathfrak{U}}(\hat{\tau},\hat{\zeta})|_{\widehat{\mathit{\Lambda}}_{\delta_0}}$, is a solution of the original system \eqref{e:mainsys8} with parameters \eqref{e:para} within the domain $ \widehat{\mathit{\Lambda}}_{\delta_0}$.  This solution is uniquely determined by the initial data restricted to $\widehat{\Sigma}_0$.
\end{lemma}
\begin{proof}
	Since $\widehat{\mathfrak{U}}$ is a solution to the extended system \eqref{e:mainsys6} on the spacetime $[0,1) \times \mathbb{T}^n_{[-\frac{\pi}{2},\frac{\pi}{2}]}$, and assume $\acute{\mathfrak{U}}$ solves the original system \eqref{e:mainsys8} with parameters \eqref{e:para} in the domain $\widehat{\mathit{\Lambda}}_{\delta_0}$, let us define $\mathfrak{W} := \widehat{\mathfrak{U}} - \acute{\mathfrak{U}}$. If $\widehat{\mathfrak{U}} |_{\ttau = -1} = \acute{\mathfrak{U}}|_{\ttau = -1}$, it follows that $\mathfrak{W}$ vanishes on $\widehat{\Sigma}_0$. Additionally, within the domain $\widehat{\mathit{\Lambda}}_{\delta_0}$, the extended system \eqref{e:mainsys6} coincides with the original system \eqref{e:mainsys8} under the given parameters \eqref{e:para}. Therefore, by subtracting the original system from the extended system within $\widehat{\mathit{\Lambda}}_{\delta_0}$, we obtain a linear system for $\mathfrak{W}$ in this region, as $\widehat{\mathfrak{U}}$ and $\acute{\mathfrak{U}}$ are already known. 
		\begin{equation} 
	\widehat{\mathfrak{A}}^0(\hat{\tau},\hat{\zeta}, \widehat{\mathfrak{U}})\del{\hat{\tau}} \mathfrak{W}+\frac{1}{A\hat{\tau}} \gamma \cos^2(\hat{\zeta}^i) 	\widehat{\mathfrak{A}}^{i}(\hat{\tau},\hat{\zeta}, \widehat{\mathfrak{U}}) \del{\hat{\zeta}^i} \mathfrak{W}  =   \mathfrak{H}(\hat{\tau},\hat{\zeta},\widehat{\mathfrak{U}},\acute{\mathfrak{U}},\del{\hat{\tau}}\acute{\mathfrak{U}},\del{\hat{\zeta}} \acute{\mathfrak{U}})  .  \label{e:diffeq}
\end{equation}  
Next, if we can verify $ \widehat{\Gamma}_{\delta_0}\cup\widehat{\Sigma}_1 \bigl(\cup_{k=2}^n \widehat{\Sigma}_{\pm k}\bigr)$ are weakly spacelike for the system described above, we can apply the well-known theorem \cite[Theorem $4.5$]{Lax2006} (the similar proofs to \cite[Theorem $4.5$]{Lax2006} is applied to the system \eqref{e:diffeq} with the decay estimates provided in  \cite[Theorem $3.8$]{Beyer2020}).  This allows us to conclude that $\mathfrak{W}$ vanishes within $\widehat{\mathit{\Lambda}}_{\delta_0}$. Consequently, we can assert that $\acute{\mathfrak{U}}=\widehat{\mathfrak{U}}$ for all  $(\hat{\tau},\hat{\zeta})\in\widehat{\mathit{\Lambda}}_{\delta_0}$.

\underline{$(1)$ $\widehat{\Gamma}_{\delta_0}$ is weakly spacelike.}
Consider the normal vectors to the hypersurface $\widehat{\Gamma}_{\delta_0}$, expressed in the coordinates $(\hat{\tau},\hat{\zeta})$ and $(\ttau,\txi)$ as $(\hat{\nu}_\Lambda)_\mu=\del{\mu}\widehat{\Psi}(\hat{\tau},\hat{\zeta})$ and $(\nu_\Lambda)_\mu=\del{\mu}\Psi(\tilde{\tau},\tilde{\zeta})$, respectively (refer to \eqref{e:surf0}).  Utilizing the transformation provided in  \eqref{e:transf3},  we derive the relationship between these normal vectors.
\begin{equation*}
(	\hat{\nu}_{\Lambda})_0 =(\nu_{\Lambda})_0=1 \AND 	(\hat{\nu}_{\Lambda})_i =(\nu_{\Lambda})_i \frac{\sec ^2(\hat{\zeta}^1 )}{\gamma }  . 
\end{equation*}
In order to verify $\widehat{\mathit{\Lambda}}_{\delta_0}$ is weakly spacelike,  we  refer to the definition of weakly spacelike surfaces in \cite[\S $4.3$]{Lax2006}. According to this definition, it is necessary to demonstrate that, on these surfaces
\begin{equation*}
	(\hat{\nu}_\Lambda)_0\widehat{\mathfrak{A}}_\phi^0 +	(\hat{\nu}_\Lambda)_i \frac{1}{A\hat{\tau}}    \gamma\cos^2 \hat{\zeta}^i  \widehat{\mathfrak{A}}^{i}_\phi   =\left( (\nu_\Lambda)_0\mathfrak{A}^0_\phi +	(\nu_\Lambda)_i \frac{1}{A\ttau}  \mathfrak{A}^i_\phi \right)^{\widehat{}} \geq 0  . 
\end{equation*}
Next let us verify $(\nu_\Lambda)_0\mathfrak{A}^0_\phi +	(\nu_\Lambda)_i \frac{1}{A\ttau}  \mathfrak{A}^i_\phi\geq 0$ holds on this hypersurface.

Using Lemma \ref{t:Tsurf2},	the outward-pointing conormal of  $\Gamma_{\delta_0}$,  which serves as the top boundary of  $\mathit{\Lambda}_{\delta_0}$ in the $(\ttau,\txi)$ coordinate, is given by
	\begin{align*} 
		 \nu_\Lambda=\left(1,A \ttau \mathtt{q}  \delta^1_i \right) =\begin{cases}
		 	\nu_\Lambda^l  
		 	=\left(1,A \ttau \mathtt{q}_l \delta^1_i \right),  \quad &-\frac{1}{\delta_0}<\txi^1 \leq \frac{2-51 \Xi_0 }{102  \Xi_0} \frac{1}{\delta _0}    \\
		 	\nu_\Lambda^r 
		 	=\left(1,A \ttau \mathtt{q}_r \delta^1_i \right)  ,  \quad  & \txi^1 >\frac{2-51 \Xi_0 }{102  \Xi_0} \frac{1}{\delta _0}
		 \end{cases}  , 
	\end{align*}
where  (recall $\cc \in (0, \bigl(1+\frac{51}{2} \Xi_0\bigr)^{-1})$)   
	\begin{align}\label{e:qlr} 
	\mathtt{q}=\begin{cases}
		\mathtt{q}_l=\frac{51 }{2 } \cc \left( 1- \frac{1}{2(2+\txi^1\delta_0)^2}\right) \in[\frac{51}{4}\cc,\frac{51}{2}\cc),  \quad &-\frac{1}{\delta_0}<\txi^1 \leq \frac{2-51 \Xi_0 }{102  \Xi_0} \frac{1}{\delta _0}    \\
		\mathtt{q}_r=\frac{51}{2}  \left(1-\frac{3}{4 (\delta _0 \txi^1 +2)^2}\right) \in[17 ,\frac{51}{2}) ,  \quad  & \txi^1 >\frac{2-51 \Xi_0 }{102  \Xi_0} \frac{1}{\delta _0}  
	\end{cases}  . 
\end{align}

On the hypersurface $\Gamma_\Lambda$, we calculate
	\begin{align*}
		& (\nu_\Lambda)_0\mathfrak{A}^0 +	(\nu_\Lambda)_i \frac{1}{A\ttau}  \mathfrak{A}^i =	\mathfrak{A}^0 +  A \ttau \mathtt{q}  \delta^1_i \frac{1}{A\ttau}  \mathfrak{A}^i  \notag  \\  
		=  &    \p{1+\mathtt{q}  \frac{100}{51}   &   	\mathtt{q} S \frac{\widetilde{\underline{\chi_\uparrow}}}{B}  \delta^{1j}  + 	\mathscr{K}^j & 0 & 0 & 0 & 0 \\
	  \mathtt{q} S \frac{\widetilde{\underline{\chi_\uparrow}}}{B}  \delta^{1j}  \delta_{jk} + \mathscr{K}^j\delta_{jk}   & (1+ \mathtt{q} \frac{100}{51}    ) ( S+\phi \widetilde{\mathscr{L} } ) \delta^j_k  & 0 & 0 & 0 & 0\\
		0&0&2+ \mathtt{q} \frac{200}{51}   &0&0&0\\
		0&0&0& (1+ \mathtt{q} \frac{100}{51}  )\delta^l_s&0 & 0 \\
		0&0&0&0&1+\mathtt{q} \frac{100}{51}   & 0\\
		0&0&0&0&0&2+ \mathtt{q} \frac{200}{51}  } 
	\end{align*}
	where   $\mathscr{K}^j:=\mathscr{K}^j(\tau; \mu \mfu_0, \mu \mfv, \eta \mathfrak{z},\mu\mathfrak{B}_i)=\mathtt{q} \phi	\widetilde{\mathscr{Z}}^{1j}   + \phi\mu \tilde{R}  \mathfrak{B}_i \delta^{ij}$.  We already know from Proposition \ref{t:verfuc} that the solution  $\|\widehat{\mathfrak{U}}\|_{\Li}\leq C\sigma$. With the help of Lemma \ref{t:Gest2} and Proposition \ref{t:fginv0}, and considering \eqref{e:S1} along with the points on the hypersurface \eqref{e:surf}, 
	it follows that there exists a constant $C>0$ independent of $\delta_0$,  such that, for $\txi^1 >\frac{2-51 \Xi_0 }{102  \Xi_0 \delta _0} (> -\frac{1}{2\delta_0}) $,  we have estimates
	\begin{align}
		\left|S-\frac{1}{4} \right|  
		 \leq & C(-\ttau)^{\frac{1}{2} }   = C\exp\left(-\frac{1}{2} \left( \frac{51A}{2}-\frac{51A}{ 4\delta_0\txi^1 + 8}  \right) \left(   \txi^1 + \frac{1}{2\delta_0} \right)\right) < C\exp\left(-  \frac{51A }{8} \left(   \txi^1 + \frac{1}{2\delta_0} \right)\right) \notag  \\
	    = &  
	     C\exp\left(-\frac{1}{\delta_0}   \frac{A}{ 8 \Xi _0}  \right)  \exp\left(-  \frac{51A }{8} \left( 
	     \txi^1-\frac{2-51 \Xi_0 }{102 \delta _0 \Xi_0 }\right)\right)     
	       ,  \label{e:Sest1a} 
\end{align} 
and similarly, 
\begin{equation} 
	    \left|S  \frac{\widetilde{\underline{\chi_\uparrow}}}{B} -1\right|  
	    \leq   C(-\ttau)^{\frac{1}{2} }    <	     C\exp\left(-\frac{1}{\delta_0}   \frac{A}{ 8 \Xi _0}  \right)  \exp\left(- \frac{51A }{8} \left( 
	    \txi^1-\frac{2-51 \Xi_0 }{102 \delta _0 \Xi_0 }\right)\right)  .   
	       \label{e:Sest1b}
	   \end{equation} 
Additionally, for $\txi^1 >-\frac{1}{  \delta _0 } $,  the following estimates hold: 
	   \begin{equation}\label{e:KLest1} 
		|\mathscr{K}^j|\leq    C\sigma_0 e^{-  \frac{153}{\delta_0} }   \sigma e^{-51 \txi^1 } 
		\AND 	|\widetilde{\mathscr{L}} |\leq     C\sigma_0  e^{-  \frac{153}{\delta_0} }   \sigma e^{-51 \txi^1 } ,   
	\end{equation} 
and  by Lemma \ref{t:Thpst}, we also obtain
	\begin{equation}\label{e:Sest2}
		S\frac{\chi_\uparrow}{B}=(2+\Xi)^2
\frac{B}{\chi_\uparrow}  \AND 		S =(2+\Xi)^2
\left(\frac{B}{\chi_\uparrow}\right)^2  . 
	\end{equation}

  We \textit{claim} all the sequential principal minors of the matrix $(\nu_\Lambda)_0\mathfrak{A}^0 +	(\nu_\Lambda)_i \frac{1}{A\ttau}  \mathfrak{A}^i $ are \textit{nonnegative}, and let us prove this claim as follows. Firstly, observe that by  \eqref{e:cc}, we obtain an estimate for $\txi^1 >-\frac{1}{  \delta _0 } $, 
    \begin{equation}\label{e:Xiest2}
  	1  -  \cc  \left(  1+   \frac{51 }{2 } \Xi \right)  \left( 1- \frac{1}{2(2+\txi^1\delta_0)^2}\right)>1  -  \cc  \left(  1+   \frac{51 }{2 } \Xi \right)   \geq 1  -  \cc  \left(  1+   \frac{51 }{2 } \Xi_0 \right)     >0  . 
  \end{equation}

We will first verify the first $2n+4$ key principal minors, denoted as $\mathrm{D}_\ell$ for $\ell = 1, 2, \cdots, 2n+4$. The remaining principal minors can be calculated directly, given the positive diagonal elements  $2+\mathtt{q}\frac{200}{51}$ and $1+\mathtt{q}\frac{100}{51}$.  By choosing the scale of data  $\sigma$ small enough (as indicated in Proposition   \ref{t:verfuc}.$(2)$ where $\|\mathfrak{U}(\ttau)\|_{H^k} \leq C\sigma$),  and using estimates from  \eqref{e:KLest1}, \eqref{e:Sest2} and \eqref{e:Xiest2}, we obtain the following result for  $-\frac{1}{\delta_0}<\txi^1 \leq \frac{2-51 \Xi_0 }{102 \delta _0 \Xi_0 }  $, 
  	\begin{align*}
  	\mathrm{D}_2= & \left(1+ \mathtt{q}_l \frac{100}{51}   \right)^2 ( S+\phi \widetilde{\mathscr{L} }  ) -	\left( \mathtt{q}_l S \frac{\widetilde{\underline{\chi_\uparrow}}}{B}    + 	\mathscr{K}^1\right)^2 \notag  \\  
  	= & \left(1+ 50 \cc   \left( 1- \frac{1}{2(2+\txi^1\delta_0)^2}\right)    \right)^2    
  	\phi \widetilde{\mathscr{L} }   + \left(  1  -  \cc  \left(  1+   \frac{51 }{2 } \Xi \right)  \left( 1- \frac{1}{2(2+\txi^1\delta_0)^2}\right)       - \frac{\chi_\uparrow}{B}	\frac{\mathscr{K}^1 }{ (2+\Xi) 
  		}    \right)   \notag  \\
  	& \times \left(  1       +  \cc  \left(101   + \frac{51 }{2 }\Xi\right) \left( 1- \frac{1}{2(2+\txi^1\delta_0)^2}\right)    + \frac{\chi_\uparrow}{B}	\frac{\mathscr{K}^1 }{ (2+\Xi) 
  	 }   \right)  (2+\Xi)^2 
  	\left(\frac{B}{\chi_\uparrow}\right)^2        \notag  \\ 
\overset{\eqref{e:Xiest2}}{\geq}  & -C\sigma_0 e^{-  \frac{153}{\delta_0} }   \sigma e^{-51 \txi^1 }+ \left(  1  -  \cc  \left(  1+   \frac{51 }{2 } \Xi \right)  \left( 1- \frac{1}{2(2+\txi^1\delta_0)^2}\right)       -C\sigma_0 e^{-  \frac{153}{\delta_0} }   \sigma e^{-51 \txi^1 }  \right)   \notag  \\
  	& \times \left(  1       +  \cc  \left(101   + \frac{51 }{2 }\Xi \right) \left( 1- \frac{1}{2(2+\txi^1\delta_0)^2}\right)   -C\sigma_0 e^{-  \frac{153}{\delta_0} }   \sigma e^{-51 \txi^1 }  \right)  (2+\Xi)^2 
  	\left(\frac{B}{\chi_\uparrow}\right)^2    >0 .    
  \end{align*}

On the other hand, let us estimate $\mathrm{D}_2$ for $\txi^1 >\frac{2-51 \Xi_0 }{102  \Xi_0} \frac{1}{\delta _0}$. To do this, we first introduce a decreasing function $\mathrm{Q}$,  which is essential for estimating  $\mathrm{D}_2$. Let
\begin{equation*}
	\mathrm{Q}(	\txi^1):=(\delta_0\txi^1+ 2 )^2   \exp\left(-  \frac{51A }{8} \left( 
	\txi^1-\frac{2-51 \Xi_0 }{102 \delta _0 \Xi_0 }\right)\right) . 
\end{equation*}
We can verify $\mathrm{Q}(	\txi^1)$ is decreasing for $\txi^1 >\frac{16 \delta _0-102 A}{51 A \delta _0}$ by direct differentiating it. If taking $\delta_0$ such that $0<\delta_0<\frac{153}{32}A$, then it follows that $\frac{16 \delta _0-102 A}{51 A \delta _0}<-\frac{1}{2 \delta _0}<\frac{2-51 \Xi_0 }{102  \Xi_0} \frac{1}{\delta _0}$. Therefore,   for $\txi^1 >\frac{2-51 \Xi_0 }{102  \Xi_0} \frac{1}{\delta _0}$, $\mathrm{Q}$ is decreasing, and thus
\begin{equation}\label{e:Qest1}
 	\mathrm{Q}(	\txi^1) \leq 	\mathrm{Q}\left(\frac{2-51 \Xi_0 }{102  \Xi_0} \frac{1}{\delta _0}\right) =\frac{\left(153 \Xi _0+2\right){}^2}{10404 \Xi _0^2} .
\end{equation}
A similar approach shows that if $0<\delta _0\leq \frac{153}{4}$, then for $\txi^1 >\frac{2-51 \Xi_0 }{102  \Xi_0} \frac{1}{\delta _0}>-\frac{1}{2\delta_0}$, $\mathrm{P}(\txi^1)=e^{-51 \txi^1 }  (\delta _0 \txi^1 +2 ){}^2$ is decreasing, and thus
\begin{equation}\label{e:Qest2}
	\mathrm{P}(	\txi^1) \leq 	\mathrm{P}\left(\frac{2-51 \Xi_0 }{102  \Xi_0} \frac{1}{\delta _0}\right) =\frac{\left(153 \Xi _0+2\right){}^2 }{10404 \Xi _0^2} e^{\frac{51 \Xi _0-2}{2 \delta _0 \Xi _0}}.
\end{equation} 
Then, using \eqref{e:Sest1a}--\eqref{e:Sest1b}, and taking $\delta_0$ to be sufficiently small, for $\txi^1 >\frac{2-51 \Xi_0 }{102  \Xi_0} \frac{1}{\delta _0}$, we obtain
    	\begin{align*}
  	\mathrm{D}_2= & \left(1+ \mathtt{q}_r \frac{100}{51}   \right)^2 ( S+\phi \widetilde{\mathscr{L} } ) -	\left( \mathtt{q}_r S \frac{\widetilde{\underline{\chi_\uparrow}}}{B}    + 	\mathscr{K}^1 \right)^2 \notag  \\ 
  	\geq &     \frac{1}{4} \left(1+ \mathtt{q}_r \frac{100}{51}   \right)^2    -	  \mathtt{q}_r^2 -  C\exp\left(-\frac{1}{\delta_0}   \frac{A}{ 8 \Xi _0}  -  \frac{51A }{8} \left( 
  	\txi^1-\frac{2-51 \Xi_0 }{102 \delta _0 \Xi_0 }\right)\right)  -  C\sigma_0     \sigma \exp\left(-  \frac{153}{\delta_0}-51 \txi^1 \right)   \notag  \\ 
\overset{\eqref{e:qlr}}{=} & \frac{  (202 \mathtt{q}_r+51)}{10404}  \frac{153}{4(\delta_0\txi^1+2)^2} -  C\exp\left(-\frac{1}{\delta_0}   \frac{A}{ 8 \Xi _0}   -  \frac{51A }{8} \left( 
  	\txi^1-\frac{2-51 \Xi_0 }{102 \delta _0 \Xi_0 }\right)\right)  -  C\sigma_0     \sigma \exp\left(-  \frac{153}{\delta_0}-51 \txi^1 \right)   \notag  \\ 
  \overset{\mathtt{q}_r\geq 17}{\geq} &  \frac{1}{ (\delta_0\txi^1+ 2 )^2} \left( \frac{205}{16}      -  C \exp\left(-\frac{1}{\delta_0}   \frac{A}{ 8 \Xi _0}  \right)  (\delta_0\txi^1+ 2 )^2  \exp\left(-  \frac{51A }{8} \left( 
  	\txi^1-\frac{2-51 \Xi_0 }{102 \delta _0 \Xi_0 }\right)\right)   \right. \notag \\
  	& \left.-  C(\delta_0\txi^1+ 2 )^2 \sigma_0    \sigma \exp\left(-  \frac{153}{\delta_0}-51 \txi^1 \right)   \right)  \notag  \\
  	 \overset{\eqref{e:Qest1}\&\eqref{e:Qest2}}{\geq} &    \frac{1}{ (\delta_0\txi^1+ 2 )^2} \left( \frac{205}{16}      -  C \exp\left(-\frac{1}{\delta_0}   \frac{A}{ 8 \Xi _0}  \right)  \frac{\left(153 \Xi _0+2\right){}^2}{10404 \Xi _0^2}    -  C \sigma \sigma_0     \frac{\left(153 \Xi _0+2\right){}^2 }{10404 \Xi _0^2} \exp\left(-\frac{255 \Xi _0+2}{2 \delta _0 \Xi _0}   \right)\right)  \notag  \\
  	 \geq &0  .   
  \end{align*}   
  
For $\ell=2,\cdots,n$, we calculate for $-\frac{1}{\delta_0}<\txi^1 \leq \frac{2-51 \Xi_0 }{102 \delta _0 \Xi_0 }  $, by taking $\sigma$  to be sufficiently small, 
  \begin{align*}
  	\mathrm{D}_{\ell+1}= & \left(\left(1+ \mathtt{q}_l \frac{100}{51}   \right)^2 ( S+\phi \widetilde{\mathscr{L} } ) -	\left( \mathtt{q}_l S \frac{\widetilde{\underline{\chi_\uparrow}}}{B}    + 	\mathscr{K}^1\right)^2 -\sum_{i=2}^{\ell}(\mathscr{K}^i)^2 \right) \left(1+ \mathtt{q}_l \frac{100}{51}   \right)^{\ell-1} ( S+\phi \widetilde{\mathscr{L} } )^{\ell-1}  \notag  \\ 
  	\geq  & \left( \mathrm{D}_2 -C\sigma_0 e^{-  \frac{153}{\delta_0} }   \sigma e^{-51 \txi^1 } \right) \left(1+ 50 \cc   \left( 1- \frac{1}{2(2+\txi^1\delta_0)^2}\right)  \right)^{\ell-1} \left( (2+\Xi)^2
  	\left(\frac{B}{\chi_\uparrow}\right)^2-C\sigma_0 e^{-  \frac{153}{\delta_0} }   \sigma e^{-51 \txi^1 } \right)^{\ell-1}  \notag  \\
  	> & 0 . 
  \end{align*}

On the other hand,  similarly, for $\txi^1 >\frac{2-51 \Xi_0 }{102 \delta _0 \Xi_0 }$, by taking $\delta_0$ sufficiently small,
    \begin{equation*}
  	\mathrm{D}_{\ell+1}=  \left(\mathrm{D}_2 -\sum_{i=2}^{\ell}(\mathscr{K}^i)^2 \right) \left(1+ \mathtt{q}_r \frac{100}{51}   \right)^{\ell-1} ( S+\phi \widetilde{\mathscr{L} } )^{\ell-1}  \geq 0 . 
  \end{equation*}
Therefore, $(\nu_\Lambda)_0\mathfrak{A}^0 +	(\nu_\Lambda)_i \frac{1}{A\ttau}  \mathfrak{A}^i $ is nonnegative, confirming that this surface is indeed weakly spacelike.

\underline{$(2)$ $\widehat{\Sigma}_{k}$ ($k=1,\pm2,\pm3,\cdots$) is weakly spacelike.}
The normal vector to  $\widehat{\Sigma}_k$ is given by $\nu_\Sigma:=(0,\delta^k_i)  $. 
To verify that $\widehat{\Sigma}_k$ is weakly spacelike, we evaluate:
\begin{equation*}
	(\nu_\Sigma)_0\widehat{\mathfrak{A}}^0_\phi +	(\nu_\Sigma)_i \frac{1}{A\hat{\tau}} \gamma \cos^2(\hat{\zeta}^i)  \widehat{\mathfrak{A}}^i_\phi =  \frac{1}{A\hat{\tau}} \gamma \cos^2(\hat{\zeta}^k)  \widehat{\mathfrak{A}}^k_\phi=  \frac{1}{A\hat{\tau}} \gamma \cos^2\frac{\pi}{2}  \widehat{\mathfrak{A}}^k_\phi =0 . 
\end{equation*}
Therefore, $\Sigma_k$ is weakly spacelike, which completes the proof. 
\end{proof}

\subsection{Verifications of Fuchsian systems}\label{s:verfuc}		
In this section, we provide a comprehensive proof for the verification of Fuchsian systems. Specifically, we aim to prove Proposition \ref{t:verfuc}, which involves verifying that the system \eqref{e:mainsys6} satisfies the conditions \ref{c:2}--\ref{c:7} in Appendix \ref{s:fuc}.

\begin{proof}[The proof of Proposition \ref{t:verfuc}.$(1)$] 
	First, we verify that the system  \eqref{e:mainsys6} satisfies all the conditions \ref{c:2}--\ref{c:7} for a Fuchsian system \eqref{e:model1} as specified in Appendix   \ref{s:fuc}.

	\underline{$(1)$ Verifications of \ref{c:2}, \ref{c:4} and  \ref{c:6}:} 
For the system \eqref{e:mainsys6},  corresponding to \eqref{e:model1}--\eqref{e:model2}, we have
	$\mathbf{P}= \mathds{1}$, $\mathbf{P}^\perp=0$, $ B^{0} =\widehat{\mathfrak{A}}^0_\phi$ and $B_2^i =\frac{1}{A } \gamma \cos^2(\hat{\zeta}^i) \widehat{\mathfrak{A}}^{i}_\phi$, $\mathbf{B} =\frac{1}{A } \widehat{\mathfrak{A}}_\phi $, $H=\widehat{\mathfrak{F}}_\phi$ and $B_0^i=0$, and there is
	\begin{gather*}
		\tilde{\mathbf{B}} =\frac{1}{A} \tilde{\mathfrak{A}}_\phi  
		= \frac{1}{A} 	  \p{  \frac{286}{3} &     -q^j+51   \delta^{1j}      &0  & 0 & -8 \frac{\hat{\eta}}{\hat{\mu}} & 8  \\
		51   \delta^{1}_k      &  26  \delta^j_k  &0& 6   \delta^l_k & 0 & 0\\
		- 8 \hat{\phi} \hat{\mu}  & 0 & \frac{ 40 }{3 }  & 0 &  -16 \hat{\phi} \hat{\eta}   & 0  \\0&  \frac{2  }{3       }     \delta^j_{s} &0&  \frac{302}{3}   \delta^l_s  & 0 & 0 \\ 0 & 0 &0 &0 & 100   &0  \\
		- 8    & 0 & 0 & 0 & 	-16   \frac{\hat{\eta}}{\hat{\mu}}    & \frac{640}{3}   }   \AND	\tilde{B}^0=\tilde{\mathfrak{A}}^0_\phi=  	\p{1 &   0   & 0 & 0 & 0 & 0 \\
				0     &  \hat{S}  \delta^j_k & 0 & 0 & 0 & 0 \\
				0&0& 2 &0&0 & 0 \\
				0&0&0&\delta^l_s &0& 0 \\
				0&0&0&0&1& 0 \\
				0&0&0&0&0& 2 }   ,   \\ 
		\tilde{B}_2^i(\hat{\tau},\hat{\zeta}^i)=\frac{1}{A} \gamma \cos^2(\hat{\zeta}^i )  \tilde{\mathfrak{A}}^{i}_\phi(\hat{\tau} ) = \frac{1}{A}\gamma \cos^2(\hat{\zeta}^i ) \p{ \frac{100}{51} \delta^i_1 & 	   \hat{S} \frac{\widehat{\underline{\chi_\uparrow}}}{B}  \delta^{ij}    & 0 & 0 & 0 &0\\
		\hat{S} \frac{\widehat{\underline{\chi_\uparrow}}}{B}  \delta^i_k   & \frac{100}{51}  \hat{S}  \delta^i_1 \delta^j_k   & 0 & 0 & 0 &0\\
			0&0& \frac{200}{51} \delta^i_1  &0&0&0\\
			0&0&0&\frac{100}{51} \delta^i_1  \delta^l_s &0&0\\
			0&0&0&0&  \frac{100}{51} \delta^i_1  & 0 \\0&0&0&0&0& \frac{200}{51} \delta^i_1  }   , 
	\end{gather*}
	such that
	\begin{gather*}
		\mathbf{B}(\hat{\tau}, \hat{\zeta}^i,  \widehat{\mathfrak{U}})-\tilde{\mathbf{B}} (\hat{\tau}, \hat{\zeta}^i)  = \frac{1}{A} (\widehat{\mathfrak{A}}_\phi(\hat{\tau}, \hat{\zeta}^i,  \widehat{\mathfrak{U}}) -\tilde{\mathfrak{A}}_\phi(\hat{\tau}, \hat{\zeta}^i ))   = \mathrm{O}(\widehat{\mathfrak{U}}) , \quad 	B^0(\hat{\tau}, \hat{\zeta}^i,  \widehat{\mathfrak{U}})-\tilde{B}^0(\hat{\tau}) =	\widehat{\mathfrak{A}}^0_\phi(\hat{\tau}, \hat{\zeta}^i,  \widehat{\mathfrak{U}})-\tilde{\mathfrak{A} }^0_\phi (\hat{\tau})= \mathrm{O}(\widehat{\mathfrak{U}}) ,  \\ [\mathbf{P},\tilde{\mathbf{B}}]=[\mathds{1},\tilde{\mathfrak{A}}_\phi/A]=0   
		\AND  	\mathbf{P}  (B_2^i(\hat{\tau}, \hat{\zeta}^i,  \widehat{\mathfrak{U}})-\tilde{B}_2^i(\hat{\tau})) \mathbf{P}  =  \frac{1}{A} \lambda \cos^2(\hat{\zeta}^i/\lambda) ( \widehat{\mathfrak{A}}^{i}_\phi  (\hat{\tau}, \hat{\zeta}^i,  \widehat{\mathfrak{U}})-\tilde{\mathfrak{A}}^{i}_\phi(\hat{\tau} ) )     =   \mathrm{O}( \widehat{\mathfrak{U}} ), 
	\end{gather*}
	for all $ (\hat{\tau}, \hat{\zeta}^i,  \widehat{\mathfrak{U}}) \in[-1,0]\times \Tbb^n_{[-\frac{\pi}{2} ,\frac{\pi}{2} ]} \times B_R(\Rbb^{4+2n}) $. 
We focus on verifying the key conditions, while the others are straightforward.

	From the expressions of  $B^i=\frac{1}{A\htau}  \gamma \cos^2(\hat{\zeta}^i ) \widehat{\mathfrak{A}}^{i}_\phi$, $\mathbf{B}=\frac{1}{A} \widehat{\mathfrak{A}}_\phi$ and $B^0=\widehat{\mathfrak{A}}^0_\phi$, using Lemma \ref{t:gb2}.\ref{l:2.3},  the revised coordinate transform  \eqref{e:coord2}, Theorem \ref{t:mainthm0} and Proposition \ref{t:limG}, we have
	\begin{align*}
		\tilde{B}^i_2 \in  & C^{0}([-1,0],C^{\infty}(\Tbb^n_{[-\frac{\pi}{2} ,\frac{\pi}{2} ]}   , \mathbb M_{(4+2n)\times (4+2n)})), \\
		B^i_2 \in  & C^{0}([-1,0],C^{\infty}(\Tbb^n_{[-\frac{\pi}{2} ,\frac{\pi}{2} ]} \times B_R(\Rbb^{4+2n}) , \mathbb M_{(4+2n)\times (4+2n)})), \\
		B^i \in  & C^{0}([-1,0),C^{\infty}(\Tbb^n_{[-\frac{\pi}{2} ,\frac{\pi}{2} ]} \times B_R(\Rbb^{4+2n}) , \mathbb M_{(4+2n)\times (4+2n)})), \\
		\mathbf{B} \in  & C^{0}([-1,0],C^{\infty}(\Tbb^n_{[-\frac{\pi}{2} ,\frac{\pi}{2} ]} \times B_R(\Rbb^{4+2n}) , \mathbb M_{(4+2n)\times (4+2n)})), \\
		B^0\in & C^{0}([-1,0],C^{\infty}( \Tbb^n_{[-\frac{\pi}{2} ,\frac{\pi}{2} ]} \times B_R(\Rbb^{4+2n}) , \mathbb M_{(4+2n)\times (4+2n)})). 
	\end{align*}

To verify that $B^0\in C^{1}([-1,0),C^{\infty}( \Tbb^n_{[-\frac{\pi}{2} ,\frac{\pi}{2} ]} \times B_R(\Rbb^{4+2n}) , \mathbb M_{(4+2n)\times (4+2n)})$, we need to compute  $\del{\htau}\widehat{\mathfrak{A}}^0$. Specifically, it suffices to verify that  $\del{\htau}\bigl(1/(1+\hat{\uf})\bigr) \in C^0([-1,0))$, $\del{\htau}\bigl(1/ \hat{\uf}\bigr) \in C^0([-1,0))$,  $\del{\htau}\widehat{\underline{\chi_\uparrow}} \in C^0([-1,0))$ and  $\del{\htau} \hat{S} \in C^0([-1,0)) $. First, using Lemma \ref{t:gb2}.\ref{l:2.5} and \eqref{e:Gdef2}, with the help of Propositions \ref{t:limG} and \ref{t:fginv0}, we obtain\footnote{In  fact, by \eqref{e:transf3} and \eqref{e:transf1}, we have $\del{\htau}\hat{\uf}=\widehat{\del{\ttau}\tilde{\uf}}=\widehat{\widetilde{\del{\tau}\uf}}$, for simplicity, we denote $\widehat{\widetilde{\del{\tau}\uf}}$ simply by $\widehat{\del{\tau}\uf}$. Similarly, we will use $\del{\hat{\tau}}\widehat{\underline{\chi_\uparrow}}$ and $\del{\hat{\tau}}\widehat{\underline{\mathfrak{G}}}$ to represent their respective transformations. }
	\begin{equation}\label{e:dt1/f1}
		\del{\htau} \biggl(\frac{1}{1+\hat{\uf}}\biggr)  =-\frac{\del{\htau}\hat{\uf}}{(1+\hat{\uf})^2}	=-\frac{\widehat{\del{\ttau} \tilde{\uf}}}{(1+\hat{\uf})^2}    =-  \frac{\hat{\mathsf{b}}_\uparrow^{\frac{2}{3}} (1+ \hat{\uf} )^{-\frac{5}{3}}}{A B  \hat{\uf  } (-\htau)^{\frac{2}{3A}+1}}	\hat{\underline{f_0} } =  \frac{ \widehat{\underline{\chi_\uparrow}}}{AB\htau(1+\hat{\uf}) }    \in C^0([-1,0))   . 
	\end{equation}	
Similarly, 
		\begin{equation}\label{e:dt1/f2}
		\del{\htau} \biggl(\frac{1}{\hat{\uf}}\biggr)  =-\frac{\del{\htau}\hat{\uf}}{(1+\hat{\uf})^2} \frac{(1+\hat{\uf})^2}{ \hat{\uf}^2}	  =  \frac{ \widehat{\underline{\chi_\uparrow}}}{AB\htau   \hat{\uf}  } \frac{(1+\hat{\uf}) }{ \hat{\uf} }	   \in C^0([-1,0))   . 
	\end{equation}	
Next, applying Lemmas \ref{t:gb2}.\ref{l:2.3},  \ref{t:iden1},  \ref{t:dtchi}, and identities  \eqref{e:iden3} and \eqref{e:chig}, direct calculations (see \cite[Lemma $3.7$]{Liu2023} for details) imply
	\begin{equation}\label{e:dtchi2}
	\del{\htau}\widehat{\underline{\mathfrak{G}}} =	\del{\htau}\widehat{\underline{\chi_\uparrow}} = \widehat{ \del{\tau} \underline{\chi_\uparrow}  }=\widehat{\del{\tau}\mathsf{b}_\uparrow \underline{\del{\mathfrak{t}} \chi_\uparrow}  }  
		=      \frac{   \hat{ \underline{\mathfrak{G}} }  \widehat{ \underline{\chi_\uparrow}  } }{3 A  B \htau      }  +  \frac{ \widehat{ \underline{\chi_\uparrow}}^{2} }{A B \htau   \hat{\uf}    }  +\frac{2}{3}     \frac{ \widehat{\underline{\chi_\uparrow}}^{\frac{3}{2}} }{A  B^{\frac{1}{2}} \htau  \hat{\uf}^{\frac{1}{2}}    } \in C^0([-1,0))   . 
	\end{equation} 
Combining \eqref{e:R1} \eqref{e:L1},  \eqref{e:dt1/f1}--\eqref{e:dtchi2}, and performing direct calculations, we conclude that $\del{\htau}\hat{S}\in C^0([-1,0)) $. Therefore, we have
	\begin{equation*}  \widehat{\mathscr{L}}\in C^{1}([-1,0),C^{\infty}(\Tbb^n_{[-\frac{\pi}{2} ,\frac{\pi}{2} ]} \times B_R(\Rbb^{4+2n}), \Rbb)   \AND
		 \hat{R} \in C^{1}([-1,0),C^{\infty}(\Tbb^n_{[-\frac{\pi}{2} ,\frac{\pi}{2} ]} \times B_R(\Rbb^{4+2n}), \Rbb) . 
	\end{equation*}    
These results imply that $B^0\in C^{1}([-1,0),C^{\infty}(\Tbb^n_{[-\frac{\pi}{2} ,\frac{\pi}{2} ]} \times B_R(\Rbb^{4+2n}) , \mathbb M_{(4+2n)\times (4+2n)})$.

To estimate the parameter $\mathtt{b}$ defined in Theorem \ref{t:fuc}, where $\tilde{B}^0$ and $\tilde{\mathfrak{A}}^i_\phi$ are both only $\htau$-dependent (thus  $D\tilde{B}^0=0$ and   $D\tilde{\mathfrak{A}}^i_\phi=0$), we have 
		\begin{align}
		\mathtt{b}:= & \sup_{T_0\leq t<0} \bigl(\||  \tilde{\mathbf{B}} D (\tilde{\mathbf{B}}^{-1} )  \tilde{B}_2^i |_{\mathrm{op}}\|_{\Li}+\|| \tilde{\mathbf{B}} D (\tilde{\mathbf{B}}^{-1} \tilde{B}_2^i) |_{\mathrm{op}}\|_{\Li}\bigr)  \notag  \\
		= & \frac{1}{A} \gamma \sup_{T_0\leq t<0} \bigl(\||    \tilde{\mathfrak{A}}_\phi   D (  \tilde{\mathfrak{A}}_\phi  ^{-1} )   \cos^2(\hat{\zeta}^i )  \tilde{\mathfrak{A}}^{i}_\phi |_{\mathrm{op}}\|_{\Li}+\||  \tilde{\mathfrak{A}}_\phi   D (  \tilde{\mathfrak{A}}_\phi  ^{-1}   \cos^2(\hat{\zeta}^i )  )  \tilde{\mathfrak{A}}^{i}_\phi |_{\mathrm{op}}\|_{\Li}\bigr) \notag  \\
		= & \frac{1}{A} \gamma \sup_{T_0\leq t<0} \bigl(\||    \tilde{\mathfrak{A}}_\phi   D (  \tilde{\mathfrak{A}}_\phi  ^{-1} )   \cos^2(\hat{\zeta}^i )  \tilde{\mathfrak{A}}^{i}_\phi |_{\mathrm{op}}\|_{\Li}+\||   \bigl(\tilde{\mathfrak{A}}_\phi   D (  \tilde{\mathfrak{A}}_\phi  ^{-1}  ) \cos^2(\hat{\zeta}^i )   +    D (  \cos^2(\hat{\zeta}^i )  )  \bigr) \tilde{\mathfrak{A}}^{i}_\phi |_{\mathrm{op}}\|_{\Li}\bigr)  \notag  \\
		\overset{(\star)}{\leq} & C_0\gamma. \label{e:ttb}
	\end{align}
To obtain $(\star)$, noting that $\hat{\zeta}\in\Tbb^n_{[-\frac{\pi}{2},\frac{\pi}{2} ]}$, using \eqref{e:eta2}--\eqref{e:phi1}, we find that  $0\leq\hat{\phi}\hat{\mu},\hat{\phi}\hat{\eta} \in C^\infty([-\frac{\pi}{2},\frac{\pi}{2} ])$ and $\lim_{\hat{\zeta}^1\rightarrow \pm\infty} \hat{\phi}\hat{\mu} =0$ as well as $\lim_{\hat{\zeta}^1\rightarrow \pm\infty} \hat{\phi}\hat{\eta} =0$. Additionally, 
\begin{equation*}
	 |\cos^2(\hat{\zeta}^i )  D(\tilde{\mathfrak{A}}_\phi^{-1})|\leq C_1 \del{\hat{\zeta}}(\hat{\phi}\hat{\eta} )  \cos^2(\hat{\zeta}^i )   +C_2\del{\hat{\zeta}}(\hat{\phi}\hat{\eta} )  \cos^2(\hat{\zeta}^i )  = C_1 \del{\tilde{\zeta}}(\tilde{\phi}\tilde{\eta} )   +C_2\del{\tilde{\zeta}}(\tilde{\phi}\tilde{\eta} )     <C\sigma. 
\end{equation*}
Thus, there exists a constant $C_0>0$, such that  \eqref{e:ttb} holds. 
	
With these settings, there exists a constant $R>0$ (shrinking it if necessary), such that conditions \ref{c:2}, \ref{c:4} and \ref{c:6} are satisfied. This is evident from the fact that   $\mathbf{P}=\mathds{1}$,  $\mathbf{P}^\perp =0$ and the symmetries of $\widehat{\mathfrak{A}}^0_\phi$ and $\widehat{\mathfrak{A}}^{i}_\phi$.

	\underline{$(2)$ Verifications of \ref{c:3}:} 
	From Lemma \ref{t:Seq1}--\ref{t:mainsys1}, Lemma \ref{t:sigsys} and \eqref{e:mainsys5}, we note $	\widehat{\mathfrak{F}}_\phi$ can be   expressed as
	\begin{equation*}
		\widehat{\mathfrak{F}}_\phi=\frac{\hat{\underline{\mathfrak{G}}}}{\htau} \widehat{\mathscr{F}}_1(\htau, \hat{\zeta}, \widehat{\mathfrak{U}}) +\frac{1}{\htau \hat{\uf}^{\frac{1}{2}}}\widehat{\mathscr{F}}_2(\htau, \hat{\zeta}, \widehat{\mathfrak{U}}) , 
	\end{equation*}
	where 
	\begin{equation}\label{e:S12}
		\widehat{\mathscr{F}}_\ell(\htau, \hat{\zeta}, \widehat{\mathfrak{U}})  \in C^0([-1,0],C^{\infty}(\Tbb^n_{[-\frac{\pi}{2},\frac{\pi}{2}]} \times B_R(\Rbb^{4+2n}) ,\Rbb^{4+2n}))   \AND 	\widehat{\mathscr{F}}_\ell(\htau, \hat{\zeta}, 0) = 0 , \quad (\ell=1,2). 
	\end{equation}

	By Proposition \ref{s:gf1/2}, we can regard $\frac{1}{\htau \hat{\uf}^{\frac{1}{2}}}\widehat{\mathscr{F}}_2(\htau, \hat{\zeta}, \widehat{\mathfrak{U}})$ as $H_0$ in Condition \ref{c:3} of Appendix \ref{s:fuc}.  It follows that   $H_0\in C^0([-1,0],C^{\infty}(\Tbb^n_{[-\frac{\pi}{2},\frac{\pi}{2}]} \times B_R(\Rbb^{4+2n})  ),\Rbb^{4+2n}) $ and  $H_0(\htau,\hat{\zeta},\widehat{\mathfrak{U}} )=\mathrm{O}(\mathrm{\widehat{\mathfrak{U}}})$ by virtue of \eqref{e:S12}.

According to Lemma \ref{t:Gest2}, $\frac{\hat{\underline{\mathfrak{G}}}}{\htau} \widehat{\mathscr{S}}_1(\htau, \hat{\zeta}, \widehat{\mathfrak{U}} )$ can be regarded as  $|\htau|^{-\frac{1}{2}}H_1$ in Condition \ref{c:3} of Appendix \ref{s:fuc}.   It satisfies  $H_1\in C^0([-1,0],C^{\infty}(\Tbb^n_{[-\frac{\pi}{2},\frac{\pi}{2}]} \times B_R(\Rbb^{4+2n}),\Rbb^{4+2n})) $ and  $H_1(\htau,\hat{\zeta},\widehat{\mathfrak{U}})=\mathrm{O}(\mathrm{\widehat{\mathfrak{U}}})$ as indicated by   \eqref{e:S12}. Therefore, we have completed the verification of Condition \ref{c:3}.

\underline{$(3)$ Verifications of \ref{c:5}:} To verify Condition \ref{c:5},  we need to show that there exist constants  $\acute{\kappa}, \, \gamma_{2}$ and $ \gamma_{1}$ such that
	\begin{equation}\label{e:pstv}
		\frac{1}{\gamma_{1}}\mathds{1}\leq B^{0}\leq \frac{1}{\acute{\kappa}} \textbf{B} \leq\gamma_{2}\mathds{1}   \quad \text{i.e., }\quad \frac{1}{\gamma_{1}} X^T \mathds{1}  X \leq  X^T \widehat{\mathfrak{A}}^{0}_\phi X \leq \frac{1}{\acute{\kappa}} X^T \frac{1}{A} \widehat{\mathfrak{A}}_\phi  X \leq  \gamma_{2} X^T\mathds{1} X
	\end{equation}
	for all $ (\htau,\hat{\zeta},\widehat{\mathfrak{U}}) \in[-1,0] \times \Tbb^n_{[-\frac{\pi}{2},\frac{\pi}{2}]} \times B_R(\Rbb^{4+2n})$ and $X :=(X_1,\cdots,X_{4+2n})^T \in  \Rbb^{4+2n}$.

	Firstly, by \eqref{e:Aph}, we can rewrite $X^T\widehat{\mathfrak{A}}_\phi X$ in a symmetric form
	\begin{equation}\label{e:eBe1}
		X^T\widehat{\mathfrak{A}}_\phi X=   X^T\acute{\mathfrak{A}}_\phi X
	\end{equation}
	where $\acute{\mathfrak{A}}_\phi$ is a symmetric matrix given by  
		\begin{align}\label{e:Aphi2}
		&\acute{\mathfrak{A}}_\phi=    \p{  \frac{286}{3} &  -\frac{q^j}{2}+   51   \delta^{1j}      & - 4\hat{\phi} \hat{\mu}  & 0 & -4 \sigma_0  & 0  \\
		-\frac{q^j}{2}+ 	51   \delta^{1}_k      &  26  \delta^j_k  &0& \frac{10 }{3}  \delta^l_k & 0 & 0\\
			- 4\hat{\phi} \hat{\mu}  & 0 & \frac{ 40 }{3 }  & 0 &  -8 \hat{\phi} \hat{\eta}   & 0  \\0& \frac{10 }{3}    \delta^j_{s} &0&  \frac{302}{3}   \delta^l_s  & 0 & 0 \\ -4 \sigma_0  & 0 & -8 \hat{\phi} \hat{\eta}  &0 & 100   &-8  \sigma_0   \\
			0    & 0 & 0 & 0 & 	-8  \sigma_0   & \frac{640}{3}   }    + \widehat{\mathscr{Z}}_{(4+2n)\times (4+2n)}(\ttau;  \hat{\phi}{\hat{\mu}} \mfu_0,    \mfu, \hat{\phi}{\hat{\mu}} \mfu_i,\hat{\phi}{\hat{\eta}} \mathfrak{z},\hat{\phi}{\hat{\mu}} \mathfrak{B}_\zeta ) 
	\end{align}
and  $\widehat{\mathscr{Z}}_{(4+2n)\times (4+2n)}(\ttau;  \hat{\phi}{\hat{\mu}} \mfu_0,    \mfu, \hat{\phi}{\hat{\mu}} \mfu_i,\hat{\phi}{\hat{\eta}} \mathfrak{z},\hat{\phi}{\hat{\mu}} \mathfrak{B}_\zeta ) $ is a $(4+2n) \times (4+2n)$ symmetric matrix with entries.

Expanding \eqref{e:Aphi2} using assumption \ref{Asp2}   (with $q^i=|q|\delta^i_1$ and $|q|\in(3,100)$), 
	\begin{align}\label{e:eBe2}
	  X^T\acute{\mathfrak{A}}_\phi X  = & \frac{286}{3}X_1^2 +26\sum_{\ell=2}^{n+1} X_\ell^2 +\frac{40}{3} X_{n+2}^2 +\frac{302}{3}\sum_{\ell=n+3}^{2n+2} X_\ell^2 + 100 X_{2n+3}^2+\frac{640}{3} X_{2n+4}^2  \notag  \\
	  & +2\cdot \biggl(-\frac{|q|}{2} + 51\biggr) X_1X_2 -2\cdot 4 \hat{\phi}\hat{\mu} X_{n+2}X_1+2\cdot \frac{10}{3}\sum_{\ell=0}^{n-1} X_{n+3+\ell}X_{2+\ell} -2\cdot 4 \sigma_0 X_{2n+3} X_1 \notag \\
	  &  -2 \cdot 8\hat{\phi} \hat{\eta} X_{2n+3} X_{n+2} -2 \cdot 8 \sigma_0  X_{2n+4} X_{2n+3}+X^T \widehat{\mathscr{Z}}_{(4+2n)\times (4+2n)}   X . 
	\end{align} 
Since $\widehat{\mathscr{Z}}_{(4+2n)\times (4+2n)}(\ttau;0)=0$,  and by continuity, there exists a constant  $\hat{R}>0$, such that for $\mathfrak{U}\in B_{\hat{R}}(\Rbb^{4+2n})$, and for sufficiently small $\sigma_0$,  using \eqref{e:eBe2} and \eqref{e:mueta2}, $|q|\in(3,100)$,  Young's inequality with $p$ (i.e., $2ab\geq -\frac{a^2}{p} - p b^2$) and Cauchy's inequality, we arrive at a lower bound for $X^T\acute{\mathfrak{A}}_\phi X$,  
	\begin{align}\label{e:XAX1}			  X^T\acute{\mathfrak{A}}_\phi X  \geq & \biggl(\frac{286}{3}-\frac{(51-|q|/2)}{p}
	-4\hat{\phi}\hat{\mu}-4\sigma_0 \biggr) X_1^2 +\biggl(26-\Bigl(51-\frac{|q|}{2}\Bigr)p-\frac{10}{3} r \biggr) X_2^2 + \sum_{\ell=3}^{n+1}\biggl(26 -\frac{10}{3}\biggr) X_\ell^2 \notag  \\
	&  +\biggl(\frac{40}{3}-4\hat{\phi}\hat{\mu}-8\hat{\phi}\hat{\eta} \biggr)  X_{n+2}^2 + \biggl(\frac{302}{3}-\frac{10}{3r} \biggr)  X_{n+3}^2  + \biggl(\frac{302}{3}-\frac{10}{3} \biggr)\sum_{\ell=n+4 }^{2n+2} X_\ell^2  \notag  \\
	&  +\biggl(100-4\sigma_0-8\hat{\phi}\hat{\eta}-8 \sigma_0 \biggr) X_{2n+3}^2 +\biggl( \frac{640}{3} - 8 \sigma_0 \biggr)  X_{2n+4}^2 - C\hat{R} \sum_{\ell=1}^{2n+4}X_\ell^2 \notag  \\
\overset{\text{Let } p=\frac{13}{25},r=\frac{1}{20}}{>}&	 \biggl(\frac{11}{78}-8\sigma_0  \biggr) X_1^2 +\frac{7}{75}X_2^2 +\frac{68}{3} \sum_{\ell=3}^{n+1} X_\ell^2   +\biggl(\frac{40}{3}-12\sigma_0 \biggr)  X_{n+2}^2 + 34 X_{n+3}^2 + \frac{292}{3} \sum_{\ell=n+4 }^{2n+2} X_\ell^2  \notag  \\
&   +\bigl(100-20\sigma_0 \bigr) X_{2n+3}^2 +\biggl( \frac{640}{3} - 8 \sigma_0 \biggr)  X_{2n+4}^2 - C\hat{R} \sum_{\ell=1}^{2n+4}X_\ell^2   
>        
 \frac{1}{50}X^T\mathds{1} X. 
	\end{align}

Similarly, to determine the upper bound of  $X^T\acute{\mathfrak{A}}_\phi X$,  we perform the following calculation
\begin{align}\label{e:XAX2}			  
	X^T\acute{\mathfrak{A}}_\phi X   
	<&	  \biggl(8 \sigma _0+\frac{439}{3}   \biggr)   X_1^2 + \frac{241}{3}X_2^2 +\frac{88}{3} \sum_{\ell=3}^{n+1} X_\ell^2 + \biggl(\frac{40}{3}+12\sigma_0 \biggr) X_{n+2}^2  + 104 \sum_{\ell=n+3}^{2n+2} X_\ell^2  \notag  \\
	& + (100+20\sigma_0) X_{2n+3}^2+ \biggl( \frac{640}{3}+8 \sigma_0 \biggr)   X_{2n+4}^2  +C\hat{R} \sum_{\ell=1}^{2n+4}X_\ell^2 <250 X^T\mathds{1} X. 
\end{align} 
By utilizing \eqref{e:A0b} and Lemma \ref{t:coef1},  we can estimate $X^T \widehat{\mathfrak{A}}^{0}_\phi X$ for $\widehat{\mathfrak{U}}\in B_{\hat{R}}(\Rbb^{4+2n})$ as follows
\begin{equation}\label{e:XA0X}
\frac{\cm^2}{2}  X^T\mathds{1} X	\leq X^T \widehat{\mathfrak{A}}^{0}_\phi X \leq \left(2+ \frac{1}{\beta}\right) X^T\mathds{1} X	  .   
\end{equation}
Combining \eqref{e:eBe1},  \eqref{e:XAX1}--\eqref{e:XA0X}  together, we conclude
\begin{equation}\label{e:XAXcomp}
\frac{\cm^2 }{2} X^T\mathds{1} X	\leq X^T \widehat{\mathfrak{A}}^{0}_\phi X   <  \frac{50 (2\beta+1)}{\beta} X^T\widehat{\mathfrak{A}}_\phi X\leq  \frac{12500(2\beta+1)}{\beta} X^T\mathds{1} X   . 
\end{equation}
Thus, we can choose constants $\acute{\kappa}$, $\gamma_1$ and $\gamma_2$ in \eqref{e:pstv} as follows, 
\begin{equation*}
	\acute{\kappa}=\frac{\beta}{50A(2\beta+1) } ,\quad \gamma_1=\frac{2}{\cm^2 } \AND  \gamma_2= \frac{12500(2\beta+1)}{\beta} . 
\end{equation*}
Therefore, the inequality \eqref{e:XAXcomp}  verifies Condition \ref{c:5}.

	\underline{$(4)$ Verifications of \ref{c:7}:} 
Given that  $\mathbf{P}=\mathds{1}$ and $\mathbf{P}^\perp=0$, the verification reduces to ensuring that there exist constants  $\theta$ and $\beta_\ell>0$ ($\ell=0,1$),  such that the following condition, derived from \eqref{e:PhP1}, is satisfied, 
	\begin{equation}\label{e:divB1}
		\mathrm{div} \widehat{\mathfrak{A}}_\phi (\htau,\hat{\zeta}, \widehat{\mathfrak{U}},\widehat{\mathfrak{W}}) =   \;\mathcal{O}\bigl(\theta   +|\htau|^{-\frac{1}{2}}\beta_0+|\htau|^{-1}\beta_1 \bigr) ,
	\end{equation}
	where\footnote{To facilitate readers' understanding, we have labeled the orders of the terms below. The detailed proofs will be provided in subsequent expositions. }
	\begin{align}\label{e:divB2}
		\mathrm{div} \widehat{\mathfrak{A}}_\phi (\htau,\hat{\zeta}, \widehat{\mathfrak{U}},\widehat{\mathfrak{W}}) :=& \underbrace{\del{\htau} \widehat{\mathfrak{A}}^0_\phi(\htau,\hat{\zeta}, \widehat{\mathfrak{U}})}_{(d)\;(-\htau)^{-\frac{1}{2}} -\text{term}}  +	\del{\widehat{\mathfrak{U}}}  \widehat{\mathfrak{A}}^0_\phi(\htau,\hat{\zeta}, \widehat{\mathfrak{U}})  \cdot (\widehat{\mathfrak{A}}^0_\phi(\htau,\hat{\zeta}, \widehat{\mathfrak{U}}) )^{-1} \Bigl[\underbrace{-	 \frac{1}{A\htau} \gamma \cos^2(\hat{\zeta}^i )  \widehat{\mathfrak{A}}^{i}_\phi (\htau,\hat{\zeta}, \widehat{\mathfrak{U}}) \cdot \widehat{\mathfrak{W}}_i }_{(a)\;\htau^{-1}-\text{term}} \notag  \\
		&\hspace{-3cm}  +\underbrace{ \frac{1}{ A\htau}	 \widehat{\mathfrak{A}}_\phi (\htau,\hat{\zeta}, \widehat{\mathfrak{U}})  \widehat{\mathfrak{U}}}_{(b)\; \htau^{-1}-\text{term}}   +  \underbrace{ \widehat{ \mathfrak{F}}_\phi(\htau,\hat{\zeta}, \widehat{\mathfrak{U}})}_{(e)\;(-\htau)^{-\frac{1}{2}} -\text{term}}  \Bigr] +\underbrace{\frac{1}{A\htau} \del{ \hat{\zeta}^i} \left( \gamma  \cos^2(\hat{\zeta}^i ) \widehat{\mathfrak{A}}^{i}_\phi(\htau,\hat{\zeta}, \widehat{\mathfrak{U}})\right)  +  \frac{\gamma  \cos^2(\hat{\zeta}^i ) }{A\htau} 	\del{\widehat{\mathfrak{U}}}    \widehat{\mathfrak{A}}^{i}_\phi(\htau,\hat{\zeta}, \widehat{\mathfrak{U}})  \cdot \widehat{\mathfrak{W}}_i }_{(c)\;\htau^{-1}-\text{term}}
	\end{align}
	for $(\htau, \hat{\zeta}, \widehat{\mathfrak{U}}, \widehat{\mathfrak{W}}_i) \in [-1,0)\times \Tbb^n_{[-\frac{\pi}{2},\frac{\pi}{2}]} \times  B_{\tilde{R}}(\Rbb^{4+2n}) \times  B_{\tilde{R}}(\mathbb{M}_{(4+2n) \times n})$. 
	Since $\mathbf{P}^\perp=0$, the left-hand sides of \eqref{e:PhP2}--\eqref{e:PhP4} are identically zero, which trivially satisfies these conditions.

Next, let us compute  $\del{\htau} \widehat{\mathfrak{A}}^0_\phi(\htau,\hat{\zeta}, \widehat{\mathfrak{U}})$. As described in $(1)$,   $\del{\htau}\hat{R}$  depends on $\del{\htau} \widehat{\mathscr{L}} $, $\del{\htau} \hat{S}$,  $\del{\htau}\bigl(1/(1+\hat{\uf})\bigr) $ and $\del{\htau}\widehat{\underline{\chi_\uparrow}}$ due to \eqref{e:R1}. Moreover,  $\del{\htau} \widehat{\mathscr{L}} $ and $\del{\htau} \hat{S}$ depend on $\del{\htau}\bigl(1/\hat{\uf}\bigr) $ , $\del{\htau}\bigl(1/(1+\hat{\uf})\bigr) $ and $\del{\htau}\widehat{\underline{\chi_\uparrow}}$ due to \eqref{e:L1}.  Therefore, the non-vanishing terms in $\del{\htau} \widehat{\mathfrak{A}}^0_\phi(\htau,\hat{\zeta}, \widehat{\mathfrak{U}})$ depend on $\del{\htau}\bigl(1/\hat{\uf}\bigr) $ , $\del{\htau}\bigl(1/(1+\hat{\uf})\bigr) $ and $\del{\htau}\widehat{\underline{\chi_\uparrow}}$.  Utilizing  \eqref{e:dt1/f1}--\eqref{e:dtchi2}  along with  Lemma \ref{t:Gest2} and Proposition \ref{s:gf1/2}, we find that there exist constants  $\tilde{\theta}$ and $\tilde{\beta}_{0}$ such that
	\begin{equation}\label{e:H6.1}
		\del{\htau} \widehat{\mathfrak{A}}^0_\phi =    \;\mathcal{O}\bigl(\tilde{\theta} +\tilde{\beta}_{0} \ |\htau|^{-\frac{1}{2}}\bigr).  
	\end{equation}

	Next, note the following facts: 
	\begin{enumerate}[leftmargin=*] 
		\item[(a)] The derivative $\del{\widehat{\mathfrak{U}}} \hat{R}$ depends on $\del{\widehat{\mathfrak{U}}} \hat{\mathscr{L}}$, where we have $\del{\widehat{\mathfrak{U}}} \hat{\mathscr{L}}=\;\mathcal{O}\bigl(\hat{\theta}   \bigr)$ and $\del{\widehat{\mathfrak{U}}} \widehat{\mathscr{Z}^{ij}}=\;\mathcal{O}\bigl(\hat{\theta}   \bigr)$ for some constants $\hat{\theta} >0$  based on  \eqref{e:L1} and \eqref{e:ZIj}; 
		\item[(b)] Note the following relation, 
		\begin{equation}\label{e:dzcosA}
			\del{ \hat{\zeta}^i} \left(  \gamma \cos^2(\hat{\zeta}^i ) \widehat{\mathfrak{A}}^{i}_\phi(\htau,\hat{\zeta}, \widehat{\mathfrak{U}})\right)  = -2  \gamma\cos(\hat{\zeta}^i) \sin(\hat{\zeta}^i) \widehat{\mathfrak{A}}^{i}_\phi(\htau,\hat{\zeta}, \widehat{\mathfrak{U}}) +   \gamma \cos^2(\hat{\zeta}^i )\del{ \hat{\zeta}^i}   \widehat{\mathfrak{A}}^{i}_\phi(\htau,\hat{\zeta}, \widehat{\mathfrak{U}})  . 
		\end{equation}
		\item[(c)] 	To estimate \eqref{e:dzcosA}, the first term on the RHS is bounded.  The key task is to verify the second term.  The only non-vanishing terms are $  \gamma \cos^2(\hat{\zeta}^i )\del{\hat{\zeta}^i}\widehat{\mathscr{Z}}^{ij}_\phi$ and $  \gamma \cos^2(\hat{\zeta}^i )\del{\hat{\zeta}^i}\widehat{\mathscr{L}}_\phi$.  		
To estimate these terms, one can directly compute them. We  focus on how to estimate the factors $ \gamma \cos^2(\hat{\zeta}^i )\del{\hat{\zeta}^i}(\hat{\phi}\hat{\mu})$ and $ \gamma \cos^2(\hat{\zeta}^i )\del{\hat{\zeta}^i}(\hat{\phi}\hat{\eta})$ in some terms. As an example, let's consider $ \gamma \cos^2(\hat{\zeta}^i )\del{\hat{\zeta}^i}(\hat{\phi}\hat{\mu})$ to illustrate the estimation process, 
		\begin{align}
		|	\gamma \cos^2(\hat{\zeta}^i )\del{\hat{\zeta}^i}(\hat{\phi}\hat{\mu})|= & |\gamma \cos^2(\hat{\zeta}^i )\hat{\mu}\del{\hat{\zeta}^i}\hat{\phi}+\gamma \cos^2(\hat{\zeta}^i )\hat{\phi}\del{\hat{\zeta}^i}\hat{\mu}|
			\overset{\eqref{e:transf3}}{=}   |\hat{\mu} \widehat{\del{\tilde{\zeta}^i} \phi}+ \hat{\phi}\widehat{\del{\tilde{\zeta}^i} \mu} |  \notag  \\ \overset{\eqref{e:phi1},\eqref{e:eta2}\&\eqref{e:mueta2}}{ \leq } &
			| C\hat{\mu} -51\delta^1_i\hat{\phi}\hat{\mu} |\leq C\hat\mu\leq C\sigma_0.  \label{e:gamcos}
		\end{align}
  Therefore, it follows from \eqref{e:L1} and \eqref{e:ZIj} that  $|\del{\hat{\zeta}^i}\widehat{\mathscr{L}} | \lesssim \sigma_0$  and $|\del{\hat{\zeta}^i}\widehat{\mathscr{Z}^{ij}}| \lesssim \sigma_0$. 
	\end{enumerate}
With the above observations, by directly calculating $\del{\widehat{\mathfrak{U}}}  \widehat{\mathfrak{A}}^0_\phi$, $\del{\widehat{\mathfrak{U}}}  \widehat{\mathfrak{A}}^{i}_\phi$  and $\del{\hat{\zeta}^i} \widehat{\mathfrak{A}}^{i}_\phi$, and employing \cite[Lemmas C.$1$ and C.$2$]{Liu2018} to calculate $(\widehat{\mathfrak{A}}^0_\phi )^{-1}$, we conclude that there exist constants $\hat{\theta}$ and $\tilde{R}$, such that
	\begin{equation}\label{e:H6.2}
		\del{\widehat{\mathfrak{U}}}\widehat{\mathfrak{A}}^0_\phi =\;\mathcal{O}\bigl(\hat{\theta}   \bigr) , \quad  \del{\widehat{\mathfrak{U}}}  \widehat{\mathfrak{A}}^{i}_\phi=\;\mathcal{O}\bigl(\hat{\theta}   \bigr)  , \quad   	\del{ \hat{\zeta}^i} \bigl(  \gamma \cos^2(\hat{\zeta}^i ) \widehat{\mathfrak{A}}^{i}_\phi \bigr)   =\;\mathcal{O}\bigl(\hat{\theta} \sigma_0  \bigr)   \AND (\widehat{\mathfrak{A}}^0 )^{-1} =\;\mathcal{O}\bigl(\hat{\theta}   \bigr)
	\end{equation}
	for $\widehat{\mathfrak{U}}\in B_{\tilde{R}}(\Rbb^{4+2n})$. 

	Additionally, from our previous verification of Condition \ref{c:3} (see $(2)$), there exist constants $\bar{\theta}$, $\hat{\beta}_0$ and $\tilde{R}$, such that
	\begin{equation}\label{e:H6.5}
		\widehat{\mathfrak{F}}_\phi  =   \;\mathcal{O}\bigl( \bar{\theta}+ \hat{\beta}_0  |\htau|^{-\frac{1}{2}} \bigr)
	\end{equation}
	for $\widehat{\mathfrak{U}}\in B_{\tilde{R}}(\Rbb^{4+2n})$.

With the help of \eqref{e:H6.1}--\eqref{e:H6.5}, directly examining \eqref{e:divB2} term by term, we can conclude there are constants $
\theta>0, \;   \beta_0>0 $ and $ \beta_1 >0$,
such that \eqref{e:divB1} holds for $(\htau, \hat{\zeta}, \widehat{\mathfrak{U}}, \widehat{\mathfrak{W}}_i) \in [-1,0)\times \Tbb^n_{[-\frac{\pi}{2},\frac{\pi}{2}]} \times  B_{\tilde{R}}(\Rbb^{4+2n}) \times  B_{\tilde{R}}(\mathbb{M}_{(4+2n) \times n})$.

Moreover,  by properly shrinking the ball $B_{\tilde{R}}(\Rbb^N) \times  B_{\tilde{R}}(\mathbb{M}_{(4+2n) \times n}) \ni (\mathfrak{U},\;\mathfrak{W}_i)$ and taking $\sigma_0$ sufficiently small  (note, with the help of \eqref{e:dzcosA} and \eqref{e:gamcos},  the $\htau^{-1}$-terms $(a)$, $(b)$ and $(c)$ in \eqref{e:divB2} all include a factor $\widehat{\mathfrak{U}}$,  $\widehat{\mathfrak{W}}_i$, or can be bounded by $\sigma_0$ and $\gamma$), we can further select $\gamma$ (which can be freely chosen) and set $\beta_1=C_1\gamma>0$ (absorbing $\sigma$ and $\tilde{R}$ into $C_1\gamma$) in \eqref{e:divB1}  such that $\beta_1$ satisfies
	\begin{equation}\label{e:bt1}
		\beta_1 +2k(k+1)C_0\gamma =(C_1  +2k(k+1)C_0)\gamma < \frac{\cm^2\beta}{50A(1+2\beta)}  ,
	\end{equation} 
	where $k\in \Zbb_{\frac{n}{2}+3}$ is a constant and $C_0>0$ is given by \eqref{e:ttb}. 
	This completes the verification of Condition \ref{c:7}.
	
	Now after verifying all the conditions \ref{c:2}--\ref{c:7}, we conclude that \eqref{e:mainsys3} is a Fuchsian system as described in Appendix \ref{s:fuc}, thereby completing the proof of Proposition  \ref{t:verfuc}.$(1)$. 
\end{proof}
\begin{proof}[The proof of Proposition \ref{t:verfuc}.$(2)$] 	
This can be proven by directly applying Theorem \ref{t:fuc},  as \eqref{e:mainsys5} is indeed a Fuchsian system. The constants $\kappa,\, \gamma_1$ and  $\beta_1$ meet the requirements of \eqref{e:kpbt1} due to the choice of $\beta_1$ as specified in \eqref{e:bt1} (using \eqref{e:ttb} to verify it satisfies \eqref{e:kpbt1}). This completes the proof. 
\end{proof}

%-------------------New Sec---------------------

\section{Proofs of the Main Theorems}\label{s:pfmthm}
We begin by focusing on the proof of Theorem \ref{t:mainthm2}. Once established, Theorem \ref{t:mainthm1} follows directly from Theorem \ref{t:mainthm2} by an exponential-logarithmic transformation, the details of which are omitted here for brevity.

\subsection{The proof of Theorem \ref{t:mainthm2}}
To begin, we select arbitrary constants  $\sigma_0\in(0,\sigma_\star)$ and $\delta_0\in(0,\delta_\star)$, where $\sigma_\star$ and $\delta_\star$ are some small, positive  constants. 
	
\subsubsection{Step $1$: Variable transformations.} Before delving into the proof, we first recall the complete set of variable transformations frequently utilized in the subsequent analysis. By applying the transformations defined in \eqref{e:v1}--\eqref{e:v5} and  \eqref{e:fv1}--\eqref{e:fv6},   as well as the coordinate transformations \eqref{e:coord6},  \eqref{e:coord5}--\eqref{e:coordi5} combined with \eqref{e:coord2} and \eqref{e:coordi2},  the definitions  \eqref{e:eta2} of $\mu$ and $\eta$, and the relations $	\mft=\mathsf{b}_\uparrow (\tau) =\mathsf{b}_\uparrow \circ g(t,x^i)$ and $f(\mft)=\uf(\tau) =\uf\circ\mfg(\mft)=  f \circ \mathsf{b}_\uparrow \circ g(t,x^i)$, we arrive at  the composition transformations, 
\begin{align}
	\hat{\mfu}_0(\htau,\hat{\zeta} ) = & 	 \mfu_0(\ttau,\txi ) =	\frac{e^{ \frac{153}{\delta_0} }   (- \tau)^{   \frac{100 }{A} }   e^{  51  \zeta^1   }}{  \sigma_0 } 
	\frac{ \underline{\varrho_0}(\tau,\zeta )  - \underline{f_0}(\tau) }{ \underline{f_0}(\tau) }   
	=  \frac{e^{  \frac{153}{\delta_0} }   (-g(t,x))^{  \frac{100 }{A} }   e^{  51 x^1   }}{   \sigma_0 } 
	\frac{ \varrho_0(t,x )  - f_0 \circ \mathsf{b}_\uparrow \circ g(t,x ) }{ f_0 \circ \mathsf{b}_\uparrow \circ g(t,x )}     
	, \label{e:fv1.b}  \\
	\hat{\mfu}_i(\htau,\hat{\zeta} )  =& \mfu_i(\ttau,\txi )   =  \frac{e^{  \frac{153}{\delta_0} }   (- \tau)^{  \frac{100 }{A} }   e^{ 51 \zeta^1   }}{  \sigma_0 }  \frac{ \underline{\varrho_i} (\tau,\zeta ) }{1+\underline{f}(\tau)}     
	=  \frac{e^{  \frac{153}{\delta_0} }   (- g(t,x))^{    \frac{100 }{A} }   e^{  51 x^1   }}{  \sigma_0 }   \frac{  \varrho_i  (t,x ) }{1+f \circ \mathsf{b}_\uparrow \circ g(t,x ) }      
	, \label{e:fv2.b}   \\
	\hat{\mfu}(\htau,\hat{\zeta} )   =& 		\mfu(\ttau,\txi )  = 	  \frac{\underline{\varrho}(\tau,\zeta )- \uf(\tau) }{\uf(\tau )}    
	=  	    \frac{ \varrho(t,x )- f \circ \mathsf{b}_\uparrow \circ g(t,x ) }{f \circ \mathsf{b}_\uparrow \circ g(t,x )}   ,\label{e:fv3.b}  \\
	\hat{\mfv}(\htau,\hat{\zeta} )   = 	& \mfv(\ttau,\txi)  =   \frac{e^{  \frac{153}{\delta_0} }   (- \tau)^{ \frac{100 }{A} }   e^{ 51 \zeta^1   }}{  \sigma_0 }  \frac{\underline{\varrho}(\tau,\zeta )- \uf(\tau) }{\uf(\tau )}  
	=    \frac{e^{  \frac{153}{\delta_0} }   (- g(t,x))^{ \frac{100 }{A} }   e^{ 51 x^1   }}{   \sigma_0 }      \frac{ \varrho(t,x )- f \circ \mathsf{b}_\uparrow \circ g(t,x ) }{f \circ \mathsf{b}_\uparrow \circ g(t,x )}  	, \label{e:fv6.b}      \\
	\hat{\mathfrak{B}}_j(\htau,\hat{\zeta} )  =  &	\mathfrak{B}_j(\ttau,\txi )  =    \frac{e^{  \frac{153}{\delta_0} }   (- \tau)^{ \frac{100 }{A} }   e^{ 51 \zeta^1   }}{  \sigma_0 } 
	\frac{B \ufo (\tau ) }{\underline{\chi_\uparrow}(\tau) \uf (\tau) }  
	\mathsf{b}_j(\tau,\zeta)   
	, \label{e:fv4.b} 
	\\
	\hat{\mathfrak{z}}(\htau,\hat{\zeta})  =&\mathfrak{z}(\ttau,\txi )  =   \frac{e^{   \frac{153}{ \delta_0} } (-\tau)^{\frac{100}{A}}e^{51\zeta^1}  }{  \sigma_0^2    }   \biggl[\biggl(\frac{\mathsf{b}(\tau,\zeta )}{\mathsf{b}_\uparrow(\tau)}\biggr)^{\frac{1}{3}}-1\biggr]    .     
  \label{e:fv5.b}    
\end{align}  
These transformations lead to the following expressions for the initial data (at $\htau=-1$, corresponding to $t=t_0$), 
\begin{align}
		\hat{\mfu}_0(-1,\hat{\zeta} ) = & 	 	\frac{e^{  \frac{153}{\delta_0} }     e^{ 51 \zeta^1   }}{   \sigma_0 } 
	\frac{ \underline{\varrho_0}(-1,\zeta )  - \underline{f_0}(-1) }{ \underline{f_0}(-1) }  
	=  \frac{e^{  \frac{153}{\delta_0} }     e^{51  x^1   }}{  \sigma_0 } 
	\frac{ \varrho_0(t_0,x )  - \beta_0}{ \beta_0}     
	, \label{e:infv1.b}  \\
	\hat{\mfu}_i(-1,\hat{\zeta})   =& 	\frac{e^{  \frac{153}{\delta_0} }    e^{  51  \zeta^1   }}{   \sigma_0 }  \frac{ \underline{\varrho_i} (-1,\zeta ) }{1+\underline{f}(-1)}   
	=  	\frac{e^{  \frac{153}{\delta_0} }    e^{ 51 x^1   }}{  \sigma_0 }   \frac{  \varrho_i  (t_0,x ) }{1+\beta}      
	, \label{e:infv2.b}   \\
	\hat{\mfu}(-1,\hat{\zeta} )  =&   \frac{\underline{\varrho}(-1,\zeta )- \uf(-1) }{\uf(-1 )}   
	=  	    \frac{ \varrho(t_0,x )-\beta}{\beta}   ,\label{e:infv3.b}  \\
	\hat{\mfv}(-1,\hat{\zeta})  = 	&    \frac{e^{  \frac{153}{\delta_0} }    e^{ 51\zeta^1   }}{   \sigma_0 }  \frac{\underline{\varrho}(-1,\zeta )- \uf(-1) }{\uf(-1 )}   
	=   \frac{e^{  \frac{153}{\delta_0} }   e^{  51 x^1   }}{   \sigma_0 }      \frac{ \varrho(t_0,x )- \beta }{\beta}  	, \label{e:infv6.b}      \\
	\hat{\mathfrak{B}}_j(-1,\hat{\zeta})  =  & \frac{e^{  \frac{153}{\delta_0} }    e^{  51 \zeta^1   }}{   \sigma_0 } 
	\frac{B \ufo (-1 ) }{\underline{\chi_\uparrow}(-1) \uf (-1) }  
	\mathsf{b}_j(-1,\zeta )    
	, \label{e:infv4.b} 
	\\
	\hat{\mathfrak{z}}(-1,\hat{\zeta})  =&  \frac{ e^{  \frac{153}{\delta_0} }   e^{51\zeta^1}   }{   \sigma_0^2       }   \biggl[\biggl(\frac{\mathsf{b}(-1,\zeta )}{\mathsf{b}_\uparrow(-1)}\biggr)^{\frac{1}{3}}-1\biggr]   
	.    \label{e:infv5.b}    
\end{align} 
 
Given the initial data in  \ref{Eq2} (recalling $	\varrho|_{t=t_0} =	\beta +  \psi(x)$ and  $\del{t}\varrho|_{t=t_0}=	 \beta_0  +  \psi_0(x)$), and using the expressions from \eqref{e:infv1.b}--\eqref{e:infv6.b}, we obtain
	\begin{equation}\label{e:dt.1} 
 \psi_0(x ) =\beta_0   \sigma_0 e^{- \frac{153}{\delta_0} }     e^{ - 51 x^1   } \hat{\mfu}_0(-1,\hat{\zeta} ), \quad   \psi(x )  =	\beta\hat{\mfu}(-1,\hat{\zeta} ) =  \beta \sigma_0 e^{- \frac{153}{\delta_0} }   e^{ -  51 x^1   }	\hat{\mfv}(-1,\hat{\zeta} ) 
	\end{equation}
	and 
	\begin{equation}\label{e:dt.2} 
	  \del{i}\psi(x ) =  (1+\beta) \sigma_0 e^{- \frac{153}{\delta_0} }    e^{ - 51  x^1   } \hat{\mfu}_i(-1,\hat{\zeta} )   
	\end{equation}  
	
	Next, note that on the initial hypersurface $t=t_0$ (which corresponds to  $\tau=-1$), the coordinate relation $x^i= \gamma^{-1} \tan  \hat{\zeta}^i $ holds. For any function  $F(x^i)=F(\gamma^{-1} \tan  \hat{\zeta}^i)=\hat{F}(\hat{\zeta}^i)$ with a compact domain, and using the definition of Sobolev norms \S\ref{s:funsp}, along with the fact that $D_{x^i}=\gamma \cos^2(\hat{\zeta}^i)D_{\hat{\zeta}^i}$, we obtain
	\begin{equation*}
		\|F\|_{H^k(B_1(0))}^2=\sum_{0\leq \alpha\leq k} \int_{B_1(0)}(D_x^\alpha F)^2 d^n x=\sum_{0\leq \alpha\leq k} \int_{B_r(0)}\bigl((\gamma \cos^2(\hat{\zeta}^i)D_{\hat{\zeta}^i})^\alpha F\bigr)^2  \gamma^{-n} \prod_{i=1}^n \sec^2\hat{\zeta}^i d^n \hat{\zeta}\sim \|\hat{F}\|_{H^k(B_{r}(0))}^2
	\end{equation*}
	where $r=\arctan \gamma\in(-\frac{\pi}{4},\frac{\pi}{4})$ and $A\sim B$ means there exist constants $C_1,C_2>0$ such that $C_1 A<B<C_2A$. The notation $\sim$ holds because the derivatives applied to $\cos^2(\hat{\zeta}^i)$ result in  $\cos^2(\hat{\zeta}^i)$, $\sin^2(\hat{\zeta}^i)$ or $\cos (\hat{\zeta}^i)\sin (\hat{\zeta}^i)$, all of which are bounded.  In addition, $\sec^2(\hat{\zeta}^i)\in[1,2)$ within $B_r(0)$. Consequently, using these facts, we can conclude that  $\|F\|_{H^k(B_0(0))}\sim\|\hat{F}\|_{H^k(B_r(0))}$ by straightforward calculations.

Based on the assumptions of Theorem \ref{t:mainthm2} and using \eqref{e:dt.1}--\eqref{e:dt.2},  the inequality of the data $\|  \psi \|_{H^k(B_1(0))} +	\|\del{i} \psi \|_{H^{k}(B_1(0))} +  \|  \psi_0\|_{H^k(B_1(0))} \leq  e^{  - \frac{153}{\delta_0} }    \sigma_0^2 $ leads to the following bounds
	\begin{align}
		\|\hat{\mfu}_0(-1)\|_{H^k(B_r(0))} \leq &\frac{C \| 
			e^{  51x^1   }\psi_0 \|_{H^k(B_1(0))}}{  \beta_0 \sigma_0 e^{- \frac{153}{\delta_0} }    }  \leq\frac{C \| 
		 \psi_0 \|_{H^k(B_1(0))}}{  \beta_0 \sigma_0 e^{- \frac{153}{\delta_0} }    }    \leq \frac{ C }{      \beta_0 }    \sigma_0   ,  \label{e:uin1} \\
		\|\hat{\mfu}_i(-1)\|_{H^k(B_r(0))} \leq 	& \frac{ C \|  e^{   51 x^1   } \del{i}\psi \|_{H^k(B_1(0))}}{   \sigma_0 e^{- \frac{153}{\delta_0} }  (1+\beta) }   
		\leq \frac{C}{     1+\beta }     \sigma_0    , \label{e:uin2} \\
		\|\hat{\mfu}(-1)\|_{H^k(B_r(0))} \leq &  \frac{C \|\psi\|_{H^k(B_1(0))}}{\beta}  \leq  \frac{C e^{  - \frac{153}{\delta_0} }    }{\beta} \sigma_0^2  \leq  \frac{C   }{\beta} \sigma_0^2    ,  \label{e:uin3} \\
		\|\hat{\mfv}(-1)\|_{H^k(B_r(0))} \leq &   \frac{  C\|    e^{ 51 x^1   }  \psi  \|_{H^k(B_1(0))} }{   \sigma_0 e^{- \frac{153}{\delta_0} } \beta}      \leq \frac{ 	C }{ \beta}     \sigma_0  .  \label{e:uin4}
	\end{align}
	In addition, by Lemmas  \ref{t:gb1}.\ref{l:1.2} and   \ref{t:gb2}.\ref{l:2.3}, we have $\mathsf{b}(-1,\zeta )= t_0$ and $\mathsf{b}_\uparrow(-1)=t_0$, which implies  $\mathsf{b}_j(-1,\zeta )= 0$ and $\bigl(\frac{\mathsf{b}(-1, \zeta)}{\mathsf{b}_\uparrow(-1)}\bigr)^{\frac{1}{3}}-1=0$. Using \eqref{e:infv4.b} and \eqref{e:infv5.b}, it follows that $ \hat{\mathfrak{B}}_j(-1,\hat{\zeta}) =  	\hat{\mathfrak{z}}(-1,\zeta )=0 $. Therefore, we obtain \begin{equation}\label{e:Bzin}
		\|\hat{\mathfrak{B}}_j(-1)\|_{H^k(B_r(0))}= \|	\hat{\mathfrak{z}}(-1 ) \|_{H^k(B_r(0))}=0  . 
	\end{equation}
	By taking sufficiently small constants $\sigma_\star>0$, such that $4C\sigma_\star<\sigma$ for $C$ (where $C$ is the maximum of the coefficients in \eqref{e:uin1}--\eqref{e:uin4}). 
Then since $\sigma_0<\sigma_\star$,  \eqref{e:uin1}--\eqref{e:Bzin} imply the initial data of $\widehat{\mathfrak{U}}$ satisfies
\begin{equation}\label{e:dt3}
	\|\widehat{\mathfrak{U}}_0\|_{C^3} \leq\|\widehat{\mathfrak{U}}_0\|_{H^k(B_r(0))} \leq \sigma .
\end{equation}

\subsubsection{Step $2$: Estimates of extended Fuchsian and original variables.}
By using Lemma \ref{t:mainsys1}, Lemma \ref{t:mainsys2} and Lemma \ref{t:sigsys}, 
we rewrite the main equation \ref{Eq2}  into the singular system  \eqref{e:mainsys3}. By introducing a cutoff function $\phi$, we handle the large values of $\mu$ and $\eta$  in the singular equation \eqref{e:mainsys3},  replacing $\mu$ and $\eta$ with $\phi\mu$ and $\phi\eta$ respectively, leading to the revised system \eqref{e:mainsys5}.  To bring the revised system \eqref{e:mainsys5} into the Fuchsian form described in Appendix \ref{s:fuc},  we compactify the spatial domain into a torus using the coordinate transformations \eqref{e:coord6b} and \eqref{e:coordi6}. Therefore,  Proposition \ref{t:verfuc} confirms that the system takes the Fuchsian form  \eqref{e:mainsys6} (parameterized by $\delta_0$). On the extended spacetime domain $[-1,0)\times \Tbb^n_{[-\frac{\pi}{2},\frac{\pi}{2}]}$, there exists a constant $  \sigma=\sigma(\delta_0)>0$,  such that if $
	\|\widehat{\mathfrak{U}}_0\|_{H^k} \leq \sigma$, 
then there exists a unique solution $\widehat{\mathfrak{U}}$ satisfying the regularity conditions specified in Proposition \ref{t:verfuc} and the energy estimate $\|\widehat{\mathfrak{U}}(\htau)\|_{H^k(\Tbb^n_{[-\frac{\pi}{2},\frac{\pi}{2}]})}  \leq C \|\widehat{\mathfrak{U}}_0 \|_{H^k(\Tbb^n_{[-\frac{\pi}{2},\frac{\pi}{2}]})} $.

By taking  $\sigma_\star$ sufficiently small, there is a constant $\sigma>0$, such that 	$\|\widehat{\mathfrak{U}}_0\|_{H^k} \leq \sigma$ as shown in \eqref{e:dt3}. Therefore, the unique solution $\widehat{\mathfrak{U}}$ with the regularity \eqref{e:solreg} exists in the extended spacetime domain $[-1,0)\times \Tbb^n_{[-\frac{\pi}{2},\frac{\pi}{2}]}$ and for $-1\leq \hat{\tau}<0$, we have
\begin{equation}\label{e:est1}
	\|\widehat{\mathfrak{U}}(\htau)\|_{C^3(\Tbb^n_{[-\frac{\pi}{2},\frac{\pi}{2}]})}  \leq C	\|\widehat{\mathfrak{U}}(\htau)\|_{H^k(\Tbb^n_{[-\frac{\pi}{2},\frac{\pi}{2}]})}  \leq C \|\widehat{\mathfrak{U}}_0 \|_{H^k(\Tbb^n_{[-\frac{\pi}{2},\frac{\pi}{2}]})} \leq C\sigma. 
\end{equation}

Applying  Lemma \ref{t:extori2}, 
we obtain there is a  unique solution 	$\widehat{\mathfrak{U}}:=(\hat{\mfu}_0,\hat{\mfu}_j,\hat{\mfu},\hat{\mathfrak{B}}_l, \hat{\mathfrak{z}},\hat{\mfv})^T$ 
to the original system \eqref{e:mainsys3}  with parameters \eqref{e:para} on the domain $\widehat{\mathit{\Lambda}}_{\delta_0}$. 
Using the Sobolev embedding theorem, the estimate  \eqref{e:est1} can be expressed as
\begin{equation}\label{e:est2}
	\|(\hat{\mfu}_0,\hat{\mfu}_j,\hat{\mfu},\hat{\mathfrak{B}}_l, \hat{\mathfrak{z}},\hat{\mfv})\|_{C^3(\mathit{\Lambda_{\delta_0}})}=\|\widehat{\mathfrak{U}}(\ttau)\|_{C^3(\mathit{\Lambda_{\delta_0}})}  \leq C\sigma . 
\end{equation}

Before proceeding, we control the decay factor by a simple exponential decay in $\zeta$. 
\begin{lemma}\label{t:dcy}
		If $(\tau,\zeta) \in\underline{\mathit{\Lambda}_{\delta_0}}$, then the decay factor 
	\begin{equation*}
		e^{ -  \frac{153}{ \delta_0} }(-g(t,x))^{-\frac{100}{A}}e^{-51 x^1}=	e^{ -  \frac{153}{ \delta_0} }(-\tau)^{-\frac{100}{A}}e^{-51\zeta^1}<e^{  - \frac{50}{ \delta_0}} e^{ -\frac{\zeta^1}{2} }  . 
	\end{equation*} 
\end{lemma}
\begin{proof}
	If $(\tau,\zeta) \in \underline{\mathit{\Lambda}_{\delta_0}}$, noting $(-\tau)^{-\frac{100}{A}}$ achieves the upper bound on the hypersurface $\hat{\Gamma}_{\delta_0}$ for any $\zeta$. Using the results from Lemma \ref{t:Tsurf3} and noting that $\sqrt{\Diamond_r(\zeta^1)}>0
	$ and $\sqrt{\Diamond_l(\zeta^1)}>0
	$, 
we directly obtain the desired estimate. 
\end{proof}

Now let us transform the variables back to $\mathit{\Lambda}_{\delta_0}$. 
By using \eqref{e:fv1.b}--\eqref{e:fv5.b}, with the help of Lemma \ref{t:dcy}, the estimate \eqref{e:est2} becomes
\begin{align}
(1-C \sigma_0^2 e^{  - \frac{50}{ \delta_0}} e^{ -\frac{x^1}{2} }) f_0\circ\mathsf{b}_\uparrow \circ g  \leq 	\varrho_0 & \leq (1+C \sigma_0^2 e^{  - \frac{50}{ \delta_0}} e^{ -\frac{x^1}{2} }) f_0\circ\mathsf{b}_\uparrow \circ g , \\
 -C \sigma_0^2 e^{  - \frac{50}{ \delta_0}} e^{ -\frac{x^1}{2} }  (1+f \circ\mathsf{b}_\uparrow \circ g ) \leq 	\varrho_i & \leq  C \sigma_0^2 e^{  - \frac{50}{ \delta_0}} e^{ -\frac{x^1}{2} }  (1+f \circ\mathsf{b}_\uparrow \circ g)  , \\
 (1-C \sigma_0 ) f \circ\mathsf{b}_\uparrow \circ g  \leq 	\varrho & \leq (1+C \sigma_0 ) f \circ\mathsf{b}_\uparrow \circ g , \label{e:rhoreg0}   \\
(1-C \sigma_0^2 e^{  - \frac{50}{ \delta_0}} e^{ -\frac{x^1}{2} }) f \circ\mathsf{b}_\uparrow \circ g  \leq 	\varrho & \leq (1+C \sigma_0^2 e^{  - \frac{50}{ \delta_0}} e^{ -\frac{x^1}{2} }) f \circ\mathsf{b}_\uparrow \circ g  , \\
-\frac{\chi_\uparrow}{B} \frac{f}{f_0}C\sigma_0^2  e^{  - \frac{50}{ \delta_0}} e^{ -\frac{x^1}{2} }\leq \mathsf{b}_j & \leq  \frac{\chi_\uparrow}{B} \frac{f}{f_0}C\sigma_0^2  e^{  - \frac{50}{ \delta_0}} e^{ -\frac{x^1}{2} } , \\
(1-C\sigma_0^3e^{  - \frac{50}{ \delta_0}} e^{ -\frac{x^1}{2} })^3 \mathsf{b}_\uparrow \leq \mathsf{b} & \leq (1+C\sigma_0^3e^{  - \frac{50}{ \delta_0}} e^{ -\frac{x^1}{2} })^3 \mathsf{b}_\uparrow \label{e:breg0}
\end{align}
for $(t,x)\in\mathit{\Lambda}_{\delta_0}$.   

In addition, by Lemma \ref{t:Tsurf5} and the estimates from \eqref{e:breg0}, we can further estimate the  hypersurface (recalling $\tau_\Gamma$ is defined in  Lemma \ref{t:Tsurf3})
\begin{align}\label{e:lidest}
	&(1-C\sigma_0^3e^{  - \frac{50}{ \delta_0}} e^{ -\frac{x^1}{2} })^3 \mathsf{b}_\uparrow \left(\tau_\Gamma(x;\delta_0)\right)  \leq  \mathsf{b}  \left(\tau_\Gamma(x;\delta_0) ,x\right)   \leq (1+C\sigma_0^3e^{  - \frac{50}{ \delta_0}} e^{ -\frac{x^1}{2} })^3 \mathsf{b}_\uparrow\left(\tau_\Gamma(x;\delta_0)  \right)   . 
\end{align}
Note
\begin{gather*}
	\lim_{x^1\rightarrow +\infty} \tau_\Gamma (x;\delta_0)=0 \AND \lim_{x^1\rightarrow +\infty} e^{-\frac{x^1}{2}}=0; \\ 
	\lim_{\delta_0\rightarrow 0+} \tau_\Gamma (x;\delta_0) =0 \AND \lim_{\delta_0\rightarrow 0+} e^{-\frac{100}{\delta_0}}=0 . 
\end{gather*}
Taking the limit as $x^1\rightarrow +\infty$ in \eqref{e:lidest},  and recalling $\mathfrak{g}(t_m)=0$ in \eqref{e:ctm2} and Lemma \ref{t:gb2}.\ref{l:2.3}, we obtain 
\begin{equation}
	\lim_{a\rightarrow +\infty} \mathfrak{T}(a\delta^i_1,\delta_0)=\mathsf{b}_\uparrow (0)=t_m \AND \lim_{\delta_0\rightarrow 0+} \mathfrak{T}(x,\delta_0)=\mathsf{b}_\uparrow (0)=t_m. \label{e:limT}
\end{equation}

\subsubsection{Step $3$: Proofs of Theorem \ref{t:mainthm2}.\ref{b1}, \ref{b4}, \ref{b2} and \ref{b3}.} In order to prove Theorem \ref{t:mainthm2}.\ref{b1}, we firstly give a lemma on the estimates of $\mathsf{b}_\uparrow\circ g$. 
\begin{lemma}\label{t:bgest}
	The composition  $\mathsf{b}_\uparrow \circ g (t,x)$ satisfies the estimate
	\begin{equation*}
		t_0+(1-C \sigma_0^2 e^{  - \frac{50}{ \delta_0}} e^{ -\frac{x^1}{2} }) (t-t_0) \leq \mathsf{b}_\uparrow \circ g (t,x)\leq 		t_0+(1+C\sigma_0^2 e^{  - \frac{50}{ \delta_0}} e^{ -\frac{x^1}{2} } ) (t-t_0) 
	\end{equation*}
for $(t,x)\in\mathit{\Lambda}_{\delta_0}$.    
\end{lemma}
\begin{proof}
	By using \eqref{e:tmeq1} and \eqref{e:tmeqi1}, along with \eqref{e:dtb} and \eqref{e:fv3.b}, we derive
	\begin{equation}\label{e:dtbg1}
		\underline{	\del{t}(\mathsf{b}_\uparrow \circ g) } =(\del{\tau}\mathsf{b}_\uparrow) \underline{ \del{t} g  }=  \frac{\del{\tau}\mathsf{b}_\uparrow}{\del{\tau}\mathsf{b}}  =    \frac{\mathsf{b}_\uparrow^{\frac{2}{3}}  (1+\underline{f} )^{\frac{1}{3}}  \underline{\varrho} }{ \mathsf{b}^{\frac{2}{3}}   ( 1 + \underline{\varrho} )^{\frac{1}{3}}  \underline{f}    }  =    \biggl(\frac{\mathsf{b}_\uparrow }{\mathsf{b}}\biggr)^{\frac{2}{3}}  \frac{    (1 +   \hat{\mfv} \sigma_0 e^{ -  \frac{153}{ \delta_0} }(-\tau)^{-\frac{100}{A}}e^{-51\zeta^1} ) }{ \bigl( 1  + \frac{\uf}{1+\uf}  \hat{\mfv} \sigma_0 e^{ -  \frac{153}{ \delta_0} }(-\tau)^{-\frac{100}{A}}e^{-51\zeta^1} \bigr)^{\frac{1}{3}}    }    . 
	\end{equation}
	By using \eqref{e:breg0}, we obtain
	\begin{equation}\label{e:bbest2}
	  \bigl(1+ C\sigma_0^3e^{  - \frac{50}{ \delta_0}} e^{ -\frac{x^1}{2} } \bigr)^{-2}   \leq \biggl(\frac{\mathsf{b}_\uparrow }{\mathsf{b}}\biggr)^{\frac{2}{3}}   \leq \bigl(1 - C\sigma_0^3e^{  - \frac{50}{ \delta_0}} e^{ -\frac{x^1}{2} } \bigr)^{-2}    .
	\end{equation} 
	Then by \eqref{e:dtbg1} and \eqref{e:bbest2},  there is a constant $C>0$, such that 
	\begin{equation*}
			1-C \sigma_0^2 e^{  - \frac{50}{ \delta_0}} e^{ -\frac{x^1}{2} } \leq 	\del{t}(\mathsf{b}_\uparrow \circ g ) \leq 1+C\sigma_0^2 e^{  - \frac{50}{ \delta_0}} e^{ -\frac{x^1}{2} } . 
	\end{equation*}
Integrating this inequality proves the lemma.  
\end{proof}

Now, combining  \eqref{e:rhoreg0} with Lemma \ref{t:bgest}, and noting that $f$ is an increasing function by Lemma \ref{t:f0fg},     we establish
\begin{align}
	\mathbf{1}_{-}(x^1) f_0\bigl(	t_0+\mathbf{1}_{-}(x^1) (t-t_0)\bigr)  \leq 	\varrho_0(t,x) & \leq \mathbf{1}_{+}(x^1) f_0\bigl(	t_0+\mathbf{1}_{+}(x^1) (t-t_0) \bigr) ,\\
	-C \sigma_0^2 e^{  - \frac{50}{ \delta_0}} e^{ -\frac{x^1}{2} }  (1+f \bigl(	t_0+\mathbf{1}_{-}(x^1) (t-t_0)\bigr)  ) \leq 	\varrho_i & \leq  C \sigma_0^2 e^{  - \frac{50}{ \delta_0}} e^{ -\frac{x^1}{2} }  \bigl(1+f \bigl(	t_0+\mathbf{1}_{+}(x^1) (t-t_0) \bigr)\bigr)  , \\
	\mathbf{1}_{-}(x^1) f \bigl(	t_0+\mathbf{1}_{-}(x^1) (t-t_0)\bigr)  \leq 	\varrho & \leq \mathbf{1}_{+}(x^1) f \bigl(	t_0+\mathbf{1}_{+}(x^1) (t-t_0) \bigr)  ,  \\
	-\frac{\chi_\uparrow}{B} \frac{f}{f_0}C\sigma_0^2  e^{  - \frac{50}{ \delta_0}} e^{ -\frac{x^1}{2} }\leq \mathsf{b}_j & \leq  \frac{\chi_\uparrow}{B} \frac{f}{f_0}C\sigma_0^2  e^{  - \frac{50}{ \delta_0}} e^{ -\frac{x^1}{2} }  , \\
	(1-C\sigma_0^3e^{  - \frac{50}{ \delta_0}} e^{ -\frac{x^1}{2} })^3 \mathsf{b}_\uparrow \leq \mathsf{b} & \leq (1+C\sigma_0^3e^{  - \frac{50}{ \delta_0}} e^{ -\frac{x^1}{2} })^3 \mathsf{b}_\uparrow \label{e:breg0a}
\end{align}
for $(t,x)\in\mathit{\Lambda}_{\delta_0}$ where 
\begin{equation*}
	\mathbf{1}_{-}(x^1):=1-C \sigma_0^2 e^{  - \frac{50}{ \delta_0}} e^{ -\frac{x^1}{2} } \AND	\mathbf{1}_{+}(x^1):=1+C \sigma_0^2 e^{  - \frac{50}{ \delta_0}} e^{ -\frac{x^1}{2} }  . 
\end{equation*}  
This completes the proof of \ref{b1} and the Corollary \ref{t:homdom} yields \ref{b4}. With the help of Theorem \ref{t:mainthm0}, we arrive at the estimates for \ref{b2} and \ref{b3}. 
Moreover, \ref{b1} together with \eqref{e:limT} (note $x^1\geq -\frac{1}{\delta_0}$) implies \ref{b5}.

\subsubsection{Step $4$: Proofs of $\varrho\in C^2(\mathcal{K}\cup \mathcal{H})$.} 
We prove that $\varrho\in C^2(\mathcal{K}\cup \mathcal{H})$ by following these steps. 

\underline{$(a)$ $\mathsf{b}_\uparrow\in C^2([-1,0))$.} First, we can directly verify that $\mfg\in C^2([t_0,t_m))$ by using \eqref{e:ctm2} and Lemma \ref{t:gb2},   given that  $f\in C^2([t_0,t_m))$. The inverse mapping theorem then implies that $\mathsf{b}_\uparrow\in C^2([-1,0))$.

\underline{$(b)$ $g\in C^2(\mathcal{K})$.} We begin by noting that \ref{t:verfuc} and Lemma \ref{t:extori2} provide us with the regularity for
$\widehat{\mathfrak{U}} $ which belongs to $  C^0([-1,0),H^k(\Tbb^n_{[-\frac{\pi}{2},\frac{\pi}{2}]}) ) \cap C^1([-1,0),H^{k-1}(\Tbb^n_{[-\frac{\pi}{2},\frac{\pi}{2}]}) ) \cap \Li ([-1,0),H^k(\Tbb^n_{[-\frac{\pi}{2},\frac{\pi}{2}]}))$. Utilizing the Sobolev embedding theorem, we further conclude that $\widehat{\mathfrak{U} } \in C^1([-1,0)\times \Tbb^n_{[-\frac{\pi}{2},\frac{\pi}{2}]})$. By applying the coordinate transformation  \eqref{e:coord6b}, this result extends to  $\mathfrak{U}  \in C^1([-1,0)\times \Rbb^n)$.  

Next, recall the coordinate transformations  
\begin{equation}\label{e:coord6}
	\ttau=\tau=g(t,x^i)    \AND  \txi^i=  \frac{100\delta^i_1 }{51A} \ln(- \tau)+  \zeta^i =\frac{100\delta^i_1 }{51A} \ln(- g(t,x^i))+  x^i, 
\end{equation}
along with the variable transformations in \eqref{e:fv1.b}--\eqref{e:fv5.b}. Utilizing the results above and recalling that  $\uf(\tau)= f\circ \mathsf{b}_\uparrow(\tau)$, 
in terms of coordinate $(\tau,\zeta^k)$,  $\underline{\varrho_0}$, $\underline{\varrho_i} $, $\underline{\varrho} $, $\mathsf{b}_j$ and $\mathsf{b}$ are all in $C^1(\underline{\mathcal{K}})$. In addition, by differentiating \eqref{e:tmeqi1} with respective to $\tau$, and substituting  $\del{\tau}\mathsf{b}$ and $\del{\tau} \underline{\varrho}$ using \eqref{e:tmeqi1} and \eqref{e:dtrho}, we arrive at
\begin{equation*}
	\partial_{\tau}^2 \mathsf{b} =   \frac{2 \mathsf{b}^{1/3}    (\underline{\varrho}  +1)^{2/3}}{3 A^2 B^2    \underline{\varrho}^2 \left(-\tau\right){}^{\frac{4}{3 A}+2} }   -\frac{  \mathsf{b}^{4/3}   (\underline{\varrho}  +1)^{2/3} }{A^2 B^2 \underline{\varrho}^3  \left(-\tau\right){}^{\frac{4}{3 A}+2}}   \underline{\varrho_0}     +\frac{  \mathsf{b}^{4/3} \left(-\tau\right){}^{-\frac{4}{3 A}-2} }{3 A^2 B^2 \underline{\varrho}^2   (\underline{\varrho}  +1)^{1/3}   }   \underline{\varrho_0}   + \frac{\left(\frac{2}{3 A}+1\right)  \mathsf{b}^{2/3} (\underline{\varrho}  +1)^{1/3}}{A B (-\tau )^{ \frac{2}{3 A}+2} \underline{\varrho}  }. 
\end{equation*}
Consequently, we have  $	\partial_{\tau}^2 \mathsf{b}  \in C^1(\underline{\mathcal{K}})$.  Since $\mathsf{b}_j \in C^1(\underline{\mathcal{K}})$, it follows that $\partial_{\zeta^i}\partial_{\zeta^j} \mathsf{b}  \in C^0(\underline{\mathcal{K}})$ and  $\partial_{\tau} \partial_{\zeta^j}  \mathsf{b}  \in C^0(\underline{\mathcal{K}})$, respectively. Thus, we conclude that  $\mathsf{b}\in C^2(\underline{\mathcal{K}})$. 

By applying   \eqref{e:tmeqi1},  \eqref{e:Joc} and \eqref{e:breg0a}, and noting $\mathsf{b}_\uparrow(\tau)\geq\mft_0>0$ for $\tau\in[-1,0)$, we can conclude that the Jacobian determinant of the coordinate transformation \eqref{e:coordi2} is nonzero in $\underline{\mathcal{K}}$. Thus by the \textit{inverse mapping theorem}, with the help of $\mathsf{b}\in C^2(\underline{\mathcal{K}})$, we reach the coordinate system defined by  \eqref{e:coord2} exists and that $g\in C^2(\mathcal{K})$.

\underline{$(c)$ $\mfu\in C^2(\tilde{\mathcal{K}})$.} Next, we also need to verify $\mfu\in C^2(\tilde{\mathcal{K}})$. Specifically, it remains to show that  $ \del{\ttau}\mfu,\; \del{\txi^i}\mfu \in C^1(\tilde{\mathcal{K}})$. To proceed, we begin by differentiating  $\tilde{\underline{\varrho}}=\tilde{\uf}(1+\mfu)$ (from \eqref{e:fv3.b}) with respect to $\ttau$. This yields $\del{\ttau}\tilde{\underline{\varrho}}=\del{\ttau}\tilde{\uf}(1+\mfu)+\tilde{\uf}\del{\ttau}\mfu$. By applying \eqref{e:transf1} and \eqref{e:transf2}, we obtain $\widetilde{\del{\tau}\underline{\varrho}}-\frac{100\delta^i_1}{51 A\tau}\widetilde{\del{\zeta^i}\underline{\varrho}}=\widetilde{\del{\tau}\uf}(1+\mfu)+\tilde{\uf}\del{\ttau}\mfu$. Applying \eqref{e:dzrho0}, \eqref{e:dtrho} and Lemma \ref{t:gb2}.\ref{l:2.5}, and recalling the coordinate transformation \eqref{e:coordi5}, we derive
\begin{equation}\label{e:dtu6}
\del{\ttau}\mfu = \frac{1}{ \tilde{\uf}} \biggl( 	 \frac{\tilde{\mathsf{b}}^{\frac{2}{3}}   ( \tilde{\underline{\varrho}} +1)^{\frac{1}{3}}}{A B  \tilde{ \underline{\varrho} } \left(-\ttau\right){}^{\frac{2}{3 A}+1}} \widetilde{\underline{\varrho_0}  }-\frac{100\delta^i_1}{51A\ttau}  (\widetilde{\underline{\varrho_i}} +\widetilde{\mathsf{b}_{i} }  \widetilde{\underline{\varrho_0}   }) - \frac{\tilde{\mathsf{b}}_\uparrow^{\frac{2}{3}} (1+ \tilde{\uf} )^{\frac{1}{3}}}{A B \tilde{ \uf  } (-\ttau)^{\frac{2}{3A}+1}}	\tilde{\underline{f_0}   }(1+\mfu) \biggr) \in C^1(\tilde{\mathcal{K}}) . 
\end{equation}
A similar calculation applies to $\del{\txi^i}\mfu$. In fact, differentiating $\tilde{\underline{\varrho}}=\tilde{\uf}(1+\mfu)$   with respect to $\txi^i$ yields $\del{\txi^i}\tilde{\underline{\varrho}}= \tilde{\uf}\del{\txi^i}\mfu$. 
By \eqref{e:transf2}, we obtain $ \widetilde{\del{\zeta^i}\underline{\varrho}}= \tilde{\uf}\del{\txi^i}\mfu$.  Applying \eqref{e:dzrho0} gives
\begin{equation}\label{e:dzu6}
\del{\txi^i}\mfu  = \frac{1}{	\tilde{\uf}} \bigl(  \widetilde{\underline{\varrho_i}} +	\widetilde{\mathsf{b}_{i}} 	\widetilde{ \underline{\varrho_0} } \bigr)  \in C^1(\tilde{\mathcal{K}}) . 
\end{equation}
From \eqref{e:dtu6} and \eqref{e:dzu6}, we then conclude that $\mfu\in C^2(\tilde{\mathcal{K}})$.

Further, noting \eqref{e:fv3.b} and the coordinate transformation \eqref{e:coord6}, we have
\begin{equation*}
	\varrho(t,x^k)=f \circ \mathsf{b}_\uparrow \circ g(t,x^k)\Biggl(1+\mfu\biggl(g(t,x^i)   ,\frac{100\delta^i_1 }{51A} \ln(- g(t,x^i))+  x^i \biggr)\Biggr) \in C^2(\mathcal{K} ) . 
\end{equation*}
Utilizing the regularity established in parts $(a), \; (b)$ and $(c)$, along with $\varrho\equiv f\in C^2([t_0,t_m))$ for  $(t,x) \in \mathcal{H}$, we conclude $\varrho\in C^2(\mathcal{K}\cup \mathcal{H})$.  
This completes the proof of Theorem \ref{t:mainthm2}.

%-----------------------------------------------------------------New Section------------------------------------------------------------------------------------

\appendix

\section{Complementary lemmas and proofs}\label{s:pfclm} 
\subsection{Proofs of Claim \ref{t:clm2}}
Let us provide the proof of Claim \ref{t:clm2}.    
\begin{proof}[The proof of Claim \ref{t:clm2}]
	Let $s=\ln t$, so that  $\del{s}=t\del{t}$.  For convenience, we slightly abuse the notation $\varrho$, while $\varrho(t,r)$ and $\varrho(s,r)$ are clear from the context. By multiplying \ref{Eq2} by $t^2$ on both sides, we derive
	\begin{align*}
		&t^2 \partial^2_t \varrho-  	\biggl(\cm^2 \frac{ 	t^2 (\del{t}\varrho )^2}{(1+\varrho )^2}  + 4 (1-\cm^2) (1+\varrho  ) \biggr)  \delta^{ij} \partial_i \partial_j \varrho +\frac{4}{3  } t \del{t}\varrho  -
		\frac{2}{3 }  \varrho  (1+  \varrho )\notag  \\
		& \hspace{2cm} -\frac{4}{3} \frac{	t^2 (\del{t}\varrho )^2}{1+\varrho }   -  \biggl(\cm^2 \frac{ 	t^2 (\del{t}\varrho )^2}{(1+\varrho )^2}  + 4 (1-\cm^2) (1+\varrho  ) \biggr)  q^i \del{i}\varrho  +   \mathtt{K}^{ij}\del{i}\varrho\del{j}\varrho  =  0 . 
	\end{align*}
Next, using  $\del{s}=t\del{t}$ and noting that $t^2 \partial^2_t \varrho=t\del{t}(t\del{t}\varrho)-t\del{t} \varrho=   \partial_{s}^2 \varrho- \del{s} \varrho$, we obtain
	\begin{align*}
		& \partial_{s}^2 \varrho -  	\biggl(\cm^2 \frac{   (\del{s}\varrho )^2}{(1+\varrho )^2}  +4(1-\cm^2) (1+\varrho  ) \biggr)  \delta^{ij} \partial_i\partial_j  \varrho +\frac{1}{3  }   \del{s}\varrho  -
		\frac{2}{3 }  \varrho  (1+  \varrho ) -\frac{4}{3} \frac{  (\del{s}\varrho )^2}{1+\varrho }  \notag  \\
		& \hspace{2cm} - \biggl(\cm^2 \frac{ 	  (\del{s}\varrho )^2}{(1+\varrho )^2}  + 4 (1-\cm^2) (1+\varrho  ) \biggr)  q^i\del{i}\varrho + \mathtt{K}^{ij}\del{i}\varrho\del{j}\varrho   =  0  , 
	\end{align*}
which is consistent with \ref{Eq1}. Similarly, by reversing these steps, one can derive \ref{Eq2} from \ref{Eq1}. This completes the proof. 
\end{proof}

\subsection{Estimates on useful quantities}
\begin{lemma}\label{t:Thpst}
	Suppose the initial data satisfies Assumption \ref{Asp1} in \S\ref{s:mainthm}, and let $S(\tau)$,  $\underline{\chi}(\tau)$ and $\underline{\mathfrak{G}}(\tau)$ be defined by \eqref{e:S1},  \eqref{e:Gdef2} and \eqref{e:chig}, respectively. Then, for $\tau\in[-1,0]$, 
	\begin{enumerate}
		\item\label{r:1} 	$
			\underline{\mathfrak{G}}(\tau) \geq 0$ and $S (\tau)\leq \ck+  \frac{\ck+\cm^2}{\uf(\tau)}$; 
		\item\label{r:2} If $\ck=\frac{1}{4}$,  then the following identity holds:	\begin{equation*}
			\sqrt{S} \frac{ \underline{\chi_\uparrow}}{B}=2\sqrt{\frac{\underline{\chi_\uparrow}}{B}\left(\frac{\cm^2}{4}\frac{\underline{\chi_\uparrow}}{B}+\left(\frac{1}{4}-\cm^2\right)\frac{1+\uf}{\uf}\right)} =2+\Xi . 
		\end{equation*}
		where $\Xi(\tau)$ is a decreasing function. 
	\end{enumerate}

\end{lemma}
\begin{proof} 
	By Assumption \ref{Asp1}, the identity \eqref{e:keyid3} gives us $
	  \underline{\mathfrak{G}}(-1) =   \underline{\chi_\uparrow}(-1)  -4 B =\frac{t_0^{2}\beta_0^2}{(1+\beta)^{2}\beta} B-4 B\geq 0 $. 
Using Proposition \ref{t:limG} (which states $\lim_{\tau\rightarrow 0}\underline{\mathfrak{G}}(\tau)=0$) and Corollary \ref{t:dtchi2} (which asserts that $|\underline{\mathfrak{G}}|$ is decreasing), we conclude, by the continuity of $  \underline{\mathfrak{G}}(\tau)$ and a proof by contradiction, that $  \underline{\mathfrak{G}}(\tau) \geq 0$ for all $\tau\in[-1,0]$. Additionally, from \eqref{e:S1} and the non-negativity of $	 \underline{\mathfrak{G}}( \tau) $,  we derive the estimate for $S(\tau)$, thus completing the proof of \eqref{r:1}. 
For \eqref{r:2}, applying \eqref{e:S1},  \eqref{e:chig} and the condition  $\underline{\mathfrak{G}}(\tau)\geq 0$, we directly obtain the desired identity and conclude that $\Xi(\tau)$ is decreasing.
\end{proof}

\subsection{The boundary between charted and uncharted domains}\label{s:ctbdry}  
This Appendix provides alternative expressions for the hypersurface $\widehat{\Gamma}_{\delta_0}$, as defined in \eqref{e:surf0} in \S\ref{s:reorg}, using different coordinate systems: $(\ttau,\txi)$, $(\tau,\zeta)$, and $(t,x)$.
As the proofs of the following lemmas rely on straightforward calculations, we will omit the details.
 \begin{lemma}\label{t:Tsurf2}
 	In terms of the  coordinates $(\ttau,\txi)$, the hypersurface 	$\widehat{\mathfrak{T}}_{\delta_0}(\hat{\zeta} )$ can be expressed by 
 	\begin{align}\label{e:surf}
 		\ttau  = \widetilde{\mathfrak{T}}_{\delta_0}(\tilde{\zeta} ) =
 		\begin{cases}
 			-\exp\left(-\cc\left( \frac{51A}{2}-\frac{51A}{ 4\delta_0\txi^1  + 8}  \right) \left(   \txi^1 + \frac{1}{\delta_0} \right)\right)  ,  \quad &-\frac{1}{\delta_0}\leq\txi^1 \leq \frac{2-51 \Xi_0 }{102  \Xi_0} \frac{1}{\delta _0}   \\
 			-\exp\left(- \left( \frac{51A}{2}-\frac{51A}{ 4\delta_0\txi^1  + 8}  \right) \left(   \txi^1 + \frac{1}{2\delta_0} \right)\right)  ,  \quad  & \txi^1 >\frac{2-51 \Xi_0 }{102  \Xi_0} \frac{1}{\delta _0}  
 		\end{cases} . 
 	\end{align}     
 	where 
 	\begin{equation}\label{e:cc}
 		\cc= \frac{1}{1+\frac{51}{2} \Xi_0}   \AND \Xi_0:= 
 		\Xi(t_0)=\sup_{t\in[t_0,t_m)} \Xi(t)   .  
 	\end{equation} 
 Alternatively, solving for $\tilde{\zeta}$, we have: 
 	\begin{align*}
 		\txi^1 =   
 		\begin{cases}
 			\frac{1}{204} \sqrt{ 	\triangle_l(\ttau)  }-\frac{(51 \Xi_0 +2) \ln (-\ttau )}{102 A}-\frac{5}{4 \delta _0} , \quad &-1 \leq \ttau \leq -\exp\left(-\frac{51 A \Xi_0 +A}{153 \delta _0 \Xi_0 ^2+2 \delta _0 \Xi_0 }\right)   \\
 			\frac{1}{102} \sqrt{	\triangle_r(\ttau)}-\frac{\ln (-\ttau )}{51 A}-\frac{1}{\delta _0},\quad & \ttau >-\exp\left(-\frac{51 A \Xi_0 +A}{153 \delta _0 \Xi_0 ^2+2 \delta _0 \Xi_0 }\right)  
 		\end{cases}   , 
 	\end{align*} 
 	where 
 	\begin{align*}
 		\triangle_r(\ttau):=&\frac{4 (\ln (-\ttau ))^2}{A^2}-\frac{408 \ln (-\ttau )}{A \delta _0}+\frac{2601}{\delta _0^2} ,  \\
 		\triangle_l(\ttau):=&\frac{4 (51 \Xi_0 +2)^2 (\ln (-\ttau ))^2}{A^2}-\frac{612 (51 \Xi_0 +2) \ln (-\ttau )}{A \delta _0}+\frac{2601}{\delta _0^2}  . 
 	\end{align*}
 \end{lemma}

 \begin{lemma}\label{t:Tsurf3}
 	In terms of the coordinates $(\tau,\zeta)$,  the hypersurface 	$\widehat{\mathfrak{T}}_{\delta_0}(\hat{\zeta} )$ can be expressed by   
 	\begin{align*}%\label{e:surf3}
 		\tau  :=\tau_\Gamma(\zeta;\delta_0) =
 		\begin{cases}
 			-\exp\left(\frac{A }{200 \delta _0 (\Xi_0 +2)} \sqrt{  
 				\Diamond_l(\zeta^1)} -\frac{A (51 \Xi_0 +127)}{100 \delta _0 (\Xi_0 +2)}-\frac{A \zeta^1  (51 \Xi_0 +202)}{200 (\Xi_0 +2)}\right)  ,  \quad &-\frac{1}{\delta_0}\leq\zeta^1 \leq \frac{-153 \Xi_0 ^2+204 \Xi_0 +4}{306 \delta _0 \Xi_0^2+4 \delta _0 \Xi_0 }  \\
 			-\exp\left(\frac{A }{200 \delta _0}\sqrt{  	\Diamond_r(\zeta^1) }-\frac{51 A}{100 \delta _0}-\frac{101 A \zeta^1}{200} \right)  ,  \quad  & \zeta^1 >\frac{-153 \Xi_0 ^2+204 \Xi_0 +4}{306 \delta _0 \Xi_0^2+4 \delta _0 \Xi_0 }  
 		\end{cases} , 
 	\end{align*}    
 	where
 	\begin{align*}
 		\Diamond_r(\zeta^1):= & \delta _0 \zeta^1  \left(\delta _0 \zeta^1 +204\right)+2754 , \\
 		\Diamond_l(\zeta^1):= & \delta _0 \zeta^1  (51 \Xi_0 +2) \left(\delta _0 \zeta^1  (51 \Xi_0 +2)+4 (51 \Xi_0 +77)\right)+4 (51 \Xi_0  (51 \Xi_0 +104)+829) . 
 	\end{align*} 
 	Alternatively, solving for  $\zeta$, we have:
 	\begin{align*}
 		\zeta^1 =   
 		\begin{cases}
 			\frac{1}{204} \sqrt{\triangle_l(\tau) }-\frac{\Xi_0  \ln (-\tau )}{2 A}-\frac{101 \ln (-\tau )}{51 A}-\frac{5}{4 \delta _0} , \quad &-1 \leq \tau \leq -\exp\left(\frac{-51 A \Xi_0 -A}{153 \delta _0 \Xi_0^2+2 \delta _0 \Xi_0 }\right)   \\
 			\frac{1}{102} \sqrt{	\triangle_r(\tau)}-\frac{101 \ln (-\tau )}{51 A}-\frac{1}{\delta _0},\quad & \tau >-\exp\left(-\frac{51 A \Xi_0 +A}{153 \delta _0 \Xi_0 ^2+2 \delta _0 \Xi_0 }\right)  
 		\end{cases}  .
 	\end{align*}  
 \end{lemma}

 \begin{lemma}\label{t:Tsurf5}
 	Suppose $g(t,x)$ and $\mathsf{b}(\tau,\zeta)$, given by \eqref{e:coord2} and \eqref{e:coordi2}, exist. Then in terms of the coordinates $(t,x)$, the hypersurface	$\widehat{\mathfrak{T}}_{\delta_0}(\hat{\zeta} )$ can be expressed as 
 	\begin{align*}%\label{e:surf2} 
 		t :=&\mathsf{b}(\tau_\Gamma(x;\delta_0),x)  \notag  \\
 		=&\begin{cases}
 			\mathsf{b}\left(	-\exp\left(\frac{A }{200 \delta _0 (\Xi_0 +2)} \sqrt{\Diamond_l(x^1)} -\frac{A (51 \Xi_0 +127)}{100 \delta _0 (\Xi_0 +2)}-\frac{A x^1  (51 \Xi_0 +202)}{200 (\Xi_0 +2)}\right) ,x\right)  ,  \quad &-\frac{1}{\delta_0}\leq x^1 \leq \frac{-153 \Xi_0 ^2+204 \Xi_0 +4}{306 \delta _0 \Xi_0^2+4 \delta _0 \Xi_0 } \\
 			\mathsf{b}\left(		-\exp\left(\frac{A }{200 \delta _0}\sqrt{\Diamond_r(x^1)}-\frac{51 A}{100 \delta _0}-\frac{101 A x^1}{200} \right)  ,x\right)  ,     \quad  & x^1 >\frac{-153 \Xi_0 ^2+204 \Xi_0 +4}{306 \delta _0 \Xi_0^2+4 \delta _0 \Xi_0 } 
 		\end{cases} . 
 	\end{align*} 
 Additionally, $x^1$ can be expressed in terms of $g(t,x)$ as follows:
 	\begin{align*}
 		x^1 :=& \mathrm{X}(g(t,x)) \notag  \\
 		=&    
 		\begin{cases}
 			\frac{1}{204} \sqrt{	\triangle_l(g(t,x)) }-\frac{\Xi_0  \ln (-g(t,x) )}{2 A}-\frac{101 \ln (-g(t,x) )}{51 A}-\frac{5}{4 \delta _0} , \quad &-1 \leq g(t,x) \leq -\exp\left(\frac{-51 A \Xi_0 -A}{153 \delta _0 \Xi_0^2+2 \delta _0 \Xi_0 }\right)   \\
 			\frac{1}{102} \sqrt{\triangle_r(g(t,x))}-\frac{101 \ln (-g(t,x) )}{51 A}-\frac{1}{\delta _0},\quad & g(t,x) >-\exp\left(-\frac{51 A \Xi_0 +A}{153 \delta _0 \Xi_0 ^2+2 \delta _0 \Xi_0 }\right)  
 		\end{cases}  .
 	\end{align*}  
 \end{lemma}

The following lemma implies an important result that $p_m \in \overline{\mathit{\Lambda}_{\delta_0} \cap \mathcal{I}}$. 
\begin{lemma}\label{t:inI}
	For every point $(t,a\delta_1^i)\in \Gamma_{\delta_0}$ where $a\geq 0$,  we have $(t,a\delta_1^i) \in \mathcal{I}$ where $\mathcal{I}$ is defined in \eqref{e:cai}. 
\end{lemma}
\begin{proof}
	Firstly note that $\frac{153}{\delta_0}-\frac{2 (51 \Xi_0 +2)  (\ln(-g(t,x)) )}{A }>0$ and $\frac{102}{\delta_0}-\frac{2 (\ln(-g(t,x)) )}{A }>0$ due to the fact that $g(t,x)\in[-1,0)$.  
	Moreover, we have
	\begin{equation*}
		\sqrt{\triangle_r(g(t,x))} <\frac{102}{\delta_0}-\frac{2 (\ln(-g(t,x)) )}{A } \AND 	\sqrt{\triangle_l(g(t,x))} < \frac{153}{\delta_0}-\frac{2 (51 \Xi_0 +2)  (\ln(-g(t,x)) )}{A } . 
	\end{equation*}
By applying Lemma \ref{t:Tsurf5}, we arrive at the following inequality
	\begin{equation}
		x^1 =\mathrm{X}(g(t,x))  
		< \begin{cases}
			-\frac{\left(\Xi _0+2\right) }{A} \ln (-g(t,x))  -\frac{1}{2\delta_0}  , \quad &-1 \leq g(t,x) \leq -\exp\left(\frac{-51 A \Xi_0 -A}{153 \delta _0 \Xi_0^2+2 \delta _0 \Xi_0 }\right) \\
			-\frac{ 2  }{A} \ln (-g(t,x)) , \quad & g(t,x) >-\exp\left(-\frac{51 A \Xi_0 +A}{153 \delta _0 \Xi_0 ^2+2 \delta _0 \Xi_0 }\right)   
		\end{cases}  .  \label{e:x1upbd}
	\end{equation}  

	Next let us prove $(t,a\delta_1^i) \in \mathcal{I}$ by contradictions. We assume there exists a point $(t_1,a_1\delta_1^i)\in \Gamma_{\delta_0}$ such that $(t_1,a_1\delta_1^i)\notin \mathcal{I}$ with $a_1\geq 0$. In this case, since $(t_1,a_1\delta_1^i)$  lies within the homogeneous domain $\mathcal{H}$, we have $g(t_1,a_1\delta^i_1)=\mfg(t_1)$. Utilizing  \eqref{e:cah},  \eqref{e:keyid3}, Lemmas \ref{t:Tsurf5} and \ref{t:Thpst}.\eqref{r:2} we obtain
	\begin{equation*} 
		\mathrm{X}(\mfg(t_1)) \geq   1 +  \int_{t_0}^{t_1}   \sqrt{   \frac{\cm^2 f_0^2(y)}{(1+f(y))^2 } +\left(1-4\cm^2\right)\frac{1+f(y)}{y^2 }   } dy   = 1+\int_{t_0}^{t_1} \frac{ f_0 (y)}{1+f(y) }   \frac{B}{\chi_\uparrow} (2+\Xi) dy.    %\label{e:ctrasp}
	\end{equation*}
Combining this result with \eqref{e:x1upbd}, and recalling Lemma \ref{t:gb2}.\ref{l:2.1}, we derive: 
	\begin{align*} 
		& 1+\int_{t_0}^{t_1} \frac{ f_0 (y)}{1+f(y) }   \frac{B}{\chi_\uparrow} (2+\Xi) dy \notag  \\
		<	&  
		\begin{cases}
			-\frac{\left(\Xi _0+2\right) }{A} \ln (-\mfg(t_1))  -\frac{1}{2\delta_0} = (\Xi _0+2) \int^{t_1}_{t_0} \frac{f(y)(f(y)+1)}{y^2 f_0(y)} dy  -\frac{1}{2\delta_0} , \quad &-1 \leq \mfg(t_1) \leq -\exp\left(\frac{-51 A \Xi_0 -A}{153 \delta _0 \Xi_0^2+2 \delta _0 \Xi_0 }\right) \\
			-\frac{ 2  }{A} \ln (-\mfg(t_1)) = 2 \int^{t_1}_{t_0} \frac{f(y)(f(y)+1)}{y^2 f_0(y)} dy  , \quad & \mfg(t_1) >-\exp\left(-\frac{51 A \Xi_0 +A}{153 \delta _0 \Xi_0 ^2+2 \delta _0 \Xi_0 }\right)  
		\end{cases}  \notag  \\
		\overset{\eqref{e:keyid3}}{=}&  
		\begin{cases}
			(\Xi _0+2) \int^{t_1}_{t_0} \frac{f_0(y)}{1+f(y)} \frac{B}{\chi_\uparrow} dy  -\frac{1}{2\delta_0} , \quad &-1 \leq \mfg(t_1) \leq -\exp\left(\frac{-51 A \Xi_0 -A}{153 \delta _0 \Xi_0^2+2 \delta _0 \Xi_0 }\right) \\
			2 \int^{t_1}_{t_0} \frac{f_0(y)}{1+f(y)} \frac{B}{\chi_\uparrow}  dy  , \quad & \mfg(t_1) >-\exp\left(-\frac{51 A \Xi_0 +A}{153 \delta _0 \Xi_0 ^2+2 \delta _0 \Xi_0 }\right)   
		\end{cases}   . 
	\end{align*} 
	Since $\Xi>0$ (see Lemma \ref{t:char1}) and $B/\chi_\uparrow <1/4$ (using \eqref{e:chig} and Lemma \ref{t:Thpst}.\eqref{r:1}), we can further deduce 
	\begin{align*} 
		1  
		<&  
		\begin{cases}
			\int^{t_1}_{t_0} \frac{f_0(y)}{1+f(y)} \frac{B}{\chi_\uparrow} (\Xi _0-\Xi) dy   -\frac{1}{2\delta_0}< \frac{\Xi _0}{4} 	\ln \frac{1+f(t_1)}{1+\beta}  -\frac{1}{2\delta_0}  \overset{(\star)}{<} 0, \quad &-1 \leq \mfg(t_1) \leq -\exp\left(\frac{-51 A \Xi_0 -A}{153 \delta _0 \Xi_0^2+2 \delta _0 \Xi_0 }\right) \\
			- \int^{t_1}_{t_0} \frac{f_0(y)}{1+f(y)} \frac{B}{\chi_\uparrow}   \Xi dy  <0  , \quad & \mfg(t_1) >-\exp\left(-\frac{51 A \Xi_0 +A}{153 \delta _0 \Xi_0 ^2+2 \delta _0 \Xi_0 }\right) 
		\end{cases}   . 
	\end{align*} 
	The inequality $(\star)$ holds if $\delta_0$ sufficiently small, say $\delta_0<\bigl(\frac{\Xi_0}{2} \ln \frac{1+f(t_1)}{1+\beta}\bigr)^{-1}$. This leads to a contradiction $1<0$,  thereby completing the proof. 
\end{proof}

%%%--------------------NEW SEC-----------------------

\section{The reference ODE}\label{s:ODE0}
\subsection{Analysis of the reference ODE}\label{s:ODE}
In this section, we revisit a crucial type of ODE that was originally developed and analyzed in our previous work  \cite[\S$2$]{Liu2022b}. For convenience, we present the relevant results here without proofs; detailed derivations can be found in \cite[\S$2$]{Liu2022b} and \cite{Liu2023}. Specifically, we consider solutions $f(t)$ to the following ODE:
\begin{gather}
	f^{\prime\prime}(\mft)+\frac{\ca}{\mft}  f^\prime(\mft)-\frac{\cb}{\mft^2} f(\mft)(1+  f(\mft))-\frac{\cc (  f^\prime(\mft))^2}{1+f(\mft)}=   0 , \label{e:feq0}\\
	f(t_0)= \mf>0 \AND
	f^\prime(t_0)=   \mf_0>0.  \label{e:feq1}
\end{gather}
where $\mf,\mf_0>0$ are positive constants and the parameters $\ca,\; \cb$ and $\cc$ satisfy
\begin{equation}\label{e:abcdk}
	\ca>1, \quad \cb>0 \AND 1<\cc < 3/2 .
\end{equation}

\begin{remark}
	The equations \eqref{e:feq0}--\eqref{e:feq1} reduce to \eqref{e:feq0b}--\eqref{e:feq1b}  when  the parameters are set to  $
		\ca=\frac{4}{3}$, $ \cb=\frac{2}{3}$ and $ \cc=\frac{4}{3}  $. 
\end{remark}

To simplify the notations, we denote
\begin{equation*}%\label{e:deftr}
	\triangle:=\sqrt{(1-\ca)^2+4\cb}>-\ba, \; \ba=1-\ca<0, \; \bc=1-\cc<0 . 
\end{equation*}
We also introduce constants $\mathtt{A}$,  $\mathtt{B}$, $\mathtt{C}$, $\mathtt{D}$ and $\mathtt{E}$ which depend on the initial data $\mf$ and $\mf_0$ from \eqref{e:feq0}--\eqref{e:feq1}, as well as on the parameters $\ca$, $\cb$ and $\cc$,
\begin{gather*}
	\mathtt{A}:= \frac{t_0^{   -\frac{\ba-\triangle}{2}  }}{\triangle}\biggl(  \frac{t_0   \mf_0}{(1+\mf)^2} - \frac{\ba+\triangle}{2}  \frac{\mf  } {1+\mf} \biggr), \quad
	\mathtt{B}:=    \frac{t_0^{-\frac{\ba+\triangle}{2} }}{\triangle} \Bigl( \frac{\ba-\triangle}{2}  \frac{\mf }{1+\mf}  -\frac{t_0 \mf_0}{(1+\mf)^2} \Bigr)<0, \quad 	\mathtt{E}:=   
	\frac{\bc  \mf_0 t_0^{1-\ba}  }{  \ba  (1+\mf) } >0, \\
	\mathtt{C}:=   \frac{2 } {2+\ba+\triangle} \Bigl( \ln( 1+\mf) +\frac{\ba+\triangle}{2\cb} \frac{t_0\mf_0}{1+\mf}\Bigr)t_0^{-\frac{\ba+\triangle}{2 }} >0 \AND 
	\mathtt{D}:=  
	\frac{ \ba+\triangle   }{2+\ba+\triangle}  \Bigl(  \ln( 1+\mf)  - \frac{1 } { \cb} \frac{t_0\mf_0}{1+\mf}\Bigr) t_0. 
\end{gather*}

We define the following two critical times $t_\star$ and $t^\star$.
\begin{definition}\label{t:tdef}
	Suppose $\mathtt{A},\;\mathtt{B}, \;\mathtt{E},\;\ba$ and $\triangle$ are defined above, then
	\begin{enumerate}
		\item
		Let $\mathcal{R}:=\{t_r>t_0 \;|\;\mathtt{A} t_r^{\frac{\ba-\triangle}{2} } + \mathtt{B} t_r^{\frac{\ba+\triangle}{2} } + 1 =0\}$ and define $t_\star:=\min \mathcal{R}$.
		\item If $t_0^{\ba}> \mathtt{E}^{-1}$, we define $t^\star :=    (t_0^{\ba}- \mathtt{E}^{-1} )^{1/\ba}\in(0,\infty)$, i.e.,  $\mft=t^\star$ solves $1-\mathtt{E}  t_0^{\ba} +  \mathtt{E}   \mft^{\ba}=0$.
	\end{enumerate}
\end{definition}

With these definitions in place, we are now in a position to state the main theorem concerning the ODE \eqref{e:feq0}--\eqref{e:feq1}. The proof of this theorem can be found in \cite[\S$2$]{Liu2022b}.
\begin{theorem}\label{t:mainthm0}
	Suppose constants $\ca$, $\cb$ and $\cc$ are defined by  \eqref{e:abcdk}, $t_\star$ and $t^\star$ are defined above and the initial data $\mf, \mf_0>0$, then
	\begin{enumerate}
		\item $t_\star \in[0,\infty)$ exists and $t_\star>t_0$;
		\item there is a constant $t_m\in [t_\star,\infty]$, such that there is a unique solution $f\in C^2([t_0,t_m))$ to the equation \eqref{e:feq0}--\eqref{e:feq1}, and
		\begin{equation*}%\label{e:limf}
			\lim_{\mft\rightarrow t_m} f(\mft)=+\infty \AND \lim_{\mft\rightarrow t_m} f_0(\mft)=+\infty .
		\end{equation*}
		\item  $f$ satisfies upper and lower bound estimates,
		\begin{align*}%\label{e:fbds}
			1+f(\mft)>&\exp \bigl( \mathtt{C} \mft^{\frac{\ba+\triangle}{2} }  +\mathtt{D}  \mft^{-1}\bigr)  ,    &&\text{for}\quad \mft\in(t_0,t_m);
			\\
			1+f(\mft) < & \bigl(\mathtt{A} \mft^{\frac{\ba-\triangle}{2} } + \mathtt{B} \mft^{\frac{\ba+\triangle}{2} } + 1 \bigr)^{-1} ,   && \text{for}\quad \mft\in(t_0,t_\star).
		\end{align*}
	\end{enumerate}		
	Furthermore, if the initial data satisfies
	$\mf_0 >  \ba(1+\mf) /(\bc t_0 )$,
	then
	\begin{enumerate}
		\setcounter{enumi}{3}
		\item 	$t_\star$  and $t^\star$  exist and finite, and  $t_0<t_\star<t^\star<\infty$;
		\item there is a finite time $t_m\in [t_\star,t^\star)$, such that there is a solution $f\in C^2([t_0,t_m))$ to the equation \eqref{e:feq0} with the initial data \eqref{e:feq1},   and 	\begin{equation*}
			\lim_{\mft\rightarrow t_m} f(\mft)=+\infty \AND \lim_{\mft\rightarrow t_m} f_0(\mft)=+\infty .
		\end{equation*}
		\item the solution $f$ has improved lower bound estimates, for $\mft\in(t_0,t_m)$,
		\begin{equation*}\label{e:ipvest}
			(1+\mf)  \bigl(1-\mathtt{E}  t_0^{\ba} +  \mathtt{E}   \mft^{\ba} \bigr)^{1/\bc}  < 1+f(\mft) .
		\end{equation*}
	\end{enumerate}
\end{theorem}

\subsection{Analysis of the reference solutions} \label{t:refsol}
\subsubsection{The time transformation function $\mfg(\mft)$ and rough estimates of $f$ and  $\del{\mft}f(\mft)$}\label{t:ttf}
Let\footnote{Note the sign difference between  $\mfg$ in this article and $g$ in \cite{Liu2022b}. }
\begin{align}\label{e:gdef0a}
	\mfg(\mft):=-\exp\Bigl(-A\int^{\mft}_{t_0} \frac{f(s)(f(s)+1)}{s^2 f_0(s)} ds \Bigr)<0
\end{align}
where $A\in(0,2\cb/(3-2\cc))$ is a constant.
The following lemma gives an alternative representation of $\mfg(\mft)$ only involving $f$ without $f_0$.  It also outlines the fundamental properties of the function $\mfg(\mft)$ and expresses $f_0$ in terms of $f$ and $\mfg$.

\begin{lemma}\label{t:f0fg}
	Suppose $f\in C^2([t_0,t_1))$ ($t_1>t_0$) solves the equation \eqref{e:feq0}--\eqref{e:feq1}, $\mfg(\mft)$ is defined by \eqref{e:gdef0a}, and  denote $f_0(\mft):=\del{\mft} f(\mft)$, then
	\begin{enumerate}	
		\item  $f_0$ can be expressed by
		\begin{equation}%\label{e:f0frl0}
			f_0(\mft)=B^{-1} \mft^{-\ca} (-\mfg(\mft))^{-\frac{\cb}{A}}(1+f(\mft))^{\cc} >0  \label{e:f0aa}
		\end{equation}
		for $\mft\in [t_0,t_1)$ where $B:= (1+\mf)^\cc/( t_0^{ \ca} \mf_0)>0 $ is a constant depending on the data;
		\item If the data $\mf>0$, then $f(\mft)>0$ for $\mft \in[t_0,t_1)$;
		\item $\mfg(\mft)$ can be represented by
		\begin{equation*}%\label{e:gdef2}
			\mfg (\mft) =-\Bigl(1+ \cb B \int^\mft_{t_0} s^{\ca-2} f(s)(1+f(s))^{1-\cc}  ds \Bigr)^{-\frac{A}{\cb}}\in [-1,0),
		\end{equation*}
		for $\mft\in[t_0,t_1)$,  and $\mfg(t_0)=-1$; 	
		\item
		$\mfg(t)$ is strictly increasing and  invertible in $[t_0,t_1)$.
	\end{enumerate}
\end{lemma}

\begin{remark}\label{t:dtg1}
	There are two useful identities from \cite[eqs. $(2.7)$ and $(2.8)$]{Liu2022b}. We list them in this remark.
	\begin{align*}
		\del{\mft}\mfg(\mft) = & A  B (-\mfg(\mft) )^{\frac{\cb}{A}+1}  \mft^{\ca-2} f(\mft) (1+f(\mft))^{1-\cc} , %\label{e:dtg0}
		 \\
		\del{\mft} (-\mfg(\mft)  )^{-\frac{\cb}{A}}= &  \frac{\cb}{A}(-\mfg(\mft) )^{-\frac{\cb}{A}-1}\del{\mft} \mfg (\mft)  = \cb B \mft^{\ca-2} f(\mft)  (1+f(\mft))^{1-\cc}.  %\label{e:dtgf}
	\end{align*}
\end{remark}

\subsubsection{Estimates of two crucial quantities $\chi(\mft)$ and $\xi(\mft)$}
In this section, we  estimate two important quantities $\chi(\mft)$ and $\xi(\mft)$ which are  frequently utilized in our analysis (see \cite{Liu2022b} for details).  The first quantity is defined by
\begin{equation}\label{e:Gdef0}
	\chi(\mft):=\frac{\mft^{2-\ca} f_0(\mft)}{(1+f(\mft))^{2-\cc} f(\mft) (-\mfg)^{\frac{\cb}{A}}(\mft)} \overset{\eqref{e:f0aa}}{=} \frac{  (-\mfg)^{-\frac{2\cb}{A}}(\mft) \mft^{2(1-\ca)}}{B f(\mft) (1+f(\mft))^{2(1-\cc)}} >0 .
\end{equation}

\begin{lemma}\label{t:gmap}
	Suppose $\mfg(\mft)$ is defined by \eqref{e:gdef0a} and  $\cc\in(1,3/2)$,  and  $f\in C^2([t_0,t_m))$ (where $[t_0,t_m)$ is the maximal interval of existence of $f$ given by Theorem \ref{t:mainthm0}) solves ODE \eqref{e:feq0}--\eqref{e:feq1}. Then $
	\lim_{\mft\rightarrow t_m} \mfg(\mft)=0$.
\end{lemma}
\begin{remark}
	Due to this lemma, it is convenient to continuously extend $\mfg(\mft)$ from $[t_0,t_m)$ to $[t_0,t_m]$ by letting $\mfg(t_m):=\lim_{\mft\rightarrow t_m} \mfg(\mft)=0$, then $\mfg^{-1} (0)=t_m$.
\end{remark}

\begin{proposition}\label{t:limG}
	Suppose $\cc\in(1,3/2)$, $\cb>0$, $\ca>1$,  $\chi$ is defined by \eqref{e:Gdef0} and  $f\in C^2([t_0,t_m))$ (where $[t_0,t_m)$ is the maximal interval of existence of $f$ given by Theorem \ref{t:mainthm0}) solves ODE \eqref{e:feq0}--\eqref{e:feq1}.
	Then there is a function $\mathfrak{G} \in C^1([t_0,t_m))$, such that for $\mft\in [t_0,t_m)$,
	\begin{equation}\label{e:limG}
		\chi(\mft)=\frac{2\cb B}{3-2\cc}+\mathfrak{G}(\mft)
	\end{equation}
	where $\lim_{\mft\rightarrow t_m}\mathfrak{G}(\mft)=0$.
	Moreover, there is a constant $C_\chi>0$ such that $0<\chi(\mft) \leq C_\chi$ in $[t_0,t_m)$, and there are continuous extensions of $\chi$ and $\mathfrak{G}$ such that $\chi\in C^0([t_0,t_m])$ and $\mathfrak{G}\in C^0([t_0,t_m])$ by letting $\chi(t_m):=2\cb B/(3-2\cc)$ and $\mathfrak{G}(t_m):=0$.
\end{proposition}

The second crucial quantity is defined as
\begin{equation}\label{e:xidef}
	\xi(\mft):=1/[-\mfg(\mft) (1+f(\mft)) ] . 
\end{equation}
The following proposition establishes that $\xi$ is bounded and that its limit vanishes as $\mft$ tends to $t_m$. 
\begin{proposition}\label{t:fginv0}
	Suppose $f\in C^2([t_0,t_m))$ (where $[t_0,t_m)$ is the maximal interval of existence of $f$ given by Theorem \ref{t:mainthm0}) solves ODE \eqref{e:feq0}--\eqref{e:feq1}, $\mfg(t)$ is defined by \eqref{e:gdef0a} and $\xi(\mft)$ is given by \eqref{e:xidef}, then $\xi\in C^1([t_0,t_m))$ and
	\begin{equation*}%\label{e:fginv}
		\lim_{\mft\rightarrow t_m} \xi(\mft)= 0 .
	\end{equation*}
	Moreover, there is a constant $C_\star>0$, such that $0<\xi(\mft)  \leq C_\star$ for every $\mft\in[t_0,t_m)$, and there is a continuous extension of $\xi$ such that  $\xi\in C^0([t_0,t_m])$ by letting $\xi(t_m):=0$.
\end{proposition}

We compute the derivative $\del{\mft}\chi$ in the following lemma. The detailed proof is provided in  \cite[Lemma B.$6$]{Liu2023}. 
\begin{lemma}\label{t:dtchi}
	Suppose $\chi(\mft)$ is defined by \eqref{e:Gdef0} and $\mathfrak{G}$ is given by \eqref{e:limG}, then
	\begin{equation*}%\label{e:dtchi}
		\del{\mft}\chi= \del{\mft}\mathfrak{G} = -\frac{(3-2\cc)  \mathfrak{G} f^{\frac{1}{2}} \chi^{\frac{1}{2}}}{B^{\frac{1}{2}} \mft}   - \frac{\chi^{\frac{3}{2}}}{B^{\frac{1}{2} } \mft f^{\frac{1}{2}}}   +2(1-\ca) \frac{\chi }{\mft} .
	\end{equation*}
\end{lemma} 
\begin{corollary}\label{t:dtchi2}
	Under the conditions of Lemma \ref{t:dtchi}, the function $|\mathfrak{G}(\mft)|$ is decreasing and $-|\mathfrak{G}(t_0)|<\mathfrak{G}(\mft)<|\mathfrak{G}(t_0)|$.    
\end{corollary}
\begin{proof}
	By Lemma \ref{t:dtchi}, we obtain $	 \del{\mft}\mathfrak{G}   <-\frac{(3-2\cc)  \mathfrak{G} f^{\frac{1}{2}} \chi^{\frac{1}{2}}}{B^{\frac{1}{2}} \mft} $, 
which, in turn, yields $
	\del{\mft} \ln |\mathfrak{G}(\mft)| <-\frac{(3-2\cc)   f^{\frac{1}{2}} \chi^{\frac{1}{2}}}{B^{\frac{1}{2}} \mft}  <0$. 
We conclude $|\mathfrak{G}(\mft)|$ is decreasing.  
We complete this proof. 
\end{proof}

	\begin{lemma}\label{t:Gest2}
	Suppose $\mathfrak{G}$ is given by \eqref{e:chig}, then
	$\tilde{\underline{\mathfrak{G}}}(\ttau)$ has an estimate
	\begin{equation*}
		|\tilde{\underline{\mathfrak{G}}}(\ttau)|\lesssim (-\ttau)^{\frac{1}{2}} 
	\end{equation*}
	for $\ttau\in[-1,0)$ and the function $(-\ttau)^{-\frac{1}{2}} \tilde{\underline{\mathfrak{G}}}(\ttau)$ can be continuously extended to $\ttau\in[-1,0]$, that is, $(-\ttau)^{-\frac{1}{2}} \tilde{\underline{\mathfrak{G}}}(\ttau)\in C^0([-1,0])  $.  
\end{lemma} 
\begin{proof}
	The proof of this lemma is similar to \cite[Lemma $3.8$]{Liu2023} and we omit the details. 
\end{proof}

The following proposition, along with its generalization, is proven in \cite[Corollary B.1]{Liu2023} and \cite[Proposition B.7]{Liu2023}, respectively. 
\begin{proposition}\label{s:gf1/2}
	Suppose $\ca=4/3$, $\cb=2/3$,  $\cc=4/3$ and $0<
	A<2$, then
	\begin{equation*}
		\lim_{t\rightarrow t_m} \biggl(\frac{1}{\mfg f^{\frac{1}{2}}}\biggr)= 0 .
	\end{equation*}
\end{proposition}

%-------------------------------------------------------------New Section----------------------------------------------------------------------------------------

 \section{Fuchsian global initial value problem on  compact domains}\label{s:fuc}
 The Fuchsian global initial value problem (Fuchsian GIVP) originates from Oliynyk \cite[Appendix B]{Oliynyk2016a}. 
 In this Appendix, we introduce, without proofs, a tailored version of the theorem established in \cite{Beyer2020},  which serves as a fundamental tool for the analysis in this article.  For further generalizations and applications of the Fuchsian GIVP, see \cite{Liu2018,Liu2028c,Liu2018a,Beyer2020b,Fajman2021,LeFloch2021,Ames2022,Ames2022a,Marshall2023,Liu2022,Liu2022a,Liu2023,Beyer2023a,Beyer2023b,Fajman2023,Oliynyk2024}. 
 Consider the following singular symmetric hyperbolic system.
 \begin{align}
 	B^{\mu}(t,x,u)\partial_{\mu}u =&\frac{1}{t}\textbf{B}(t,x,u)\textbf{P}u+H(t,x,u)\quad&&\text{in}\;[T_{0},T_{1})\times\mathbb{R}^{n},  \label{e:model1}\\
 	u =&u_{0}(x) &&\text{in}\;\{T_{0}\}\times\mathbb{R}^{n},\label{e:model2}
 \end{align}
 where $T_{0}<T_{1}\leq0$,  and we require the following \textbf{Conditions}\footnote{The notations in this Appendix, for example, $\mathrm{O}(\cdot)$, $\mathcal{O}(\cdot)$ and $B_R(\Rbb^N)$, are defined in \cite[\S$2.4$]{Beyer2020}. }:
 \begin{enumerate}[leftmargin=*,label={(H\arabic*)}]
 	\item \label{c:2} $\textbf{P}$ is a constant, symmetric projection operator, i.e., $\textbf{P}^{2}=\textbf{P}$, $\textbf{P}^{T}=\textbf{P}$ and $\partial_\mu \textbf{P}=0$. We also denote $\textbf{P}^\perp:=\mathds{1}-\textbf{P}$ the \textit{complementary projection}. 
 	
 	\item \label{c:3} $u=u(t,x)$ and $H(t,x,u)$ are $\Rbb^{N}$-valued maps, $H\in C^{0}([T_{0},0),C^{\infty}(\Tbb^n\times B_R(\Rbb^{N}),\Rbb^N))$ %and satisfies $H(t,x,0)=0$.
 	can be expanded as
 	\begin{equation*}
 		H(t,x,u)= H_0(t,x,u)+|t|^{-\frac{1}{2}}H_1(t,x,u) 
 	\end{equation*}
 	where $H_0,H_1 \in C^0([T_0,0],C^{\infty}(\Tbb^n\times B_R(\Rbb^{N}),\Rbb^N))$ satisfy   there exist constants $\lambda_a\geq 0$, $a=1,2$, such that
 	\begin{equation*}
 		H_0(t,x,u)=\mathrm{O}(u) ,  \quad  \textbf{P}H_1(t,x,u)=\mathcal{O}(\lambda_1 u) , \quad \textbf{P}^\perp H_1(t,x,u)=\mathcal{O}(\lambda_2 \textbf{P} u)  
 	\end{equation*} 
 	for all $(t,x)\in[T_0,0)\times \Tbb^n$. 
 	
 	\item \label{c:4} $B^{\mu}=B^{\mu}(t,x,u)$ and $\textbf{B}=\textbf{B}(t,x,u)$ are $\mathbb M_{N\times N}$-valued maps, and there is a constant $R>0$, such that $B^{i}\in   C^{0}([T_{0},0),C^{\infty}(\Tbb^n\times  B_R(\Rbb^{N}), \mathbb M_{N\times N})$,  $\textbf{B}\in   C^{0}([T_{0},0],C^{\infty}(\Tbb^n\times  B_R(\Rbb^{N}), \mathbb M_{N\times N})$, $B^{0}\in C^{1}([T_{0},0),C^{\infty}(\Tbb^n\times  B_R(\Rbb^{N}), \mathbb M_{N\times N})$ and they satisfy
 	\begin{equation*}\label{e:comBP}
 		(B^{\mu})^{T}=B^{\mu},\quad [\textbf{P}, \textbf{B}]=\textbf{PB}-\textbf{BP}=0 , 
 	\end{equation*}
 	and $	B^i$ can be expanded as
 	\begin{equation*}
 		B^i(t,x,u)=B_0^i(t,x,u)+|t|^{-1}  B_2^i(t,x,u)
 	\end{equation*}
 	where $B_0^i, B_2^i\in C^{0}([T_{0},0],C^{\infty}(\Tbb^n\times  B_R(\Rbb^{N}), \mathbb M_{N\times N}))$.
 	
 	Suppose there is $\tilde{B}^0, \tilde{\mathbf{B}} \in C^0([T_0,0], C^\infty(\Tbb^n, \mathbb M_{N\times N}))$, such that
 	\begin{equation*}
 		[\mathbf{P}, \tilde{\mathbf{B}}] =0, \quad
 		B^0(t,x,u)-\tilde{B}^0(t,x)= \mathrm{O}(u) , \quad
 		\mathbf{B}(t,x,u)-\tilde{\mathbf{B}} (t,x)= \mathrm{O}(u)
 	\end{equation*}
 	for all $ (t,x,u) \in[T_{0},0]\times \Tbb^n\times B_R(\Rbb^N) $.
 	
 	Moreover, there is $\tilde{B}_2^i \in C^0([T_0,0], C^\infty(\Tbb^n, \mathbb M_{N\times N}))$, such that
 	\begin{gather*}
 		\mathbf{P} B_2^i(t,x,u) \mathbf{P}^\perp =   \mathrm{O}(\mathbf{P} u), \quad
 		\mathbf{P}^\perp B_2^i(t,x,u) \mathbf{P} =   \mathrm{O}(\mathbf{P} u),  \\
 		\mathbf{P}^\perp B_2^i(t,x,u) \mathbf{P}^\perp =   \mathrm{O}(\mathbf{P} u\otimes \mathbf{P} u),\quad
 		\mathbf{P}  (B_2^i(t,x,u)-\tilde{B}_2^i(t,x)) \mathbf{P}  =   \mathrm{O}( u),
 	\end{gather*}
 	for all $ (t,x,u) \in[T_{0},0]\times \Tbb^n\times B_R(\Rbb^N) $.
 	
 	\item \label{c:5}
 	There exists constants $\kappa,\,\gamma_{1},\,\gamma_{2}$ such that
 	\begin{equation*}
 		\frac{1}{\gamma_{1}}\mathds{1}\leq B^{0}\leq \frac{1}{\kappa} \textbf{B} \leq\gamma_{2}\mathds{1} \label{e:Bineq}
 	\end{equation*}
 	for all $ (t,x,u) \in[T_{0},0]\times \Tbb^n\times B_R(\Rbb^N) $.
 	
 	\item \label{c:6} For all $(t,x,u)\in[T_{0},0]\times\Tbb^n \times B_R(\Rbb^{N})$, assume
 	\begin{equation*}
 		\textbf{P}^{\bot}B^{0}(t,\textbf{P}^\perp u)\textbf{P}=\textbf{P}B^{0}(t,\textbf{P}^\perp u)\textbf{P}^{\bot}=0.
 	\end{equation*}
 	\item \label{c:7}
 	There exist
 	constants $\theta$ and $\beta_{\ell}\geq 0$, $\ell=0,\cdots,7$, such that
 	\begin{align*}
 		\mathrm{div} B(t,x, u,w):= &  \del{t} B^0 (t,x,u)+D_u B^0(t,x,u) \cdot(B^0(t,x, u))^{-1}\Bigl[-B^i(t,x,u)\cdot w_i   +\frac{1}{t} \mathbf{B}(t,x,u)\mathbf{P} u \notag  \\
 		&+H_0(t,x,u) +|t|^{-\frac{1}{2}}H_2(t,x,u)\Bigr]+\del{i}B^i(t,x,u)  +D_u B^i(t,x,u) \cdot w_i  ,
 	\end{align*}
 	where $w=(w_i)$ and $(t,x,u, w) \in [T_0,0)\times \Tbb^n \times B_R(\Rbb^N) \times B_R(\mathbb{M}_{N \times n})$, satisfies
 	\begin{align}
 		\mathbf{P} 	\mathrm{div} B \mathbf{P}  = & \;\mathcal{O}\bigl(\theta   +|t|^{-\frac{1}{2}}\beta_0+|t|^{-1}\beta_1 \bigr),  \label{e:PhP1}\\
 		\mathbf{P} 	\mathrm{div} B \mathbf{P}^\perp  = &\; \mathcal{O}\biggl(\theta  +|t|^{-\frac{1}{2}} \beta_2 +\frac{|t|^{-1}\beta_3}{R}\mathbf{P}u  \biggr),  \label{e:PhP2}\\
 		\mathbf{P}^\perp 	\mathrm{div} B \mathbf{P}  = & \; \mathcal{O}\biggl(\theta  +|t|^{-\frac{1}{2}} \beta_4  +\frac{|t|^{-1}\beta_5}{R}\mathbf{P}u   \biggr) \label{e:PhP3}
 		\intertext{and}
 		\mathbf{P}^\perp  	\mathrm{div} B \mathbf{P}^\perp  =  & \; \mathcal{O}\biggl(\theta  +\frac{|t|^{-\frac{1}{2}}\beta_6}{R }\mathbf{P}u   +\frac{|t|^{-1}\beta_7}{R^2}\mathbf{P}u \otimes \mathbf{P}u \biggr)  . \label{e:PhP4}
 	\end{align}	
 \end{enumerate}

 Now let us present the global existence theorem for the Fuchsian system (see \cite{Beyer2020} for detailed proofs).
 \begin{theorem}\label{t:fuc}
 	Suppose that $k\in\Zbb_{> \frac{n}{2}+3}$, $u_{0}\in H^{k}(\mathbb T^{n})$ and conditions \ref{c:2}--\ref{c:7} are fulfilled, and the constants $\kappa, \gamma_1,  \beta_1,\beta_3,\beta_5,\beta_7$ from the conditions \ref{c:2}--\ref{c:7} satisfy
 	\begin{equation}\label{e:kpbt1}
 		\kappa>\frac{1}{2} \gamma_1 \max\Bigl\{\sum^3_{\ell=0} \beta_{2\ell+1}, \beta_1+2k(k+1)  \mathtt{b} \Bigr\}
 	\end{equation}
 	where
 	\begin{equation*}
 		\mathtt{b}:= \sup_{T_0\leq t<0} \bigl(\||\mathbf{P} \tilde{\mathbf{B}} D (\tilde{\mathbf{B}}^{-1} \tilde{B}^0)(\tilde{B}^0)^{-1} \mathbf{P} \tilde{B}_2^i \mathbf{P}|_{\mathrm{op}}\|_{\Li}+\||\mathbf{P}\tilde{\mathbf{B}} D (\tilde{\mathbf{B}}^{-1} \tilde{B}_2^i)\mathbf{P}|_{\mathrm{op}}\|_{\Li}\bigr).
 	\end{equation*}
 	Then there exist constants $\delta_0, \delta>0$ satisfying $\delta<\delta_0$, such that if
 	\begin{equation*}
 		\|u_0\|_{H^k} \leq \delta,
 	\end{equation*}
 	then there exists a unique solution
 	\begin{equation*}
 		u\in C^0([T_0,0),H^k(\Tbb^n) ) \cap C^1([T_0,0),H^{k-1}(\Tbb^n) ) \cap \Li ([T_0,0),H^k(\Tbb^n))
 	\end{equation*}
 	of the initial value problem \eqref{e:model1}--\eqref{e:model2} such that $\Pbp u(0):=\lim_{t\nearrow 0}\Pbp u(t)$ exists in $H^{s-1}(\Tbb^n)$.
 	
 	\noindent Moreover, for $T_0\leq t<0$, the solution $u$ satisfies the energy estimate
 	\begin{equation*}\label{e:ineq1}
 		\|u(t)\|_{H^k(\Tbb^n)}^2  - \int^t_{T_0}\frac{1}{\tau}\|\mathbf{P} u(\tau)\|^2_{H^k(\Tbb^n)}d\tau \leq C(\delta_0,\delta_0^{-1}) \|u_0\|^2_{H^k(\Tbb^n)}  .  
 	\end{equation*} 
 \end{theorem}

\section*{Acknowledgement}
		%We would like to thank ... for their helpful discussions, comments and advice.
		C.L. is partially supported by the Fundamental Research Funds for the Central Universities, HUST: $2020$kfyXJJS$037$ ($5003011036$ and $5003011047$). 
		%We also thank the referee for their comments and criticisms, which have served to improve the content and exposition of this article.

\bigskip

\textbf{Data Availability} Data sharing is not applicable to this article as no datasets were
generated or analysed during the current study.

\bigskip

\textbf{Declarations}

\bigskip

\textbf{Conflict of interest} The authors declare that they have no conflict of interest.

\bibliographystyle{amsplain}
\bibliography{Reference_Chao}

\providecommand{\bysame}{\leavevmode\hbox to3em{\hrulefill}\thinspace}
\providecommand{\MR}{\relax\ifhmode\unskip\space\fi MR }
\providecommand{\MRhref}[2]{%
  \href{http://www.ams.org/mathscinet-getitem?mr=#1}{#2}
}
\providecommand{\href}[2]{#2}
\begin{thebibliography}{10}

\bibitem{Abbrescia2022}
Leonardo Abbrescia and Jared Speck, \emph{The emergence of the singular
  boundary from the crease in {3D} compressible {Euler flow}}, arxiv:2207.07107
  (2022).

\bibitem{Abbrescia2023}
\bysame, \emph{{The relativistic Euler equations: ESI} notes on their
  geo-analytic structures and implications for shocks in {1D} and
  multi-dimensions}, Classical and Quantum Gravity \textbf{40} (2023), no.~24,
  243001.

\bibitem{Alinhac1995}
Serge Alinhac, \emph{Blowup for nonlinear hyperbolic equations}, Birkh{\"a}user
  Boston, 1995.

\bibitem{Ames2022a}
Ellery Ames, Florian Beyer, James Isenberg, and Todd Oliynyk, \emph{Stability
  of asymptotic behavior within polarised ${T}^2$-symmetric vacuum solutions
  with cosmological constant}, Phil. Trans. R. Soc. A 380: 20210173 (2022)
  \textbf{380} (2022), no.~2222.

\bibitem{Ames2022}
\bysame, \emph{Stability of {AVTD} behavior within the polarized
  {$T^2$}-symmetric vacuum spacetimes}, Ann. Henri Poincar\'e (2022)
  \textbf{23} (2022), no.~7, 2299--2343.

\bibitem{Bardeen1980}
James~M. Bardeen, \emph{Gauge-invariant cosmological perturbations}, Physical
  Review D \textbf{22} (1980), no.~8, 1882--1905.

\bibitem{Beyer2020b}
Florian Beyer and Todd~A. Oliynyk, \emph{Relativistic perfect fluids near
  {Kasner} singularities}, arXiv:2012.03435 (2020).

\bibitem{Beyer2023a}
\bysame, \emph{Localized big bang stability for the {Einstein-Scalar} field
  equations}, Archive for Rational Mechanics and Analysis \textbf{248} (2023),
  no.~1.

\bibitem{Beyer2023b}
\bysame, \emph{{Past stability of FLRW solutions to the Einstein-Euler-scalar
  field equations and their big bang singularites}}, arXiv:2308.07475 (2023).

\bibitem{Beyer2020}
Florian Beyer, Todd~A. Oliynyk, and J.~Arturo Olvera-Santamar{\'{\i}}a,
  \emph{The {Fuchsian} approach to global existence for hyperbolic equations},
  Communications in Partial Differential Equations \textbf{46} (2020), no.~5,
  1--82.

\bibitem{Bonnor1957}
W.~B. Bonnor, \emph{Jeans' formula for gravitational instability}, Monthly
  Notices of the Royal Astronomical Society \textbf{117} (1957), no.~1,
  104--117.

\bibitem{Buckmaster2022b}
Tristan Buckmaster, Gonzalo Cao-Labora, and Javier Gómez-Serrano, \emph{Smooth
  imploding solutions for {3D} compressible fluids}, arXiv:2208.09445 (2022).

\bibitem{Buckmaster2021}
Tristan Buckmaster, Theodore~D. Drivas, Steve Shkoller, and Vlad Vicol,
  \emph{Simultaneous development of shocks and cusps for {2D Euler }with
  azimuthal symmetry from smooth data}, arXiv:2106.02143 (2021).

\bibitem{Buckmaster2020}
Tristan Buckmaster, Steve Shkoller, and Vlad Vicol, \emph{Formation of shocks
  for {2D} isentropic compressible {Euler}}, Communications on Pure and Applied
  Mathematics \textbf{75} (2020), no.~9, 2069--2120.

\bibitem{Buckmaster2022a}
\bysame, \emph{Formation of point shocks for {3D} compressible {Euler}},
  Communications on Pure and Applied Mathematics \textbf{76} (2022), no.~9,
  2073--2191.

\bibitem{Buckmaster2022}
\bysame, \emph{Shock formation and vorticity creation for {3D Euler}},
  Communications on Pure and Applied Mathematics \textbf{76} (2022), no.~9,
  1965--2072.

\bibitem{Christodoulou2007}
Demetrios Christodoulou, \emph{The formation of shocks in 3-dimensional
  fluids}, 2007, pp.~17--30.

\bibitem{Christodoulou2014}
Demetrios Christodoulou and Shuang Miao, \emph{Compressible flow and {Euler}'s
  equations (vol. 9 of the surveys of modern mathematics series)},
  International Press of Boston, Incorporated, 2014.

\bibitem{Courant2008}
R.~Courant and D.~Hilbert, \emph{Methods of mathematical physics}, 1. auflage
  ed., vol.~2, Wiley Classics Library, no.~20, Wiley-VCH, Weinheim, 2008.

\bibitem{Ellis2009}
George F.~R. Ellis, \emph{Republication of: Relativistic cosmology}, General
  Relativity and Gravitation \textbf{41} (2009), no.~3, 581--660.

\bibitem{Fajman2023}
David Fajman, Maximilian Ofner, Todd~A. Oliynyk, and Zoe Wyatt, \emph{The
  stability of relativistic fluids in linearly expanding cosmologies}, Int.
  Math. Res. Not. 2024 (2024), 4328-4383 \textbf{2024} (2023), no.~5,
  4328--4383.

\bibitem{Fajman2024}
David Fajman, Maximilian Ofner, and Zoe Wyatt, \emph{Relativistic fluids in
  cosmological spacetimes}, Classical and Quantum Gravity \textbf{41} (2024),
  no.~23, 233001.

\bibitem{Fajman2021}
David Fajman, Todd~A. Oliynyk, and Zoe Wyatt, \emph{Stabilizing relativistic
  fluids on spacetimes with non-accelerated expansion}, Commun. Math. Phys.
  383, 401-426 (2021) \textbf{383} (2021), no.~1, 401--426.

\bibitem{Fosalba1998}
Pablo Fosalba and Enrique Gazta{\~{n}}aga, \emph{Cosmological perturbation
  theory and the spherical collapse model {\textemdash} i. gaussian initial
  conditions}, Monthly Notices of the Royal Astronomical Society \textbf{301}
  (1998), no.~2, 503--523.

\bibitem{Ginsberg2024}
Daniel Ginsberg and Igor Rodnianski, \emph{The stability of irrotational shocks
  and the {Landau} law of decay}, arXiv:2403.13568 (2024).

\bibitem{Guo2018}
Yan Guo, Mahir Had{\v{z}}i{\'c}, and Juhi Jang, \emph{Continued gravitational
  collapse for {N}ewtonian stars}, Archive for Rational Mechanics and Analysis
  \textbf{239} (2021), 431--552.

\bibitem{Guo2022}
Yan Guo, Mahir Hadžić, Juhi Jang, and Matthew Schrecker, \emph{Gravitational
  collapse for polytropic gaseous stars: Self-similar solutions}, Archive for
  Rational Mechanics and Analysis \textbf{246} (2022), no.~2–3, 957--1066.

\bibitem{Hawking1966}
S.~W. Hawking, \emph{Perturbations of an expanding universe}, The Astrophysical
  Journal \textbf{145} (1966), 544.

\bibitem{Holzegel2016}
Gustav Holzegel, Sergiu Klainerman, Jared Speck, and Willie Wai-Yeung Wong,
  \emph{Small-data shock formation in solutions to {3D} quasilinear wave
  equations: An overview}, Journal of Hyperbolic Differential Equations
  \textbf{13} (2016), no.~01, 1--105.

\bibitem{Hoermander1997}
Lars H{\"o}rmander, \emph{Lectures on nonlinear hyperbolic differential
  equations}, Math{\'e}matiques et Applications, Springer Berlin Heidelberg,
  1997.

\bibitem{Hwang1999}
Jai-Chan Hwang and Hyerim Noh, \emph{Relativistic hydrodynamic cosmological
  perturbations}, General Relativity and Gravitation \textbf{31} (1999), no.~8,
  1131--1157.

\bibitem{Hwang2005}
Jai-chan Hwang and Hyerim Noh, \emph{Second-order perturbations of a
  zero-pressure cosmological medium: Proofs of the relativistic-newtonian
  correspondence}, Physical Review D \textbf{72} (2005), no.~4, 044011.

\bibitem{Hwang2006}
\bysame, \emph{Second-order perturbations of a zero-pressure cosmological
  medium: Comoving versus synchronous gauge}, Physical Review D \textbf{73}
  (2006), no.~4, 044021.

\bibitem{Hwang2007}
\bysame, \emph{Second-order perturbations of cosmological fluids: Relativistic
  effects of pressure, multicomponent, curvature, and rotation}, Physical
  Review D \textbf{76} (2007), no.~10, 103527.

\bibitem{Hwang2013a}
\bysame, \emph{Fully non-linear and exact perturbations of the friedmann world
  model}, Monthly Notices of the Royal Astronomical Society \textbf{433}
  (2013), no.~4, 3472--3497.

\bibitem{Jeans1902}
J.~H. Jeans, \emph{The stability of a spherical nebula}, Philos. Trans. R. Soc.
  Lond. A \textbf{199} (1902), 1--53.

\bibitem{John1984}
F.~John and S.~Klainerman, \emph{Almost global existence to nonlinear wave
  equations in three space dimensions}, Communications on Pure and Applied
  Mathematics \textbf{37} (1984), no.~4, 443--455.

\bibitem{Kichenassamy2021}
Satyanad Kichenassamy, \emph{Nonlinear wave equations}, {CRC} Press, may 2021.

\bibitem{Klainerman1984}
Sergiu Klainerman, \emph{Long time behaviour of solutions to nonlinear wave
  equations}, Proceedings of the InternationalCongress of Mathematicians, Vol.
  1, 2 (Warsaw, 1983) (1984), 1209--1215.

\bibitem{Lax2006}
Peter Lax, \emph{Hyperbolic partial differential equations}, American
  Mathematical Society, December 2006.

\bibitem{LeFloch2021}
Philippe~G. LeFloch and Changhua Wei, \emph{Nonlinear stability of
  self-gravitating irrotational {Chaplygin} fluids in a {FLRW} geometry},
  Annales de l'Institut Henri Poincar{\'{e}} C, Analyse non lin{\'{e}}aire
  \textbf{38} (2021), no.~3, 787--814.

\bibitem{Li2017}
Tatsien Li and Yi~Zhou, \emph{Nonlinear wave equations}, Springer Berlin
  Heidelberg, 2017.

\bibitem{Lifshitz1946}
E.~M. {Lifshitz}, \emph{{On the gravitational stability of the expanding
  universe}}, Zhurnal Eksperimentalnoi i Teoreticheskoi Fiziki \textbf{16}
  (1946), 587--602.

\bibitem{Liu2022b}
Chao Liu, \emph{Blowups for a class of second order nonlinear hyperbolic
  equations: A reduced model of nonlinear {Jeans} instability},
  arXiv:2208.06788 (2022).

\bibitem{Liu2023a}
\bysame, \emph{Fully nonlinear gravitational instabilities for expanding
  {Newtonian} universes with inhomogeneous pressure and entropy: beyond the
  {T}olman's solution}, Physical Review D \textbf{107} (2023), no.~12, 123534.

\bibitem{Liu2023}
\bysame, \emph{Fully nonlinear gravitational instabilities for expanding
  spherical symmetric newtonian universes with inhomogeneous density and
  pressure}, arXiv:2305.13211 (2023).

\bibitem{Liu2028c}
Chao Liu and Todd~A. Oliynyk, \emph{Cosmological newtonian limits on large
  spacetime scales}, Communications in Mathematical Physics \textbf{364}
  (2018), no.~3, 1195--1304.

\bibitem{Liu2018}
\bysame, \emph{Newtonian limits of isolated cosmological systems on long time
  scales}, Annales Henri Poincar{\'{e}} \textbf{19} (2018), no.~7, 2157--2243.

\bibitem{Liu2022a}
Chao Liu, Todd~A. Oliynyk, and Jinhua Wang, \emph{Future global existence and
  stability of de {S}itter-like solutions to the {Einstein-Yang-Mills}
  equations in spacetime dimensions $n\geq 4$}, Journal of the European
  Mathematical Society (accepted) (2022).

\bibitem{Liu2022}
Chao Liu and Yiqing Shi, \emph{Rigorous proof of the slightly nonlinear {Jeans}
  instability in the expanding {Newtonian} universe}, Physical Review D
  \textbf{105} (2022), no.~4, 043519.

\bibitem{Liu2018a}
Chao Liu and Changhua Wei, \emph{Future stability of the {FLRW} spacetime for a
  large class of perfect fluids}, Annales Henri Poincar{\'{e}} \textbf{22}
  (2021), 715–770.

\bibitem{Luk2018}
Jonathan Luk and Jared Speck, \emph{Shock formation in solutions to the {2D}
  compressible {Euler} equations in the presence of non-zero vorticity},
  Inventiones mathematicae \textbf{214} (2018), no.~1, 1--169.

\bibitem{Luk2021}
\bysame, \emph{The stability of simple plane-symmetric shock formation for {3D
  compressible Euler} flow with vorticity and entropy}, Analysis \& {PDE}
  \textbf{17} (2024), no.~3, 831--941.

\bibitem{Marshall2023}
Elliot Marshall and Todd~A. Oliynyk, \emph{On the stability of relativistic
  perfect fluids with linear equations of state $p=k\rho$ where $1/3<k<1$},
  Lett. Math. Phys. 113 (2023), 102 \textbf{113} (2023), no.~5.

\bibitem{Merle2022}
Frank Merle, Pierre Raphaël, Igor Rodnianski, and Jeremie Szeftel, \emph{On
  the implosion of a compressible fluid {I}: {Smooth} self-similar inviscid
  profiles}, Annals of Mathematics \textbf{196} (2022), no.~2.

\bibitem{Merle2022a}
Frank Merle, Pierre Raphaël, Igor Rodnianski, and Jeremie Szeftel, \emph{On
  the implosion of a compressible fluid {II}: {Singularity formation}}, Annals
  of Mathematics \textbf{196} (2022), no.~2.

\bibitem{Miao2018}
Shuang Miao, \emph{On the formation of shock for quasilinear wave equations
  with weak intensity pulse}, Annals of PDE \textbf{4} (2018), no.~1.

\bibitem{Miao2016}
Shuang Miao and Pin Yu, \emph{On the formation of shocks for quasilinear wave
  equations}, Inventiones mathematicae \textbf{207} (2016), no.~2, 697--831.

\bibitem{ViatcehslavMukhanov2013}
Viatcheslav Mukhanov, \emph{Physical foundations of cosmology}, Cambridge
  University Press, November 2013.

\bibitem{Neal2023}
Isaac Neal, Calum Rickard, Steve Shkoller, and Vlad Vicol, \emph{A new type of
  stable shock formation in gas dynamics}, Communications on Pure and Applied
  Analysis \textbf{0} (2023), no.~0, 0--0.

\bibitem{Neal2023a}
Isaac Neal, Steve Shkoller, and Vlad Vicol, \emph{A characteristics approach to
  shock formation in {2D Euler }with azimuthal symmetry and entropy},
  arXiv:2302.01289 (2023).

\bibitem{Noh2004}
Hyerim Noh and Jai-chan Hwang, \emph{Second-order perturbations of the
  friedmann world model}, Physical Review D \textbf{69} (2004), no.~10, 104011.

\bibitem{Noh2005}
\bysame, \emph{Relativistic-newtonian correspondence of the zero pressure but
  weakly nonlinear cosmology}, Classical and Quantum Gravity \textbf{22}
  (2005), no.~16, 3181--3188.

\bibitem{Oliynyk2016a}
Todd~A. Oliynyk, \emph{{Future stability of the FLRW fluid solutions in the
  presence of a positive cosmological constant}}, Communications in
  Mathematical Physics \textbf{346} (2016), 293--312.

\bibitem{Oliynyk2024}
Todd~A. Oliynyk, \emph{On the fractional density gradient blow-up conjecture of
  {Rendall}}, Communications in Mathematical Physics \textbf{405} (2024),
  no.~8.

\bibitem{Peebles2020}
P.~J.~E. Peebles, \emph{The large-scale structure of the universe}, Princeton
  University Press, November 1980.

\bibitem{Rendall2005}
Alan~D Rendall, \emph{Theorems on existence and global dynamics for the
  {Einstein equations}}, Living Reviews in Relativity \textbf{8} (2005), no.~1,
  6.

\bibitem{Scherrer2001}
Robert~J. Scherrer and Enrique Gazta{\~{n}}aga, \emph{The real- and
  redshift-space density distribution functions for large-scale structure in
  the spherical collapse approximation}, Monthly Notices of the Royal
  Astronomical Society \textbf{328} (2001), no.~1, 257--265.

\bibitem{Shkoller2024}
Steve Shkoller and Vlad Vicol, \emph{The geometry of maximal development and
  shock formation for the {Euler} equations in multiple space dimensions},
  Inventiones mathematicae (2024).

\bibitem{Speck2016b}
Jared Speck, \emph{Lecture notes on shock formation in quasilinear wave
  equations: An overview of the nearlyplane symmetric gime}, Unpublished notes
  (2016).

\bibitem{Speck2016a}
Jared Speck, \emph{Shock formation in small-data solutions to {3D} quasilinear
  wave equations}, 2016, Includes bibliographical references and index. -
  Electronic reproduction; Providence, Rhode Island; American Mathematical
  Society; 2016. - Description based on print version record.

\bibitem{Speck2017a}
\bysame, \emph{Shock formation for {2D} quasilinear wave systems featuring
  multiple speeds: Blowup for the fastest wave, with non-trivial interactions
  up to the singularity}, Annals of PDE \textbf{4} (2017), no.~1.

\bibitem{Speck2020}
Jared Speck, \emph{Stable {ODE}-type blowup for some quasilinear wave equations
  with derivative-quadratic nonlinearities}, Analysis \& {PDE} \textbf{13}
  (2020), no.~1, 93--146.

\bibitem{Speck2016}
Jared Speck, Gustav Holzegel, Jonathan Luk, and Willie Wong, \emph{Stable shock
  formation for nearly simple outgoing plane symmetric waves}, Annals of PDE
  \textbf{2} (2016), no.~2.

\bibitem{Weinberg1976}
Steven Weinberg, \emph{Gravitation and cosmology}, Wiley, New York [u.a.],
  1976.

\bibitem{Zeldovich1971}
Ya.~B. Zel'dovich and I.~D. Novikov, \emph{Relativistic astrophysics 2: The
  structure and evolution of the universe}, University of Chicago Press,
  Chicago, 1971.

\end{thebibliography}

\end{document}